\documentclass[10pt]{article}
\usepackage{setspace}
\usepackage{amsmath,amssymb,amsfonts,bbm,verbatim,multirow,color}
\usepackage{epic,eepic,psfrag,epsfig}
\usepackage[sf,bf,SF,footnotesize]{subfigure}
\usepackage{graphicx}
\usepackage{algorithm}
\usepackage{algpseudocode}
\usepackage{amsthm,breakcites}
\usepackage{amscd}
\usepackage{epsfig}
\usepackage{mathtools}
\usepackage{nicematrix}
\usepackage{tabularx}
\usepackage{graphicx}
\usepackage{array}
\usepackage{multirow}
\usepackage[title,toc,titletoc,page]{appendix}
\usepackage{rotating}
\usepackage{fullpage}

\usepackage[round]{natbib}
\usepackage{setspace}
\usepackage{enumerate}
\usepackage{array}
\usepackage{authblk}
\usepackage[small]{caption}
\RequirePackage[colorlinks,citecolor=blue,linkcolor=blue,urlcolor=blue,breaklinks]{hyperref}

\newtheorem{thm}{Theorem}

\newtheorem{lemma}[thm]{Lemma}

\newtheorem{prop}[thm]{Proposition}

\newtheorem{eg}{Example}

\newcommand*{\vertbar}{\rule[-1ex]{0.5pt}{2.5ex}}
\newcommand*{\horzbar}{\rule[.5ex]{2.5ex}{0.5pt}}

\DeclareMathOperator*{\tr}{tr}

\def\hat{\widehat}

\newcommand{\tb}[1]{\textcolor{blue}{(TB: #1)}}

\newcommand{\ab}[1]{\textcolor{orange}{(AB: #1)}}

\usepackage{xcolor}
\usepackage[draft,inline,nomargin,index]{fixme}
\fxsetup{theme=color,mode=multiuser}
\FXRegisterAuthor{tb}{atb}{\color{red} TB}
\FXRegisterAuthor{ab}{aab}{\color{green} AB}

\allowdisplaybreaks
\begin{document}
\openup .01em

\title{Tests of Missing Completely At Random based on sample covariance matrices}
\author{Alberto Bordino and Thomas B. Berrett}
\affil{Department of Statistics, University of Warwick}
\date{}

\maketitle
\begin{abstract}
     We study the problem of testing whether the missing values of a potentially high-dimensional dataset are Missing Completely at Random (MCAR). We relax the problem of testing MCAR to the problem of testing the compatibility of a collection of covariance matrices, motivated by the fact that this procedure is feasible when the dimension grows with the sample size. Our first contributions are to define a natural measure of the incompatibility of a collection of correlation matrices, which can be characterised as the optimal value of a Semi-definite Programming (SDP) problem, and to establish a key duality result allowing its practical computation and interpretation. By analysing the concentration properties of the natural plug-in estimator for this measure, we propose a novel hypothesis test, which is calibrated via a bootstrap procedure and demonstrates power against any distribution with incompatible covariance matrices. By considering key examples of missingness structures, we demonstrate that our procedures are minimax rate optimal in certain cases. We further validate our methodology with numerical simulations that provide evidence of validity and power, even when data are heavy tailed. Furthermore, tests of compatibility can be used to test the feasibility of positive semi-definite matrix completion problems with noisy observations, and thus our results may be of independent interest.
\end{abstract}

\section{Introduction}
\label{Sec:Intro}

Incomplete data are a common occurrence in almost all areas of statistical application, and the mechanisms leading to such data are diverse. For example, subjects in a survey may choose not to respond to certain questions, leading to missing values, or a practitioner may wish to combine data collected in different studies, where different variables were recorded in each. With incomplete data, traditional approaches become unreliable or even inapplicable, leading to a significant effect on the conclusions that can be drawn from the data. The most common approaches to dealing with missing values are to remove any incomplete observations, and thus to perform a complete-case analysis, or to replace any missing entry with a representative value, using an imputation method \citep[e.g.][]{yates33, van2011mice, buhlmann_impuation_2011}. However, the validity of such procedures, and the choice of an appropriate one, depends crucially on the mechanism that determines the missingness. Mechanisms have traditionally been classified as Missing Completely At Random (MCAR), Missing At Random (MAR) and Missing Not At Random (MNAR) \citep[e.g.][]{little2002statistical} according to the dependence structure between the variables themselves and their missingness, with such assumptions being required to link observations to targets of inference.

The typical formal setting is to suppose that we observe independent and identically distributed copies of a random object $X \circ \Omega$, where $X$ takes values in some product space $\mathcal{X} = \prod_{j=1}^d \mathcal{X}_j$, where $\Omega$ takes values in $\{0,1\}^d$ and where we define the operator $\circ$ by
\[(x \circ \omega)_j = 
\begin{cases}
x_j \quad \text{ if } \omega_j = 1, \\
\text{NA} \quad \text{ if } \omega_j = 0. \\
\end{cases}
\]
The assumptions named above then control the dependence between the data $X$ and the missingness indicator $\Omega$. The simplest case of MCAR is when these are independent, denoted $X \perp \!\!\! \perp \Omega$, so that the data we observe is representative of the population, even if it is incomplete\footnote{This scenario is sometimes referred to as \emph{everywhere MCAR}, which should not be confused with \emph{realised MCAR} (see \cite{seaman13} for a more detailed discussion on this distinction)}. Under MCAR we can often employ statistical methodologies that are easy to interpret and make good use of all incomplete data, with solid theoretical guarantees having been developed in various modern statistical problems such as high-dimensional regression \citep{loh2012high}, 
high-dimensional or sparse principal component analysis \citep{zhu2019, elsner2019}, classification \citep{cai2018high,sell2023nonparametric}, and precision matrix and changepoint estimation \citep{follain2021}. MCAR also allows for the use of a simple complete-case analysis which, in certain cases, such as when we have small sample sizes, can be preferable to complex procedures \citep[e.g.][]{milosevic23}.  However, if MCAR does not hold, which is common in practice, alternative methods may be required. 

Hypothesis tests can be used to guide practitioners in deciding whether or not missingness assumptions are reasonable. The goal of this work is to study the problem of testing the hypothesis of MCAR, which has been the subject of much research in the missing data literature. Most prior work has been developed within the context of parametric models. For example, \cite{little1988test} works under the hypothesis that the data are Gaussian in the setting that all pairs of variable are observed together (see Section \ref{sec:simul} for further details). \cite{fuchs1982maximum} considers discrete data in the setting that a large number of complete cases are available. In both cases the methods are likelihood ratio tests, with the MLEs calculated using the EM algorithm \citep{rubin77EM} and validity and power guarantees based on classical asymptotics. More recently, \citet{berrett2022optimal} provided a nonparametric formulation of the problem and methodology that was proved to be widely powerful under minimal assumptions. 
The key insight of~\citet{berrett2022optimal} is to relate the problem of testing MCAR to the problem of testing \emph{compatibility}, for which we now recall the definition. For $S \subseteq [d]:=\{1,\ldots,d\}$ denote by $\{\Omega=\mathbbm{1}_S\}$ the event that $X_j$ is observed if and only if $j \in S$, write $\mathbb{S}=\{S : \mathbb{P}(\Omega=\mathbbm{1}_S) >0\}$ for the set of all possible observation patterns and write $P_S$ for the distribution of the observation $X_S|\{\Omega = \mathbbm{1}_S \}$. We say that the collection $(P_S : S \in \mathbb{S})$ is compatible if there exists a distribution $P$ on $\mathcal{X}$ with marginal distribution $P_S$ on $\mathcal{X}_S$ for all $S \in \mathbb{S}$.
Under MCAR, the distribution $P_S$ is equal to the marginal distribution of the population distribution $ \mathcal{L}(X)$ on $\mathcal{X}_S := \prod_{j \in S} \mathcal{X}_j$, so it must be the case that $(P_S : S \in \mathbb{S})$ is compatible. Hence, if $P_{\mathbb{S}}:=(P_S : S \in \mathbb{S})$ is \emph{incompatible}, then the data cannot be MCAR. In fact, it is shown~\cite[][Proposition~1]{berrett2022optimal} that this reasoning is tight in that it is not possible to rule out MCAR based on observations of $X \circ \Omega$ if $P_\mathbb{S}$ is compatible. In general, fully testing the compatibility of a collection of distributions requires us to look at complex interactions between the distributions, and methods for doing so will have sample complexity that is exponential in the dimension $d$. Our work aims to provide methods that are valid and powerful without strong assumptions while being effective as the dimension grows.


Our methodology will be based on testing the compatibility of collections of covariance matrices, which can be estimated consistently even for large $d$. Earlier studies have employed the covariance matrix to assess MCAR. As briefly discussed above, \cite{little1988test} studied a likelihood ratio test of MCAR, effectively examining the homogeneity of means and covariances under the assumption of normality. However, Little expressed scepticism about its effectiveness unless the sample size is exceptionally large and the assumption of normality holds. This scepticism was further validated in simulations by \cite{kim2002tests}, who also developed a test for consistency of means and covariances based on generalised least squares. Both of these approaches work by comparing the sample covariance matrix associated to a given missingness pattern to the corresponding submatrix of an estimated complete covariance matrix. More recently, \cite{jamshidian2010tests} developed $k$-sample tests of the equality of covariance matrices, given complete data, based on Hawkins' test \citep{Hawkins1981}. Using empirical evidence, they then argued that these tests could be combined with imputation techniques to test the homogeneity of covariance matrices calculated using incomplete data. These methodologies can be effective when the corresponding assumptions are met and when a complete covariance matrix can be consistently estimated. 


Our method works by directly checking the compatibility of the observed sample covariance matrices, making no assumptions on the form of the underlying distributions and not requiring the estimation of a complete covariance matrix. In particular, this second point means that our test can be applied with any collection $\mathbb{S}$ of missingness patterns. More precisely, at the population level, we will consider $\Sigma_{\mathbb{S}} = (\Sigma_S:S \in \mathbb{S})$, a collection of suitably-normalised covariance matrices $\Sigma_S$ associated to the law of $X_S|\{\Omega = \mathbbm{1}_S\}$, and design a statistical test to check if $\Sigma_{\mathbb{S}}$ is compatible, meaning that each $\Sigma_S$ can be obtained by marginalising a general $d \times d$ positive-semi-definite matrix $\Sigma$, i.e.~$(\Sigma)_S = \Sigma_S$. If MCAR holds then for each $S \in \mathbb{S}$ we must have $\mathrm{Cov}(X_S |\Omega = \mathbbm{1}_S) = (\mathrm{Cov}(X))_S$, the block of the covariance matrix of $X$ corresponding to the variables in $S$, so that the collection $(\Sigma_S : S \in \mathbb{S})$ must be compatible. Hence, if we can reject the hypothesis $H_0: \Sigma_{\mathbb{S}} \textit{ compatible}$, then we can reject the hypothesis of MCAR. See Figure \ref{fig:methodology} for a pictorial summary of the key concepts so far. 

More generally, one can consider the problem of testing the compatibility of moments of order $p \geq 1$ and, if it is found that these moments are incompatible, one can reject MCAR. For $p = 1$, this problem reduces to testing the compatibility of mean vectors, which essentially boils down to testing the equality of means. This has been studied in the statistical literature for over a century, and we refer to existing methods for solving this problem \citep[e.g.][]{wilks1946means, little1988test}. In order to have power against a wider range of alternatives, while limiting the complexity of the testing procedure, we restrict attention in this work to the natural $p=2$ problem. Here there are still various ways in which compatibility can fail. For example, we can rule out $H_0$ if $\Sigma_{\mathbb{S}}$ is \emph{inconsistent}, in the sense that that there are two observation patterns $S_1, S_2 \in \mathbb{S}$ for which $(\Sigma_{S_1})_{S_1 \cap S_2} \neq (\Sigma_{S_2})_{S_1 \cap S_2}$, meaning that there exists a pair of variables whose covariance takes different values in different observation patterns. Testing the consistency of covariance matrices reduces to testing the equality of smaller covariance matrices, which has again been previously studied \citep[e.g.][]{Hawkins1981}. The corresponding nonparametric problem of testing the consistency of distributions was studied by \cite{li2015nonparametric,michel2021pklm} in the context of testing MCAR. However, it is possible to test more than consistency. As a concrete example, consider the case where $d=3$, where $\mathbb{S}=\{\{1,2\},\{1,3\},\{2,3\} \}$, and where
\[
	\Sigma_{ij} = \begin{pmatrix} 1 & \rho_{ij} \\ \rho_{ij} & 1 \end{pmatrix}
\]
with $\rho_{23}=\rho_{13}=-\rho_{12}=\rho$. Then $\Sigma_{\mathbb{S}}$ is compatible if and only if $\rho \leq 1/2$, even though it is always consistent. This is because the only matrix $\Sigma$ such that $(\Sigma)_S = \Sigma_S$ for all $S \in \mathbb{S}$ is \[
\Sigma = \begin{pmatrix}
    1 & -\rho & \rho \\
    -\rho & 1 & \rho \\
    \rho & \rho & 1
\end{pmatrix},
\]
whose eigenvalues are $1+\rho$ and $1-2\rho$, which therefore is not positive semi-definite if $\rho > 1/2$. 
\begin{figure}
    \centering
    \includegraphics[scale=0.4]{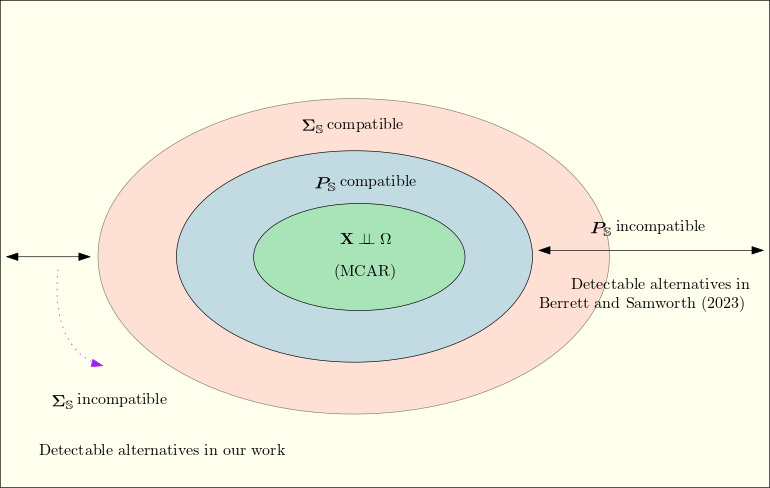}
    \caption{Our framework. We relax the methodology in \cite{berrett2022optimal} and consider (suitably normalised) covariance matrices instead of full distributions. The price we pay is to create an extra ring (red area) in which we cannot detect departure from $H_0$ just by looking at $\Sigma_\mathbb{S}$. For example, if the collection of third-moment tensors were inconsistent, but $\Sigma_\mathbb{S}$ were compatible, we would not be able to reject MCAR, although $P_\mathbb{S}$ would be incompatible.}
    \label{fig:methodology}
\end{figure}

The above example is relatively simple because any pair of variables is observed together so that the full covariance matrix can be estimated. However, we can characterise compatibility for any $\mathbb{S}$ (see Proposition \ref{Prop:ConsCompatCheck}). While the compatibility of distributions can be characterised using linear programming \citep[e.g.][]{kellerer1984duality}, characterising the compatibility of covariance matrices requires ideas from semi-definite programming (SDP), which studies linear optimisation problems over spectrahedra \citep[e.g.][]{blekherman2012sdp, boyd_vandeberghe}. If $\Sigma_{\mathbb{S}}$ is consistent, compatibility is equivalent to the feasibility of a positive semi-definite matrix completion problem, where we observe a partial symmetric matrix $A = (a_{ij})$ for positions $(i, j)$ in a certain set of edges, and aim to construct a positive semi-definite completion of $A$. This problem is extensively studied owing to its widespread applications in diverse fields such as probability, statistics, systems engineering and geophysics; see, for example, \cite{laurent2009matrix} and the references therein for an introduction to the topic. Statistical questions associated with such problems are relatively under-explored, though we mention the recent work~\cite{waghmare2022completion} that provides estimated completions of covariance operators in settings where completions always exist. A distinct but related problem that has received more attention in the statistics literature is low-rank matrix completion. In particular, significant contributions \citep{candes_recht_2009, candes2010completion, recht2011} have been made in the realm of convex optimization, where a low-rank matrix is recovered from partial observations after introducing a nuclear norm penalty. In our work we make no low-rank assumptions and our main interest is in answering the question of whether or not a positive semi-definite completion exists.

We now briefly outline our main contributions. In Section \ref{sec:bootstrap_test}, we introduce a data-driven bootstrap test for MCAR, detailed in Algorithm \ref{alg:Corrbootstrap_test}. We also state our main result (Theorem \ref{thm:bootstrap_guarantees}), that this procedure is uniformly valid over expanding subsets of the null hypothesis, excluding distributions close to the boundary, and uniformly powerful against alternatives separated from the null hypothesis. Additionally, we establish its asymptotic validity over the entire null hypothesis in a specific example (Proposition \ref{prop:cycle_boundary}). This test is based on a numerical measure, $R(\Sigma_\mathbb{S})$, which quantifies the incompatibility of a collection of correlation matrices $\Sigma_\mathbb{S}$. We define this measure in Section \ref{Sec:cov_compatibility} and establish key properties, including a useful and interpretable dual representation (Proposition~\ref{Prop:CorrDuality}). In Section \ref{Sec:oracle_test}, we shift our focus to the empirical estimation of this index and analyse its concentration properties. Assuming non-singularity of the correlation matrices, $\Sigma_S \succeq c I_{|S|}$ for all $S \in \mathbb{S}$, we present an oracle test based on knowledge of $c > 0$, and we state a result on its validity and power (Theorem~\ref{thm:oracle_test}). This result is crucial for the proof of Theorem \ref{thm:bootstrap_guarantees}. In Section \ref{sec:minimax_LB} we analyse the performance of our oracle test in various examples and show that its separation rate is near-minimax optimal in some cases, while studying properties of the associated semi-definite programmes. In Section \ref{sec:simul} we validate our methodology in numerical experiments. Section \ref{sec:proof} contains the proofs of our main results. The Appendix contains background and auxiliary results.

We conclude the introduction with some notation that is used throughout the paper. In general, we will denote covariance matrices by $\Omega$ and correlation matrices (or other suitably normalised covariance matrices) by $\Sigma$. For $d \in \mathbb{N}$, we write $[d] := \{1, \cdots, d\}$, and indicate with $|B|$ the cardinality of the set $B$. Given $a, b \geq 0$, we write $a \lesssim b$
to mean that there exists a universal constant $C > 0$ such that $a \leq Cb$. 
We use $a \wedge b$ for $\min\{a,b\}$, and $a \vee b$ for $\max\{a,b \}$. We will denote with $\boldsymbol{0}_d$ the null vector of dimension $d$, with $\boldsymbol{1}_d$ the all-one vector, with $\boldsymbol{e}_j$ the $j$-th element of the canonical basis of $\mathbb{R}^d$, with $\boldsymbol{O}_{d_1, d_2}$ the zero matrix of dimension $d_1 \times d_2$, with $\boldsymbol{O}_{d} := \boldsymbol{O}_{d,d}$, and with $I_d$ the identity matrix of dimension $d$. We will omit the subscript with the dimension $d$ when it is clear from the context. For symmetric matrices $A, B$ of dimension $d$, we write $A \succeq 0$ to mean that $A$ is positive semi-definite, write $A \succeq B$ to mean that $A - B \succeq 0$, write $\operatorname{diag}(A)$ to indicate the vector whose elements coincide with the diagonal entries of the matrix $A$, and $\operatorname{diag}(\boldsymbol{v})$ for a vector $\boldsymbol{v}=(v_1,\ldots,v_d)$ to indicate a diagonal matrix with diagonal elements equal to $v_i$. We will indicate the trace of $A$ with $\operatorname{tr}(A)$, the determinant with either $|A|$ or $\operatorname{det}(A)$, and the minimum and maximum eigenvalues of $A$ with $\lambda_{\mathrm{max}}(A)$ and $\lambda_{\mathrm{min}}(A)$, respectively. We use $\| \cdot \|_p$ for the $l_p$-norm of a vector. We will use  $\|\cdot\|_*$ for nuclear norm, or Schatten-1 norm of a matrix, $\|\cdot\|_2$ for the spectral norm, and $\|\cdot\|_F$ for the Frobenius norm. For random elements $X, Y$, we write $X \perp \!\!\! \perp Y$ to mean that $X$ and $Y$ are independent. For $\sigma>0$, a random variable $X$ with mean $\mu=\mathbb{E}[X]$ is said to be $\sigma$-subgaussian if
\[
\mathbb{E}\left[e^{\lambda(X-\mu)}\right] \leq e^{\sigma^2 \lambda^2 / 2} \quad \text { for all } \lambda \in \mathbb{R},
\]
while, for $(\nu,\alpha) \in (0,\infty)^2$, it is said to be $(\nu,\alpha)$-subexponential if
$$
\mathbb{E}\left[e^{\lambda(X-\mu)}\right] \leq e^{\frac{\nu^2 \lambda^2}{2}} \quad \text { for all }|\lambda|<\frac{1}{\alpha}.
$$
A random vector $X$ in $\mathbb{R}^n$ is said to be $\sigma$-subgaussian if every one-dimensional projection, i.e. $v^TX$ with $v \in \mathbb{R}^n$ and $\|v\|=1$, is $\sigma$-subgaussian in the sense defined above.

\section{Statistical setting and bootstrap test of MCAR}\label{sec:bootstrap_test}
We aim to test $X \perp \!\!\! \perp \Omega$ by examining the collection of suitably normalised covariance matrices across different missingness patterns, and we will see in the sequel (e.g.~Section~\ref{sec:minimax_LB}) that the structure of these patterns has a strong influence of the complexity of the problem. Instead of assuming that we have access to i.i.d. copies from $X \circ \Omega$, it is convenient to condition on the realisations of $\Omega$, so that the collection of patterns $\mathbb{S}$ and the samples sizes $(n_S : S \in \mathbb{S})$ are fixed.
This does not impose a significant constraint, and all the theoretical results can be adapted to the unconditional model. We then assume that for each $S \in \mathbb{S}$ we have access to an independent sample
\[
    X_{S,1},\ldots,X_{S,n_S} \overset{\text{i.i.d.}}{\sim} P_S
\]
for some sample size $n_S$ and some distribution $P_S$ with mean $\mu_S$, correlation matrix $\Sigma_S$, vector of variances $\sigma_S^2$, and covariance matrix $\Omega_{S} = \operatorname{diag}^{1/2}( \sigma_{S}^2) \cdot \Sigma_{S} \cdot \operatorname{diag}^{1/2}(\sigma_{S}^2)$. We write $\hat{\mu}_S$, $\hat{\Sigma}_S$, $\hat{\sigma}_S^2$ and $\hat{\Omega}_S$ for the sample mean, the sample correlation matrix, the vector of sample variances and the sample covariance matrix, respectively, of $X_{S,1},\ldots,X_{S,n_S}$ for each $S \in \mathbb{S}$. Additionally, we write $\hat{\mu}_\mathbb{S} = (\hat{\mu}_S: S \in \mathbb{S})$, $\hat{\Sigma}_\mathbb{S} = (\hat{\Sigma}_S : S \in \mathbb{S})$, $\hat{\sigma}_\mathbb{S}^2 = (\hat{\sigma}_S^2 : S \in \mathbb{S})$ and $\hat{\Omega}_\mathbb{S} = (\hat{\Omega}_S : S \in \mathbb{S})$ for the collections of these estimators.

We now propose a bootstrap test of MCAR that can be applied without any knowledge of unknown parameters. This will be based on the incompatibility index $R(\cdot) \in [0,1]$ defined in Section \ref{Sec:cov_compatibility}, which acts on correlation matrices and characterises compatibility, in the sense that $R(\Sigma_\mathbb{S}) = 0$ if and only if $\Sigma_\mathbb{S}$ is compatible, and its regularised version $R_z(\cdot)$ defined in Equation \eqref{eq:reg_R}. Algorithm \ref{alg:Corrbootstrap_test}, which is implemented the R-package \texttt{MCARtest} \citep{berrett2022MCARtest}, tackles the testing problem $H_0: R(\Sigma_\mathbb{S}) = 0$ using as test statistic the plug-in estimate $R(\hat{\Sigma}_\mathbb{S})$, and is calibrated through a bootstrap procedure. As already outlined in Section \ref{Sec:Intro}, rejection of $H_0$ is sufficient to reject MCAR.

\begin{algorithm}[!htbp]
\caption{MCAR bootstrap test checking compatibility of correlation matrices}\label{alg:Corrbootstrap_test}
\begin{algorithmic}[1]
\State Given data $X_\mathbb{S}$, discard all patterns $S \in \mathbb{S}$ such that $n_S \leq |S|+1$.
\State Compute $\hat{\Sigma}_{\mathbb{S}} = \operatorname{SampleCorr}X_\mathbb{S}$ and $\hat{c} = \min_{S \in \mathbb{S}} \lambda_{\mathrm{min}}(\hat{\Sigma}_S)$.
\State Compute $R(\hat{\Sigma}_\mathbb{S})$ and the dual decomposition $\hat{\Sigma}_\mathbb{S} = (1-R(\hat{\Sigma}_\mathbb{S}))\hat{Q}_\mathbb{S} + R(\hat{\Sigma}_\mathbb{S})\hat{\Sigma}'_\mathbb{S}$.
\If{$R(\hat{\Sigma}_\mathbb{S}) \geq 3/4$}
    \State \Return  $p_R = 0$.
\Else{}
    \State For all $S \in \mathbb{S}$, for all $i \in [n_S]$ \textbf{do} $\tilde{X}_{S,i} := \hat{Q}_S^{1/2}\hat{\Sigma}_S^{-1/2}\mathrm{diag}^{-1/2}(\hat{\sigma}_S^2) (X_{S,i} - \hat{\mu}_S)$.
    \For{$b \in [B]$} 
        \State For all $S \in \mathbb{S}$, let $\tilde{X}_{S}^{(b)} = (\tilde{X}_{S,i}^{(b)} : i \in [n_S])$ be a nonparametric bootstrap sample from $\tilde{X}_{S}$. 
        \State Compute $\hat{\Sigma}_{\mathbb{S},b} = \operatorname{SampleCorr}\tilde{X}_\mathbb{S}^{(b)}$.
        \State Compute $R_{\hat{c}/2}(\hat{\Sigma}_{\mathbb{S},b})$.
\EndFor
\State \Return  $p_R := (1+B)^{-1}(1 + \sum_{b=1}^B \mathbbm{1}\{R_{\hat{c}/2}(\hat{\Sigma}_{\mathbb{S},b}) \geq R(\hat{\Sigma}_\mathbb{S})\})$.
\EndIf
\end{algorithmic}
\end{algorithm}

The intuition behind the procedure is as follows. From Proposition \ref{Prop:CorrDuality} in Section \ref{Sec:cov_compatibility} we can write $\hat{\Sigma}_\mathbb{S} = (1-R(\hat{\Sigma}_\mathbb{S}))\hat{Q}_\mathbb{S} + R(\hat{\Sigma}_\mathbb{S})\hat{\Sigma}'_\mathbb{S}$, where $\hat{Q}_\mathbb{S}$ can be thought as the closest compatible sequence of correlation matrices to $\hat{\Sigma}_\mathbb{S}$, and can be computed at the same time as the test statistic $R(\hat{\Sigma}_\mathbb{S})$, and $\hat{\Sigma}_\mathbb{S}'$ is an arbitrary sequence of correlation matrices. If $R(\hat{\Sigma}_\mathbb{S}) \geq 3/4$, we reject the null hypothesis outright, as there is strong evidence against it. Otherwise, we proceed by calibrating our test using a bootstrap procedure. We transform the original data by calculating $\tilde{X}_{S} := \hat{Q}_S^{1/2}\hat{\Sigma}_S^{-1/2}\mathrm{diag}^{-1/2}(\hat{\sigma}_S^2) (X_S - \hat{\mu}_S)$ for all $S \in \mathbb{S}$. This transformation means that the sample correlation matrices of $\tilde{X}_\mathbb{S}:=(\tilde{X}_S : S \in \mathbb{S})$ are given by $\hat{Q}_\mathbb{S}$, which is compatible. Fixing $B \in \mathbb{N}_+$, for each $b \in [B]$ and $S \in \mathbb{S}$ we generate $\tilde{X}_{S}^{(b)}$ as a nonparametric bootstrap sample from $\tilde{X}_{S}$ and calculate the sample correlation matrix $\hat{\Sigma}_{S,b} = \operatorname{SampleCorr}(\tilde{X}_{S}^{(b)})$. Then, for each $b \in [B]$ we compute the corresponding test statistic $R_{\hat{c}/2}(\hat{\Sigma}_{S,b} : S \in \mathbb{S})$. These will serve as surrogates of the null, and will be employed to generate a $p$-value of the form $p_R := (1+B)^{-1}(1 + \sum_{b=1}^B \mathbbm{1}\{R_{\hat{c}/2}(\hat{\Sigma}_{\mathbb{S},b}) \geq R(\hat{\Sigma}_\mathbb{S})\})$. The reason why we are forced to use the regularised version $R_{\hat{c}/2}(\cdot)$ at the bootstrap level instead of the standard $R(\cdot)$, and why we treat the case $R(\hat{\Sigma}_\mathbb{S}) \geq 3/4$ separately, is that it is notably challenging to characterise the spectrum of $\hat{Q}_\mathbb{S}$. This approach might inflate the size of our test in general. However, this inflation is vanishingly small since $R_{\hat{c}/2}(\hat{\Sigma}_{\mathbb{S},1}) \mid X_\mathbb{S}  \overset{\mathbb{P}}{\rightarrow} R(\hat{\Sigma}_{\mathbb{S},1}) \mid X_\mathbb{S}$, as demonstrated, for instance, in the proof of Proposition \ref{prop:cycle_boundary}.

The theoretical properties of Algorithm \ref{alg:Corrbootstrap_test} are based on analysing the concentration properties of the plug-in estimator  $R(\hat{\Sigma}_\mathbb{S})$ and the bootstrap statistics $R_{\hat{c}/2}(\hat{\Sigma}_{\mathbb{S},b})$, which require several technical insights. These are developed in Section~\ref{Sec:oracle_test}.  In light of the results presented therein, we are able to prove that our bootstrap test is uniformly valid over expanding subsets of the null, and uniformly powerful over alternatives separated from the null. To this end, let
\[
\bar{\mathcal{P}}_{\mathbb{S}}(0) := \{P_\mathbb{S} \in \mathcal{P}_\mathbb{S} : R(\Sigma_\mathbb{S}) = 0\}.
\]
For $P_\mathbb{S} \in \mathcal{P}_{\mathbb{S}}$ and $\epsilon > 0$, define 
\[
B_\epsilon(P_\mathbb{S}) = \{P'_\mathbb{S} \in  \mathcal{P}_{\mathbb{S}} : \max_{S \in \mathbb{S}}\|\Sigma_S - \Sigma^\prime_S \|_2 \leq \epsilon \} \quad \text{ and } \quad \bar{\mathcal{P}}_{\mathbb{S}}(0)^{-\epsilon} = \{P_\mathbb{S} \in \bar{\mathcal{P}}_{\mathbb{S}}(0) : B_\epsilon(P_\mathbb{S}) \subseteq \bar{\mathcal{P}}_{\mathbb{S}}(0)\}. 
\]

\begin{thm}\label{thm:bootstrap_guarantees}
    Suppose we observe $X_{S,1}, \ldots, X_{S, n_S} \overset{\text{i.i.d.}}{\sim} P_S$ for each $S \in \mathbb{S}$ independently, where each $P_S$ is $\nu$-subgaussian with mean $\mu_S$. Denote by $\Sigma_\mathbb{S}$ the collection of population correlation matrices and by $\Omega_\mathbb{S}$ the collection of population covariance matrices, and assume that $\Sigma_{\mathbb{S}} \succeq_\mathbb{S} cI_\mathbb{S}$ for $c >0$. 
    Let $\sigma_\mathrm{min}^2 := \min\limits_{j \in [d]} \min\limits_{S \in \mathbb{S}_j}(\Omega_S)_{jj}$ and define 
\begin{align}\label{eq:critical_oracle}
        C_{\alpha, c} = \frac{K_1 \nu^4}{c \sigma^4_{\mathrm{min}}} \max_{S \in \mathbb{S}} \sqrt{\frac{|S| + \log(|\mathbb{S}|/\alpha)}{n_S}}. 
    \end{align}
\begin{enumerate}
    \item[(i)] Then, for a universal constant $K_1 >0$ chosen sufficiently large, for all $\alpha \in (0,1)$, if $ C_{\alpha, 2} \leq 1$,  we have
    \[
    \sup_{P_\mathbb{S} \in \bar{\mathcal{P}}_{\mathbb{S}}(0)^{-C_{\alpha, 2}}} \mathbb{P}_{P_\mathbb{S}}\left\{ p_R \leq \alpha \right\} \leq \alpha.
    \]
    \item[(ii)] Moreover, if $\alpha \in (0,1)$ and $\beta \in (0, 1 -\alpha)$ are such that $B \geq 2(1-\alpha)/\alpha$ and $\tilde{C}_{\alpha, \beta} \leq 1/2$, and if $P_\mathbb{S} \in \mathcal{P}_\mathbb{S}$ is such that $R(\Sigma_\mathbb{S}) > 2\tilde{C}_{\alpha, \beta}$, then we have $\mathbb{P}_{P_\mathbb{S}}\left\{p_R > \alpha \right\} \leq \beta$, where 
\begin{align*}
    \tilde{C}_{\alpha, \beta} & = \max\left( \max_{S \in \mathbb{S}}\frac{K_2 \nu^5}{\sigma_{\mathrm{min}}^5 c^{3/2}}\sqrt{\frac{\log(n_S) \log(|\mathbb{S}|/\alpha\beta)\{|S| + \log(|\mathbb{S}|n_S/\alpha\beta)  \}}{n_S}}, \right. \\
    & \hspace{5cm} \left. \max_{S \in \mathbb{S}}\frac{K_3 \nu^6}{\sigma_{\mathrm{min}}^6 c^2}\frac{\{|S| + \log(|\mathbb{S}|/\alpha\beta)\} \log^2(|\mathbb{S}|n_S/\alpha\beta)}{n_S}, C_{\alpha \beta, c} \right),
\end{align*}
and $K_2, K_3 > 0$ are sufficiently large universal constants.
\end{enumerate}
      
\end{thm}
First, observe that the initial part of the result implies asymptotic validity across the entire null hypothesis, except at its boundary. Additionally, compared to the oracle test given in Theorem \ref{thm:oracle_test} in Section~\ref{Sec:oracle_test}, the separation required to achieve uniform power includes just an extra logarithmic term in $n_S$ and $|\mathbb{S}|$, being of the order
\[
\tilde{C}_{\alpha, \beta} \lesssim \max_{S \in \mathbb{S}} \sqrt{ \frac{|S| \log n_S \log(|\mathbb{S}|/\alpha\beta)}{n_S}},
\]
when $c$, $\sigma^2_\mathrm{min}$ and $\nu^2$ are fixed. This, combined with Theorems \ref{thm:cycle_minimax} and \ref{block_minimax}, also demonstrates that our bootstrap test is essentially minimax optimal up to logarithmic factors when $\mathbb{S}$ is either a $d$-cycle or a block $3$-cycle. 

Regarding the behaviour at the boundary of the null hypothesis, it is important to note that classical bootstrap procedures can be inconsistent for certain parameter values on the boundary of the associated parameter spaces \citep{andrews2000bootstrap, samworth03restoring, cavaliere17bootstrap}.  While alternative approaches like the $m \text{ out of } n$ bootstrap can restore consistency in these cases, they may still be outperformed by the inconsistent standard bootstrap in practice. For example, this phenomenon is highlighted in \cite{samworth03restoring} through an analysis of the parametric bootstrap in the context of the Hodges and Stein estimators. As for the bootstrap test given in Algorithm \ref{alg:Corrbootstrap_test}, its asymptotic validity on the boundary crucially depends on the geometrical properties of the null hypothesis, which are in turn induced by the structure of the collection of missingness patterns $\mathbb{S}$. While this is generally very complex, we can still demonstrate the asymptotic validity of the bootstrap test over the entire null hypothesis in certain specific cases, such as for the $d$-cycle.  

To this end, recall that for bivariate random vectors $(X_1, Y_1), \ldots, (X_n, Y_n) \overset{\mathrm{i.i.d.}}{\sim} P$ with mean $(\xi, \eta)$, variances $(\sigma^2, \tau^2)$, and finite fourth moments, we have that $\sqrt{n}(\hat{\rho} - \rho) \overset{d}{\rightarrow} N(0, \gamma^2(P))$ \citep[Example 5.4.3]{lehmann99largesample}, where $\rho$ and $\hat{\rho}$ are Pearson correlation and its sample version, respectively. Here, the definition of $\gamma^2(P)$ is as follows: let $(S)_{ij} = \sigma_{ij}$ with
\begin{align}\label{eq:def_gammaPearson}
\sigma_{11} & =\operatorname{Var}[(X-\xi)(Y-\eta)]=E\left[(X-\xi)^2(Y-\eta)^2-\rho^2 \sigma^2 \tau^2\right], \nonumber \\
\sigma_{12} & =\sigma_{21}=\operatorname{Cov}\left[(X-\xi)(Y-\eta),(X-\xi)^2\right]=E\left[(X-\xi)^3(Y-\eta)-\rho \sigma^3 \tau\right],\nonumber \\
\sigma_{22} & =\operatorname{Var}(X-\xi)^2=E(X-\xi)^4-\sigma^4, \nonumber \\
\sigma_{13} & =\sigma_{31}=\operatorname{Cov}\left[(X-\xi)(Y-\eta),(Y-\eta)^2\right]=E\left[(X-\xi)(Y-\eta)^3\right]-\rho \sigma \tau^3, \nonumber \nonumber \\
\sigma_{23} & =\sigma_{32}=\operatorname{Cov}\left[(X-\xi)^2,(Y-\eta)^2\right]=E\left[(X-\xi)^2(Y-\eta)^2\right]-\sigma^2 \tau^2, \nonumber \\
\sigma_{33} & =\operatorname{Var}(Y-\eta)^2=E(Y-\eta)^4-\tau^4 .
\end{align}
Let $f(u,v,w) = u/\sqrt{v w}$, and define $\gamma^2(P) = a^T S a$, where $a = \left( \partial f / \partial u, \; \partial f / \partial v, \; \partial f / \partial w \right)
$ evaluated at $u=\rho \sigma \tau, v=\sigma^2, w=\tau^2$. 

\begin{prop}\label{prop:cycle_boundary}
    Consider a $d$-cycle, i.e.~$\mathbb{S}=\{\{1,2\},\{2,3\},\ldots,\{d,1\}\}$, and write $\Sigma_{\mathbb{S}}:= (\Sigma_{\{1,2\}}, \cdots, \Sigma_{ \{d,1\}})$ for the collection of $2 \times 2$ correlation matrices with
\[
\Sigma_{j,j+1} = \begin{pmatrix} 1 & \cos \theta_{j} \\ \cos \theta_{j} & 1 \end{pmatrix},
\]
where $\theta_j \in [0, \pi]$ for all $j \in [d]$. Let $P_{\{j, j+1\}}$ be the distribution of $(X_j, X_{j+1}) \mid \Omega = \mathbbm{1}_{\{j, j+1\}}$, and suppose  $P_{\{j, j+1\}}$ has finite fourth moments for all $j \in [d]$.  Assume that $\min\{1+\cos\theta_j , 1 - \cos\theta_j \} \geq c$ for all $j \in [d]$ for a universal constant $c \in (0,1)$, and suppose there exists $j \in [d]$ such that $\gamma^2(P_{\{j, j + 1 \}}) > 0$. For all $P_\mathbb{S} \in \bar{\mathcal{P}}_{\mathbb{S}}(0)$ and $\alpha < 1/2$, we have
    \[
    \lim_{n_\mathbb{S} \rightarrow \infty} \lim_{B \rightarrow \infty} \mathbb{P}_{P_\mathbb{S}}\left\{p_R \leq \alpha \right\} \leq \alpha.
    \]
    Furthermore, if $P_\mathbb{S} \in \partial \bar{\mathcal{P}}_{\mathbb{S}}(0):= \bar{\mathcal{P}}_{\mathbb{S}}(0) \setminus \cup_{\epsilon > 0} \bar{\mathcal{P}}_{\mathbb{S}}(0)^{-\epsilon}$, the probability of rejection tends exactly to $\alpha$.
\end{prop}

The condition $\gamma^2(P_{{j, j + 1 }}) > 0$ ensures that $R(\hat{\Sigma}_\mathbb{S})$ and $R_{\hat{c}/2}(\hat{\Sigma}_{\mathbb{S},b})$ do not have a degenerate limit, and holds for all sufficiently regular distributions. For instance, in the case of elliptical distributions, we have $\gamma^2(P) = (1+\kappa)(1-\rho^2)^2 > 0$ (see Theorem 5.1.6 in \cite{Muirhead1982AspectsOM}), since $\rho^2 < 1$ by assumption, and the kurtosis $\kappa$ is guaranteed to satisfy $\kappa > -1$ \citep{kurtosisElliptical86}.

\section{Measure of incompatibility for covariance matrices}
\label{Sec:cov_compatibility}

In this section we develop our index $R(\cdot)$ of the incompatibility of population covariance matrices, which will be defined as the optimal value of a semi-definite programme. Standardising the covariance matrices is necessary to have a well-posed problem, and we choose to work with correlation matrices because this is most natural from a statistical point of view. Other standardisation are possible, though, each leading to a different measure of incompatibility, with different properties. In Appendix \ref{sec:test_cov} we introduce another measure of incompatibility based on a different standardisation, analyse its properties, and derive an oracle test of MCAR based on its estimation from data. In order to define and study the properties of $R(\cdot)$, we must first introduce some basic algebraic objects for collections of symmetric and positive semi-definite matrices. Our key notation is collected in Table~\ref{Table:Notation} below. \\

\begin{table}[!ht]
    \centering
    \setlength\extrarowheight{2pt}
   \begin{tabularx}{\textwidth}{|c|X|X|}
        \hline
        \textbf{Notation} & \textbf{Definition} & \textbf{Meaning}\\
        \hline
        $\mathbb{S}$ & A subset of the power set of $[d]$ & Set of all missingness patterns \\
        $\mathbb{S}_j$ & $\{S \in \mathbb{S} : j \in S\}$ & Set of all patterns that contain $j$ \\
        $\mathbb{S}_{jj'}$ & $\{S \in \mathbb{S} : j, j' \in S\}$ & Set of all patterns that contain $(j,j')$ \\
         $\mathcal{M} \equiv \mathcal{M}_d$  & $\{ X \in \mathbb{R}^{d \times d} : X = X^T\}$  & Space of symmetric matrices \\
         $\mathcal{M}_\mathbb{S}$ & $\{ (X_S : S \in \mathbb{S}) : X_S \in \mathcal{M}_{|S|} \text{ for all } S \in \mathbb{S}\}$ & Space of collections of symmetric matrices \\
         $\langle X,Y \rangle$  & $ \mathrm{tr}(X Y)$ for $X,Y \in \mathcal{M}$ & Frobenius inner product \\
          $\langle X_\mathbb{S}, Y_\mathbb{S} \rangle_\mathbb{S}$ & $\sum_{S \in \mathbb{S}} \mathrm{tr}(X_S Y_S)$ for $X_\mathbb{S}, Y_\mathbb{S} \in \mathcal{M}_\mathbb{S}$ & Sum of Frobenius inner products \\
         $\mathcal{P}^*$ & $\{ \Sigma \in \mathcal{M} : \Sigma \succeq 0 \}$ & Cone of PSD matrices \\
        $\mathcal{P}$ & $ \{ \Sigma \in \mathcal{P}^* : \mathrm{diag}(\Sigma)=I_d \}$ & Set of correlation matrices \\
        $\Sigma_\mathbb{S} \succeq_\mathbb{S} 0$ & $\Sigma_\mathbb{S} \succeq_\mathbb{S} 0$ if and only if $\Sigma_S \succeq 0$ for all $S \in \mathbb{S}$ & Loewner order for collections of matrices \\
         $\mathcal{P}_\mathbb{S}^*$ & $\{ \Sigma_\mathbb{S} \in \mathcal{M}_\mathbb{S} : \Sigma_\mathbb{S} \succeq_\mathbb{S} 0 \}$ & Set of collections of PSD matrices \\
         $\mathcal{P}_\mathbb{S}$ & $\{ \Sigma_\mathbb{S} \in \mathcal{P}_\mathbb{S}^* : \mathrm{diag}(\Sigma_S)=I_{|S|} \, \text{for all } S \in \mathbb{S} \}$ & Set of collections of correlation matrices \\
         $A$ & $A : \mathcal{M} \rightarrow \mathcal{M}_\mathbb{S}$ with $ (A X)_S = (X_{jj'})_{j,j' \in S}$ & Marginalisation operator on matrices \\
         $\mathcal{P}_\mathbb{S}^{0,*}$ & $ \{ A \Sigma : \Sigma \in \mathcal{P}^*\}$ &  Set of compatible collections of PSD matrices \\
         $\mathcal{P}_\mathbb{S}^{0}$ & $\{ A \Sigma : \Sigma \in \mathcal{P}\}$ &  Set of compatible collections of correlation matrices \\
          $\mathcal{Y}$ & $\{\operatorname{diag}(v) : v \in \mathbb{R}^d \text{ and } \sum_{j = 1}^d v_j = 0\}$ & Space of diagonal matrices with null trace \\
        $\boldsymbol{O}_\mathbb{S}$ & $(\boldsymbol{O}_{|S|} : S \in \mathbb{S})$ & Collection of zero matrices \\
        $I_{\mathbb{S}}$ &  $(I_{|S|} : S \in \mathbb{S})$ & Collection of identity matrices \\
        \hline
    \end{tabularx}
    \caption{Table of definitions commonly used in the main text.}
    \label{Table:Notation}
\end{table}

Crucially, we say that an element of $\mathcal{P}_\mathbb{S}^*$ is \emph{compatible} if and only if it is an element of $\mathcal{P}_{\mathbb{S}}^{0,*}$.
It turns out that we can characterise compatibility through the adjoint of the linear operator $A$ defined in Table~\ref{Table:Notation} above, which maps a symmetric matrix $X$ into a collection of symmetric matrices indexed by $\mathbb{S}$ according to $ (A X)_S = (X_{jj'})_{j,j' \in S}$ for all $S \in \mathbb{S}$.
\begin{prop}\label{Prop:BasicLinearProperties}
The adjoint operator $A^* : \mathcal{M}_\mathbb{S} \rightarrow \mathcal{M}$ of $A$ is given by
	\[
		(A^* X_\mathbb{S})_{jj'} = \sum_{S \in \mathbb{S}} \mathbbm{1}_{\{j,j' \in S\}} (X_S)_{jj'} = \sum_{S \in \mathbb{S}_{jj'}} (X_S)_{jj'},
	\]
	where we recall that $\mathbb{S}_{jj'} = \{S \in \mathbb{S} : j,j' \in S\} = \mathbb{S}_j \cap \mathbb{S}_{j'}$.
\end{prop}
Now, the following proposition fully characterises compatibility in terms of the non-negativity of a collection of linear functionals. 
\begin{prop}
\label{Prop:ConsCompatCheck}
For $\Sigma_\mathbb{S} \in \mathcal{P}_\mathbb{S}^*$ we have $\Sigma_\mathbb{S} \in \mathcal{P}_\mathbb{S}^{0,*}$ if and only if
\[
	\langle X_\mathbb{S}, \Sigma_\mathbb{S} \rangle_\mathbb{S} \geq 0 \quad \text{for all } X_\mathbb{S} \in \mathcal{M}_\mathbb{S} \text{ satisfying } A^* X_\mathbb{S} \succeq 0.
\]
\end{prop}
The proof can be found in Section \ref{sec:proof}. This is an extension of well-known characterisation of the feasibility of positive semi-definite matrix completion \citep[e.g.][]{laurent2009matrix}. Indeed, when $\Sigma_\mathbb{S}$ is consistent, we can show that $\langle X_\mathbb{S}, \Sigma_\mathbb{S} \rangle_\mathbb{S} = \langle A^* X_\mathbb{S}, \Sigma \rangle$, where $\Sigma$ is the $d \times d$ symmetric matrix with $\Sigma_{jj'}=(\Sigma_S)_{jj'}$ for all $S \in \mathbb{S}_{jj'}$, if $\mathbb{S}_{jj'} \neq \emptyset$, and $\Sigma_{jj'}=0$ if $\mathbb{S}_{jj'} = \emptyset$. Here $\Sigma$ can be thought of as a partial matrix that is padded with zeros in unobserved positions. Since $A^* X_\mathbb{S}$ is also zero in these positions, the value of $\Sigma$ there is arbitrary. Now our characterisation reduces to checking that $\langle A^* X_\mathbb{S}, \Sigma \rangle \geq 0$ for all $X_\mathbb{S} \in \mathcal{M}_\mathbb{S}$ satisfying $A^* X_\mathbb{S} \succeq 0$, which is equivalent to checking $\langle X, \Sigma \rangle \geq 0$ for all $X \in \mathcal{M}$ with $X_{jj'}=0$ if $\mathbb{S}_{jj'} = \emptyset$, which coincides with (4) in \cite{laurent2009matrix}. 

Proposition~\ref{Prop:ConsCompatCheck} provides a characterisation of compatibility, but in order to assess the significance of departures from the null hypothesis and thus to define hypothesis tests, we will need a numerical measure of incompatibility. A natural way to do this is to minimise $\langle X_\mathbb{S}, \Sigma_\mathbb{S} \rangle_\mathbb{S}$ over a subset of $\{ X_\mathbb{S} \in \mathcal{M}_\mathbb{S} : A^* X_\mathbb{S} \succeq 0\}$ that still characterises compatibility, but gives finite minimal values. First, observe that checking the compatibility of covariance matrices is equivalent to checking the consistency of the variances $\sigma^2_S = \operatorname{diag}(\mathrm{Cov}(X_S|\Omega=\mathbbm{1}_S))$ for $S \in \mathbb{S}$, and the compatibility of the correlation matrices $\operatorname{Corr}(X_S|\Omega = \mathbbm{1}_S)$ for $S \in \mathbb{S}$. Here, we will focus on the latter issue, and postpone the discussion of testing the consistency of variances to Section~\ref{sec:simul}. This is because this problem essentially reduces to testing the equality of variances, which is well understood in the statistical literature \citep{brown12variances, miao09levene}.  
Now, whenever $\Sigma_\mathbb{S}$ is a collection of correlation matrices we define
\begin{align}
\label{Eq:Primal}
	R(\Sigma_\mathbb{S}) &:= \sup\biggr\{-\frac{1}{d} \langle \Sigma_\mathbb{S}, X_\mathbb{S} \rangle_\mathbb{S} : X_\mathbb{S}+X_\mathbb{S}^0 \succeq_\mathbb{S} 0,  A^* X_\mathbb{S} + Y \succeq 0 \text{ for some } Y \in \mathcal{Y} \biggr\} \nonumber \\
    &= 1- \frac{1}{d}\inf\{ \langle \Sigma_\mathbb{S}, Y_\mathbb{S} \rangle_\mathbb{S} : Y_\mathbb{S} \succeq_\mathbb{S} 0,  A^* Y_\mathbb{S} + Y \succeq I_d \text{ for some } Y \in \mathcal{Y} \}
\end{align}
where $X_\mathbb{S}^0 = (X_S^0 : S \in \mathbb{S})$ with $X_S^0 = \operatorname{diag}\left(1 /\left|\mathbb{S}_j\right|: j \in S\right)$, and where $\mathcal{Y}$ is the set of diagonal $d \times d$ matrices with trace equal to zero. We refer to $R(\cdot)$ as an index of incompatibility, borrowing the terminology from \cite{berrett2022optimal}. The objective function of this optimisation problem is a one-to-one mapping of the linear functional appearing in our characterisation of compatibility. Moreover, by choosing $Y=\boldsymbol{O}$ and noting that $X_\mathbb{S}^0 \succ_{\mathbb{S}} 0$, we can see that for any $X_\mathbb{S}$ that satisfies $A^* X_\mathbb{S} \succeq 0$, the collection $\epsilon X_\mathbb{S}$ is feasible for $\epsilon>0$ sufficiently small. Thus, by Proposition~\ref{Prop:ConsCompatCheck} we have that $R(\Sigma_\mathbb{S})>0$ whenever $\Sigma_\mathbb{S}$ is incompatible. On the other hand, when $\Sigma_\mathbb{S} = A\Sigma$ is compatible and $(X_\mathbb{S},Y)$ is feasible we have $\langle \Sigma_\mathbb{S}, X_\mathbb{S} \rangle_\mathbb{S} = \langle \Sigma, A^* X_\mathbb{S} \rangle =  \langle \Sigma, A^* X_\mathbb{S} + Y \rangle \geq 0$, where the second equality holds because $\Sigma$ has a constant diagonal. Combining this with the observation that $X_\mathbb{S} = \boldsymbol{O}_\mathbb{S}$ is feasible, we see that $R(\Sigma_\mathbb{S})=0$ when $\Sigma_\mathbb{S}$ is compatible. 

In the above argument we did not use the specific form of the lower bound $X_\mathbb{S} \succeq_\mathbb{S} -X_\mathbb{S}^0$ anywhere, and it would also have been possible to optimise over the restricted set of $X_\mathbb{S}$ that are feasible with $Y=\boldsymbol{O}$. The specific choice of the feasible set in the definition of $R(\Sigma_\mathbb{S})$ was made because it leads to an interpretable dual formulation. While there exist semi-definite programs for which strong duality does not hold, Slater's condition is satisfied in our problem, so we do not encounter such issues (see Appendix~\ref{sec:SDP} for an introduction to the theory of semi-definite programming). This is formalised in the result below. 
\begin{prop}
\label{Prop:CorrDuality}
For $\Sigma_\mathbb{S} \in \mathcal{P}_\mathbb{S}$ we have
\begin{align}
\label{Eq:Dual}
    R(\Sigma_{\mathbb{S}}) &=\inf \{ \epsilon \in [0,1] : \Sigma_\mathbb{S} \in (1-\epsilon) \mathcal{P}_\mathbb{S}^0 + \epsilon \mathcal{P}_\mathbb{S} \} \nonumber \\
    &= 1-\frac{1}{d} \sup\{ \mathrm{tr}(\Sigma) : A\Sigma \preceq_\mathbb{S} \Sigma_\mathbb{S}, \, \Sigma_{11}=\ldots=\Sigma_{dd}, \, \Sigma \succeq 0 \}.
\end{align}
\end{prop}

This result shows that our measure of incompatibility $R(\Sigma_\mathbb{S})$
can be interpreted as the smallest amount of perturbation $\epsilon$ that a compatible collection  of correlation matrices must be corrupted by to result in the input collection  $\Sigma_{\mathbb{S}}$. It is immediate from this representation that $R(\Sigma_\mathbb{S})$ takes values in $[0,1]$. Moreover, it follows from Slater's condition that the optimal value of the dual problem is attained. Thus, writing $\lambda^* = 1- R(\Sigma_{\mathbb{S}})$, we can write $\Sigma_{\mathbb{S}}$ as 
\begin{align}\label{eq:dual_decomposition}
    \Sigma_{\mathbb{S}} = \lambda^*A\Sigma + (1-\lambda^*)\Sigma_{\mathbb{S}}',
\end{align}
where $\Sigma \in \mathcal{P}$ and $\Sigma_{\mathbb{S}}' \in \mathcal{P}_\mathbb{S}$. By the maximality of $\lambda^* = 1-R(\Sigma_\mathbb{S})$, it must be the case that $R(\Sigma_{\mathbb{S}}')=1$. 
Indeed, if this were not the case, it would be possible to write $\Sigma_{\mathbb{S}}' = \lambda'A\Sigma' + (1-\lambda')\Sigma_{\mathbb{S}}''$ for some $\lambda' \in (0,1]$, which would contradict the fact that $\lambda^*$ is optimal. 
This argument shows that, whenever $\mathbb{S}$ is such that there exists an incompatible collection  $\Sigma_\mathbb{S}$, the maximal value $R(\Sigma_\mathbb{S})=1$ is attainable, so that the quantity $R(\Sigma_\mathbb{S})$ is on an interpretable scale between compatibility at one extreme and maximal incompatibility at the other. 

We remark that this dual interpretation of $R(\cdot)$ aligns with a similar representation of the incompatiblity of collections of distributions defined by~\cite{berrett2022optimal}. In this earlier work it is shown that the incompatibility of collections of distributions can be understood through linear programming techniques. Our work here, however, shows that we must consider the more complex problem of semi-definite programming to understand the incompatibility of collections of covariance matrices. Despite this additional complexity, since Slater's condition is satisfied for our problem, the primal-dual interior point method has a computational complexity which is polynomial in the number of constraints and the dimension of the variable square matrix (Section 6.4.1. of \cite{nesterov94}, Section 5.7. of \cite{boyd_vandeberghe}). This ensures that $R(\Sigma_\mathbb{S})$ can be always computed efficiently without additional assumptions. 
We conclude this section with some basic properties of $R(\cdot)$ that will be used in later proofs.
\begin{prop}\label{prop:properties}
    The following hold:
    \begin{enumerate}
        \item[(i)] $R(\cdot)$ is convex.
        \item[(ii)] $R(\cdot)$ is continuous.
        \item[(iii)] Suppose $\mathbb{S} \subseteq \mathbb{S}'$ and $\Sigma_{\mathbb{S}} \subseteq_{\mathbb{S}} \Sigma_{\mathbb{S}'}$, where the inclusion $\subseteq_{\mathbb{S}}$ means that every correlation matrix in $\Sigma_{\mathbb{S}}$ is also in $\Sigma_{\mathbb{S}'}$. Then $R(\Sigma_{\mathbb{S}}) \leq R(\Sigma_{\mathbb{S}'})$. 
    \end{enumerate}
\end{prop}

It is interesting to observe that property (iii) says that $R$ is monotone with respect to the inclusion operator, so that additional information can only make a collection less compatible.  


\section{Concentration bound and an oracle test}\label{Sec:oracle_test}
Having introduced our population-level measure of incompatibility, in this section we analyse the concentration properties of the plug-in estimator  $ R(\hat{\Sigma}_\mathbb{S})$. This will give us tools that enable us to prove theoretical guarantees (Theorem \ref{thm:bootstrap_guarantees}) for the bootstrap procedures defined in Section \ref{sec:bootstrap_test}. Furthermore, it will lead naturally to the definition of the oracle test defined in Theorem \ref{thm:oracle_test}. Now,
the analysis of $R(\hat{\Sigma}_\mathbb{S})$ is challenging, as it is defined as the optimal value of a semi-definite program with an unbounded feasible set. In fact, without further assumptions, it is not possible to restrict attention to a compact feasible set. On the other hand, most statistical techniques for the analysis of suprema of empirical processes require feasible sets to be totally bounded so that, for example, covering arguments can be applied.

Fortunately, under the assumption that $\Sigma_\mathbb{S} \succ_\mathbb{S} 0$, our dual problem~\eqref{Eq:Dual} is strictly feasible and hence Slater's condition implies that the optimal value is attained in the primal problem~\eqref{Eq:Primal}. This assumption is reasonable in many areas of application, and similar assumptions of invertibility have been used frequently in the literature \citep{buhlmann2006lasso, cai2011precision}. In fact, if we assume the stronger condition that $\Sigma_\mathbb{S} \succeq_\mathbb{S} c I_\mathbb{S}$ for some $c>0$, we will see that the optimal value is always attained in a compact set whose size depends on $c$. 
Indeed, the strict feasibility of the dual problem~\eqref{Eq:Dual} implies that there exists $X_{\mathbb{S}} \in \mathcal{M}_\mathbb{S}$ such that $X_\mathbb{S} + X_\mathbb{S}^0 \succeq_\mathbb{S} 0$ and $R(\Sigma_{\mathbb{S}}) = -d^{-1} \langle X_{\mathbb{S}}, \Sigma_{\mathbb{S}} \rangle_{\mathbb{S}}$. This in turn implies that
\[
    \sum_{S \in \mathbb{S}}\|X_S + X_S^0\|_* 
 = \langle X_{\mathbb{S}} + X_{\mathbb{S}}^0, I_{\mathbb{S}} \rangle_{\mathbb{S}} \leq \frac{1}{c} \langle X_{\mathbb{S}} + X_{\mathbb{S}}^0, \Sigma_{\mathbb{S}} \rangle_{\mathbb{S}} = \frac{d}{c}\{1-R(\Sigma_\mathbb{S})\} \leq \frac{d}{c},
\]
so that we have a bound on the sum of the nuclear norms of the matrices in the sequence $X_\mathbb{S} + X_\mathbb{S}^0$. In finding the optimal value of the primal problem~\eqref{Eq:Primal}, then, we may restrict attention to
\begin{equation}
\label{Eq:CompactFeasibleSet}
    \mathcal{F}_c := \biggl\{ X_\mathbb{S} \in \mathcal{M}_\mathbb{S} :  X_{\mathbb{S}} + X_{\mathbb{S}}^0 \succeq_{\mathbb{S}} 0 , \sum_{S \in \mathbb{S}}\|X_S + X_S^0\|_* \leq d/c,  A^* X_\mathbb{S} + Y \succeq 0 \text{ for some } Y \in \mathcal{Y} \biggr\},
\end{equation}
which is compact. Formally, defining the regularised version of $R(\cdot)$ as
\begin{align}\label{eq:reg_R}
    R_z(\Sigma_\mathbb{S}) := \sup \left\{-d^{-1} \langle X_{\mathbb{S}}, \Sigma_{\mathbb{S}} \rangle_{\mathbb{S}} : X_\mathbb{S} \in \mathcal{F}_z \right\} \text{ for } z \leq 1,
\end{align}
we have just shown that $R_c(\Sigma_\mathbb{S}) = R(\Sigma_\mathbb{S})$ whenever $\Sigma_\mathbb{S} \succeq_\mathbb{S} c I_\mathbb{S}$. This regularised index of incompatibility is used to compute the bootstrap statistic $R_{\hat{c}/2}(\hat{\Sigma}_{\mathbb{S}, 1})$ in Algorithm \ref{alg:Corrbootstrap_test}, and is such that $R_{z_1}(\Sigma_\mathbb{S}) \geq R_{z_2}(\Sigma_\mathbb{S})$ if $z_2 \geq z_1$.

Before moving on to describe how to construct a statistical test under this new assumption, we give a brief discussion of the norm on $\mathcal{M}_\mathbb{S}$ defined by
\[
    \|X_{\mathbb{S}}\|_{*, \mathbb{S}} := \sum_{S \in \mathbb{S}}\|X_S\|_*,
\]
which reduces to $\sum_{S \in \mathbb{S}} \operatorname{tr}(X_S)$ in case that $X_{\mathbb{S}} \succeq_{\mathbb{S}} 0$. For each $S \in \mathbb{S}$, the nuclear norm $\|X_S\|_*$ can be thought of as the $\ell_1$ norm applied to the eigenvalues of $X_S$. As these are then summed to give $\|X_{\mathbb{S}}\|_{*, \mathbb{S}}$, we see that $\| \cdot \|_{*, \mathbb{S}}$ can be thought of as an $\ell_1$ norm on $\mathcal{M}_\mathbb{S}$. It is easy to see that the dual norm of $\|\cdot\|_{*, \mathbb{S}}$ with respect to the inner product $\langle \cdot, \cdot \rangle_{\mathbb{S}}$ is 
\[
\|X_{\mathbb{S}}\|_{2, \mathbb{S}} := \max_{S \in \mathbb{S}}\|X_S\|_2,
\]
 where $\|\cdot\|_2$ is the usual spectral norm of a matrix. This follows after writing the sequence of matrices in block-diagonal form, and allows us to derive the following generalisation of Holder's inequality in the space of matrix collections,
 \begin{equation}
 \label{Eq:Holder}
 |\langle X_{\mathbb{S}}, Y_{\mathbb{S}} \rangle_{\mathbb{S}}| \leq \|X_{\mathbb{S}}\|_{*, \mathbb{S}} \|Y_{\mathbb{S}}\|_{2, \mathbb{S}} = \sum_{S \in \mathbb{S}}\|Y_S\|_* \cdot \max_{S \in \mathbb{S}}\|X_S\|_2.
 \end{equation}
 This inequality will be used in the proof of the following result, which provides valid critical values for the test statistic $R(\hat{\Sigma}_{\mathbb{S}})$ and gives conditions under which the resulting test has large power.

\begin{thm}\label{thm:oracle_test}
    Suppose that the assumptions of Theorem \ref{thm:bootstrap_guarantees} hold, and recall the definition of $C_{\alpha, c} \equiv C_\alpha$ from~\eqref{eq:critical_oracle}. Then, for a universal constants $K_1 > 0$ chosen sufficiently large, for all $\alpha \in (0,1)$, the test that rejects $H_0 : R(\Sigma_\mathbb{S}) = 0$ if and only if $R(\hat{\Sigma}_{\mathbb{S}}) \geq C_{\alpha}$ has Type I error bounded by $\alpha$. Moreover, for $\beta \in (0,1-\alpha)$, if $R(\Sigma_\mathbb{S}) > C_{\alpha} + C_{\beta}$, then~$\mathbb{P}\{R(\hat{\Sigma}_{\mathbb{S}}) \leq C_{\alpha}\} \leq \beta$.
\end{thm}

In proving this result we give concentration inequalities for the random quantities $R(\hat{\Sigma}_\mathbb{S})$. The analysis of $R(\hat{\Sigma}_\mathbb{S})$ is crucially based on the fact that, under $H_0$ and in light of the inequality~\eqref{Eq:Holder}, we can control the oscillation $|R(\hat{\Sigma}_\mathbb{S}) - R(\Sigma_\mathbb{S})|$ using $\max_{S \in \mathbb{S}} \|\hat{\Sigma}_{S} - \Sigma_{S}\|_2$, where the $\Sigma_{S}$ are the Pearson population correlation matrices and $\hat{\Sigma}_{S}$ are the corresponding Pearson sample correlation matrices. To this end, we derive a tail bound for the spectral norm $\|\hat{\Sigma} - \Sigma\|_2$, where $\Sigma$ is the population correlation matrix and $\hat{\Sigma}$ is the sample correlation matrix of complete data, which may be of independent interest. This can be found in Section \ref{sec:proof}.  


As well as providing a critical value for our test, Theorem~\ref{thm:oracle_test} also gives upper bounds on the minimax separation rate for this testing problem. When $c>0$, $\sigma^2_\mathrm{min}$ and $\nu>0$ are fixed, our analysis gives an upper bound on the minimax rate of the order 
\[
C_{\alpha}\lesssim \max_{S \in \mathbb{S}}\sqrt{\frac{|S| + \log(|\mathbb{S}|/\alpha)}{n_S}}.
\]
whenever $n_S \gtrsim |S|$ for all $S \in \mathbb{S}$. This is our main regime of interest, and we see in our examples in Section~\ref{sec:minimax_LB} below that reliable testing is only possible when sample sizes are large compared with dimensions, up to logarithmic factors.

We conclude this section by illustrating the behaviour of this bound in certain examples where the expression for $C_\alpha$ can be simplified. The corresponding upper bounds on the minimax separation rate will be complemented by lower bounds in Section~\ref{sec:minimax_LB} to follow. 

\begin{eg}\label{ex:cycle_separation} In the $d$-cycle example, with $d \geq 3$ and $\mathbb{S}= \{ \{1,2\} , \{2,3\} ,\ldots, \{d-1,d\}, \{d,1\}\}$, we have $|S|=2$ for all $S \in \mathbb{S}$ and $|\mathbb{S}| = d$ so that 
\[
C_{\alpha} \lesssim  \max\limits_{S \in \mathbb{S}} \sqrt{\frac{\log(d/\alpha)}{n_S}} =   \sqrt{\frac{\log(d/\alpha)}{\min\limits_{S \in \mathbb{S}} n_S}}.
\]
Combined with Theorem~\ref{thm:cycle_minimax} below, this reveals that, in this specific example, testing the compatibility of the correlation matrices is no harder than testing the consistency of the variances, up to constant factors.
\end{eg}

\begin{eg}\label{block_separation}
    Consider the block-$3$-cycle $\mathbb{S} = \{[2d], [d]\cup([3d]\setminus[2d]),[3d]\setminus[d]\}$, with $d \geq 1$. Then
         \[
C_{\alpha} \lesssim  \max\limits_{S \in \mathbb{S}} \sqrt{\frac{d + \log(1/\alpha)}{n_S}} =  \sqrt{\frac{d + \log(1/\alpha)}{\min\limits_{S \in \mathbb{S}}  n_S}}.
\]
We prove the minimax optimality, up to logarithmic factors, of this rate in Theorem \ref{block_minimax}. In particular, this shows that the optimal separation rates for this testing problem are not significantly faster than the optimal rates for the estimation of $\Sigma_\mathbb{S}$ with operator norm loss.
\end{eg}

\begin{eg}\label{ex:all_but_one_separation}  Consider $\mathbb{S} = \operatorname{Pow}([d])$,  where $\operatorname{Pow}(\cdot)$ stands for the power set.  This corresponds to the case where we observe all possible missingness patterns from a dataset of dimension $d$.  In this case,  we have $\max_{S \in \mathbb{S}} |S| = d$ and $\log |\mathbb{S}| = d \log 2$, so that 
\[
C_{\alpha} \lesssim  \max\limits_{S \in \mathbb{S}} \sqrt{\frac{d + \log(1/\alpha)}{n_S}} =  \sqrt{\frac{d  +\log(1/\alpha)}{\min\limits_{S \in \mathbb{S}} n_S}}.
\]
This shows that $C_{\alpha}$ is at worse of the order $\sqrt{d/n}$.
\end{eg}

\section{Optimality and examples}\label{sec:minimax_LB}
In this section, we assess the optimality of the oracle test given in Theorem \ref{thm:oracle_test}, in the settings of Examples~\ref{ex:cycle_separation} and \ref{block_separation}, i.e.~when $\mathbb{S}= \{ \{1,2\} , \{2,3\} ,\ldots, \{d-1,d\}, \{d,1\}\}$, the $d$-cycle, and when $\mathbb{S} = \{[2d], [d]\cup([3d]\setminus[2d]),[3d]\setminus[d]\}$, the block $3$-cycle. These collections $\mathbb{S}$ provide examples where our methodology is provably near rate-optimal. For a given dimension $d$, these two examples further demonstrate the range of optimal rates that can arise for different collections $\mathbb{S}$. Assuming for simplicity that $n_S=n$ for all $S \in \mathbb{S}$, we will see that the optimal rate in the $d$-cycle case is $\{\log(d)/n\}^{1/2}$, while for the block 3-cycle it is $(d/n)^{1/2}$ up to logarithmic factors. Together, these results show that the structure of $\mathbb{S}$ can have a significant effect on the difficulty of the problem.

We will characterise the optimality of a testing procedure using the minimax framework, where we aim at finding the smallest separation between the null and the alternative hypotheses such that there exists a test that can distinguish between $H_0$ and $H_1$ up to a given level of error. More precisely, given $\rho>0$, we are interested in testing
\[
         H_0: R(\Sigma_\mathbb{S}) = 0 \quad \text{ vs. } \quad H_1: R(\Sigma_\mathbb{S}) > \rho, 
         \]
and our goal is to find the smallest value of $\rho$ such that there exists a test with uniform error control. Write $\Psi \equiv \Psi_\mathbb{S}(n_\mathbb{S})$ for the set of all tests, that is measurable functions of the data $(X_{S,i} : S \in \mathbb{S}, i \in [n_S])$ taking values in $\{0,1\}$. Recall that  $\bar{\mathcal{P}}_\mathbb{S}(0)$ denotes the set of all collections of distributions on $(\mathbb{R}^S : S \in \mathbb{S})$ such that the associated correlation matrices satisfy $R(\Sigma_\mathbb{S}) = 0$, and write $\mathcal{P}_\mathbb{S}(\rho)$ for the set of all collections of distributions on $(\mathbb{R}^S : S \in \mathbb{S})$ such that the associated correlation matrices satisfy $R(\Sigma_\mathbb{S}) > \rho$. Given a  collection of distributions $P_\mathbb{S} = (P_S : S \in \mathbb{S})$ on $(\mathbb{R}^S : S \in \mathbb{S})$ and a collection of sample sizes $n_\mathbb{S}=(n_S : S \in \mathbb{S})$, we write $P_\mathbb{S}^{\otimes n_\mathbb{S}}$ for the distribution of the entire dataset $(X_{S,i} : S \in \mathbb{S}, i \in [n_S])$ when each observation is independent and $X_{S,i} \sim P_S$ for each $i \in [n_S]$ and $S \in \mathbb{S}$. For a fixed $\eta \in (0,1)$ we may then define the minimax separation to be
\[
\rho^* \equiv \rho_\mathbb{S}^*(n_\mathbb{S},\eta) := \inf\left\{\rho > 0: \inf_{\varphi \in \Psi} \biggl( \sup_{P_{\mathbb{S},0} \in \bar{\mathcal{P}}_{\mathbb{S}}(0)} P_{\mathbb{S},0}^{\otimes n_{\mathbb{S}}}(\varphi = 1) + \sup_{P_{\mathbb{S},1} \in \mathcal{P}_{\mathbb{S}}(\rho)} P_{\mathbb{S},1}^{\otimes n_{\mathbb{S}}} (\varphi = 0) \biggr) \leq \eta \right\}.
\]
In our analysis we take $\eta=3/4$, but this is an arbitrary choice and any constant value in $(0,1)$ would result in the same qualitative behaviour. In common with previous work on minimax testing, we prove lower bounds on $\rho^*$ by constructing suitable (prior) distributions $\mu_0, \mu_1$ whose support is contained in $\bar{\mathcal{P}}_{\mathbb{S}}(0), \mathcal{P}_{\mathbb{S}}(\rho)$, respectively. In our proofs it will be sufficient to consider mean-zero Gaussian distributions with suitable priors over their covariance matrices. Having chosen these priors we can bound the minimal error probability by writing
\begin{align*}
    \sup_{P_{\mathbb{S},0} \in \bar{\mathcal{P}}_{\mathbb{S}}(0)} P_{\mathbb{S},0}^{\otimes n_{\mathbb{S}}}(\varphi = 1)  + \sup_{P_{\mathbb{S},1} \in \mathcal{P}_{\mathbb{S}}(\rho)} P_{\mathbb{S},1}^{\otimes n_{\mathbb{S}}} (\varphi = 0) &\geq \mathbb{E}_{\mu_0} P_{\mathbb{S},0}^{\otimes n_{\mathbb{S}}}(\varphi = 1) + \mathbb{E}_{\mu_1} P_{\mathbb{S},1}^{\otimes n_{\mathbb{S}}}(\varphi = 0) \\
    & \geq 1 - \operatorname{TV}\bigl( \mathbb{E}_{\mu_0}P_{\mathbb{S},0}^{\otimes n_{\mathbb{S}}} , \mathbb{E}_{\mu_1}P_{\mathbb{S},1}^{\otimes n_{\mathbb{S}}} \bigr),
\end{align*}
where 
$\mathbb{E}_{\mu_i}P_{\mathbb{S},i}^{\otimes n_{\mathbb{S}}}$ denotes the mixture distribution of the dataset resulting from generating $P_{\mathbb{S},i} \sim \mu_i$ then, conditionally on  $P_{\mathbb{S},i}$, generating the data. Then, the idea behind this method of finding a lower bound on $\rho^*$ is to find priors $\mu_0, \mu_1$  with the largest separation possible such that no test can successfully distinguish between $\mathbb{E}_{\mu_0}P_{\mathbb{S},0}^{\otimes n_{\mathbb{S}}}$ and $\mathbb{E}_{\mu_1}P_{\mathbb{S},1}^{\otimes n_{\mathbb{S}}}$.

\subsection{Cycles}\label{subsec:cycles}
Recall that we refer to $\mathbb{S}$ as a $d$-cycle when $\mathbb{S} = \mathbb{S}_d:=\{\{1,2\},\ldots, \{d,1\}\}$. In this subsection, additions in subscripts of the form $(j, j+1)$ for $j \in [d]$ are intended modulo $d$, where $d$ in the size of the cycle, so that $(0,1)$ and $(d,d+1)$ are equivalent to $(d,1)$. We also write $\Sigma_{\mathbb{S}_d}:= (\Sigma_{\{1,2\}}, \cdots, \Sigma_{ \{d,1\}})$ for a collection of $2 \times 2$ correlation matrices with
\[
\Sigma_{j,j+1} = \begin{pmatrix} 1 & \cos \theta_{j} \\ \cos \theta_{j} & 1 \end{pmatrix},
\]
and $\theta_j \in [0, \pi]$ for all $j \in [d]$.

Now, recall that Theorem \ref{thm:oracle_test} implies an upper bound on the minimax separation of the form $\rho^* \lesssim \sqrt{\log d/n}$. While it is straightforward to show a lower bound of $\rho^* \gtrsim 1/\sqrt{n}$ using a standard Le Cam two-point argument, matching the logarithmic dependence on $d$ in the numerator is technically challenging. Nonetheless, we will argue that $\sqrt{\log d/n}$ is the optimal rate of convergence in the combined problem where we test both the compatibility of the correlation matrices and the consistency of the variances.

To this end, we now define an analogous index of inconsistency for the collection of variances. Writing $\sigma_\mathbb{S}^2=(\sigma_S^2 : S \in \mathbb{S})$ for the collection of individual variances, we fix our units of measurement such that 
\[
\bar{\operatorname{av}}_j(\sigma^2_{\mathbb{S}}) :=  |\mathbb{S}_j|^{-1} \sum_{S \in \mathbb{S}_j} \sigma_{S,j}^2 = 1, 
\]
for all $j \in [d]$, where $\sigma_{S,j}^2$ is the $j$-th element of $\sigma_{S}^2$. This is a natural constraint, analogous to the standardisation of variables in complete-data problems, that does not remove information that may be present in the individual variances. For such $\sigma_\mathbb{S}^2$, define 
\[
V(\sigma_{\mathbb{S}}^2) := 1 - \min_{j \in [d]}\min_{S \in \mathbb{S}_j} \sigma_{S,j}^2 = \max_{j \in [d]} \max_{S \in \mathbb{S}_j} \left(1 - \sigma_{S,j}^2\right). 
\]
Under the hypothesis $\bar{\operatorname{av}}_j(\sigma^2_{\mathbb{S}}) = 1$ for all $j \in [d]$, it is clear that $V(\sigma_{\mathbb{S}}^2) = 0$ if and only if $\sigma^2_{S,j} = 1$ for all $j \in [d]$, $S \in \mathbb{S}_j$. On the other hand, we have $V(\sigma_{\mathbb{S}}^2) > 0$, if and only if there exists at least one variance strictly less than $1$. It is clear from the definition that $V$ is bounded by one, and that this extreme value is attainable when $\mathbb{S}$ is non-trivial and there exists $j$ such $\sigma_{S,j}^2=0$ for some $S \in \mathbb{S}_j$. The following result gives a dual representation for $V(\sigma^2_\mathbb{S})$, providing justification for our specific measure of inconsistency.
\begin{prop}\label{prop:measure_consistency}
    If $\bar{\operatorname{av}}_j(\sigma^2_{\mathbb{S}}) = 1$ for all $j \in [d]$, then
    \[
    V(\sigma_{\mathbb{S}}^2) = \inf\left\{ \epsilon \in [0,1] : \sigma_{\mathbb{S}}^2 = (1-\epsilon)A_{V} \boldsymbol{1}_d + \epsilon {\sigma'}_{\mathbb{S}}^2 \text{ with } \bar{\operatorname{av}}_j(\sigma'^2_{\mathbb{S}}) = 1 \text{ for all } j \in [d] \right\},
    \]
    where $(A_V \sigma^2)_S = (\sigma^2_k)_{k \in S}$.
\end{prop}
This result gives a dual representation for $V(\sigma^2_\mathbb{S})$, which is analogous to Proposition \ref{Prop:CorrDuality} and leads to a similar interpretation, based on the idea of finding the smallest perturbation to make the collection consistent. \\
We now characterise the optimal rate of convergence in the combined problem where we test both the compatibility of the correlation matrices and the consistency of the variances. Together, the following results show that the oracle test based on the plug-in estimator $R(\hat{\Sigma}_\mathbb{S}) + V(\hat{\sigma}^2_\mathbb{S})$ is optimal up to constant factors, and further imply that testing the consistency of the variances captures the essential statistical difficulty of the problem in the case of a $d$-cycle.

\begin{thm}\label{thm:combined_oracle_test}
   Suppose that the assumptions of Theorem \ref{thm:bootstrap_guarantees} hold, and further assume that the sequence of variances $\sigma^2_\mathbb{S}$ satisfies  $\bar{\operatorname{av}}_j(\sigma^2_{\mathbb{S}}) = 1$ for all $j \in [d]$. 
    Define  $C_{\alpha}^{(R + V)} = \max\{2C_{\alpha/2} , C_{\alpha}^{(V)}\}$, with $C_\alpha \equiv C_{\alpha, c}$ as in \eqref{eq:critical_oracle} and
\[
C_{\alpha}^{(V)} = K_2 \nu^2 \sqrt{\frac{\log\left( \sum_{S \in \mathbb{S}} |S|/\alpha\right)}{\min_{S \in \mathbb{S}} n_S}},
\]
for a universal constant $K_2 > 0$. Then, for $K_1, K_2 > 0$ chosen sufficiently large, for all $\alpha \in (0,1)$, if  $C_{\alpha}^{(R+V)} \leq \min\{ 2, 16\nu^2\}$, the test that rejects $H_0 : R(\Sigma_\mathbb{S}) + V(\sigma^2_\mathbb{S})= 0$ if and only if $R(\hat{\Sigma}_{\mathbb{S}}) + V(\hat{\sigma}^2_\mathbb{S}) \geq C_{\alpha}^{(R+V)}$ has Type I error bounded by $\alpha$. Moreover, for all $\beta \in (0,1-\alpha)$, if $C_{\beta}^{(R+V)} \leq \min\{ 2, 16\nu^2\}$ and $R(\Sigma_\mathbb{S}) + V(\sigma^2_\mathbb{S}) > C_{\alpha}^{(R + V)} + C_{\beta}^{(R + V)}$, then $\mathbb{P}\{R(\hat{\Sigma}_{\mathbb{S}}) + V(\hat{\sigma}^2_\mathbb{S}) \leq C_{\alpha}^{(R + V)}\} \leq \beta.$
\end{thm}

\begin{thm}\label{thm:cycle_minimax}
    Let $\mathbb{S} = \mathbb{S}_d$ for $d \geq 3$, with sample sizes $n_{\mathbb{S}} = (n_1, \ldots, n_d)$, and consider testing
    \[
         H_0^\prime: R(\Sigma_\mathbb{S}) + V(\sigma^2_\mathbb{S}) = 0 \quad \text{ vs. } \quad H_1^\prime: R(\Sigma_\mathbb{S}) + V(\sigma^2_\mathbb{S}) > \rho, 
         \]
    for $\rho > 0$. Call $\rho^*_{R+V}$ the minimax separation of this testing problem. There exists a universal constant $c_1 > 0$ such that 
    \[
    \rho_{R+V}^* \geq c_1 \sqrt{\frac{\log d}{\min_{j \in [d]}n_j}}.
    \]
\end{thm}

Furthermore, we can give a relatively explicit expression for $R(\cdot)$ for a general $d$-cycle.
\begin{prop}\label{prop:KKT_cycle}
Let $\Sigma$ be the optimum solution to the dual problem~\eqref{Eq:Dual}, and let $\boldsymbol{\varphi}^* = (\varphi_1^*, \ldots, \varphi_d^*)$ be such that $\Sigma_{j,j+1} = \cos\varphi_j^*$ for each $j \in [d]$. Then:
\begin{itemize}
    \item[(i)] $R(\Sigma_{\mathbb{S}_d}) = |\rho_j - (1-R(\Sigma_{\mathbb{S}_d}))\cos\varphi_j^*|, \text{ for all } j \in [d]$;
\item[(ii)]$\boldsymbol{\varphi}^* = (\varphi_1^*, \ldots, \varphi_d^*)$ is unique, and $\boldsymbol{\varphi}^*(\theta_1, \ldots, \theta_d)$ is continuous for varying $(\theta_1, \ldots, \theta_d) \in [0, \pi]^d$;
    \item[(iii)] if $\theta_1 = \max_{j \in [d]} \theta_j$, with $\theta_2,\ldots,\theta_d \leq \pi/2$, then \[
    1-R(\Sigma_{\mathbb{S}_d}) = \frac{1-\epsilon_j \cos\theta_j}{1-\epsilon_j \cos\varphi_j^*}, \quad \text{ for all } j \in [d],
    \]
    where $\boldsymbol{\epsilon}_d = (\epsilon_1, \ldots, \epsilon_d) = (-1, +\boldsymbol{1}_{d-1})$. Also, $\varphi_1^* = \sum_{j = 2}^d \varphi_j^*$.
\end{itemize}
\end{prop}

Observe that part (ii) only says that the entries of $\Sigma$ corresponding to the cycle pattern are unique, not the whole $\Sigma$ itself. Indeed, given the unique optimal $\boldsymbol{\varphi}^*$, there may exist infinitely many positive semi-definite completions. In fact, if $A$ is a partial symmetric matrix admitting a positive semi-definite completion, then there exists a unique positive semi-definite completion with maximum determinant~\citep[][Theorem~2]{GRONE1984109}. For a general sequence of angles $(\theta_1, \ldots, \theta_d)$, it is sufficient to use the transformation given in  Proposition \ref{prop:reduciton3} in Appendix \ref{sec:extra_prop_R} to reduce to the case where at most one angle is larger than $\pi/2$, choose $\boldsymbol{\epsilon}_d$ as outlined above, and perform the inverse transformation to obtain the signs for the original $(\theta_1, \ldots, \theta_d)$. As an immediate corollary of this, it is easy to see that, under the same set of hypotheses, we have  \[
 1-R(\Sigma_{\mathbb{S}_d}) = \frac{1+ \cos\theta_1}{1+ \cos\varphi_1^*},
\]
where $\varphi_1^*$ is the solution of \[
\begin{cases}
    \varphi_1^* = \sum_{j = 2}^d \varphi_j^* \\
    \\
    \cos \varphi_j^* = 1 - \frac{1-\cos\theta_j}{1+\cos\theta_1}(1+\cos\varphi_1^*), \quad \text{ for all }  j \in \{2, \ldots, d\}.
\end{cases}
\]
For further properties of $R(\cdot)$ in the case of a $d$-cycle, we refer the reader to Appendix \ref{sec:extra_prop_R}.

\subsection{Block cycles}\label{subsec:block}

So far, we have studied with particular care the case of a $d$-cycle, which is a relatively simple high-dimensional setting, since it is a collection of $d$ two-dimensional distributions. We now describe an evolution of this setting, where we consider a block-matrix version of the $3$-cycle. In this case the number of variables per missingness pattern is large and we will see that the minimax separation rates are correspondingly much larger than in the $d$-cycle, though the number of variables is of the same order.
\begin{thm}\label{block_minimax}
        Let $\mathbb{S} = \{[2d], [d]\cup([3d]\setminus[2d]),[3d]\setminus[d]\}$ for some $d \geq 1$. Writing $n_{\mathbb{S}} = (n_1, n_2, n_3)$ for the sample sizes within each pattern, there exists a universal constant $c_1 > 0$ such that
        \[
    \rho^* \geq c_1 \sqrt{\frac{d}{(n_1 \wedge n_2) \log^4 (ed)}}
    \]
    whenever $n_1 \wedge n_2 \geq d/2$. 
\end{thm}
This result shows that, up to logarithmic factors in $d$, the minimax separation rates for this testing problem are the same as the minimax estimation rates for estimating $\Sigma_\mathbb{S}$ in the operator norm distance. This is related to the fact that $R(\Sigma_\mathbb{S})$ is a non-smooth functional of $\Sigma_\mathbb{S}$. Indeed, the following result shows that we can construct examples of $\Sigma_\mathbb{S}$ such that $R(\Sigma_\mathbb{S})$ can be bounded below using the function $x \mapsto \max(0,x)$; see below for more discussion of the relevant literature.
\begin{prop}\label{prop:block_3_cycle}
    Consider the set of patterns $\mathbb{S} = \{[2d], [d]\cup([3d]\setminus[2d]),[3d]\setminus[d]\}$ for some $d \geq 1$, and suppose that \[
        \Sigma_{\mathbb{S}} = \Bigg\{\begin{pmatrix}
            I_d & P \\
            P^T & I_d
        \end{pmatrix},\begin{pmatrix}
            I_d & -P \\
            -P^T & I_d
        \end{pmatrix}, \begin{pmatrix}
            I_d & \beta I_d \\
            \beta I_d & I_d
        \end{pmatrix}\Bigg\},
        \]
        for some $P \in \mathbb{R}^{d \times d}$ such that $\|P\|_2 \leq 1$ and some $\beta \in [0,1]$. Then:
        \begin{itemize}
            \item[(i)] $R(\Sigma_{\mathbb{S}}) = 0$ if and only if $\|P\|_2^2 \leq \frac{1-\beta}{2}$,
            \item[(ii)] $R(\Sigma_{\mathbb{S}}) \geq \frac{3}{4d}\sum_{j = 1}^d (\sigma_j^2(P) -  \frac{1-\beta}{2})_+$, where $\sigma_j(P)$ is the $j$-th singular value of $P$.
        \end{itemize}
\end{prop}
This shows that, for $\Sigma_\mathbb{S}$ of the form above, we can relate our testing problem to the problem of testing whether the vector of squared singular values of $P$ belongs to the orthant $(-\infty,(1-\beta)/2]^d$, or is separated from it in the $\ell_1$ distance. In a Gaussian location model a similar problem, measuring separation with the $\ell_2$ distance, was considered by~\citet{carpentier18convex}, and part of our lower bound construction is inspired by this work. However, the consideration of singular values of matrices rather than Gaussian means means  that new technical tools are required. In this regard, the techniques of \cite{verzelen2021schatten}, who consider the estimation of quantities of the form $\sum_{j=1}^d \sigma_j(P)^q$ for $q >0$, are useful. We also mention that such problem are related to the estimation of $\ell_1$ distances, for which good references include \citet{cai2011testing} and \citet{weissman16}.


\section{Simulations}\label{sec:simul}
In this section, we empirically validate the performance of the bootstrap test described in Algorithm \ref{alg:Corrbootstrap_test}. To also detect departures from the null hypothesis caused by inconsistencies in either the means or variances, we introduce two separate bootstrap procedures addressing these aspects individually. Specifically, Algorithm \ref{alg:Meanbootstrap_test}, which provides the $p$-value $p_M$, is designed to detect inconsistencies in the collection of means $\mu_\mathbb{S}$, while Algorithm \ref{alg:Varbootstrap_test} returns the $p$-value $p_V$, focusing on inconsistencies in the variances $\sigma^2_\mathbb{S}$. To create a more comprehensive test, we propose an \textit{omnibus} procedure that combines these three $p$-values using a Bonferroni correction, which consists in rejecting the null whenever $\min\{p_R, p_V, p_M \} < \alpha/3$. 
Although alternative methods for combining $p$-values, such as Fisher’s method \citep{fisher48method}, could be considered, we chose this approach for its simplicity and its better control of the Type-I error. Furthermore, there are various alternative approaches for measuring inconsistency in means and variances, but we chose these because they align with the spirit of Algorithm \ref{alg:Corrbootstrap_test}. Additionally, this is a relatively classical problem since it reduces to testing the equality of means and variances. Consequently, alternative methods might achieve better practical performance. However, this is beyond the scope of our work, as the real novelty lies in addressing the more challenging problem of testing the compatibility of correlation matrices.  

Will will compare Algorithm \ref{alg:Corrbootstrap_test} (represented by the purple line in the plots) and the \textit{omnibus} approach (represented by the blue line in the plots) with Little's test \citep{little1988test}. Little's test can be applied when all pairs of variables are observed together, so that the EM algorithm~\citep{rubin77EM} can be applied to find estimators $\hat{\mu}$ and $\hat{\Lambda}$ of the mean and covariance matrix of the data under the null hypothesis of MCAR. Little's test is a generalised likelihood ratio test whose validity is based on the assumption that the data $(X_{S,i} : S \in \mathbb{S}, i \in [n_S])$ are Gaussian. Define \[
d_\mu^2= \sum_{S \in \mathbb{S}} n_S (\bar{X_S} - \hat{\mu}_{|S})\Lambda_{|S}^{-1}(\bar{X_S} - \hat{\mu}_{|S}),
\]
\[
d^2_{\mathrm{cov}} =  \sum_{S \in \mathbb{S}} n_S[\operatorname{tr}(\hat{\Omega}_S \hat{\Lambda}_{|S}^{-1})-|S|-\log|\hat{\Omega}_S| + \log|\hat{\Lambda}_{|S}|],
\]
and
\[
d^2_{\mathrm{aug}} = d_\mu^2+ d^2_{\mathrm{cov}}.
\]
Then, under MCAR, $d^2_{\mathrm{aug}}$ converges in law to a $\chi^2$-distribution with \[
f = \sum_{S \in \mathbb{S}} \frac{1}{2}|S|(|S|+3) - \frac{1}{2}d(d+3),
\]
degrees of freedom by Wilks' theorem. Based on these asymptotic results, Little's test rejects MCAR if and only if $d^2_{\mathrm{aug}} > \chi^2_f(1-\alpha)$, 
where $\chi^2_f(1-\alpha)$ is such that $\mathbb{P}\{W \geq \chi^2_f(1-\alpha)\} = \alpha$, and where $W$ is $\chi^2$-distributed with $f$ degrees of freedom. Using similar asymptotics, one can define a test based on $d^2_{\mathrm{cov}}$, which ignores the means and only considers the partial covariance matrices, and another one based only on $d^2_\mu$, which discards the collection of covariance matrices and makes use of the means only. For the test based on $d_\mu^2$ we will use the R-function \texttt{mcar\_test} from the R-package \texttt{naniar} \citep{naniar}, while the other two tests based on $d^2_{\mathrm{aug}}$ and $d^2_{\mathrm{cov}}$ can be found in the R-package \texttt{MCARtest} \citep{berrett2022MCARtest} under the name \texttt{little\_test}. In the following subsections, we will compare our procedures with Little's test based on $d^2_\mu$ (represented by the green line in the plots) and $d^2_{\mathrm{aug}}$ (represented by the black line in the plots). We will also include a \textit{combined} approach (represented by the orange line in the plots), which checks the compatibility of the covariance matrices using $\min\{p_R,p_V\}$ but checks the consistency of the means using $d_\mu^2$ instead of Algorithm \ref{alg:Meanbootstrap_test}. The $p$-values are again combined using a Bonferroni correction. Section \ref{sec:simul_teory} is focused on the case where $\mathbb{S}$ is a $d$-cycle, which is the best-studied theoretical setting we considered in the paper, while Section \ref{sec:simul_practice} is closer to real-world data applications.

\begin{algorithm}[!htbp]
\caption{Means only: MCAR bootstrap test checking consistency of means}\label{alg:Meanbootstrap_test}
\begin{algorithmic}[1]
\State Given data $X_\mathbb{S}$, discard all patterns $S \in \mathbb{S}$ such that $n_S \leq 10$.
\State Compute $\hat{\mu}_\mathbb{S} = \operatorname{SampleMean}X_\mathbb{S}$, i.e. $\hat{\mu}_{S,j} = n_S^{-1} \sum_{i \in n_S} X_{S, ij}$ for all $S \in \mathbb{S}$ and $j \in S$.
\State Compute $M(\hat{\mu}_{\mathbb{S}}) = \max_{S \in \mathbb{S}} \|\hat{\mu}_S - \hat{\mu}_{|S}\|_1 / \max_{S \in \mathbb{S}} |S|$, where $(\hat{\mu})_j = \hat{\mu}_j = |\mathbb{S}_j|^{-1}\sum_{S\in \mathbb{S}_j} \hat{\mu}_{S,j}$.
\State Rotate the original data $X_\mathbb{S}$, i.e. for all $S \in \mathbb{S}$, for all $i \in [n_S]$ \textbf{do} $\tilde{X}_{S,i} = X_{S,i}-\hat{\mu}_S + \hat{\mu}_{|S}$.
\For{$b \in [B]$} 
    \State For all $S \in \mathbb{S}$, let $\tilde{X}_{S}^{(b)} = (\tilde{X}_{S,i}^{(b)} : i \in [n_S])$ be a nonparametric bootstrap sample from $\tilde{X}_{S}$.
    \State Compute $\hat{\mu}_{\mathbb{S},b} = \operatorname{SampleMean}X_\mathbb{S}^{(b)}$ and  $M(\hat{\mu}_{\mathbb{S},b})$.
\EndFor
\State \Return  $p_M := (1+B)^{-1}(1 + \sum_{b=1}^B \mathbbm{1}\{M(\hat{\mu}_{\mathbb{S},b}) \geq M(\hat{\mu}_{\mathbb{S}})\})$.
\end{algorithmic}
\end{algorithm}

\begin{algorithm}[!htbp]
\caption{Variance only: MCAR bootstrap test checking consistency of variances}\label{alg:Varbootstrap_test}
\begin{algorithmic}[1]
\State Given data $X_\mathbb{S}$, discard all patterns $S \in \mathbb{S}$ such that $n_S \leq 10$.
\State Compute $\hat{\sigma}^2_\mathbb{S} = \operatorname{SampleVar}X_\mathbb{S}$, i.e.~$\hat{\sigma}^2_{S,j} = n_S^{-1}\sum_{i \in n_S}X^2_{S,ij} - \hat{\mu}_{S,j}^2$ for all $S \in \mathbb{S}$ and $j \in  S$.
\State Rescale the data such that $|\mathbb{S}_j|^{-1}\sum_{S\in \mathbb{S}_j} \hat{\sigma}^2_{S,j} = 1$ for all $j \in [d]$.
\State Compute $V(\hat{\sigma}^2_{\mathbb{S}}) = 1 - \min_{S \in \mathbb{S}} \min_{j \in S}\hat{\sigma}^2_{S,j}$.
\State Rotate the original data $X_\mathbb{S}$, i.e. for all $S \in \mathbb{S}$, for all $i \in [n_S]$ \textbf{do} $\tilde{X}_{S,i} = \operatorname{diag}(\hat{\sigma}^{-1}_{S}) X_{S,i}$.
\For{$b \in [B]$} 
    \State For all $S \in \mathbb{S}$, let $\tilde{X}_{S}^{(b)} = (\tilde{X}_{S,i}^{(b)} : i \in [n_S])$ be a nonparametric bootstrap sample from $\tilde{X}_{S}$. 
    \State Compute $\hat{\sigma}^2_{\mathbb{S},b} = \operatorname{SampleVar}\tilde{X}_{\mathbb{S},b}$ and rescale the data such that $|\mathbb{S}_j|^{-1}\sum_{S\in \mathbb{S}_j} \hat{\sigma}^2_{S,b,j} = 1$ for all $j \in [d]$.
    \State Compute $V(\hat{\sigma}^2_{\mathbb{S}, b})$.
\EndFor
\State \Return  $p_V := (1+B)^{-1}(1 + \sum_{b=1}^B \mathbbm{1}\{V(\hat{\sigma}^2_{\mathbb{S},b}) \geq V(\hat{\sigma}^2_{\mathbb{S}})\})$.
\end{algorithmic}
\end{algorithm}

\subsection{Correlation matrices and simulations for $d$-cycles}\label{sec:simul_teory}

We compare Algorithm \ref{alg:Corrbootstrap_test} with Little's procedures in the settings given in Theorem \ref{thm:cycle_minimax}, namely in the case of a $d$-cycle. For our first settings, we set $n_\mathbb{S} = (n_S)_{S \in \mathbb{S}} = (200, \ldots, 200)$, and simulate $X_{\{j,j+1\},i} \overset{\text{i.i.d.}}{\sim} N(\boldsymbol{0}_2, \Sigma_{\{j,j+1 \}})$ for $i \in [200]$ and $j \in [d]$, where \[
\Sigma_{\mathbb{S}_d} = \left\{\begin{pmatrix}
    1 & \cos\theta_1 \\
    \cos\theta_1 & 1
\end{pmatrix}, \ldots,  \begin{pmatrix}
    1 & \cos\theta_d \\
    \cos\theta_d & 1
\end{pmatrix}\right\}, 
\] 
for certain values of $\theta_1, \ldots, \theta_d \in [0, \pi]$. This makes sense only for $d=3$, while for $d \geq 4$ there exists at least one pair of variables that are never observed together, making the EM algorithm to estimate $\hat{\Lambda}$ inapplicable. As for the case $d = 3$, in Figure~\ref{fig:3cycle_gaussian} we set $B = 99$ and $\alpha = 0.05$, and we vary $\theta_1 \in [\theta_2 + \theta_3, (\theta_2 + \theta_3 + \pi)/2]$, with $(\theta_2, \theta_3)$ equal to $(\pi/3, \pi/6)$. We repeat the experiment $H =  500$ times, and report the average decision. 
\begin{figure}[!ht]
\minipage{0.49\textwidth}
  \includegraphics[width=\linewidth]{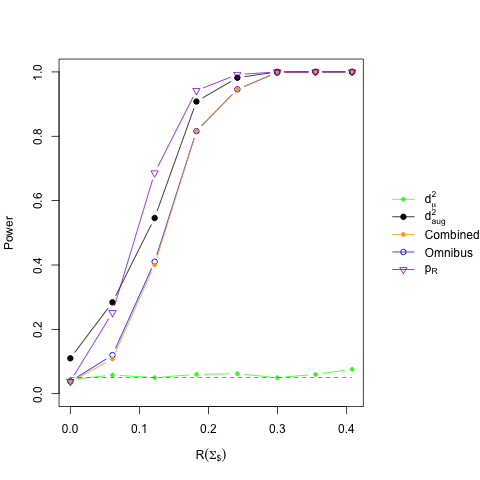}
  \caption{$3$-cycle with Gaussian data. We generate $X_{\{j,j+1\},i} \overset{\text{i.i.d.}}{\sim} N(\boldsymbol{0}_2, \Sigma_{\{j,j+1 \}})$ for all $i \in [200], j \in [3]$, where we vary $\theta_1 \in [\theta_2 + \theta_3, (\theta_2 + \theta_3 + \pi)/2]$, with $(\theta_2, \theta_3)=(\pi/3, \pi/6)$. We repeat the experiment $H =  500$ times, and report the average decision. The nominal level $\alpha = 0.05$ in red, $B = 99$.}\label{fig:3cycle_gaussian}
\endminipage\hfill
\minipage{0.49\textwidth}%
  \includegraphics[width=\linewidth]{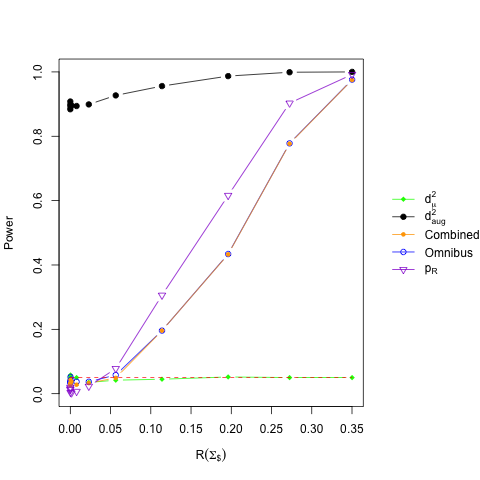}
  \caption{$3$-cycle with lognormal data. We generate $X_{\{j,j+1\},i} \overset{\text{i.i.d.}}{\sim} \log N(\boldsymbol{0}_2, \Sigma_{\{j,j+1 \}})$ for all $i \in [200], j \in [3]$, where we vary $\theta_1 \in [\pi/2 + \pi/12, \pi/12]$, while fixing $(\theta_2, \theta_3)=(3\pi/4, \pi/4)$. We use again $H =  500$, $B = 99$ and $\alpha = 0.05$. $R(\Sigma_\mathbb{S})$ is estimated using an independent sample.}\label{fig:3cycle_lognormal}
\endminipage
\end{figure}
    
The simulation results depicted in Figure~\ref{fig:3cycle_gaussian} show that for $d = 3$, both the \textit{omnibus} and \textit{combined} procedures perform similarly to Little's test based on $d^2_{\mathrm{aug}}$, though the Type I error of Little's test is slightly inflated. This outcome supports a conjecture in \cite{little1988test}, where it is suggested that even under normality, the asymptotic null distribution of $d^2_\mathrm{aug}$ is unlikely to be reliable unless the sample size is large.  Notably, Algorithm \ref{alg:Corrbootstrap_test} demonstrates the highest power, which is not surprising since it is specifically designed to detect incompatible correlation matrices. As expected, the test based on $d^2_\mu$ shows no power, as it is only sensitive to inconsistencies in the means, which are consistent in this particular scenario. For higher dimensions ($d \geq 4$), Little's test cannot be applied, while our test remains valid since it has no constraints on $\mathbb{S}$. In Figure \ref{fig:d_cycle_simul}, we show the power function of our bootstrap tests in the case of a $d$-cycle, with $d \in \{100, 200\}$, with $\theta_2 = \ldots = \theta_d = \frac{\pi}{2(d-1)}$, and varying $\theta_1$ in $[\pi/2, 5\pi/8]$. We repeat the procedure $H = 100$ times, and report the average decision as an estimate of the power function.
\begin{figure}[!htbp]
     \begin{center}
        \subfigure{%
            \label{fig:100_cycle}
            \includegraphics[width=0.5\textwidth, height=0.35\textheight]{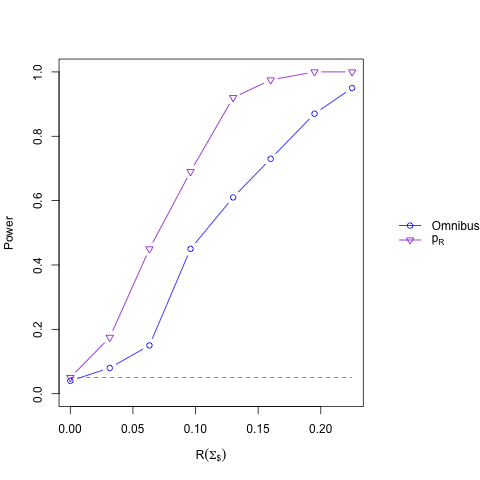}

            \label{fig:200_cycle}
           \includegraphics[width=0.5\textwidth, height=0.35\textheight]{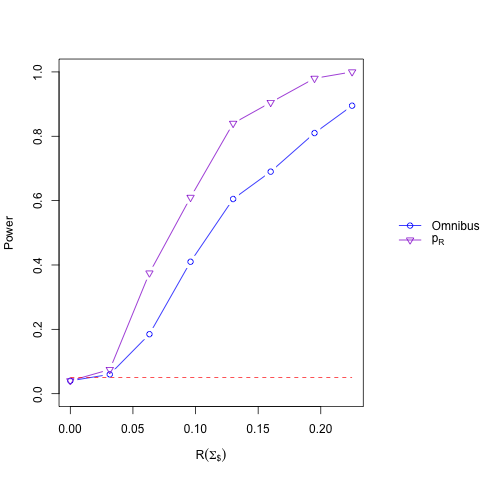}
        }
    \end{center}
    \caption{Simulation of the power functions of our method with $B=99$(blue) for $d = 100$ (left), and $d = 200$ (right), with Gaussian data. In each example, we fix $\theta_2 = \ldots = \theta_d = \frac{\pi}{2(d-1)}$, and vary $\theta_1$ in $[\pi/2, 5\pi/8]$. For each of this setting, we repeat the experiment $H =  100$ times, and report the average decision. The nominal level $\alpha = 0.05$ in red.}%
   \label{fig:d_cycle_simul}
\end{figure}

Our simulations so far have used Gaussian data, so that Little's test is valid. We now repeat our simulations with a heavy-tailed data distribution in order to assess the robustness of the methods. To this end, we consider again a $3$-cycle, and generate $X_{\{j,j+1\},i} \overset{\text{i.i.d.}}{\sim} \log N(\boldsymbol{0}_2, \Sigma_{\{j,j+1 \}})$ for all $i \in [200], j \in [3]$, where $\log N(\boldsymbol{0}_2, \Sigma_{\{j,j+1 \}})$ stands for the log-normal distribution, meaning that if $Y \sim \log N(\boldsymbol{0}_2, \Sigma_{\{j,j+1 \}})$ then $Y_i = e^{Z_i}$, with $Z \sim N(\boldsymbol{0}_2, \Sigma_{\{j,j+1 \}})$. Here we vary $\theta_1 \in [\pi/2 + \pi/12, \pi/12]$, with  $(\theta_2, \theta_3)=(3\pi/4, \pi/4)$. Figure~\ref{fig:3cycle_lognormal} shows the analogue of Figure \ref{fig:3cycle_gaussian}, in the case of artificial data from a multivariate log-normal distribution rather than a Gaussian distribution. Note however that the covariance matrices do not coincide with the original $\Sigma_{\{j,j+1 \}}$, hence on the $x$-axis we decided to estimate $R(\Sigma_\mathbb{S})$ using $R(\hat{\Sigma}_\mathbb{S})$ using an independent sample from the same log-normal distribution. As in our previous results, Little's test based on $d^2_\mu$ exhibits no power but maintains Type-I error control, even when we deviate from the Gaussian setting, which aligns with a conjecture made in \cite{little1988test}. In contrast, Little's test based on $d^2_\mathrm{aug}$ fails to control the Type-I error.
On the other hand, our three tests exhibit similar behavior to that in the Gaussian setting, with only a slight decrease in power.

\subsection{Omnibus approach}\label{sec:simul_practice}
In this subsection we compare the \textit{omnibus} and the \textit{combined} approaches with Little's tests in settings which are closer to real-data applications. In this regard, we generate complete artificial data according to various distributions, and then delete entries using the R package \texttt{missMethods} \citep{R_missMethods}. MCAR data are generated with the function \texttt{delete\_MCAR}, where each entry of the data matrix is deleted independently of the others with probability $p \in (0,1)$. Deviations from the null are generated by partitioning the columns in two groups, group A where the missing values are generated, and group B which determines the missingness mechanism, with two different mechanisms being considered. First, \texttt{delete\_MAR\_1\_to\_x} sets threshold values, splits the rows into two further groups depending on whether columns in group B have values greater or smaller than the threshold, and deletes some entries in columns in group A in a such a way that the probability for a value to be missing in group A divided by the
probability for a value to be missing in group B equals 1 divided by \texttt{x}, with \texttt{x} to be specified as an input parameter. Second, \texttt{delete\_MAR\_rank} deletes each entry in a column of group A with probability proportional to the rank of the same row in the corresponding column of group B. For further details on these functions, and other methods to generate MCAR, MAR, MNAR data, refer to \cite{generateMAR}. These three functions were also chosen in the numerical analysis of a test of MCAR based on U-statistics in \cite{aleksic2023novel}. Before discussing the simulation results, we note that in these settings, under the alternative, the collection of correlation matrices is only mildly incompatible (as verified by numerical inspection), while the inconsistency in means plays a critical role in detecting departures from MCAR. Consequently, the test based on $p_R$ demonstrates low power, unlike our \textit{omnibus} approach and Little's test based on $d^2_\mu$.

\begin{figure}[!htbp]
\minipage{0.32\textwidth}
  \includegraphics[width=\linewidth]{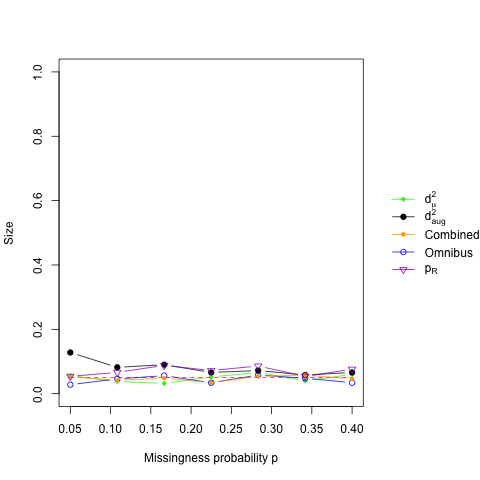}
  \caption{Size for MCAR data generate with \texttt{delete\_MCAR(p)}, for varying probability $p$ of having a missing value. Data from Clayton copula with parameter $1$ and $N(0,1)$ margins. $B = 99, H =  500$.}\label{fig:MCARnorm}
\endminipage\hfill
\minipage{0.32\textwidth}
  \includegraphics[width=\linewidth]{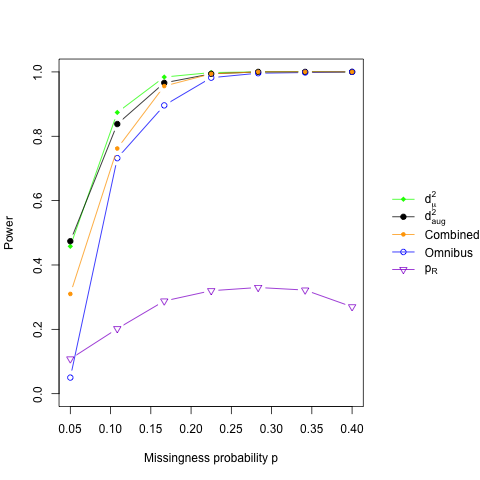}
  \caption{Power function for MAR data generate with \texttt{delete\_MAR\_1\_to\_x(p, x = 9)}, for varying probability $p$ of having a missing value. Data from the same Clayton copula. $B = 99, H =  500$.}\label{fig:MARnorm1}
\endminipage\hfill
\minipage{0.32\textwidth}%
  \includegraphics[width=\linewidth]{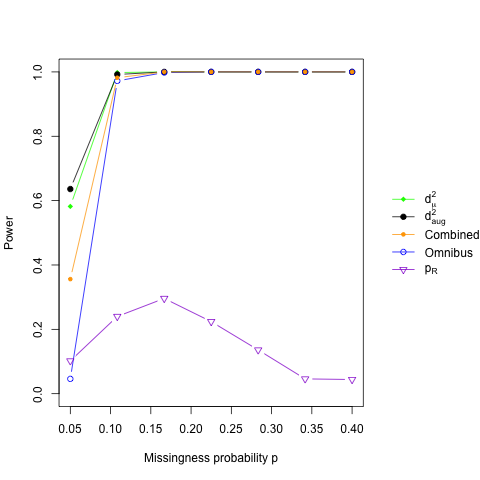}
  \caption{Power function for MAR data generate with \texttt{delete\_MAR\_rank(p)}, for varying probability $p$ of having a missing value. Data from the same Clayton copula. $B = 99, H =  500$.}\label{fig:MARnorm2}
\endminipage
\end{figure}

\begin{figure}[!htbp]
\minipage{0.32\textwidth}
  \includegraphics[width=\linewidth]{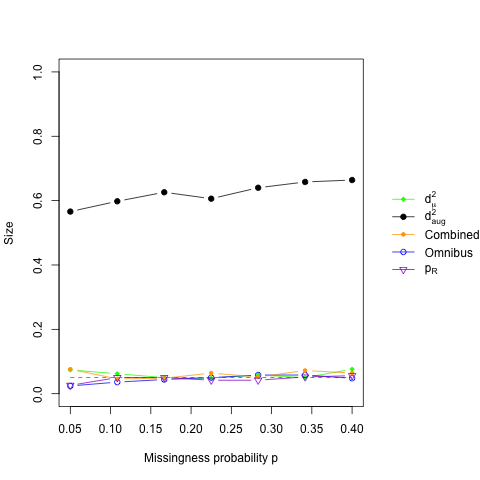}
  \caption{Size for MCAR data generate with \texttt{delete\_MCAR(p)}, for varying probability $p$ of having a missing value. Data from Clayton copula with parameter $1$ and $\operatorname{Exp}(1)$ margins. $B = 99, H =  500$.}\label{fig:MCARexp}
\endminipage\hfill
\minipage{0.32\textwidth}
  \includegraphics[width=\linewidth]{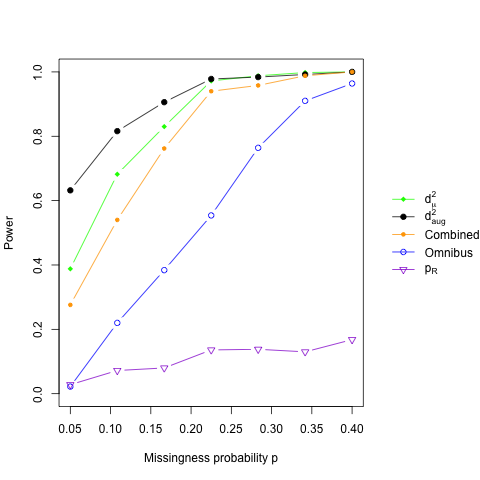}
  \caption{Power function for MAR data generate with \texttt{delete\_MAR\_1\_to\_x(p, x = 9)}, for varying probability $p$ of having a missing value. Data from the same Clayton copula. $B = 99, H =  500$.}\label{fig:MARexp1}
\endminipage\hfill
\minipage{0.32\textwidth}%
  \includegraphics[width=\linewidth]{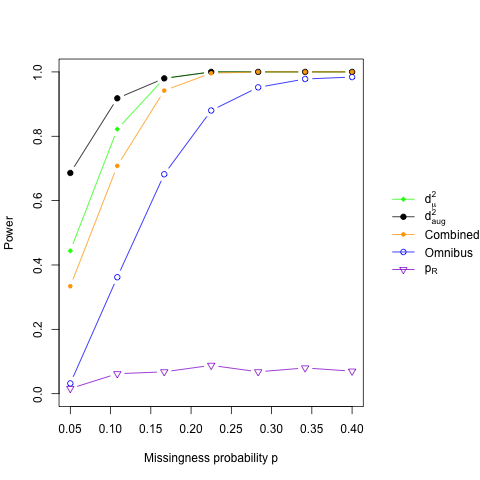}
  \caption{Power function for MAR data generate with \texttt{delete\_MAR\_rank(p)}, for varying probability $p$ of having a missing value. Data from the same Clayton copula. $B = 99, H =  500$.}\label{fig:MARexp2}
\endminipage
\end{figure}

\begin{figure}[!htbp]
\minipage{0.32\textwidth}
  \includegraphics[width=\linewidth]{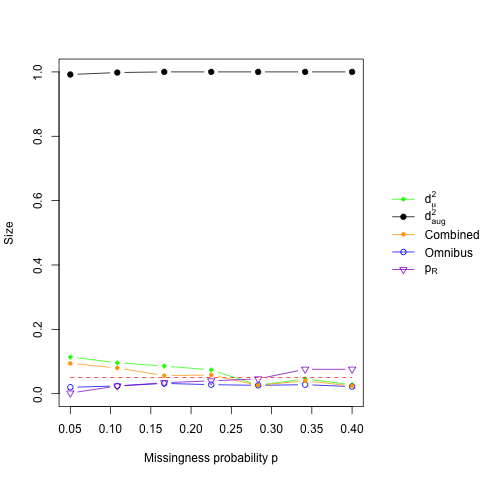}
  \caption{Size for MCAR data generate with \texttt{delete\_MCAR(p)}, for varying probability $p$ of having a missing value. Data from Clayton copula with parameter $1$ and $\log N(0,1)$ margins. $B = 99, H =  500$.}\label{fig:MCARlnorm}
\endminipage\hfill
\minipage{0.32\textwidth}
  \includegraphics[width=\linewidth]{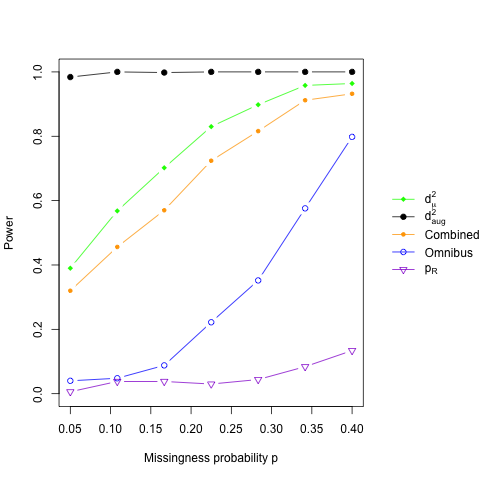}
  \caption{Power function for MAR data generate with \texttt{delete\_MAR\_1\_to\_x(p, x = 9)}, for varying probability $p$ of having a missing value. Data from the same Clayton copula. $B = 99, H =  500$.}\label{fig:MARlnorm1}
\endminipage\hfill
\minipage{0.32\textwidth}%
  \includegraphics[width=\linewidth]{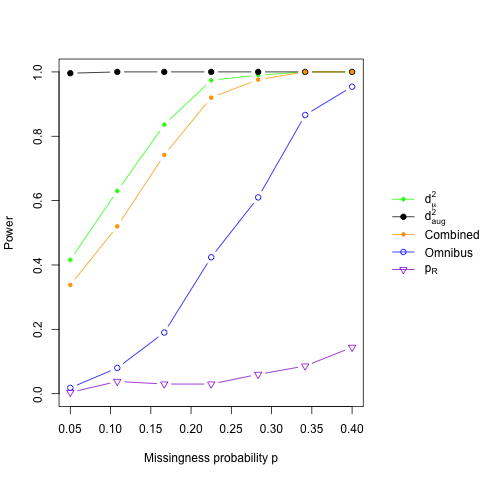}
  \caption{Power function for MAR data generate with \texttt{delete\_MAR\_rank(p)}, for varying probability $p$ of having a missing value. Data from the same Clayton copula. $B = 99, H =  500$.}\label{fig:MARlnorm2}
\endminipage
\end{figure}

For Figures \ref{fig:MCARnorm}, \ref{fig:MARnorm1}, \ref{fig:MARnorm2}, we generated $3$-dimensional datasets of sample size $n = 200$ distributed according to a Clayton copula, with parameter $1$ and Gaussian ($N(0,1)$) margins, using the function \texttt{mvdc} from the R-package \texttt{copula} \cite{R_copula}. For Figure \ref{fig:MCARnorm} we deleted the first two variables with \texttt{delete\_MCAR(p)} for different values of $p \in \{0.05, \ldots, 0.40\}$, in order to get an artificial setting coming from the null. For each $p$, we repeat the simulation $H =  500$ times, and report the average Type-I error. Alternatives to the null were generated using \texttt{delete\_MAR\_1\_to\_x}, with $x = 9$, for Figure \ref{fig:MARnorm1}, and \texttt{delete\_MAR\_rank} for Figure \ref{fig:MARnorm2}. Again, for each $p$, we repeat the simulations $500$ times, and report the average power. In this setting, Figures \ref{fig:MARnorm1} and \ref{fig:MARnorm2} demonstrate that all tests, apart from the one based on $p_R$, perform similarly in terms of power. 
The behaviour of the test based on $p_R$ can be attributed to the fact that, under the alternative, the collection of correlation matrices is only mildly incompatible, with numerical inspection suggesting that incompatibility decreases as $p$ increases. Meanwhile, Figure \ref{fig:MCARnorm} seems to indicate that all tests control the Type-I error at the nominal level $\alpha$, except possibly for Little's test based on $d^2_\mathrm{aug}$. This outcome supports once more the conjecture in \cite{little1988test} where it is suggested that even under normality, the asymptotic null distribution of $d^2_\mathrm{aug}$ is unlikely to be reliable unless the sample size is large. We then move beyond the Gaussian setting and explore different marginal distributions. Specifically, we consider exponential distributions ($\operatorname{Exp}(1)$) in Figures \ref{fig:MCARexp}, \ref{fig:MARexp1}, \ref{fig:MARexp2}, with $d= 3$, and log-normal distributions ($\log N(0,1)$) in Figures \ref{fig:MCARlnorm}, \ref{fig:MARlnorm1}, \ref{fig:MARlnorm2}, with $d = 5$. 
In the setting of log-normal data, Figures \ref{fig:MARexp1}, \ref{fig:MARexp2} show a slight loss in power for our \textit{omnibus} and \textit{combined} approaches. However, this is offset by their better control of the Type-I error (Figure \ref{fig:MCARexp}) compared to Little's test based on $d^2_\mu$, and especially $d^2_\mathrm{aug}$, which consistently fails to recognize the null and always rejects. The results for the exponential setting fall between those of the Gaussian and log-normal settings.

Overall, these simulations demonstrate the effectiveness of our procedures. The \textit{combined} approach could serve as a valid extension of Little's test based on $d^2_\mu$ when both means and covariance matrices are taken into account, rather than just means. Interestingly, incorporating $d^2_\mu$ with Algorithms \ref{alg:Corrbootstrap_test} and \ref{alg:Varbootstrap_test} does not inflate the Type-I error, which is a significant limitation of Little's test based on $d^2_\mathrm{aug}$. On the other hand, while the \textit{omnibus} approach exhibits slightly weaker performance in terms of power, it provides more rigorous control of the Type-I error. Additionally, it offers two further advantages compared to the \textit{combined} approach: (i) it can be applied to any missingness pattern $\mathbb{S}$, even when not all pairs of variables are observed simultaneously; and (ii) it remains effective even when the dimensionality $d$ is large, as it does not depend on the EM algorithm, which is known to encounter issues in practice when $d > 50$ (see the documentation of the R function \texttt{na.test} in \cite{misty24Rpackage}).

\section{Proofs}\label{sec:proof}
\subsection{Proof for Section \ref{sec:bootstrap_test}}
\begin{proof}[Proof of Theorem \ref{thm:bootstrap_guarantees}]
We start by proving (i). Let $P_\mathbb{S} \in \bar{\mathcal{P}}_\mathbb{S}^{-C_{\alpha, 2} }(0)$. Then 
\begin{align*}
    \mathbb{P}_{P_\mathbb{S}}&\left\{p_R \leq \alpha \right\} \\
    & = \mathbb{P}_{P_\mathbb{S}}\left\{1 + \sum_{i=1}^B \mathbbm{1}\{R_{\hat{c}/2}(\hat{\Sigma}_{\mathbb{S},b}) \geq R(\hat{\Sigma}_\mathbb{S})\} \leq \alpha(1+B) , R(\hat{\Sigma}_\mathbb{S}) \leq 3/4 \right\} + \mathbb{P}_{P_\mathbb{S}}\left\{R(\hat{\Sigma}_\mathbb{S}) > 3/4 \right\} \\
    & \leq   \mathbb{P}_{P_\mathbb{S}}\left\{R(\hat{\Sigma}_\mathbb{S}) > 0, R(\hat{\Sigma}_\mathbb{S}) \leq 3/4 \right\} +  \mathbb{P}_{P_\mathbb{S}}\left\{R(\hat{\Sigma}_\mathbb{S}) > 3/4 \right\} \leq 2  \mathbb{P}_{P_\mathbb{S}}\left\{R(\hat{\Sigma}_\mathbb{S}) > 0 \right\}\\
   & \leq 2\mathbb{P}_{P_\mathbb{S}}\left\{ \max_{S \in \mathbb{S}}\|\hat{\Sigma}_\mathbb{S} - \Sigma_\mathbb{S} \|_{2}  > C_{\alpha, 2}  \right\} \leq 2\sum_{S \in \mathbb{S}} \mathbb{P}_{P_\mathbb{S}}\left\{ \|\hat{\Sigma}_S - \Sigma_S \|_{2}  > C_{\alpha, 2}  \right\} \leq \alpha, 
\end{align*}
using \eqref{eq:FindCAlpha} and the computations thereafter. As for part (ii), suppose $R(\Sigma_\mathbb{S}) > 2 \rho$, with $\rho \in (0,1/2)$ to be chosen later, and observe that we can make a Type II error only under the event $\mathcal{B}_\mathbb{S} := \{R(\hat{\Sigma}_\mathbb{S}) \leq 3/4\}$. Then, for all $P_\mathbb{S} \in \mathcal{P}_\mathbb{S}$ and $B \geq 2(1-\alpha)/\alpha$, we can use Markov's inequality to show that 
\begin{align*}
\mathbb{P}_{P_\mathbb{S}}&\left\{p_R > \alpha \right\} \\
    & = \mathbb{P}_{P_\mathbb{S}}\left\{1 + \sum_{i=1}^B \mathbbm{1}\{R_{\hat{c}/2}(\hat{\Sigma}_{\mathbb{S},b}) \geq R(\hat{\Sigma}_\mathbb{S}) \} > \alpha(1+B) ,\mathcal{B}_\mathbb{S}\right\} \leq \frac{B\mathbb{P}_{P_\mathbb{S}}\left\{R_{\hat{c}/2}(\hat{\Sigma}_{\mathbb{S},1}) \geq R(\hat{\Sigma}_\mathbb{S}) , \mathcal{B}_\mathbb{S}\right\}}{\alpha(B+1)-1} \\
    & \leq \frac{2}{\alpha} \mathbb{P}_{P_\mathbb{S}}\left\{R_{\hat{c}/2}(\hat{\Sigma}_{\mathbb{S},1}) \geq R(\hat{\Sigma}_\mathbb{S}) , \mathcal{B}_\mathbb{S} \right\} \leq \frac{2}{\alpha}\left(\mathbb{P}_{P_\mathbb{S}}\left\{ R(\hat{\Sigma}_\mathbb{S}) < \rho , \mathcal{B}_\mathbb{S}\right\} + \mathbb{P}_{P_\mathbb{S}}\left\{R_{\hat{c}/2}(\hat{\Sigma}_{\mathbb{S},1}) \geq \rho , \mathcal{B}_\mathbb{S}\right\} \right) \\
    & \leq \frac{2}{\alpha} \left( \mathbb{P}_{P_\mathbb{S}}\left\{ |R(\hat{\Sigma}_\mathbb{S}) - R(\Sigma_\mathbb{S})| > \rho \right\} + \mathbb{P}_{P_\mathbb{S}}\left\{ |R_{\hat{c}/2}(\hat{\Sigma}_{\mathbb{S},1}) - R_{\hat{c}/2}(\hat{Q}_{\mathbb{S}})|  \geq \rho , \mathcal{B}_\mathbb{S} \right\} \right),
\end{align*}
where in the last inequality we used the fact that $R(\Sigma_\mathbb{S}) > 2 \rho$ and $R_{\hat{c}/2}(\hat{Q}_{\mathbb{S}}) = 0$. It is then enough to find $\rho \in (0,1/2)$ such that 
\begin{align}\label{eq:term_0}
    \mathbb{P}_{P_\mathbb{S}}\left\{ |R(\hat{\Sigma}_\mathbb{S}) - R(\Sigma_\mathbb{S})| > \rho \right\} + \mathbb{P}_{P_\mathbb{S}}\left\{ |R_{\hat{c}/2}(\hat{\Sigma}_{\mathbb{S},1}) - R_{\hat{c}/2}(\hat{Q}_{\mathbb{S}})|  \geq \rho, \mathcal{B}_\mathbb{S} \right\} \leq \frac{\alpha \beta}{2}.
\end{align}
To achieve this, we will draw on results from the proofs of Theorem \ref{thm:oracle_test} and Proposition \ref{prop:corr_concentration}, which are deferred to later sections. In particular, the analysis of the first term in \eqref{eq:term_0} follows directly from Theorem \ref{thm:oracle_test}. For the second term, we will draw on ideas from the proof of Proposition \ref{prop:corr_concentration} to demonstrate that it suffices to control the spectral norm of the covariance matrix of the bootstrap sample, which can be analysed using Proposition \ref{prop:rudelson} in Appendix \ref{sec:technical_ineq}. Throughout the following, to assist the reader in following the argument, we will explicitly reference the equations from the proofs of Theorem \ref{thm:oracle_test} and Proposition \ref{prop:corr_concentration} whenever they are used. Now, if we define the \textit{good set}
\begin{align*}
    \mathcal{A}_\mathbb{S} := & \left\{ \| \hat{\Omega}_\mathbb{S} \|_{2, \mathbb{S}} \leq 2\nu^2, \| \hat{D}_\mathbb{S}^{-1/2} \|_{2, \mathbb{S}} \leq 2/\sigma_\mathrm{min},  \hat{\Sigma}_\mathbb{S} \succeq_\mathbb{S} \frac{c}{2} I_\mathbb{S}, \right. \\
    & \left. \quad \quad 2 \max_{S \in \mathbb{S}}\|\hat{\Sigma}_S^{-1/2} \|_2 \|\hat{D}_S^{-1/2} \|_2  \max_{i \in [n_S]}    \| X_{S,i} - \mu_S \|_2 \leq \underbrace{16 \sqrt{\frac{2\nu^2 |S|}{c \sigma^2_\mathrm{min}}} + 8 \sqrt{\frac{2\nu^2\log (12|\mathbb{S}|n_S/\alpha \beta)}{c \sigma^2_\mathrm{min}}}}_{=:M} \right. \Bigg\},
\end{align*}
we have
\begin{align}\label{eq:term_1}
    \mathbb{P}_{P_\mathbb{S}}&\left\{ |R(\hat{\Sigma}_\mathbb{S}) - R(\Sigma_\mathbb{S})| > \rho \right\} + \mathbb{P}_{P_\mathbb{S}}\left\{ |R_{\hat{c}/2}(\hat{\Sigma}_{\mathbb{S},1}) - R_{\hat{c}/2}(\hat{Q}_{\mathbb{S}})| \geq \rho , \mathcal{B}_\mathbb{S} \right\} \nonumber \\
    & \leq  \mathbb{P}_{P_\mathbb{S}}\left\{ |R(\hat{\Sigma}_\mathbb{S}) - R(\Sigma_\mathbb{S})| > \rho \right\} + \mathbb{P}_{P_\mathbb{S}}\left\{ |R_{\hat{c}/2}(\hat{\Sigma}_{\mathbb{S},1}) - R_{\hat{c}/2}(\hat{Q}_{\mathbb{S}})| > \rho/2 , \mathcal{B}_\mathbb{S}\right\} \nonumber \\
    & \overset{\eqref{eq:R_to_spectral}}{\leq} \mathbb{P}_{P_\mathbb{S}}\left\{ \|\hat{\Sigma}_\mathbb{S} - \Sigma_\mathbb{S}\|_{2, \mathbb{S}} > c\rho/2, \mathcal{A}_\mathbb{S} \right\} + \mathbb{P}_{P_\mathbb{S}}\left\{ \| \hat{\Sigma}_{\mathbb{S},1} - \hat{Q}_{\mathbb{S}} \|_{2, \mathbb{S}} > c \rho/8, \mathcal{A}_\mathbb{S} , \mathcal{B}_\mathbb{S} \right\} + 2\mathbb{P}_{P_\mathbb{S}}\left\{\mathcal{A}_\mathbb{S}^\complement \right\} \nonumber \\
    & \overset{\eqref{eq:ProbSigmaHatSucceq}}{\leq} \mathbb{P}_{P_\mathbb{S}}\left\{ \|\hat{\Sigma}_\mathbb{S} - \Sigma_\mathbb{S}\|_{2, \mathbb{S}} > c\rho/2 \right\} + \mathbb{P}_{P_\mathbb{S}}\left\{ \| \hat{\Sigma}_{\mathbb{S},1} - \hat{Q}_{\mathbb{S}} \|_{2, \mathbb{S}} > c \rho/8, \mathcal{A}_\mathbb{S} , \mathcal{B}_\mathbb{S}\right\} \nonumber + 2\mathbb{P}_{P_\mathbb{S}}\left\{ \|\hat{\Omega}_\mathbb{S} - \Omega_\mathbb{S} \|_{2, \mathbb{S}} > \nu^2 \right\} \nonumber \\
    & \quad \quad + 2\mathbb{P}_{P_\mathbb{S}}\left\{ \|\hat{D}_\mathbb{S}^{-1/2}D_\mathbb{S}^{1/2} - I_\mathbb{S} \|_{2, \mathbb{S}} > 1 \right\} + 2\mathbb{P}_{P_\mathbb{S}}\left\{ \|\hat{\Sigma}_\mathbb{S} - \Sigma_\mathbb{S}\|_{2, \mathbb{S}} > c/2 \right\} \nonumber \\
    & \quad \quad + 2\mathbb{P}_{P_\mathbb{S}}\left\{2 \max_{S \in \mathbb{S}}\|\hat{\Sigma}_S^{-1/2} \|_2 \|\hat{D}_S^{-1/2} \|_2  \max_{i \in [n_S]}    \| X_{S,i} -\mu_S \|_2 > M,  \right\} \nonumber \\
    & \leq \mathbb{P}_{P_\mathbb{S}}\left\{ \|\hat{\Sigma}_\mathbb{S} - \Sigma_\mathbb{S}\|_{2, \mathbb{S}} > c\rho/2 \right\} + \mathbb{P}_{P_\mathbb{S}}\left\{ \| \hat{\Sigma}_{\mathbb{S},1} - \hat{Q}_{\mathbb{S}} \|_{2, \mathbb{S}} > c \rho/8, \mathcal{A}_\mathbb{S} , \mathcal{B}_\mathbb{S} \right\} + 2\mathbb{P}_{P_\mathbb{S}}\left\{ \|\hat{\Omega}_\mathbb{S} - \Omega_\mathbb{S} \|_{2, \mathbb{S}} > \nu^2 \right\} \nonumber \\
    & \quad \quad  + 4 \mathbb{P}_{P_\mathbb{S}}\left\{ \|\hat{D}_\mathbb{S}^{-1/2}D_\mathbb{S}^{1/2} - I_\mathbb{S} \|_{2, \mathbb{S}} > 1 \right\} + 4\mathbb{P}_{P_\mathbb{S}}\left\{ \|\hat{\Sigma}_\mathbb{S} - \Sigma_\mathbb{S}\|_{2, \mathbb{S}} > c/2 \right\} \nonumber \\
    & \quad \quad + 2\mathbb{P}_{P_\mathbb{S}}\left\{ \max_{S \in \mathbb{S}} \max_{i \in [n_S]}    \| X_{S,i} -\mu_S \|_2 >  4 \nu \sqrt{ |S|} + 2 \nu \sqrt{\log (12|\mathbb{S}|n_S/\alpha \beta)} \right\} \nonumber \\
    & \overset{\eqref{eq:conc_corr_to_cov_3},\eqref{eq:conc_corr_to_cov_4}}{\leq} 40\mathbb{P}_{P_\mathbb{S}}\left\{ \|\hat{\Omega}_\mathbb{S} - \Omega_\mathbb{S} \|_{2, \mathbb{S}} > \sigma^4_\mathrm{min}c\rho/48\nu^2 \right\} + \mathbb{P}_{P_\mathbb{S}}\left\{ \| \hat{\Sigma}_{\mathbb{S},1} - \hat{Q}_{\mathbb{S}} \|_{2, \mathbb{S}} > c \rho/8, \mathcal{A}_\mathbb{S} , \mathcal{B}_\mathbb{S} \right\} \nonumber \\
    & \quad \quad + 2\mathbb{P}_{P_\mathbb{S}}\left\{ \max_{S \in \mathbb{S}} \max_{i \in [n_S]}    \| X_{S,i} -\mu_S \|_2 > 4 \nu \sqrt{ |S|} + 2 \nu \sqrt{\log (12|\mathbb{S}|n_S/\alpha \beta)} \right\} \nonumber \\
    & \overset{\text{Prop.} \ref{prop:norm_SG}}{\leq}  40\mathbb{P}_{P_\mathbb{S}}\left\{ \|\hat{\Omega}_\mathbb{S} - \Omega_\mathbb{S} \|_{2, \mathbb{S}} > \sigma^4_\mathrm{min}c\rho/48\nu^2 \right\} + |\mathbb{S}| \max_{S \in \mathbb{S}} \mathbb{P}_{P_\mathbb{S}}\left\{ \| \hat{\Sigma}_{S,1} - \hat{Q}_{S} \|_{2} > c \rho/8, \mathcal{A}_\mathbb{S} , \mathcal{B}_\mathbb{S} \right\} + \alpha \beta/6.  
\end{align}
This shows that it is enough to choose $\rho$ such that 
\begin{align}\label{eq:last_two_bootstrap}
    \max\left(40\mathbb{P}_{P_\mathbb{S}}\left\{ \|\hat{\Omega}_\mathbb{S} - \Omega_\mathbb{S} \|_{2, \mathbb{S}} > \sigma^4_\mathrm{min}c\rho/48\nu^2 \right\}, |\mathbb{S}| \max_{S \in \mathbb{S}} \mathbb{P}_{P_\mathbb{S}}\left\{ \| \hat{\Sigma}_{S,1} - \hat{Q}_{S} \|_{2} > c \rho/8, \mathcal{A}_\mathbb{S} , \mathcal{B}_\mathbb{S} \right\} \right) \leq \alpha \beta/6.
\end{align}
As for the first term in the maximum, an analogous argument to that employed in the proof of Theorem \ref{thm:oracle_test} ensures that $40\mathbb{P}_{P_\mathbb{S}}\left\{ \|\hat{\Omega}_\mathbb{S} - \Omega_\mathbb{S} \|_{2, \mathbb{S}} > \sigma^4_\mathrm{min}c\rho/48\nu^2 \right\} \leq \alpha \beta / 6$ if $\rho \geq C_{\alpha \beta, c}$ 
 for a sufficiently large absolute constant $K_1 > 0$. As for the latter term in \eqref{eq:last_two_bootstrap}, we focus on finding a $\rho$ 
such that $\mathbb{P}_{P_\mathbb{S}}\left\{ \| \hat{\Sigma}_{S,1} - \hat{Q}_{S} \|_{2} > c \rho/8, \mathcal{A}_\mathbb{S} , \mathcal{B}_\mathbb{S} \right\} \leq \alpha \beta/6|\mathbb{S}|$. This would conclude the proof. Now, first observe that under $\mathcal{B}_\mathbb{S}$, \eqref{eq:dual_decomposition} implies that 
\[
\| \hat{Q}_S \|_2 \leq \frac{1}{1 - R(\hat{\Sigma}_\mathbb{S})} \| \hat{\Sigma}_S \|_2 \leq 4 \| \hat{\Sigma}_S \|_2  = 4 \| \hat{D}_S^{-1/2} \hat{\Omega}_S \hat{D}_S^{-1/2} \|_2 \leq 4 \| \hat{D}_S^{-1/2} \|_2^2 \| \hat{\Omega}_S \|_2,
\]
which is upper bounded by $32\nu^2/\sigma^2_\mathrm{min}$ under $\mathcal{A}_\mathbb{S}$. Hence, using similar steps as in \eqref{eq:conc_corr_to_cov} and \eqref{eq:conc_corr_to_cov_3}, for all $S \in \mathbb{S}$ and $x \in [0,1]$ we get
\begin{align}\label{eq:cov_bootstrap}
  & \mathbb{P}_{P_\mathbb{S}}\left\{ \| \hat{\Sigma}_{S,1} - \hat{Q}_{S} \|_{2} > x, \mathcal{A}_\mathbb{S} , \mathcal{B}_\mathbb{S} \right\} \nonumber \\
  & \overset{\sim \eqref{eq:conc_corr_to_cov}}{\leq} \mathbb{P}_{P_\mathbb{S}}\left\{ \| \hat{\Omega}_{S,1} - \hat{Q}_{S} \|_{2} > x/2, \mathcal{A}_\mathbb{S} , \mathcal{B}_\mathbb{S} \right\} + \mathbb{P}_{P_\mathbb{S}}\left\{ \| \hat{D}_{S,1}^{-1/2} - I_{S} \|_{2}(1+ \| \hat{D}_{S,1}^{-1/2} \|_{2} ) > \sigma^2_\mathrm{min} x/64\nu^2, \mathcal{A}_\mathbb{S} , \mathcal{B}_\mathbb{S} \right\} \nonumber \\
  & \leq \mathbb{P}_{P_\mathbb{S}}\left\{ \| \hat{\Omega}_{S,1} - \hat{Q}_{S} \|_{2} > x/2, \mathcal{A}_\mathbb{S} , \mathcal{B}_\mathbb{S} \right\} + \mathbb{P}_{P_\mathbb{S}}\left\{ \| \hat{D}_{S,1}^{-1/2} - I_{S} \|_{2} > \sigma^2_\mathrm{min} x/192\nu^2, \mathcal{A}_\mathbb{S} , \mathcal{B}_\mathbb{S} \right\} \nonumber \\
  & \hspace{3cm} + \mathbb{P}_{P_\mathbb{S}}\left\{ \| \hat{D}_{S,1}^{-1/2} \|_{2} > 2, \mathcal{A}_\mathbb{S} , \mathcal{B}_\mathbb{S} \right\} \nonumber \\
  & \leq \mathbb{P}_{P_\mathbb{S}}\left\{ \| \hat{\Omega}_{S,1} - \hat{Q}_{S} \|_{2} > x/2, \mathcal{A}_\mathbb{S} , \mathcal{B}_\mathbb{S} \right\} + 2\mathbb{P}_{P_\mathbb{S}}\left\{ \| \hat{D}_{S,1}^{-1/2} - I_{S} \|_{2} > \sigma^2_\mathrm{min} x/192\nu^2, \mathcal{A}_\mathbb{S} , \mathcal{B}_\mathbb{S} \right\} \nonumber \\
  & \overset{\sim \eqref{eq:conc_corr_to_cov_3}}{\leq} \mathbb{P}_{P_\mathbb{S}}\left\{ \| \hat{\Omega}_{S,1} - \hat{Q}_{S} \|_{2} > x/2, \mathcal{A}_\mathbb{S} , \mathcal{B}_\mathbb{S} \right\} + 4 \mathbb{P}_{P_\mathbb{S}}\left\{ \| \hat{\Omega}_{S,1} - \hat{Q}_{S} \|_{2} > \sigma^2_\mathrm{min} x/384\nu^2, \mathcal{A}_\mathbb{S} , \mathcal{B}_\mathbb{S} \right\} \nonumber \\
  & \leq  5 \mathbb{P}_{P_\mathbb{S}}\left\{ \| \hat{\Omega}_{S,1} - \hat{Q}_{S} \|_{2} > \sigma^2_\mathrm{min} x/384\nu^2, \mathcal{A}_\mathbb{S} , \mathcal{B}_\mathbb{S} \right\},
\end{align}
which shows that it is enough to give a concentration bound for the sample covariance matrix of the bootstrap sample. In this regard, 
recall that Algorithm \ref{alg:Corrbootstrap_test} generates a bootstrap sample from the rotated data $\tilde{X}_\mathbb{S}:=(\tilde{X}_S : S \in \mathbb{S})$, which for all $S \in \mathbb{S}$ and $i \in [n_S]$ is of the form $\tilde{X}_{S,i} = \hat{Q}_S^{1/2} \hat{\Sigma}_S^{-1/2} \hat{D}_S^{-1/2} (X_{S,i} - \hat{\mu}_S)$, where $\hat{Q}_S$ comes from the dual decomposition $\hat{\Sigma}_\mathbb{S} = (1-R(\hat{\Sigma}_\mathbb{S}))\hat{Q}_\mathbb{S} + R(\hat{\Sigma}_\mathbb{S})\hat{\Sigma}'_\mathbb{S}$. As a result, for all $S \in \mathbb{S}$, the sample correlation of $\tilde{X}_S$ coincides with the sample covariance, and is equal to $\hat{Q}_S$. Also, observe that, since $\tilde{X}_{S,i} = \hat{Q}_S^{1/2} \hat{\Sigma}_S^{-1/2} \hat{D}_S^{-1/2} (X_{S,i} - \hat{\mu}_S)$, we have  $\tilde{X}_{S, i}^{(1)} = \hat{Q}_S^{1/2} \hat{\Sigma}_S^{-1/2} \hat{D}_S^{-1/2} (X_{S,i}^{(1)} - \hat{\mu}_S)$, where $X_{S,i}^{(1)}$ is a bootstrap sample from the original set of data $X_S = (X_{S,1}, \ldots, X_{S, n_S})$. This implies that 
\begin{align*}
     \hat{\Omega}_{S,1} & - \hat{Q}_S = n_S^{-1}\sum_{i = 1}^{n_S} \tilde{X}_{S,i}^{(1)} \tilde{X}_{S,i}^{(1)T} - \left(n_S^{-1}\sum_{i = 1}^{n_S} \tilde{X}_{S,i}^{(1)}\right)\left(n_S^{-1}\sum_{i = 1}^{n_S} \tilde{X}_{S,i}^{(1)}\right)^T - \hat{Q}_S  \\
    & = n_S^{-1}\sum_{i = 1}^{n_S} \hat{Q}_S^{1/2} \hat{\Sigma}_S^{-1/2} \hat{D}_S^{-1/2} (X_{S,i}^{(1)} - \hat{\mu}_S) (X_{S,i}^{(1)} - \hat{\mu}_S)^T \hat{D}_S^{-1/2} \hat{\Sigma}_S^{-1/2} \hat{Q}_S^{1/2}  \\
    & \quad \quad \quad \quad - \left(n_S^{-1}\sum_{i = 1}^{n_S} \hat{Q}_S^{1/2} \hat{\Sigma}_S^{-1/2} \hat{D}_S^{-1/2} (X_{S,i}^{(1)} - \hat{\mu}_S) \right)\left(n_S^{-1}\sum_{i = 1}^{n_S} \hat{Q}_S^{1/2} \hat{\Sigma}_S^{-1/2} \hat{D}_S^{-1/2} (X_{S,i}^{(1)} - \hat{\mu}_S) \right)^T - \hat{Q}_S  \\
    & =  \hat{Q}_S^{1/2}\left\{n_S^{-1}\sum_{i = 1}^{n_S}  [\hat{\Sigma}_S^{-1/2} \hat{D}_S^{-1/2} (X_{S,i}^{(1)} - \hat{\mu}_S)] [\hat{\Sigma}_S^{-1/2} \hat{D}_S^{-1/2} (X_{S,i}^{(1)} - \hat{\mu}_S)]^T  - I_{|S|} \right. \\
    & \left. \quad \quad \quad \quad - \left(n_S^{-1}\sum_{i = 1}^{n_S} \hat{\Sigma}_S^{-1/2} \hat{D}_S^{-1/2} (X_{S,i}^{(1)} - \hat{\mu}_S) \right)\left(n_S^{-1}\sum_{i = 1}^{n_S} \hat{\Sigma}_S^{-1/2} \hat{D}_S^{-1/2} (X_{S,i}^{(1)} - \hat{\mu}_S)  \right)^T \right\} \hat{Q}_S^{1/2}.
\end{align*}
We thus get for all $x \in [0,1]$
\begin{align}\label{eq:cov_bootstrap_2}
   \mathbb{P}&\left\{\|\hat{\Omega}_{S,1} - \hat{Q}_S \|_2 > x, \mathcal{A}_\mathbb{S} , \mathcal{B}_\mathbb{S} \right\} \nonumber \\
   & \leq \mathbb{P}\left\{ \|\hat{Q}_S \|_2 \| n_S^{-1}\sum_{i = 1}^{n_S}  [\hat{\Sigma}_S^{-1/2} \hat{D}_S^{-1/2} (X_{S,i}^{(1)} - \hat{\mu}_S)] [\hat{\Sigma}_S^{-1/2} \hat{D}_S^{-1/2} (X_{S,i}^{(1)} - \hat{\mu}_S)]^T  - I_{|S|} \|_2 > x/2, \mathcal{A}_\mathbb{S} , \mathcal{B}_\mathbb{S} \right\} \nonumber \\
   & + \mathbb{P}\left\{ \|\hat{Q}_S \|_2 \|\hat{\Sigma}_S^{-1/2} \|_2^2 \|\hat{D}_S^{-1/2} \|_2^2  \|n_S^{-1}\sum_{i = 1}^{n_S} (X_{S,i}^{(1)} - \hat{\mu}_S) \|_2^2 > x/2, \mathcal{A}_\mathbb{S} , \mathcal{B}_\mathbb{S} \right\} \nonumber \\
   & \leq \mathbb{P}\left\{ \| n_S^{-1}\sum_{i = 1}^{n_S}  [\hat{\Sigma}_S^{-1/2} \hat{D}_S^{-1/2} (X_{S,i}^{(1)} - \hat{\mu}_S)] [\hat{\Sigma}_S^{-1/2} \hat{D}_S^{-1/2} (X_{S,i}^{(1)} - \hat{\mu}_S)]^T  - I_{|S|} \|_2 > \sigma^2_\mathrm{min} x/64\nu^2, \mathcal{A}_\mathbb{S} \right\} \nonumber \\
   & + \mathbb{P}\left\{ \|n_S^{-1}\sum_{i = 1}^{n_S} (X_{S,i}^{(1)} - \hat{\mu}_S) \|_2^2 > c \sigma^4_\mathrm{min} x/512\nu^2, \mathcal{A}_\mathbb{S} \right\}.
\end{align}
As for the first term in \eqref{eq:cov_bootstrap_2}, observe that, under $\mathcal{A}_\mathbb{S}$ and conditionally on the data $X_\mathbb{S}$, we have 
\begin{align*}
  \|&\hat{\Sigma}_S^{-1/2} \hat{D}_S^{-1/2} (X_{S,i}^{(1)} -  \hat{\mu}_{S}) \|_2 \leq \max_{i \in [n_S]}    \|\hat{\Sigma}_S^{-1/2} \hat{D}_S^{-1/2} (X_{S,i} -  \hat{\mu}_{S}) \|_2 \leq \|\hat{\Sigma}_S^{-1/2} \|_2 \| \hat{D}_S^{-1/2} \|_2  \max_{i \in [n_S]}    \| X_{S,i} -  \hat{\mu}_{S} \|_2 \\
  & \leq \|\hat{\Sigma}_S^{-1/2} \|_2 \| \hat{D}_S^{-1/2} \|_2 \left\{\|\hat{\mu}_{S} - \mu_S \|_2 + \max_{i \in [n_S]}    \| X_{S,i} -  \mu_S \|_2  \right\} \\
  & = \|\hat{\Sigma}_S^{-1/2} \|_2 \| \hat{D}_S^{-1/2} \|_2 \left\{\|n_S^{-1}\sum_{i = 1}^{n_S} (X_{S,i} - \mu_S) \|_2 + \max_{i \in [n_S]}    \| X_{S,i} -  \mu_S \|_2  \right\} \\
  & \leq 2 \|\hat{\Sigma}_S^{-1/2} \|_2 \| \hat{D}_S^{-1/2} \|_2  \max_{i \in [n_S]}    \| X_{S,i} -  \mu_S \|_2 \leq M \quad \text{ a.s.}.
\end{align*}
This, together with Proposition \ref{prop:rudelson} in Appendix \ref{sec:technical_ineq} implies that there exists a universal constant $K_2 > 0$ such that 
\begin{align}\label{cov_bootstrap_3}
    \mathbb{P}&\left\{ \| n_S^{-1}\sum_{i = 1}^{n_S}  [\hat{\Sigma}_S^{-1/2} \hat{D}_S^{-1/2} (X_{S,i}^{(1)} - \hat{\mu}_S)] [\hat{\Sigma}_S^{-1/2} \hat{D}_S^{-1/2} (X_{S,i}^{(1)} - \hat{\mu}_S)]^T  - I_{|S|} \|_2 > \sigma^2_\mathrm{min} x/64\nu^2, \mathcal{A}_\mathbb{S} \right\} \nonumber \\
    & \leq \mathbb{E}\left[\mathbb{P}\left\{ \| n_S^{-1}\sum_{i = 1}^{n_S}  [\hat{\Sigma}_S^{-1/2} \hat{D}_S^{-1/2} (X_{S,i}^{(1)} - \hat{\mu}_S)] [\hat{\Sigma}_S^{-1/2} \hat{D}_S^{-1/2} (X_{S,i}^{(1)} - \hat{\mu}_S)]^T  - I_{|S|} \|_2 > \sigma^2_\mathrm{min} x/64\nu^2, \mathcal{A}_\mathbb{S}  \mid X_\mathbb{S}\right\} \right] \nonumber \\
    & \leq 2 \exp \left\{-\frac{K_2 \sigma^4_\mathrm{min} n_S x^2}{\nu^4 M^2 \log n_S}\right\}.
\end{align}
As for the second term in \eqref{eq:cov_bootstrap_2}, calling $W = (W_1, \ldots, W_{n_S}) \sim \operatorname{Multinomial}(n_S, [n_S], (n_S^{-1}, \ldots, n_S^{-1}))$, we have 
\begin{align}\label{cov_bootstrap_4}
    \mathbb{P}&\left\{ \|n_S^{-1}\sum_{i = 1}^{n_S} (X_{S,i}^{(1)} - \hat{\mu}_S) \|_2^2 > c \sigma^4_\mathrm{min} x/512\nu^2, \mathcal{A}_\mathbb{S} \right\} \nonumber \\
    & = \mathbb{P}\left\{ \| n_S^{-1} \sum_{i = 1}^{n_S} (W_i - 1) (X_{S,i} - \mu_S) \|_2 > \sqrt{c \sigma^4_\mathrm{min} x/512\nu^2}, \mathcal{A}_\mathbb{S} , \max_{i \in [n_S]} |W_i - 1| \leq 3 \log(60|\mathbb{S}|n_S/\alpha\beta) \right\} \nonumber \\
    & \quad \quad + \mathbb{P}\left\{ \max_{i \in [n_S]} |W_i - 1| > 3 \log(60|\mathbb{S}|n_S/\alpha\beta) \right\} \nonumber \\
    & = \mathbb{E}\left[ \mathbb{P}\left\{ \| n_S^{-1} \sum_{i = 1}^{n_S} (W_i - 1) (X_{S,i} - \mu_S) \|_2 > \sqrt{c \sigma^4_\mathrm{min} x/512\nu^2}, \mathcal{A}_\mathbb{S} , \max_{i \in [n_S]} |W_i - 1| \leq 3 \log(60|\mathbb{S}|n_S/\alpha\beta) \mid W \right\} \right] \nonumber \\
    & \quad \quad + \mathbb{P}\left\{ \max_{i \in [n_S]} |W_i - 1| > 3 \log(60|\mathbb{S}|n_S/\alpha\beta) \right\} \overset{\text{ Prop. } \ref{prop:norm_SG}}{\leq }  5^{|S|}\exp \left\{-\frac{ n_S c \sigma^4_\mathrm{min} x}{192^2 \nu^4 \log^2(60|\mathbb{S}|n_S/\alpha\beta)}\right\} +\alpha \beta/60|\mathbb{S}|,
\end{align}
where in the last inequality we used Proposition \ref{prop:ineq_binomial} in Appendix \ref{sec:technical_ineq}, which ensures that $\mathbb{P}\left\{\max_{i \in [n_S]} |W_i - 1| > t \right\} \leq n_S e^t/(t+1)^{t+1} \leq n_S e^{-t/3}$ for $t \geq 1$. Now, combining \eqref{eq:cov_bootstrap}, \eqref{eq:cov_bootstrap_2}, \eqref{cov_bootstrap_3}, \eqref{cov_bootstrap_4} gives for all $S \in \mathbb{S}$
\begin{align*}
 \mathbb{P}_{P_\mathbb{S}}&\left\{ \| \hat{\Sigma}_{S,1} - \hat{Q}_{S} \|_{2} > c \rho/8, \mathcal{A}_\mathbb{S} \right\} \leq 10 \exp \left\{-\frac{K_2 \sigma^{10}_\mathrm{min} c^3 n_S \rho^2}{\nu^{10} \log n_S \{|S| + \log(|\mathbb{S}|n_S/\alpha\beta) \}}\right\} \\
 & \hspace{5cm} + 5^{|S|+1}\exp \left\{-\frac{ K_3 c^2 \sigma^6_\mathrm{min} n_S \rho}{\nu^6 \log^2(|\mathbb{S}|n_S/\alpha\beta)}\right\} + \alpha \beta/12|\mathbb{S}|,
\end{align*}
for sufficiently large universal constants $K_2, K_3 > 0$, which is upper bounded by $\alpha \beta/6|\mathbb{S}|$ if 
\begin{align*}
    \rho & \geq \max\left( \max_{S \in \mathbb{S}}\frac{K_2 \nu^5}{\sigma_{\mathrm{min}}^5 c^{3/2}}\sqrt{\frac{\log(n_S) \log(|\mathbb{S}|/\alpha\beta)\{|S| + \log(|\mathbb{S}|n_S/\alpha\beta)  \}}{n_S}}, \right. \\
    & \hspace{5cm} \left. \max_{S \in \mathbb{S}}\frac{K_3 \nu^6}{\sigma_{\mathrm{min}}^6 c^2}\frac{\{|S| + \log(|\mathbb{S}|/\alpha\beta)\} \log^2(|\mathbb{S}|n_S/\alpha\beta)}{n_S} \right).
\end{align*}

Taking the maximum between this and $C_{\alpha \beta,c}$ completes the proof in light of \eqref{eq:last_two_bootstrap}.
\end{proof}

\begin{proof}[Proof of Proposition \ref{prop:cycle_boundary}]
We know from \eqref{eq:barrett_charcterisation} in Appendix \ref{sec:extra_prop_R} that, in general, $\Sigma_\mathbb{S}$ is compatible if and only if 
\[
    \sum\limits_{j \in K} \theta_j \leq (|K| - 1) \pi + \sum\limits_{j \not\in K} \theta_j
\]
for all $K \subseteq [d]$ with $|K|$ odd. Additionally, we can argue as in the proof of Proposition \ref{prop:KKT_cycle} (i) to show that at most one of these inequalities can be an equality when $(\theta_1, \ldots, \theta_d)$ are bounded away from $\{0, \pi\}$, which is true by assumption. We can thus write the null as
\begin{align}\label{eq:nullPolycycle}
    \bar{\mathcal{P}}_\mathbb{S}(0) = \bigcap_{\substack{K \subseteq [d] \\ |K| \text{ odd }}} \left\{\sum\limits_{j \in K} \theta_j \leq (|K| - 1) \pi + \sum\limits_{j \not\in K} \theta_j \right\},
\end{align}
and its boundary as $\partial \bar{\mathcal{P}}_\mathbb{S}(0) = \bigcup_{\substack{K \subseteq [d] \\ |K| \text{ odd }} } F_K$, with
\begin{align}\label{eq:boundaryNull}
     F_K = \bigcap_{\substack{\tilde{K} \neq K \\ |\tilde{K}| \text{ odd }} } \left\{ \sum\limits_{j \in \tilde{K}} \theta_j < (|\tilde{K}| - 1) \pi + \sum\limits_{j \not\in \tilde{K}} \theta_j \right\} \cap \left\{ \sum\limits_{j \in K} \theta_j = (|K| - 1) \pi + \sum\limits_{j \not\in K} \theta_j \right\},
\end{align}
which shows that, in the case of a non-singular $d$-cycle, being on the boundary of the null hypothesis is equivalent to being in the relative interior of a face of the convex polyhedron defined in \eqref{eq:nullPolycycle}. Now, the first part of the result about the asymptotic validity in the interior of \eqref{eq:nullPolycycle} follows from Theorem \ref{thm:bootstrap_guarantees}, so that it is enough to analyse what happens on the boundary. In this regard, we will establish the asymptotic validity of our test for the case when $(\theta_1, \ldots, \theta_d) \in F_{\{1\}}$ (i.e.~$K = \{1\}$ in \eqref{eq:boundaryNull}), noting that the validity for the other cases can be demonstrated analogously. In this context, since on the boundary we have $R(\hat{\Sigma}_\mathbb{S}) \leq 3/4$ with high probability (w.h.p.) — a fact that follows by inverting the bound provided in Theorem \ref{thm:oracle_test} — it suffices to analyse the bootstrap procedure. For asymptotic validity, we need to show that when $(\theta_1, \ldots, \theta_d) \in F_{\{1\}}$, the random variables $R(\hat{\Sigma}_\mathbb{S})$ and $R_{\hat{c}/2}(\hat{\Sigma}_{\mathbb{S},1}) \mid X_\mathbb{S}$ converge in distribution to the same limiting law. We will now analyse each quantity separately. 

As for the former, since $\min\{1+\cos\theta_j, 1-\cos\theta_j\} \geq c$ for all $j \in [d]$ (assumption A2), we have w.h.p.
\begin{align}\label{eq:theta_hatInequalities}
    \sum\limits_{j \in \tilde{K}} \hat{\theta}_j < (|\tilde{K}| - 1) \pi + \sum\limits_{j \not\in \tilde{K}} \hat{\theta}_j \text{ for all } \tilde{K} \neq \{1\}.
\end{align}
Under this event, an analogous argument to that in the proof of Proposition \ref{prop:KKT_cycle} (iii) gives 
\begin{align}\label{eq:taylor_R_cycle}
    R(\hat{\Sigma}_\mathbb{S}) & = \left[\frac{2 \sin\left(\frac{\hat{\theta}_1 + \hat{\varphi}_1}{2}\right)}{1+\cos\hat{\varphi}_1} \sin\left(\frac{\hat{\theta}_1 - \hat{\varphi}_1}{2}\right)\right]_+ = \left[\frac{2 \sin\left(\frac{\hat{\theta}_1 + \hat{\varphi}_1}{2}\right)}{1+\cos\hat{\varphi}_1} \sin\left\{\frac{1}{2}\left(\hat{\theta}_1 - \sum_{j=2}^d \hat{\theta}_j - \hat{\beta}\right)\right\}\right]_+ \nonumber \\
    & = \left[\frac{2 \sin\left(\frac{\hat{\theta}_1 + \hat{\varphi}_1}{2}\right)}{1+\cos\hat{\varphi}_1} \left\{\sin\left(\frac{\hat{\theta}_1 - \sum_{j=2}^d \hat{\theta}_j}{2}\right) \cos \frac{\hat{\beta}}{2} - \cos\left(\frac{\hat{\theta}_1 - \sum_{j=2}^d \hat{\theta}_j}{2}\right) \sin \frac{\hat{\beta}}{2} \right\} \right]_+,
\end{align}
where $(\hat{\varphi}_1, \ldots, \hat{\varphi}_d)$ are as in Proposition \ref{prop:KKT_cycle} (iii), and $\hat{\beta} \equiv \beta(\hat{\theta}_1, \ldots, \hat{\theta}_d)$ is such that $0 \leq \hat{\beta} \leq \hat{\theta}_1 -  \sum_{j=2}^d \hat{\theta}_j$ and $ \sum_{j=2}^d \hat{\varphi}_j = \hat{\beta} +  \sum_{j=2}^d \hat{\theta}_j$. Furthermore, since $R(\hat{\Sigma}_\mathbb{S}) \overset{\mathbb{P}}{\rightarrow} R(\Sigma_\mathbb{S}) = 0$ as $n_\mathbb{S} \rightarrow \infty$, we have $\hat{\varphi}_j \overset{\mathbb{P}}{\rightarrow} \theta_j$ for all $j \in [d]$, which also implies that $\beta(\hat{\theta}_1, \ldots, \hat{\theta}_d) \overset{\mathbb{P}}{\rightarrow} 0$. This, together with a Taylor approximation of Equation \eqref{eq:taylor_R_cycle} implies that 
\[
R(\hat{\Sigma}_\mathbb{S}) = \frac{\sin \theta_1}{1+\cos\theta_1} \left\{(\hat{\theta}_1 - \theta_1) - \sum_{j = 2}^d (\hat{\theta}_j - \theta_j)\right\}_+ + o_{\mathbb{P}}(1/\sqrt{n}),
\]
where $\hat{\theta}_j := \cos^{-1}(\hat{\rho}_j)$ and $\hat{\rho}_j$ is Pearson's sample correlation coefficient. Now, it is known \citep[Example 5.4.3]{lehmann99largesample} that $\sqrt{n}(\hat{\rho}_1 - \rho_1) \overset{d}{\rightarrow} N(0, \gamma_1^2)$, where $\gamma_1^2 \equiv \gamma_1^2(P_{\{1,2\}})$ was defined in \eqref{eq:def_gammaPearson}. For example, if $P_{\{1,2\}}$ is Gaussian, it simplifies to $\gamma_1^2 = (1- \rho^2)^2$. This, together with the Delta method, implies that $\sqrt{n}\{\hat{\theta}_j - \theta_j\} \overset{d}{\rightarrow} N(0, \gamma_j^2/\sin^2 \theta_j)$ for all $j \in [d]$, which further shows that 
\begin{align}\label{eq:limiting_normal}
    \sqrt{n}R(\hat{\Sigma}_\mathbb{S}) \overset{d}{\rightarrow}  \frac{\sin \theta_1}{1+\cos\theta_1} N\left(0, \sum_{j = 1}^d \frac{\gamma_j^2}{\sin^2 \theta_j} \right)_+
\end{align}
due to the independence between $\hat{\theta}_{j_1}$ and $\hat{\theta}_{j_2}$ for $j_1 \neq j_2$. The limiting distribution has a point mass of $1/2$ at zero, and is non-degenerate for positive values, as $\sin \theta_1 > 0$ by $(\mathrm{A2})$, and there exists $j \in [d]$ such that $\gamma^2_j > 0$.

As for the convergence of $R_{\hat{c}/2}(\hat{\Sigma}_{\mathbb{S},1}) \mid X_\mathbb{S}$, we will split the proof in two steps. We will first show that $R(\hat{\Sigma}_{\mathbb{S},1}) \mid X_\mathbb{S}$ converges in distribution to \eqref{eq:limiting_normal}, using the fact that $\hat{Q}_\mathbb{S}$ is bounded away from singularity and lies on $F_{\{1\}}$ w.h.p., and then argue that $R_{\hat{c}/2}(\hat{\Sigma}_{\mathbb{S},1}) \mid X_\mathbb{S}  = R(\hat{\Sigma}_{\mathbb{S},1}) \mid X_\mathbb{S} + o_\mathbb{P}(1/\sqrt{n})$. Regarding the fact that $\hat{Q}_\mathbb{S}$ is bounded away from singularity, observe that Proposition \ref{prop:KKT_cycle} (iii) implies that 
\begin{align}\label{eq:QlessSingular}
    \min\{1+\cos\hat{\varphi}_j , 1 - \cos\hat{\varphi}_j \} \geq \min\{1+\cos\hat{\theta}_j , 1 - \cos\hat{\theta}_j \} \text{ for all } j \in [d].
\end{align}
In order to show this, observe that, under \eqref{eq:theta_hatInequalities}, we have that $(1 - \cos\hat{\theta}_1)/(1 + \cos \hat{\varphi}_1) = 1 - R(\hat{\Sigma}_\mathbb{S}) \in (0,1)$, which implies that $\hat{\varphi}_1 \leq \hat{\theta}_1$. Similarly, $\hat{\varphi}_j \geq \hat{\theta}_j$ for all $j \neq 1$. This, together with $\hat{\varphi}_1 = \sum_{j = 2}^d \hat{\varphi}_j$, shows that $\hat{\theta}_j \leq \hat{\varphi}_j \leq \hat{\varphi}_1 \leq \hat{\theta}_1$ for all $j \neq 1$, which completes the proof showing that $\hat{Q}_\mathbb{S}$ is at most as singular as $\hat{\Sigma}_\mathbb{S}$ w.h.p.. Note that this happens with high probability because the \textit{good} event \eqref{eq:theta_hatInequalities} happens with high probability. Furthermore, together with $\hat{\varphi}_1 = \sum_{j = 2}^d \hat{\varphi}_j$, the non-singularity of $\hat{Q}_\mathbb{S}$ shows that $(\hat{\varphi}_1, \ldots, \hat{\varphi}_d) \in F_{\{1\}}$ w.h.p..
Hence, the previous argument applies and gives 
\[
R(\hat{\Sigma}_{\mathbb{S},1}) \mid X_\mathbb{S} = \frac{\sin \hat{\varphi}_1}{1+\cos\hat{\varphi}_1} \left\{(\hat{\theta}_1^{(1)} - \hat{\varphi}_1) - \sum_{j = 2}^d (\hat{\theta}_j^{(1)} - \hat{\varphi}_j)\right\}_+ + o_{\mathbb{P}}(1/\sqrt{n}),
\]
where $\hat{\theta}_j^{(1)} := \cos^{-1}(\hat{\rho}_j^{(1)})$ and $\hat{\rho}_j^{(1)}$ is the Pearson's sample correlation coefficient for the bootstrap sample, and $\hat{\varphi}_1$ is such that $\hat{\varphi}_1 \overset{\mathbb{P}}{\rightarrow} \theta_1$. Now, observe that $\sqrt{n}\{ \hat{\rho}^{(1)}_1 - \cos \hat{\varphi}_1 \} \mid X_\mathbb{S} \overset{d}{\rightarrow} N(0, \gamma_1^2)$. To see why this is the case, calling $\tilde{P}_{n_{\{1,2\}}}^{\{1,2\}} := n^{-1}_{\{1,2\}} \sum_{i = 1}^{n_{\{1,2\}}} \tilde{X}_{\{1,2\},i}$ with $\tilde{X}_{\{1,2\},i} = \hat{Q}^{1/2}_{\{1,2\}}\hat{\Sigma}^{-1/2}_{\{1,2\}}\operatorname{diag}^{-1/2}(\hat{\sigma}^2_{\{1,2\}})(X_{\{1,2\},i} - \hat{\mu}_{\{1,2 \}})$, we have that $\gamma^2(\tilde{P}_{n_{\{1,2\}}}^{\{1,2\}})$ converges in probability to the $\gamma^2(\cdot)$ coefficient of the standardised distribution $\operatorname{diag}^{-1/2}(\sigma^2_{\{1,2\}})(X_{\{1,2\},1} - \mu_{\{1,2 \}})$, which is equal to $\gamma_1^2$ as $\gamma^2(\cdot)$ is invariant under standardisation (to see why, recall the definitions of $\gamma^2(P)$, $f(u,v,w)$ and $S$ in \eqref{eq:def_gammaPearson}).  We can then reproduce the same argument as before, replacing $\Sigma_\mathbb{S}$ with $\hat{Q}_\mathbb{S}$, and conclude using Slutsky's theorem that $R(\hat{\Sigma}_{\mathbb{S},1}) \mid X_\mathbb{S}$ converges in distribution to the limiting normal distribution described in \eqref{eq:limiting_normal}. 

As for the second step, observe that, since $P_{\{j, j+1\}}$ has finite fourth moments for all $j \in [d]$, we have that $\hat{\Sigma}_{\mathbb{S}, 1} \mid X_\mathbb{S} \overset{\mathbb{P}}{\rightarrow} \hat{Q}_\mathbb{S}$, hence $\lambda_\mathrm{min}(\hat{\Sigma}_{\mathbb{S}, 1}) \mid X_\mathbb{S} \geq \lambda_\mathrm{min}(\hat{Q}_\mathbb{S})/2$ w.h.p.. On the other hand, we also proved in \eqref{eq:QlessSingular} that $\lambda_\mathrm{min}(\hat{Q}_\mathbb{S}) \geq \lambda_\mathrm{min}(\hat{\Sigma}_{\mathbb{S}}) = \hat{c}$ w.h.p., hence $R_{\hat{c}/2}(\hat{\Sigma}_{\mathbb{S},1}) \mid X_\mathbb{S}  = R(\hat{\Sigma}_{\mathbb{S},1}) \mid X_\mathbb{S} + o_\mathbb{P}(1/\sqrt{n})$ using the definition of $R_z(\cdot)$ in \eqref{eq:reg_R}.

Putting all the pieces together, we have just shown that also $\sqrt{n}R_{\hat{c}/2}(\hat{\Sigma}_{\mathbb{S},1}) \mid X_\mathbb{S}$ converges in distribution to \eqref{eq:limiting_normal}. Now, calling $r_{n_\mathbb{S}}$ the critical value of the bootstrap test when $B \rightarrow \infty$ for fixed $n_\mathbb{S}$, this shows that $r_{n_\mathbb{S}}$ converges to the $(1-\alpha)$-quantile of the distribution in \eqref{eq:limiting_normal}, which is positive for $\alpha < 1/2$. Since the limiting distribution of $\sqrt{n}R(\hat{\Sigma}_\mathbb{S})$ is the same, we conclude that the probability of rejecting the null hypothesis on the boundary converges to $\alpha$. 
\end{proof}

\subsection{Proofs for Section \ref{Sec:cov_compatibility}}
\begin{proof}[Proof of Proposition \ref{Prop:BasicLinearProperties}]
For any $X \in \mathcal{M}$ and $X_\mathbb{S} \in \mathcal{M}_\mathbb{S}$ we have
\begin{align*}
	\langle AX, X_\mathbb{S} \rangle_\mathbb{S} &= \sum_{S \in \mathbb{S}} \sum_{j,j' \in S} ((AX)_S)_{jj'} (X_S)_{jj'} = \sum_{S \in \mathbb{S}} \sum_{j,j' \in S} X_{jj'} (X_S)_{jj'} \\
	&= \sum_{j,j'=1}^d X_{jj'} \sum_{S \in \mathbb{S}} \mathbbm{1}_{j,j' \in S} (X_S)_{jj'} = \langle X , A^* X_\mathbb{S} \rangle,
\end{align*}
as claimed.
\end{proof}

\begin{proof}[Proof of Proposition \ref{Prop:ConsCompatCheck}]
The strategy is to use a semi-definite programming version of Farkas' lemma. This is well known in the relevant literature, but we provide a statement and short proof for completeness; see Proposition \ref{Prop:farkas_lemma} in Appendix \ref{sec:SDP}. First, rewrite the matrix completion problem 
\[
\text{find } \Sigma \in \mathcal{M} \text{ such that } \begin{cases}
    \Sigma_{jj'} = (\Sigma_S)_{jj'}, \forall S \in \mathbb{S}_{jj'} \\
    \Sigma \succcurlyeq 0
\end{cases}
\]
as 
\begin{align}\label{eq:SDP_completion}
    \text{find } \Sigma \in \mathcal{M} \text{ such that } \begin{cases}
    \langle \Sigma, E_{jj'} \rangle = (\Sigma_S)_{jj'}, \forall S \in \mathbb{S}_{jj'} \\
    \Sigma \succcurlyeq 0
\end{cases}
\end{align}
where $E_{jj'} = (\boldsymbol{e}_{j} \boldsymbol{e}_{j'}^T + \boldsymbol{e}_{j'} \boldsymbol{e}_{j}^T)/2$ and $\boldsymbol{e}_j$ is the $j$-th column vector of the standard orthonormal basis of $\mathbb{R}^d$. In order to apply the semi-definite version of Farkas' lemma we transform our problem so that the equality constraints have zero on the right-hand side. To this end, define 
\[
H_j := \begin{pmatrix}
0 & \boldsymbol{e}_j^T/2 \\
\boldsymbol{e}_j/2 & \boldsymbol{O}
\end{pmatrix}
\text{ and }
G_{S, jj'} := \begin{pmatrix}
-(\Sigma_S)_{jj'} & \boldsymbol{0}^T \\
\boldsymbol{0} & E_{jj'}
\end{pmatrix},
\]
and consider the completion problem \begin{align}\label{eq:SDP_completion2}
    \text{find } \tilde{\Sigma} \in \mathcal{M} \text{ such that } \begin{cases}
    \langle \tilde{\Sigma}, H_j \rangle = 0, \forall j \in [d] \\
    \langle \tilde{\Sigma}, G_{S, jj'} \rangle = 0, \forall S \in \mathbb{S}_{jj'} \\
    \tilde{\Sigma} \succcurlyeq 0.
\end{cases}
\end{align}
The condition $\langle \tilde{\Sigma}, H_j \rangle = 0, \forall j \in [d]$ forces $\tilde{\Sigma}$ to be in block diagonal form 
\[
\tilde{\Sigma} := \begin{pmatrix}
\gamma_{0,0} & \boldsymbol{0}^T \\
\boldsymbol{0} & \Sigma
\end{pmatrix}.
\]
Now, observe that (\ref{eq:SDP_completion}) has a solution if and only if (\ref{eq:SDP_completion2}) has a non-zero solution. Indeed, for every solution $\Sigma_0$ of (\ref{eq:SDP_completion}), then $\operatorname{diag}(1, \Sigma_0)$ is a solution of (\ref{eq:SDP_completion2}). On the other hand, suppose that $\tilde{\Sigma}_0 = \operatorname{diag}(\gamma_{0,0}, \Sigma_0) \neq \boldsymbol{O}$ is a solution of (\ref{eq:SDP_completion2}). This implies that $\gamma_{0,0} \neq 0$, otherwise $0 = \langle \tilde{\Sigma}_0, G_{S, jj'} \rangle = -\gamma_{0,0}(\Sigma_S)_{jj'} + \Sigma_{jj'} = \Sigma_{jj'}$, which would imply $\tilde{\Sigma}_0 = \boldsymbol{O}$. Being $\gamma_{0,0} \neq 0$, we can rescale the bigger block in $\tilde{\Sigma}_0$ by $\gamma_{0,0}$, i.e. $\Sigma_0 =: \gamma_{0,0} G$, and get $0 = \langle \tilde{\Sigma}_0, G_{S, jj'} \rangle = -\gamma_{0,0}(\Sigma_S)_{jj'} + \gamma_{0,0} G_{jj'} = -(\Sigma_S)_{jj'} + G_{jj'}$, which shows that $G$ is a solution of (\ref{eq:SDP_completion}). This further implies that we can assume without loss of generality that $\gamma_{0,0} = 1$ when (\ref{eq:SDP_completion2}) admits a non-zero solution. Now, by Proposition~\ref{Prop:farkas_lemma}, we know that (\ref{eq:SDP_completion2}) has a non-zero solution $\tilde{\Sigma} = \operatorname{diag}(1, \Sigma)$ if and only if
\[
\sum\limits_{S \in \mathbb{S}}\sum\limits_{j,j' \in S} (X_S)_{jj'}G_{S, jj'} = \sum\limits_{S \in \mathbb{S}}\begin{pmatrix}
- \langle \Sigma_S, X_S \rangle & \boldsymbol{0}^T \\
\boldsymbol{0} & \frac{1}{2}X_S
\end{pmatrix} = \begin{pmatrix}
- \langle \Sigma_{\mathbb{S}}, X_{\mathbb{S}} \rangle & \boldsymbol{0}^T \\
\boldsymbol{0} & \frac{1}{2} A^* X_{\mathbb{S}}
\end{pmatrix} \nsucc 0,
\]
for all collections of matrices $X_{\mathbb{S}}$, not necessarily PSD. Now, this block matrix is positive definite if and only if both $A^* X_{\mathbb{S}} \succ 0$ and $\langle \Sigma_{\mathbb{S}}, X_{\mathbb{S}} \rangle < 0$. Hence,~\eqref{eq:SDP_completion2} has a non-zero solution if and only if $\langle \Sigma_{\mathbb{S}}, X_{\mathbb{S}} \rangle \geq 0$ for all $X_\mathbb{S}$ such that $A^* X_{\mathbb{S}} \succ 0$, and the claim follows.
\end{proof}

\begin{proof}[Proof of Proposition \ref{Prop:CorrDuality}]
Weak duality, i.e. LHS $\leq$ RHS, always holds for SDPs (see \cite{blekherman2012sdp}), but we include a short proof for the sake of completeness. In fact, for any $\Sigma_\mathbb{S} \in \mathcal{P}_\mathbb{S}$, we can rewrite \begin{align}\label{eq:RHS_dual}
    \inf \{ \epsilon \in [0,1] : \Sigma_\mathbb{S} \in (1-\epsilon) \mathcal{P}_\mathbb{S}^0 + \epsilon \mathcal{P}_\mathbb{S} \}
\end{align} 
as
\begin{align*}
\inf\{  \epsilon \in [0,1] : \Sigma_\mathbb{S} \in (1-\epsilon)  \mathcal{P}_\mathbb{S}^0 + \epsilon \mathcal{P}_\mathbb{S}\} &= 1 - \sup\{ \epsilon \in [0,1] : \Sigma_\mathbb{S} \in \epsilon \mathcal{P}_\mathbb{S}^0 + (1-\epsilon) \mathcal{P}_\mathbb{S} \} \\
	& = 1 - \frac{1}{d} \sup\{ \mathrm{tr}(\Sigma) : \Sigma \in \mathcal{P}^*, \Sigma_\mathbb{S} - A \Sigma \succeq_\mathbb{S} 0, \Sigma_{11} = \ldots = \Sigma_{dd} \}.
\end{align*}
Now, for any $Y_\mathbb{S} \in \mathcal{P}_\mathbb{S}^*$ such that $A^* Y_\mathbb{S} + Y \succeq I_d \text{ for some } Y \in \mathcal{Y}$, and any $\Sigma \in \mathcal{P}^*$ such that $\Sigma_\mathbb{S} - A \Sigma \succeq_\mathbb{S} 0$, we have
\begin{align*}
	\mathrm{tr}(\Sigma) &= \langle I_d, \Sigma \rangle = -\langle A^*Y_\mathbb{S} + Y - I_d, \Sigma \rangle + \langle A^*Y_\mathbb{S} + Y, \Sigma \rangle \leq \langle A^*Y_\mathbb{S}, \Sigma \rangle + \langle Y, \Sigma \rangle\\
	&= \langle A^*Y_\mathbb{S}, \Sigma \rangle = \langle Y_\mathbb{S}, A\Sigma \rangle_\mathbb{S} = \langle Y_\mathbb{S}, \Sigma_\mathbb{S} \rangle_\mathbb{S} - \langle Y_\mathbb{S}, \Sigma_\mathbb{S} - A\Sigma \rangle_\mathbb{S} \leq \langle Y_\mathbb{S}, \Sigma_\mathbb{S} \rangle_\mathbb{S}.
\end{align*}
This shows that (\ref{eq:RHS_dual})
is lower bounded by 
\begin{align}\label{eq:LHS_dual}
1 - \frac{1}{d} \inf\{ \langle Y_\mathbb{S}, \Sigma_\mathbb{S} \rangle_\mathbb{S} : Y_\mathbb{S} \in \mathcal{P}_\mathbb{S}^*, A^*Y_\mathbb{S} + Y \succeq I_d \}.
\end{align}
Weak duality follows upon noting that $A^* X_\mathbb{S}^0 = I_d$ and $\langle X_\mathbb{S}^0, \Sigma_\mathbb{S} \rangle_\mathbb{S} =d$ and setting $X_\mathbb{S}=Y_\mathbb{S}-X_\mathbb{S}^0$. This is not surprising, as we already mentioned that weak duality always holds for SDP problems.

We will now prove strong duality for this problem. Our strategy is to write our primal and dual problems in standard form and check Slater's condition for the primal problem (\ref{eq:LHS_dual}). We already mentioned that (\ref{eq:RHS_dual}) can be written as
\[
    1 - \frac{1}{d} \sup\{ \mathrm{tr}(\Sigma) : \Sigma \in \mathcal{P}^*, \Sigma_{11} = \ldots = \Sigma_{dd}, \Sigma_\mathbb{S} - A\Sigma \succeq_\mathbb{S} 0\}. 
\]
We now write this maximisation problem in standard form by introducing variables $(Z_S : S \in \mathbb{S}) = \Sigma_\mathbb{S} - A \Sigma \in \mathcal{P}_\mathbb{S}^*$.  Enumerating $\mathbb{S}$ as $\{S_1,\ldots,S_m\}$, we instead optimise over block-diagonal matrices of the form    
\[
    X  = \begin{pmatrix} \Sigma & 0 & \cdots & 0 \\ 0 & Z_{S_1} & \cdots & 0 \\ \vdots & \vdots & \ddots & \vdots \\ 0 & 0 & \cdots & Z_{S_m} \end{pmatrix}
\]
For such $X$ our constraints are equivalent to $X \succeq 0$,
\[
    \langle E_{jj} - E_{11}, X \rangle = 0 \quad \text{ for } j =2,\ldots,d
\]
and
\[
    \langle E_{jj'} + E_{S,jj'} , X \rangle = (\Sigma_S)_{jj'} \quad \text{ for } S \in \mathbb{S} \text{ and } j,j' \in S,
\]
where $E_{jj'} = (\boldsymbol{e}_j\boldsymbol{e}_{j'}^T +\boldsymbol{e}_{j'}\boldsymbol{e}_j^T)/2$ is the binary symmetric matrix of the same dimension as $X$ with its only non-zero entries being in the $(j,j')$-th and $(j',j)$-th positions of the top left block, and where $E_{S,jj'} = (\boldsymbol{e}_{S,j}\boldsymbol{e}_{S,j'}^T +\boldsymbol{e}_{S,j'}\boldsymbol{e}_{S,j}^T)/2$ is the binary symmetric matrix of the same dimension as $X$ with its only non-zero entries being in the $(j,j')$-th and $(j',j)$-th positions of the block occupied by $Z_S$ in $X$. Write $C$ for the diagonal matrix of the same dimension as $X$ with $I_d$ in the top left block, and all other entries equal to zero. It is now possible to write
\begin{align}
\label{Eq:StandardDual}
    \sup&\{ \mathrm{tr}(\Sigma) : \Sigma \in \mathcal{P}^*,\quad \Sigma_{11} = \ldots = \Sigma_{dd}, \quad\Sigma_\mathbb{S} - A\Sigma \succeq_\mathbb{S} 0\} \nonumber \\
    &= \sup\{ \langle C, X \rangle : X \succeq 0 \text{ is block diagonal}, \langle E_{jj} - E_{11}, X \rangle = 0 \text{ for } j =2,\ldots,d, \nonumber \\
    & \hspace{150pt} \langle E_{jj'} + E_{S,jj'} , X \rangle = (\Sigma_S)_{jj'} \quad \text{ for } S \in \mathbb{S} \text{ and } j,j' \in S \},
\end{align}
so that our dual problem (\ref{eq:RHS_dual}) is now in standard form. Our primal problem (\ref{eq:LHS_dual}) is put into standard form by writing
\begin{align}
\label{Eq:StandardPrimal}
    \inf \{ &\langle \Sigma_\mathbb{S}, Y_\mathbb{S} \rangle : A^* Y_\mathbb{S} + Y \succeq I_d, \quad Y_\mathbb{S} \succeq_\mathbb{S} 0, \quad Y_\mathbb{S} \in \mathcal{M}_\mathbb{S}, \quad Y \in \mathcal{Y} \} \nonumber \\
    & =\inf\biggl\{ \sum_{S \in \mathbb{S}} \sum_{j,j' \in S} (\Sigma_S)_{jj'} y_{S,jj'} : \sum_{S \in \mathbb{S}} \sum_{j,j' \in S} y_{S,jj'} E_{jj'} + \sum_{j=2}^d y_{jj}(E_{jj}-E_{11}) \succeq I_d, \nonumber \\
    & \hspace{100pt} \sum_{S \in \mathbb{S}} \sum_{j,j' \in S} y_{S,jj'} E_{S,jj'} \succeq 0 \text{ for all } S \in \mathbb{S}, \quad y_{jj},y_{S,jj'} \in \mathbb{R} \text{ for all } S,j,j' \biggr\}.
\end{align}
With the problems written in standard form, it is now clear that~\eqref{Eq:StandardDual} is the dual problem associated to~\eqref{Eq:StandardPrimal}; see  Theorem 3.1~in \cite{boyd_vandeberghe}. Observe further that the primal problem is strictly feasible since $Y_{\mathbb{S}} = X_{\mathbb{S}}^0$ satisfies the linear constraints with  $Y$ equal to the zero matrix. Hence, by standard duality results (Theorem 2.15~in \cite{blekherman2012sdp}, Theorem 3.1~in \cite{boyd_vandeberghe}), we have that
\begin{align*}
    & \sup\{ \langle C, X \rangle : X \succeq 0 \text{ is block diagonal}, \langle E_{jj} - E_{11}, X \rangle = 0 \text{ for } j =2,\ldots,d, \\
    & \hspace{200pt} \langle E_{jj'} + E_{S,jj'} , X \rangle = (\Sigma_S)_{jj'} \quad \text{ for } S \in \mathbb{S} \text{ and } j,j' \in S \} \\
    & = \inf\biggl\{ \sum_{S \in \mathbb{S}} \sum_{j,j' \in S} (\Sigma_S)_{jj'} y_{S,jj'} : \sum_{S \in \mathbb{S}} \sum_{j,j' \in S} y_{S,jj'} E_{jj'} + \sum_{j=2}^d y_{jj}(E_{jj}-E_{11}) \succeq I_d, \\
    & \hspace{100pt} \sum_{S \in \mathbb{S}} \sum_{j,j' \in S} y_{S,jj'} E_{S,jj'} \succeq 0 \text{ for all } S \in \mathbb{S}, \quad y_{jj},y_{S,jj'} \in \mathbb{R} \text{ for all } S,j,j' \biggr\}, 
\end{align*}
and the result follows.
\end{proof}

\begin{proof}[Proof of Proposition \ref{prop:properties}]
     (i) Convexity follows easily from basic properties of the supremum. Indeed, consider $\tilde{\Sigma}_\mathbb{S}: = \lambda \Sigma_\mathbb{S}^{(1)} + (1 - \lambda) \Sigma_\mathbb{S}^{(2)}$ with $\lambda \in [0,1]$.   Observe that $R$ is well defined at $\tilde{\Sigma}_\mathbb{S}$, as the convex combination of correlation matrices is still a correlation matrix. Then, for all $\lambda \in [0,1]$,
    \begin{align*}
        R(\tilde{\Sigma}_\mathbb{S}) &= \sup\biggr\{ -\frac{1}{d} \langle X_{\mathbb{S}}, \tilde{\Sigma}_{\mathbb{S}} \rangle : X_\mathbb{S} + X_\mathbb{S}^0 \succeq_\mathbb{S} 0, A^* X_\mathbb{S} + Y \succeq 0 \text{ for some } Y \in \mathcal{Y} \biggr\} \\
        &= \sup\biggr\{ -\frac{1}{d} \langle X_{\mathbb{S}}, \lambda \Sigma_\mathbb{S}^{(1)} + (1 - \lambda) \Sigma_\mathbb{S}^{(2)} \rangle : X_\mathbb{S} + X_\mathbb{S}^0 \succeq_\mathbb{S} 0, A^* X_\mathbb{S} + Y \succeq 0 \text{ for some } Y \in \mathcal{Y} \biggr\} \\
        &\leq \lambda \sup\biggr\{ -\frac{1}{d} \langle X_{\mathbb{S}},  \Sigma_\mathbb{S}^{(1)} \rangle: X_\mathbb{S} + X_\mathbb{S}^0 \succeq_\mathbb{S} 0, A^* X_\mathbb{S} + Y \succeq 0 \text{ for some } Y \in \mathcal{Y} \biggr\} \\
        &  \hspace{50pt} + (1-\lambda) \sup\biggr\{ -\frac{1}{d} \langle X_{\mathbb{S}}, \Sigma_\mathbb{S}^{(2)} \rangle : X_\mathbb{S} + X_\mathbb{S}^0 \succeq_\mathbb{S} 0, A^* X_\mathbb{S} + Y \succeq 0 \text{ for some } Y \in \mathcal{Y} \biggr\} \\ &= \lambda R(\Sigma_\mathbb{S}^{(1)}) + (1-\lambda) R(\Sigma_\mathbb{S}^{(2)}),
    \end{align*}
    and the convexity of $R(\cdot)$ follows.
    
    (ii) $R$ acts on $\mathcal{P}_{\mathbb{S}}$, which is the space of correlation matrices over the patterns $\mathbb{S}$. Now, the spectrahedron of all correlation matrices of dimension $p$,  \[\mathcal{E}_p = \left\{\left(x_1, \ldots, x_{\binom{p}{2}}\right) \in \mathbb{R}^{\binom{p}{2}} : \Sigma_{X} = \begin{pmatrix}
    1 & x_1 & \cdots & x_{p-1}\\
    x_1 & 1 & \cdots & x_{2p-3}\\
    \vdots & \vdots & \ddots & \vdots\\
    x_{p-1} & x_{2p-3} & \cdots & 1\\
  \end{pmatrix} \succeq 0 \right\},\]
  is called the \emph{elliptope}, and identifies a closed subset of $\mathbb{R}^{\binom{p}{2}}$. This follows from the fact that the symmetry condition $\Sigma_{X} = \Sigma_{X}^T$ defines a linear subspace of $\mathbb{R}^p$ of dimension $\binom{p}{2}$, while the PSD condition $v^T \Sigma_{X} v \geq 0$ for all $v \in \mathbb{R}^p$ defines a closed subset of $\mathbb{R}^{\binom{p}{2}}$, which is a convex cone. For further insights, refer to \cite{laurent_elliptope96}. This implies that, for every pattern $\mathbb{S}$,  $\mathcal{P}_{\mathbb{S}}$ can be identified with a closed subspace of $\mathbb{R}^s$, where $s = \sum_{S \in \mathbb{S}} \binom{|S|}{2}$. The continuity of $R$ follows from the fact that every convex function that is finite on  $\mathbb{R}^s$ is necessarily continuous (see Corollary 10.1.1. in \cite{rockafellar-1970a}).
  
To prove (iii), we will make use of the fact that the dual characterisation allows us to express $R(\Sigma_{\mathbb{S}'})$ as \[
   1 - \frac{1}{d'} \sup\{ \mathrm{tr}(\Sigma) : \Sigma \in \mathcal{P}^*, \Sigma_{11} = \ldots = \Sigma_{d'd'}, \Sigma_{\mathbb{S}'} - A_{\mathbb{S}'} \Sigma \succeq_{\mathbb{S}'} 0\},
  \]
  where $d'=|\cup_{S \in \mathbb{S}'}S|$. Now, let $\tilde{\Sigma}$ be an optimal feasible matrix for $\Sigma_{\mathbb{S}'}$, where all the diagonal elements of $\tilde{\Sigma}$ are the equal to each other by definition of $R$. Then, if we consider the restriction of $\tilde{\Sigma}$ on $\cup_{S \in \mathbb{S}}S$, call it  $\tilde{\Sigma}_{|\mathbb{S}}$, it is clear that $ \Sigma_{\mathbb{S}} - A_{\mathbb{S}} \tilde{\Sigma}_{|\mathbb{S}} \succeq_{\mathbb{S}} 0$, since $ \Sigma_{\mathbb{S}'} - A_{\mathbb{S}'} \Sigma \succeq_{\mathbb{S}'} 0$ and $\Sigma_{\mathbb{S}} \subseteq \Sigma_{\mathbb{S}'}$ by hypothesis, while $\tilde{\Sigma}_{|\mathbb{S}} \succeq 0$ follows again by Cauchy's interlacing theorem. Hence, calling $d=|\cup_{S \in \mathbb{S}}S|$, for every $\tilde{\Sigma}$ that is optimal for $\Sigma_{\mathbb{S}'}$, we can construct a feasible $\tilde{\Sigma}_{|\mathbb{S}}$ for $\Sigma_{\mathbb{S}}$ such that $1-\operatorname{tr}(\tilde{\Sigma}_{|\mathbb{S}})/d = R(\Sigma_{\mathbb{S}'})$. This completes the proof.
\end{proof}

\subsection{Proofs for Section \ref{Sec:oracle_test}}

\begin{proof}[Proof of Theorem \ref{thm:oracle_test}]
We are interested in finding $C_{\alpha} \in (0,1)$ such that $\forall \alpha \in (0,1)$ 
 \[ \mathbb{P}_{H_0} \left\{R(\hat{\Sigma}_{\mathbb{S}}) \geq C_{\alpha}  \right\} \leq \alpha. \]
 We have  
\begin{align}
\label{Eq:NonSingularityDecomposition}
   \mathbb{P}_{H_0}\left\{R(\hat{\Sigma}_{\mathbb{S}}) \geq C_\alpha  \right\} & \leq  \mathbb{P}_{H_0}\left\{R(\hat{\Sigma}_{\mathbb{S}}) \geq C_\alpha,\, \hat{\Sigma}_{\mathbb{S}} \succeq_\mathbb{S} \frac{c}{2}I_\mathbb{S} \right\} + 1 - \mathbb{P}_{H_0}\left\{\hat{\Sigma}_{\mathbb{S}} \succeq_\mathbb{S} \frac{c}{2}I_\mathbb{S} \right\}.
\end{align}
Since $\Sigma_{\mathbb{S}} \succeq_\mathbb{S} cI_\mathbb{S}$ by assumption, we may bound the second part of~\eqref{Eq:NonSingularityDecomposition} by writing
\begin{align*}
    1 & = \mathbb{P}_{H_0}\left\{\Sigma_{\mathbb{S}} \succeq_\mathbb{S} cI_\mathbb{S} \right\} = \mathbb{P}_{H_0}\left\{\Sigma_{\mathbb{S}} - \hat{\Sigma}_{\mathbb{S}} + \hat{\Sigma}_{\mathbb{S}}  \succeq_\mathbb{S} cI_\mathbb{S} \right\} \\
    & \leq \mathbb{P}_{H_0}\left\{\Sigma_{\mathbb{S}} - \hat{\Sigma}_{\mathbb{S}}  \succeq_\mathbb{S} \frac{c}{2}I_\mathbb{S} \right\} + \mathbb{P}_{H_0}\left\{ \hat{\Sigma}_{\mathbb{S}}  \succeq_\mathbb{S} \frac{c}{2}I_\mathbb{S} \right\} \\
    & \leq \mathbb{P}_{H_0}\left\{\|\hat{\Sigma}_{\mathbb{S}} - \Sigma_{\mathbb{S}}\|_{2,\mathbb{S}} \geq \frac{c}{2} \right\} + \mathbb{P}_{H_0}\left\{ \hat{\Sigma}_{\mathbb{S}}  \succeq_\mathbb{S} \frac{c}{2}I_\mathbb{S} \right\}.
\end{align*}
This implies that
\begin{align}\label{eq:ProbSigmaHatSucceq}
  1 - \mathbb{P}_{H_0}\left\{\hat{\Sigma}_{\mathbb{S}} \succeq_\mathbb{S} \frac{c}{2}I_\mathbb{S} \right\} \leq \mathbb{P}_{H_0}\left\{\|\hat{\Sigma}_{\mathbb{S}} - \Sigma_{\mathbb{S}}\|_{2,\mathbb{S}} \geq \frac{c}{2} \right\} .  
\end{align}
Now, define \[
\bar{\mathrm{tr}}(X_\mathbb{S}) = \sum_{j=1}^d |\mathbb{S}_j|^{-1} \sum_{S \in \mathbb{S}_j} (X_S)_{jj}
\]
and observe that, for all $X_\mathbb{S} \in \mathcal{M}_\mathbb{S}$, we have $\langle X_\mathbb{S}^0, X_\mathbb{S} \rangle_\mathbb{S} = \bar{\mathrm{tr}}(X_\mathbb{S})$. See Proposition \ref{Prop:BasicLinearProperties2} in Appendix \ref{sec:test_cov} for a proof of this fact. Using the arguments leading up to~\eqref{Eq:CompactFeasibleSet} above, the first term on the right-hand side of~\eqref{Eq:NonSingularityDecomposition} can be written as 
\[
\mathbb{P}_{H_0}\left\{R(\hat{\Sigma}_{\mathbb{S}}) \geq C_\alpha, \,\hat{\Sigma}_{\mathbb{S}} \succeq_\mathbb{S} \frac{c}{2}I_\mathbb{S} \right\} = \mathbb{P}_{H_0}\left\{\sup\limits_{X_{\mathbb{S}} \in \mathcal{F}_{c/2}} -\frac{1}{d} \langle X_{\mathbb{S}}, \hat{\Sigma}_{\mathbb{S}} \rangle_{\mathbb{S}}  \geq C_\alpha, \,\hat{\Sigma}_{\mathbb{S}} \succeq_\mathbb{S} \frac{c}{2}I_\mathbb{S} \right\},
\] 
where $\mathcal{F}_{c/2} =  \{X_{\mathbb{S}} + X_{\mathbb{S}}^0 \succeq_{\mathbb{S}} 0 , A^* X_\mathbb{S} + Y \succeq 0 \text{ for some } Y \in \mathcal{Y}, \langle X_{\mathbb{S}} + X_{\mathbb{S}}^0, \frac{c}{2}I_{\mathbb{S}} \rangle_{\mathbb{S}} \leq d\}$. Discarding the condition $A^* X_\mathbb{S} + Y \succeq 0 \text{ for some } Y \in \mathcal{Y}$ and enlarging our feasible to   $\mathcal{\tilde{F}}_{c/2} := \{X_{\mathbb{S}} + X_{\mathbb{S}}^0 \succeq_{\mathbb{S}} 0 ,  \langle X_{\mathbb{S}} + X_{\mathbb{S}}^0, \frac{c}{2}I_{\mathbb{S}} \rangle_{\mathbb{S}} \leq d\}$, we have
 \begin{align}\label{eq:R_to_spectral}
     \mathbb{P}_{H_0} \left\{R(\hat{\Sigma}_{\mathbb{S}}) \geq C_\alpha, \,\hat{\Sigma}_{\mathbb{S}} \succeq_\mathbb{S} \frac{c}{2}I_\mathbb{S} \right\} &= \mathbb{P}_{H_0} \left\{\hat{R} - R \geq C_\alpha, \,\hat{\Sigma}_{\mathbb{S}} \succeq_\mathbb{S} \frac{c}{2}I_\mathbb{S} \right\} \leq  \mathbb{P}_{H_0} \left\{|\hat{R} - R| \geq C_\alpha, \,\hat{\Sigma}_{\mathbb{S}} \succeq_\mathbb{S} \frac{c}{2}I_\mathbb{S} \right\} \nonumber \\
     & \leq \mathbb{P}_{H_0} \left\{\sup\limits_{X_{\mathbb{S}} \in \mathcal{\tilde{F}}_{c/2}} \left| - \frac{1}{d} \langle X_{\mathbb{S}}, \hat{\Sigma}_{\mathbb{S}} - \Sigma_{\mathbb{S}}\rangle_{\mathbb{S}} \right| \geq C_\alpha\right\} \nonumber \\
     &= \mathbb{P}_{H_0} \left\{\sup\limits_{X_{\mathbb{S}} \in \mathcal{\tilde{F}}_{c/2}} \left| \langle X_{\mathbb{S}} + X_{\mathbb{S}}^0, \hat{\Sigma}_{\mathbb{S}} - \Sigma_{\mathbb{S}}\rangle_{\mathbb{S}} - (\operatorname{\bar{tr}}(\hat{\Sigma}_{\mathbb{S}}) - d)\right| \geq d \cdot C_\alpha \right\} \nonumber \\
     &\leq \mathbb{P}_{H_0} \left\{\|\hat{\Sigma}_{\mathbb{S}} - \Sigma_{\mathbb{S}}\|_{2,\mathbb{S}} \cdot \sup\limits_{X_{\mathbb{S}} \in \mathcal{\tilde{F}}_{c/2}} \|X_{\mathbb{S}} + X_{\mathbb{S}}^0\|_{*, \mathbb{S}} \geq d \cdot C_\alpha \right\} \nonumber \\
  & \leq \mathbb{P}_{H_0} \left\{\|\hat{\Sigma}_{\mathbb{S}} - \Sigma_{\mathbb{S}}\|_{2,\mathbb{S}} \cdot 2d/c \geq d \cdot C_\alpha \right\} \nonumber \\
  & = \mathbb{P}_{H_0} \left\{\|\hat{\Sigma}_{\mathbb{S}} - \Sigma_{\mathbb{S}}\|_{2,\mathbb{S}} \geq c \cdot C_\alpha / 2 \right\},
 \end{align}
where we used Holder's inequality for collections of matrices, and the fact that $\operatorname{\bar{tr}}(\hat{\Sigma}_{\mathbb{S}}) = d$, since $\hat{\Sigma}_{\mathbb{S}}$ is a collection of sample correlation matrices. Putting all the pieces together, we have \begin{align}\label{eq:control_R_spectral}
     \mathbb{P}_{H_0}\left\{R(\hat{\Sigma}_{\mathbb{S}}) \geq C_{\alpha}  \right\} & \leq \mathbb{P}_{H_0} \left\{\|\hat{\Sigma}_{\mathbb{S}} - \Sigma_{\mathbb{S}}\|_{2,\mathbb{S}} \geq c \cdot C_\alpha / 2 \right\} + \mathbb{P}_{H_0} \left\{\|\hat{\Sigma}_{\mathbb{S}} - \Sigma_{\mathbb{S}}\|_{2,\mathbb{S}} \geq c / 2\right\} \nonumber \\
     & \leq 2 \mathbb{P}_{H_0} \left\{\|\hat{\Sigma}_{\mathbb{S}} - \Sigma_{\mathbb{S}}\|_{2,\mathbb{S}} \geq c \cdot C_\alpha / 2 \right\},
 \end{align}
 since $\mathbb{P}\{X \geq x_1\} + \mathbb{P}\{X \geq x_2\} \leq 2\mathbb{P}\{X \geq \min\{x_1, x_2\}\}$. Hence, in order to bound this probability above by $\alpha$, it is sufficient to find $C_{\alpha}$ such that 
\[
 \mathbb{P}_{H_0} \left\{\|\hat{\Sigma}_{\mathbb{S}} - \Sigma_{\mathbb{S}}\|_{2,\mathbb{S}} \geq c \cdot C_\alpha / 2 \right\} \leq \alpha/2.
\]
We have 
\begin{align}\label{eq:FindCAlpha}
    \mathbb{P}_{H_0} \Bigl\{ \|\hat{\Sigma}_{\mathbb{S}} - \Sigma_{\mathbb{S}}\|_{2,\mathbb{S}} & \geq c \cdot C_\alpha/ 2 \Bigr\}  \leq \mathbb{P}_{H_0} \left\{\max_{S \in \mathbb{S}} \|\hat{\Sigma}_{S} - \Sigma_{S}\|_2 \geq c \cdot C_\alpha / 2 \right\} \nonumber \\
  & \leq \sum_{S \in \mathbb{S}} \mathbb{P}_{H_0} \left\{ \|\hat{\Sigma}_{S} - \Sigma_{S}\|_2 \geq c \cdot C_\alpha /2   \right\} \leq |\mathbb{S}| \cdot \max_{S \in \mathbb{S}} \mathbb{P}_{H_0} \left\{ \|\hat{\Sigma}_{S} - \Sigma_{S}\|_2 \geq c \cdot C_\alpha / 2  \right\}.
\end{align}
Hence, calling 
\begin{align*}
        C_{t}(S) &:= \frac{K_1 \nu^4}{\sigma^4_{\mathrm{min}}} \sqrt{\frac{|S| + \log(1/t)}{n}} 
    \end{align*}
for all $S \in \mathbb{S}$, with $K_1 > 0$ sufficiently large universal constants, it is immediate to see using Proposition~\ref{prop:corr_concentration} that it is sufficient to take 
\[
C_{\alpha} \geq \frac{2}{c}\max_{S \in \mathbb{S}}C_{\alpha/2|\mathbb{S}|}(S),
\]
while ensuring $C_\alpha \leq 1$, in order to have $\mathbb{P}_{H_0} \left\{\|\hat{\Sigma}_{\mathbb{S}} - \Sigma_{\mathbb{S}}\|_{2,\mathbb{S}} \geq c \cdot C_\alpha / 2\right\} \leq \alpha/2$,   and the first statement follows. As for the last statement, observe that if $R(\Sigma_\mathbb{S}) > C_\alpha + C_\beta$, we have 
\begin{align*}
    \mathbb{P}\{R(\hat{\Sigma}_\mathbb{S}\} \leq C_\alpha\} &= \mathbb{P}\{R(\hat{\Sigma}_\mathbb{S}) - R(\Sigma_\mathbb{S} \leq C_\alpha - R(\Sigma_\mathbb{S})\} \leq \mathbb{P}\{R(\hat{\Sigma}_\mathbb{S}) - R(\Sigma_\mathbb{S}) \leq -C_\beta\} \\
    & \leq \mathbb{P}\{|R(\hat{\Sigma}_\mathbb{S}) - R(\Sigma_\mathbb{S})| \geq C_\beta\} \leq \beta,
\end{align*}
using the exact same concentration bound we employed to control the Type-I error.
\end{proof}

The analysis of $R(\hat{\Sigma}_\mathbb{S})$ in Theorem \ref{thm:oracle_test} is crucially based on the fact that we can control the oscillation $|R(\hat{\Sigma}_\mathbb{S}) - R(\Sigma_\mathbb{S})|$ using $\max_{S \in \mathbb{S}} \|\hat{\Sigma}_{S} - \Sigma_{S}\|_2$, where the $\Sigma_{S}$ are the Pearson population correlation matrices and $\hat{\Sigma}_{S}$ are the corresponding Pearson sample correlation matrices. We now state and prove a tail bound for the spectral norm $\|\hat{\Sigma} - \Sigma\|_2$, where $\Sigma$ is the population correlation matrix and $\hat{\Sigma}$ is the sample correlation matrix of complete data.  
\begin{prop}\label{prop:corr_concentration}
    Suppose we observe an i.i.d sample $X_1, \cdots, X_n \sim X$, where $X$ is $\nu$-subgaussian random vector in $\mathbb{R}^d$ with mean $\mu$. Let $\hat{\mu}$ be the sample mean, and let $\Omega$ and $\hat{\Omega} := n^{-1}\sum_{i =1}^n X_i X_i^T - \hat{\mu} \hat{\mu}^T$ be the population and sample covariance matrices, respectively. Let $\Sigma = D^{-1/2}\Omega D^{-1/2}$ be the population correlation matrix, where $D = \operatorname{diag}(\Omega)$, and $\hat{\Sigma} = \hat{D}^{-1/2} \hat{\Omega} \hat{D}^{-1/2}$ be the sample correlation matrix, where $\hat{D} = \operatorname{diag}(\hat{\Omega})$. Then, there exist a universal constant $K_1 > 0$ such that, for every $t \in [0,1]$ such that $C_t \leq 1$,  we have $\mathbb{P}\{\|\hat{\Sigma} - \Sigma \|_2 > C_t \} \leq t$ where 
    \begin{align*}
        C_t :=  \frac{K_1 \nu^4}{\sigma^4_{\mathrm{min}}} \sqrt{\frac{d + \log(1/t)}{n}}
    \end{align*}
    and $\sigma_\mathrm{min}^2 := \min\limits_{j \in [d]} \Omega_{jj}$.
\end{prop}
\begin{proof}[Proof of Proposition \ref{prop:corr_concentration}]
We will prove the claim by showing that it is sufficient to bound the spectral norm of the difference between the sample covariance matrix and its population counterpart, for which classical tail bounds apply. Now, observe that, for all $x \in [0,1]$,  the triangle inequality and the sub-multiplicativity of the spectral norm imply that
\begin{align}\label{eq:conc_corr_to_cov}
    \mathbb{P}&\left\{ \| \hat{\Sigma} - \Sigma \|_2 > x \right\} = \mathbb{P}\left\{ \| \hat{D}^{-1/2} \hat{\Omega} \hat{D}^{-1/2} - D^{-1/2}\Omega D^{-1/2} \|_2 > x \right\} \nonumber \\
    & \leq  \mathbb{P}\left\{ \| \hat{D}^{-1/2} \hat{\Omega} \hat{D}^{-1/2} - D^{-1/2}\hat{\Omega} D^{-1/2} \|_2 > x/2 \right\} + \mathbb{P}\left\{ \| D^{-1/2}(\hat{\Omega} - \Omega) D^{-1/2} \|_2 > x/2 \right\} \nonumber \\
    & \leq  \mathbb{P}\left\{ \| \hat{D}^{-1/2} \hat{\Omega} ( \hat{D}^{-1/2} - D^{-1/2}) \|_2 +  \| (\hat{D}^{-1/2}  - D^{-1/2})\hat{\Omega} D^{-1/2} \|_2 > x/2 \right\} +  \mathbb{P}\left\{ \| \hat{\Omega} - \Omega \|_2 > \sigma_\mathrm{min}^2 x/2 \right\} \nonumber \\
    & \leq \mathbb{P}\left\{ \| \hat{D}^{-1/2}  - D^{-1/2} \|_2 \| \hat{\Omega} \|_2 ( \| \hat{D}^{-1/2} \|_2  + \| D^{-1/2} \|_2)> x/2 \right\}  +  \mathbb{P}\left\{ \| \hat{\Omega} - \Omega \|_2 > \sigma_\mathrm{min}^2 x/2 \right\}. 
\end{align}
As for the first term in \eqref{eq:conc_corr_to_cov}, using the fact that $\|D^{-1/2}\|_2 \leq 1/\sigma_\mathrm{min}$ and $\|\Omega\|_2 \leq \nu^2$ due to subgaussianity, we further have 
\begin{align}\label{eq:conc_corr_to_cov_2}
    \mathbb{P}&\left\{ \| \hat{D}^{-1/2}  - D^{-1/2} \|_2 \| \hat{\Omega} \|_2 ( \| \hat{D}^{-1/2} \|_2  + \| D^{-1/2} \|_2)> x/2 \right\} \nonumber \\
    & \leq \mathbb{P}\left\{ \| \hat{D}^{-1/2}  - D^{-1/2} \|_2 \| \hat{\Omega} \|_2 ( \| \hat{D}^{-1/2} \|_2  + \| D^{-1/2} \|_2)> x/2, \|\hat{\Omega}\|_2 \leq 2 \nu^2, \|\hat{D}^{-1/2} \|_2 \leq 2/\sigma_\mathrm{min} \right\} \nonumber \\
    & \quad \quad + \mathbb{P}\{\|\hat{\Omega}\|_2 > 2 \nu^2 \} + \mathbb{P}\{\|\hat{D}^{-1/2}\|_2 > 2/\sigma_\mathrm{min} \} \nonumber \\
    & \leq \mathbb{P}\left\{ \| \hat{D}^{-1/2}  - D^{-1/2} \|_2   > \sigma_\mathrm{min} x/12\nu^2 \right\} + \mathbb{P}\{\|\hat{\Omega} - \Omega \|_2 > \nu^2 \} + \mathbb{P}\{\|\hat{D}^{-1/2} -  D^{-1/2} \|_2 > 1/\sigma_\mathrm{min} \} \nonumber \\
    & \leq \mathbb{P}\left\{ \| \hat{D}^{-1/2}D^{1/2} - I \|_2   > \sigma_\mathrm{min}^2 x/12\nu^2 \right\} + \mathbb{P}\{\|\hat{\Omega} - \Omega \|_2 > \nu^2 \} + \mathbb{P}\{\|\hat{D}^{-1/2}D^{1/2} - I \|_2 > 1 \} \nonumber \\
    & \leq 2\mathbb{P}\left\{ \| \hat{D}^{-1/2}D^{1/2} - I \|_2   > \sigma_\mathrm{min}^2 x/12\nu^2 \right\} + \mathbb{P}\{\|\hat{\Omega} - \Omega \|_2 > \nu^2 \},
\end{align}
where in the last step we used the fact that $x \in [0,1]$ and $\sigma_\mathrm{min}^2 \leq \|\Omega \|_2 \leq \nu^2$. As for the first term in \eqref{eq:conc_corr_to_cov_2}, for all $x > 0$ we have 
\begin{align}\label{eq:conc_corr_to_cov_3}
    \mathbb{P}&\left\{ \| \hat{D}^{-1/2}D^{1/2} - I \|_2   > x \right\} = \mathbb{P}\left\{\max_{j \in [d]} |\sigma_j/\hat{\sigma}_j - 1| > x \right\} \nonumber \\
    & \leq \mathbb{P}\left\{\max_{j \in [d]} |\sigma_j/\hat{\sigma}_j - 1| > x, \max_{j \in [d]} |\hat{\sigma}_j^2/\sigma_j^2 - 1| \leq 3/4 \right\} + \mathbb{P}\left\{ \max_{j \in [d]} |\hat{\sigma}_j^2/\sigma_j^2 - 1| > 3/4 \right\} \nonumber \\
    & \leq \mathbb{P}\left\{\max_{j \in [d]} |\hat{\sigma}_j/\sigma_j - 1| > x/2 \right\} + \mathbb{P}\left\{ \max_{j \in [d]} |\hat{\sigma}_j^2/\sigma_j^2 - 1| > 3/4 \right\} \nonumber  \\
    & \leq \mathbb{P}\left\{\max_{j \in [d]} |\hat{\sigma}_j^2/\sigma_j^2 - 1| > x/2 \right\} + \mathbb{P}\left\{ \max_{j \in [d]} |\hat{\sigma}_j^2/\sigma_j^2 - 1| > 3/4 \right\} \leq 2 \mathbb{P}\left\{\max_{j \in [d]} |\hat{\sigma}_j^2/\sigma_j^2 - 1| > x/2 \wedge 3/4 \right\} \nonumber \\
    & \leq 2\mathbb{P}\left\{\max_{j \in [d]} |\hat{\sigma}_j^2 - \sigma_j^2| > \sigma_\mathrm{min}^2 (x/2 \wedge 3/4) \right\} = 2\mathbb{P}\left\{\| \hat{D} - D \|_2 > \sigma_\mathrm{min}^2 (x/2 \wedge 3/4) \right\} \nonumber \\
    & = 2\mathbb{P}\left\{\| \operatorname{diag}(\operatorname{diag}(\hat{\Omega} - \Omega)) \|_2 > \sigma_\mathrm{min}^2 (x/2 \wedge 3/4) \right\} \leq 2\mathbb{P}\left\{\| \hat{\Omega} - \Omega \|_2 > \sigma_\mathrm{min}^2 (x/2 \wedge 3/4) \right\},
\end{align}
where in the second and third inequalities we used the fact that $|1/\sqrt{x} - 1| \leq 2 |\sqrt{x} - 1|$ for $|x - 1| \leq 3/4$ and the fact that $|\sqrt{x} - 1| \leq |x - 1|$, respectively. Combining \eqref{eq:conc_corr_to_cov}, \eqref{eq:conc_corr_to_cov_2}, \eqref{eq:conc_corr_to_cov_3} gives
\begin{align}\label{eq:conc_corr_to_cov_4}
    \mathbb{P}\left\{ \| \hat{\Sigma} - \Sigma \|_2 > x \right\} & \leq \mathbb{P}\left\{ \| \hat{\Omega} - \Omega \|_2 > \sigma_\mathrm{min}^2 x/2 \right\} + \mathbb{P}\left\{ \| \hat{\Omega} - \Omega \|_2 > \nu^2 \right\} \nonumber \\
    & + 4 \mathbb{P}\left\{ \| \hat{\Omega} - \Omega \|_2 > \sigma_\mathrm{min}^2 (3/4 \wedge \sigma_\mathrm{min}^2 x/24\nu^2) \right\} \leq 6 \mathbb{P}\left\{ \| \hat{\Omega} - \Omega \|_2 > \sigma_\mathrm{min}^4 x/24\nu^2 \right\},
\end{align}
using again the fact that $x \in [0,1]$ and $\sigma^2_\mathrm{min} \leq \nu^2$, which shows that it is enough to control $\|\hat{\Omega} - \Omega \|_2$. In this regard, for all $x > 0$ we have 
\begin{align*}
     \mathbb{P} & \left\{\|\hat{\Omega} - \Omega \|_2  > x \right\} = \mathbb{P}\left\{\|n^{-1}\sum_{i=1}^{n} X_{i} X_{i}^T - \hat{\mu} \hat{\mu}^T -\Omega \|_2 > x \right\} \\
    & = \mathbb{P}\left\{\|n^{-1}\sum_{i=1}^{n} (X_{i} - \mu)(X_{i}-\mu)^T -\Omega -(\hat{\mu} - \mu)(\hat{\mu}-\mu)^T \|_2 > x \right\} \\
    & \leq \mathbb{P}\left\{\|n^{-1}\sum_{i=1}^{n} (X_{i} - \mu)(X_{i}-\mu)^T -\Omega \|_2 > x/2 \right\} + \mathbb{P}\left\{\|\hat{\mu} - \mu\|_2^2 > x/2 \right\} \\
    & = \mathbb{P}\left\{\|n^{-1}\sum_{i=1}^{n} (X_{i} - \mu)(X_{i}-\mu)^T -\Omega \|_2 > x/2 \right\} + \mathbb{P}\left\{\|n^{-1}\sum_{i = 1}^{n}(X_{i} - \mu)\|_2 > \sqrt{x/2} \right\} \\
    & \leq 2 \cdot 9^d \exp \left\{-n \frac{x}{32 \nu^2} \wedge\left(\frac{x}{32 \nu^2}\right)^2  \right\} + 5^d \exp \left\{-\frac{n x}{16 \nu^2} \right\},
\end{align*}
where we used Propositions \ref{prop:norm_SG} and \ref{prop:concentration_cov} in Appendix \ref{sec:technical_ineq} in the last inequality. Inverting this bound leads to 
\begin{align}\label{eq:cov_concentration}
    \|\hat{\Omega} - \Omega\|_2 \leq K_1 \nu^2 \sqrt{\frac{d + \log(1/t)}{n}} \vee \frac{d + \log(1/t)}{n},
\end{align}
with probability $\geq 1 - t$, for a universal constant $K_1 > 0$ sufficiently large. Now, combining this with \eqref{eq:conc_corr_to_cov_4} and \eqref{eq:cov_concentration} shows that
\begin{align}\label{eq:cov_concentration_2}
    \|\hat{\Sigma} - \Sigma\|_2 \leq \frac{K_1 \nu^4}{\sigma^4_{\mathrm{min}}} \sqrt{\frac{d + \log(1/t)}{n}} \vee \frac{d + \log(1/t)}{n},
\end{align}
with probability $\geq 1 - t$, and the assumption that $C_t \leq 1$ allows focusing on the subgaussian regime. This  completes the proof.
\end{proof}

First, observe that the dependence on $1/\sigma_\mathrm{min}^2$ is reasonable, as the smaller the minimum variance the more problematic the normalisation matrix $D^{-1/2}$. Second, observe that since we are restricting to the case $C_t \leq 1$, i.e.~$n \gtrsim d$, the subgaussian regime prevails, and we obtain that  
\[
\|\hat{\Sigma} - \Sigma\|_2 \lesssim \sqrt{\frac{d}{n}}
\]
in probability. Similar rates, with logarithmic factors, were found in high-dimensional covariance matrix estimation with missing observations \citep{lounici14}, sample covariance matrix estimator of reduced effective rank population matrices \citep{Bunea_2015}, concentration of the adjacency matrix and of the Laplacian in random graphs \citep{oliveira2010concentration}, and in the statistical analysis of latent generalized correlation matrix estimation in transelliptical distribution \citep{Han_2017}. In particular, using the additional assumption that the data is generated according to a transelliptical distribution, \cite{Han_2017} gave an estimator $\hat{K}$ based on Kendall's tau and proved that
\[
\|\hat{K} - \Sigma\|_2 \lesssim \sqrt{\frac{r(\Sigma) \log(d)}{n}},
\]
where $r(\Sigma) := \operatorname{tr}(\Sigma)/\|\Sigma\|_2$ is the effective dimension of $\Sigma$. This is analogous to the bound given in Proposition \ref{prop:corr_concentration}, where we have an extra factor of $\nu^2/\sigma_\mathrm{min}^2$, which can be interpreted as the condition number and might lead to a suboptimal bound when it is large, and the ambient dimension $d$ in place of the intrinsic dimension $r(\Sigma)$. This would improve the bound sensibly in the case of an approximately low-rank correlation matrix.

\subsection{Proofs for Section \ref{subsec:cycles}}

\begin{proof}[Proof of Proposition \ref{prop:measure_consistency}]
Let $\sigma_\mathbb{S}^2$ be a nonnegative collection such that $\bar{\operatorname{av}}_j(\sigma^2_{\mathbb{S}}) = 1 \text{ for all } j \in [d]$. Using, for the third equality, the facts that $A_{V} \boldsymbol{1}_d$ also satisfies these properties and that $\bar{\text{av}}$ is linear, we have that
    \begin{align*}
    1 - V(\sigma_{\mathbb{S}}^2) &= 1 - \inf\left\{ \epsilon \in [0,1] : \sigma_{\mathbb{S}}^2 = (1-\epsilon)A_{V} \boldsymbol{1}_d + \epsilon {\sigma'}_{\mathbb{S}}^2 \text{ with } \bar{\operatorname{av}}_j(\sigma'^2_{\mathbb{S}}) = 1 \text{ for all } j \in [d] \right\} \\
    & = \sup\left\{ \epsilon \in [0,1] : \sigma_{\mathbb{S}}^2 = \epsilon A_{V} \boldsymbol{1}_d + (1-\epsilon) {\sigma'}_{\mathbb{S}}^2 \text{ with } \bar{\operatorname{av}}_j(\sigma'^2_{\mathbb{S}}) = 1 \text{ for all } j \in [d] \right\} \\
    & =  \sup\left\{ \epsilon \in [0,1] : \epsilon \leq \min_{j \in [d]} \min_{S \in \mathbb{S}_j} \sigma_{S,j}^2  \right\} = \min_{j \in [d]} \min_{S \in \mathbb{S}_j} \sigma_{S,j}^2,
\end{align*}
as claimed.
\end{proof}

\begin{proof}[Proof of Theorem \ref{thm:combined_oracle_test}]
We are interested in finding $C_{\alpha}^{(R+V)} \in (0,2)$ such that $\forall \alpha \in (0, 1)$ we have 
\[
\mathbb{P}_{H_0} \{R(\hat{\Sigma}_{\mathbb{S}}) + V(\hat{\sigma}_{\mathbb{S}}^2) \geq C_{\alpha}^{(R+V)} \} \leq \alpha.
\]
First, observe that \[ \mathbb{P}_{H_0}\{R(\hat{\Sigma}_{\mathbb{S}}) \geq C_{\alpha}^{(R+V)}/2  \} + \mathbb{P}_{H_0}\{V(\hat{\sigma}_{\mathbb{S}}^2) \geq C_{\alpha}^{(R+V)}/2 \},
\]
so that it is enough to find $C_{\alpha}^{(R+V)}$ such that
\[
\max \left\{\mathbb{P}_{H_0}\{R(\hat{\Sigma}_{\mathbb{S}}) \geq C_{\alpha}^{(R+V)}/2  \} , \mathbb{P}_{H_0}\{V(\hat{\sigma}_{\mathbb{S}}^2) \geq C_{\alpha}^{(R+V)}/2 \} \right\} \leq \alpha/2.
\]
As for the former, Theorem \ref{thm:oracle_test} applies, showing that 
\begin{align}\label{eq:C_alpha_R}
    \mathbb{P}_{H_0}\{R(\hat{\Sigma}_{\mathbb{S}}) \geq C_{\alpha}^{(R+V)}/2  \} \leq \alpha/2
\end{align}
if $ C_{\alpha}^{(R+V)} \geq 2C_{\alpha/2}$, as long as $C_{\alpha/2} \leq 1$. As for the second term, since $V(\sigma_{\mathbb{S}}^2) = 0$ under the null, 
\begin{align*}
\mathbb{P}_{H_0}\left\{V(\hat{\sigma}_{\mathbb{S}}^2) \geq C_{\alpha}^{(R+V)}/2  \right\} & = \mathbb{P}_{H_0}\left\{V(\hat{\sigma}_{\mathbb{S}}^2) - V(\sigma_{\mathbb{S}}^2) \geq C_{\alpha}^{(R+V)}/2  \right\} \leq  \mathbb{P}_{H_0}\left\{|V(\hat{\sigma}_{\mathbb{S}}^2) - V(\sigma_{\mathbb{S}}^2)| \geq C_{\alpha}^{(R+V)}/2  \right\} \\
& = \mathbb{P}_{H_0}\left\{|\min_{j \in [d]} \min_{S \in \mathbb{S}_j} \hat{\sigma}^2_{S,j} - \min_{j \in [d]} \min_{S \in \mathbb{S}_j} \sigma^2_{S,j}| \geq C_{\alpha}^{(R+V)}/2  \right\} \\
& \leq \mathbb{P}_{H_0}\left\{\max_{j \in [d]} \max_{S \in \mathbb{S}_j} |\hat{\sigma}^2_{S,j} - \sigma^2_{S,j}| \geq C_{\alpha}^{(R+V)}/2  \right\} \\
     &  \leq \sum_{S \in \mathbb{S}} \sum_{j \in S} \mathbb{P}_{H_0}\left\{ |\hat{\sigma}^2_{S,j} - \sigma^2_{S,j}| \geq C_{\alpha}^{(R+V)}/2  \right\} \\
     & \leq \left(\sum_{S \in \mathbb{S}} |S| \right) \max_{j \in [d]} \max_{S \in \mathbb{S}_j} \mathbb{P}_{H_0}\left\{ |\hat{\sigma}^2_{S,j} - \sigma^2_{S,j}| \geq C_{\alpha}^{(R+V)}/2  \right\}.
\end{align*}
Now, the standard Chernhoff method for subgaussian and subexponential random variables (see Propositions \ref{prop:hoeffding} and \ref{prop:SE_tail_bound} in Appendix \ref{sec:technical_ineq}) gives, for all $j \in [d]$, for all $S \in \mathbb{S}_j$, for all $x \leq 8 \nu^2$ (so that we can focus on the subgaussian regime),
\begin{align*}
    \mathbb{P}_{H_0}&\left\{ |\hat{\sigma}^2_{S,j} - \sigma^2_{S,j}| > x  \right\}
    \leq \mathbb{P}\left\{ |n_S^{-1} \sum_{i \in [n_S]} (X_{S,ij} - \hat{\mu}_{S,j})^2 - \mathbb{E}[(X_{S, 1j} - \mu_{S, j})^2] | > x \right\} \\
    & \leq \mathbb{P}\left\{ |n_S^{-1} \sum_{i \in [n_S]} (X_{S,ij} - \mu_{S,j})^2 - \mathbb{E}[(X_{S, 1j} - \mu_{S,j})^2] | > x/2 \right\} +  \mathbb{P}\left\{ |\hat{\mu}_{S, j} - \mu_{S, j} |^2 > x/2  \right\} \\
    & \leq 2\exp\left\{-\frac{n_{S} x^2}{1024 \nu^4} \right\} + 2\exp\left\{-\frac{n_{S} x}{8 \nu^2} \right\} \leq 4\exp\left\{-\frac{n_{S} x}{1024 \nu^4} \right\} \leq 4\exp\left\{-\frac{(\min_{S \in \mathbb{S}} n_{S}) x^2}{1024 \nu^4} \right\}.
\end{align*}
This implies that, if $C_{\alpha}^{(R+V)} \leq 16\nu^2$,   $\mathbb{P}_{H_0}\left\{V(\hat{\sigma}_{\mathbb{S}}^2) \geq C_{\alpha}^{(R+V)}/2  \right\} \leq \alpha/2$ is satisfied if \[
\left(\sum_{S \in \mathbb{S}} |S| \right) \exp\left\{-\frac{\min_{S \in \mathbb{S}} n_S C^2_{\alpha}}{1024 \nu^4} \right\} \leq \frac{\alpha}{8},
\]
which further shows that it is sufficient to take 
\begin{equation}\label{eq:C_alpha_V}
    C_{\alpha}^{(R+V)} \geq K_2 \nu^2 \sqrt{\frac{\log\left(\sum_{S \in \mathbb{S}} |S|/\alpha\right)}{\min_{S \in \mathbb{S}} n_S}},
\end{equation}
for a universal constant $K_2 > 0$. In order to satisfy both (\ref{eq:C_alpha_R}) and (\ref{eq:C_alpha_V}) at the same time, it is sufficient to take the maximum between the two right-hand sides, while ensuring that $C_{\alpha}^{(R+V)} \leq \min\{ 2, 16\nu^2\}$, and the first statement follows. As for the second part concerning the Type-II error, the proof follows similarly to that of the second part of Theorem \ref{thm:oracle_test}.
\end{proof}

\begin{proof}[Proof of Theorem \ref{thm:cycle_minimax}]
For the $d$-cycle, our measure of consistency of the variances is \[
V(\Sigma_{\mathbb{S}}) =  1 - \min_{j \in [d]} \min_{S \in \{\{j-1,j \}, \{j,j+1 \} \}} \sigma^2_{S,j}.
\] 
We will show that testing 
 \[
 H_0: V(\Sigma_\mathbb{S}) = 0 \quad \text{ vs. } \quad H_1(\rho): V(\Sigma_\mathbb{S}) > \rho, 
\]
requires at least a separation of the order $\sqrt{\log d/n}$, and since $H_1^\prime: R(\Sigma_\mathbb{S}) + V(\sigma^2_\mathbb{S}) > \rho$, the statement would follow. Formally, this corresponds to assuming that $\Sigma_\mathbb{S}$ is always compatible, and constructing prior distributions just on $\{V(\sigma^2_\mathbb{S}) = 0\}$ and$ \{V(\sigma^2_\mathbb{S}) > \rho\}$.  We construct a lower bound to show that the minimax separation in this case is at least $c_1 \sqrt{\log d/\min_{i \in [d]} n_i}$, where $c_1 > 0$ is a universal constant. Let
\[
P_0 = \left(N^{\otimes n_1}\left(\boldsymbol{0}_{2}, \begin{pmatrix}
            1 & 0 \\
            0 & 1
        \end{pmatrix}\right), \ldots, N^{\otimes n_d}\left(\boldsymbol{0}_{2}, \begin{pmatrix}
           1 & 0 \\
            0 & 1
        \end{pmatrix}\right) \right),
\]
and
\begin{align*}
    P_j & = \left(N^{\otimes n_1}\left(\boldsymbol{0}_{2}, \begin{pmatrix}
            1 & 0 \\
            0 & 1
        \end{pmatrix}\right), \ldots, N^{\otimes n_j}\left(\boldsymbol{0}_{2}, \begin{pmatrix}
            1 & 0 \\
            0 & 1-\delta
        \end{pmatrix}\right), \right.\\
        & \hspace{4cm} \left. N^{\otimes n_{j+1}}\left(\boldsymbol{0}_{2}, \begin{pmatrix}
            1+\delta & 0 \\
            0 & 1
        \end{pmatrix}\right),
        \ldots, N^{\otimes n_d}\left(\boldsymbol{0}_{2}, \begin{pmatrix}
           1 & 0 \\
            0 & 1
        \end{pmatrix}\right) \right),
\end{align*}
for $j \in [d]$, and $\delta > 0$. It is clear that each $P_j$ lies in $H_1(\delta)$, and that \begin{align*}
    \frac{dP_j}{dP_0}((x_{1,1},y_{1,1}), & \ldots,(x_{1,n_1},y_{1,n_1}),(x_{2,1},y_{2,1}),\ldots,(x_{d,n_d},y_{d,n_d}) \\
    & =  \prod_{h \in [n_j]}\frac{1}{(1-\delta)^{1/2}}\exp\left\{-\frac{\delta}{2(1-\delta)} y_{j,h}^2\right\} \prod_{h \in [n_{j+1}]} \frac{1}{(1+\delta)^{1/2}}\exp\left\{\frac{\delta}{2(1+\delta)} x_{j+1,h}^2\right\}.
\end{align*}
Now, using the same strategy outlined in Section \ref{sec:minimax_LB}, it is enough to control the Total Variation distance 
\begin{align*}
    4 \operatorname{TV}\Bigg\{P_0, \frac{1}{d}\sum_{j=1}^d P_j \Bigg\}^2 & \leq \operatorname{\chi^2}\left(P_0, \frac{1}{d}\sum_{j=1}^d P_j \right) = \int \frac{\{\frac{1}{d}\sum_{j=1}^d dP_j \}^2}{dP_0} - 1 \\
    & = \frac{1}{d^2}\sum_{j_1, j_2 =1}^d\int \frac{d P_{j_1} d P_{j_2}}{d P_0} - 1 = \frac{1}{d^2}\sum_{j_1, j_2 =1}^d \int \frac{d P_{j_1}}{d P_0}\frac{d P_{j_2}}{d P_0}d P_0  - 1. 
\end{align*}
Now, it is easy to see that if $j_2 \notin \{j_1 - 1, j_1, j_1 + 1\}$, then $\int \frac{d P_{j_1}}{d P_0}\frac{d P_{j_2}}{d P_0}d P_0 = 1$. This happens to be the case also when $j_2 = j_1 \pm 1$, since, for $j_2 = j_1 - 1 =: j-1$, we have 
\begin{align*}
    \int \frac{d P_{j_1}}{d P_0}\frac{d P_{j_2}}{d P_0}d P_0 & = \int (1 - \delta)^{-n_{j-1}/2} (1-\delta^2)^{-n_j/2} (1 + \delta)^{-n_{j+1}/2} \prod_{h \in [n_{j-1}]} \exp\left\{-\frac{\delta}{2(1-\delta)} y_{j-1,h}^2\right\} \\
    & \quad \quad \quad \times \prod_{h \in [n_j]} \exp\left\{\frac{\delta}{2(1+\delta)} x_{j,h}^2 -\frac{\delta}{2(1-\delta)} y_{j,h}^2\right\} \prod_{h \in [n_{j+1}]} \exp\left\{\frac{\delta}{2(1+\delta)} x_{j+1,h}^2 \right\} dP_0 \\
    & = \int \prod_{h \in [n_{j-1}]} \frac{1}{2\pi(1-\delta)^{1/2}}\exp\left\{ - \frac{1}{2} \begin{pmatrix}
        x_{j-1,h} \\ y_{j-1,h}
    \end{pmatrix}^T \begin{pmatrix}
       1 & 0 \\
       0 & \frac{1}{1-\delta}
    \end{pmatrix} \begin{pmatrix}
        x_{j-1,h} \\ y_{j-1,h}
    \end{pmatrix} \right\} d \boldsymbol{x}_{j-1} d \boldsymbol{y}_{j-1} \\
    & \times \int \prod_{h \in [n_{j}]} \frac{1}{2\pi(1-\delta^2)^{1/2}}\exp\left\{ - \frac{1}{2} \begin{pmatrix}
        x_{j,h} \\ y_{j,h}
    \end{pmatrix}^T \begin{pmatrix}
       \frac{1}{1+\delta} & 0 \\
       0 & \frac{1}{1-\delta}
    \end{pmatrix} \begin{pmatrix}
        x_{j,h} \\ y_{j,h}
    \end{pmatrix} \right\} d \boldsymbol{x}_{j} d \boldsymbol{y}_{j} \\
    & \times \int \prod_{h \in [n_{j+1}]} \frac{1}{2\pi(1+\delta)^{1/2}}\exp\left\{ - \frac{1}{2} \begin{pmatrix}
        x_{j+1,h} \\ y_{j+1,h}
    \end{pmatrix}^T \begin{pmatrix}
       \frac{1}{1+\delta} & 0 \\
       0 & 1
    \end{pmatrix} \begin{pmatrix}
        x_{j+1,h} \\ y_{j+1,h}
    \end{pmatrix} \right\} d \boldsymbol{x}_{j+1} d \boldsymbol{y}_{j+1},
\end{align*}
which is equal to $1$. Similarly, if $j_1 = j_2 = j$, 
\begin{align*}
    \int & \frac{d P_{j_1}}{d P_0}\frac{d P_{j_2}}{d P_0}d P_0  = \int (1-\delta)^{-n_j}(1+\delta)^{-n_{j+1}} \prod_{h \in [n_j]} \exp\left\{-\frac{\delta}{1-\delta} y_{j,h}^2\right\} \prod_{h \in [n_{j+1}]} \exp\left\{\frac{\delta}{1+\delta} x_{j+1,h}^2\right\} dP_0 \\
    & = (1-\delta)^{-n_j}(1+\delta)^{-n_{j+1}} \int \prod_{h \in [n_{j}]} \frac{1}{2\pi}\exp\left\{ - \frac{1}{2} \begin{pmatrix}
        x_{j,h} \\ y_{j,h}
    \end{pmatrix}^T \begin{pmatrix}
       1 & 0 \\
       0 & \frac{1+\delta}{1-\delta}
    \end{pmatrix} \begin{pmatrix}
        x_{j,h} \\ y_{j,h}
    \end{pmatrix} \right\} d \boldsymbol{x}_{j} d \boldsymbol{y}_{j} \\
    & \quad\quad\quad \quad\quad\quad \times \int \prod_{h \in [n_{j+1}]} \frac{1}{2\pi}\exp\left\{ - \frac{1}{2} \begin{pmatrix}
        x_{j+1,h} \\ y_{j+1,h}
    \end{pmatrix}^T \begin{pmatrix}
       \frac{1-\delta}{1+\delta} & 0 \\
       0 & 1
    \end{pmatrix} \begin{pmatrix}
        x_{j,h} \\ y_{j,h}
    \end{pmatrix} \right\} d \boldsymbol{x}_{j+1} d \boldsymbol{y}_{j+1} \\
    & = (1-\delta^2)^{-n_j/2} (1-\delta^2)^{-n_{j+1}/2} \int \prod_{h \in [n_{j}]} \frac{1}{2\pi}\sqrt{\frac{1+\delta}{1-\delta}}\exp\left\{ - \frac{1}{2} \begin{pmatrix}
        x_{j,h} \\ y_{j,h}
    \end{pmatrix}^T \begin{pmatrix}
       1 & 0 \\
       0 & \frac{1+\delta}{1-\delta}
    \end{pmatrix} \begin{pmatrix}
        x_{j,h} \\ y_{j,h}
    \end{pmatrix} \right\} d \boldsymbol{x}_{j} d \boldsymbol{y}_{j} \\
    & \quad\quad\quad \quad\quad\quad \times \int \prod_{h \in [n_{j+1}]} \frac{1}{2\pi}\sqrt{\frac{1-\delta}{1+\delta}}\exp\left\{ - \frac{1}{2} \begin{pmatrix}
        x_{j+1,h} \\ y_{j+1,h}
    \end{pmatrix}^T \begin{pmatrix}
       \frac{1-\delta}{1+\delta} & 0 \\
       0 & 1
    \end{pmatrix} \begin{pmatrix}
        x_{j,h} \\ y_{j,h}
    \end{pmatrix} \right\} d \boldsymbol{x}_{j+1} d \boldsymbol{y}_{j+1} \\
    & = (1-\delta^2)^{-n_j/2} (1-\delta^2)^{-n_{j+1}/2} = (1-\delta^2)^{-(n_j+n_{j+1})/2}.
\end{align*}
It follows that
\begin{align*}
    4 \operatorname{TV}\Bigg\{P_0, \frac{1}{d}\sum_{j=1}^d P_j \Bigg\}^2 & \leq \frac{1}{d^2}\sum_{j_1, j_2 =1}^d \int \frac{d P_{j_1}}{d P_0}\frac{d P_{j_2}}{d P_0}d P_0  - 1 \\
    & = \frac{1}{d^2}\sum_{j_1, j_2 =1}^d \mathbbm{1}_{j_1 = j_2} (1-\delta^2)^{-(n_{j_1} + n_{j_1+1})/2} + \mathbbm{1}_{j_1 \neq j_2} \\
    & = \frac{1}{d^2} \sum_{j =1}^d (1-\delta^2)^{-(n_j + n_{j+1})/2} - \frac{1}{d} \\
    & \leq  \frac{1}{d^2} \sum_{j =1}^d \exp\{+(n_j + n_{j+1})\delta^2/2\} - \frac{1}{d}, 
\end{align*}
from which we see that $\operatorname{TV}\left\{P_0, \frac{1}{d}\sum_{j=1}^d P_j \right\} \leq 1/2$ if $\delta \leq \sqrt{2\log(1+d)/(n_j + n_{j+1})}$ for all $j \in [d]$. The above bound on the total variation distance demonstrates that we may choose $\delta=\sqrt{\log(1+d)/\min_j n_j}$, and hence that we have
\[
    \rho^* \geq \delta = \biggl\{ \frac{\log(1+d)}{\min_j n_j} \biggr\}^{1/2},
\]
as claimed.
\end{proof}

\begin{proof}[Proof of Proposition \ref{prop:KKT_cycle}]
    (i)  We may suppose without loss of generality that $|\rho_i| \neq 1$ as, otherwise, we may perform the reduction given in Proposition \ref{prop:reduction_2}. Possibly, this reduces the $d$-cycle to a $3$-cycle: if there are no more correlations equal to $\pm 1$, then we proceed, otherwise we know $R$ exactly thanks to Example \ref{ex:explicit_theta_zero} and we can check that the claim holds. Now, calling $M_{i,i+1} = \cos\varphi_i$, we have 
    \begin{align*}
         \lambda^* = 1 - R(\Sigma_{\mathbb{S}}) &= \frac{1}{d} \sup\{ \mathrm{tr}(\Sigma) : \Sigma \succeq 0, \Sigma_{11} = \ldots = \Sigma_{dd}, \Sigma_{\mathbb{S}'} - A \Sigma \succeq_{\mathbb{S}'} 0\} \\
         &= \sup\{\lambda : M \succeq 0, M_{11} = \ldots = M_{dd} = 1, 1-\lambda \geq |\rho_i -\lambda\cos\varphi_i| \text{ for all } i \in [d] \}\\
         & = \sup\bigg\{\lambda : \lambda \leq  \min_{i \in [d]}\min\bigg\{\frac{1-\rho_i}{1 - \cos\varphi_i}, \frac{1+\rho_i}{1+ \cos\varphi_i} \bigg\} , \\
         & \quad \hspace{3cm} \sum\limits_{i \in K} \varphi_i \leq (|K| -1 ) \pi + \sum\limits_{i \not\in K} \varphi_i \text{ for all } K \subseteq [d] \text{ with }|K| \text{ odd}\bigg\}.
    \end{align*}
   Now, this implies that 
    \begin{align*}
         \frac{1}{\lambda^*} &= \inf\bigg\{z : z \geq  \max_{i \in [d]}\max\bigg\{\frac{1 - \cos\varphi_i}{1-\rho_i}, \frac{1+ \cos\varphi_i}{1+\rho_i} \bigg\} , \\
         & \quad \hspace{3cm} \sum\limits_{i \in K} \varphi_i \leq (|K| -1 ) \pi + \sum\limits_{i \not\in K} \varphi_i \text{ for all } K \subseteq [d] \text{ with }|K| \text{ odd}\bigg\}.
    \end{align*}
    Calling $g(\varphi_i) = (1 - \cos\varphi_i)/(1-\rho_i)$ and $h(\varphi_i) = (1 + \cos\varphi_i)/(1+\rho_i)$ for all $i \in [d]$, this is a linearly constrained finite minimax problem (see Chapter 2 in \cite{polak_2012_optimization}), namely
    \[
    \frac{1}{\lambda^*} = \min \max_{i \in [d]}\max\{g(\varphi_i), h(\varphi_i) \}
    \]
    under the $2^{d-1}$ linear constraints \[
    \sum\limits_{i \in K} \varphi_i \leq (|K| -1 ) \pi + \sum\limits_{i \not\in K} \varphi_i \text{ for all } K \subseteq [d] \text{ with }|K| \text{ odd},
    \]
    which is equivalent to
    \begin{align*}
        \text{ minimise } & z \\
        \text{ subject to } & g(\varphi_i) \leq z \text { for all } i \in [d], \\
        & h(\varphi_i) \leq z \text { for all } i \in [d], \\
        & \sum\limits_{i \in K} \varphi_i \leq (|K| -1 ) \pi + \sum\limits_{i \not\in K} \varphi_i \text{ for all } K \subseteq [d] \text{ with }|K| \text{ odd}.
    \end{align*}
    As a result, every optimal solution $(\varphi_1^*, \ldots, \varphi_d^*)$ must satisfy the Karush–Kuhn–Tucker (KKT) conditions (see Chapter 5 of \cite{boyd_conex_opt_2004}, Chapter 28-30 of \cite{rockafellar-1970a})
    \begin{align*}
    (i) \quad & \left(\frac{\lambda_i}{1-\rho_i} - \frac{\lambda_{i+d}}{1+\rho_i}\right)\sin\varphi_i = \mathop{\sum_{|K| odd}}\limits_{i \in K} \mu_K - \mathop{\sum_{|K| odd}}\limits_{i \in K^c} \mu_K, \text{ for all } i \in [d], \\
    (ii) \quad & \lambda_i \geq 0, \lambda_{d+i} \geq 0, \text{ for all } i \in [d], \\
    (iii) \quad & \sum_{i = 1}^d (\lambda_i + \lambda_{d+i}) = 1, \\
    (iv) \quad & \lambda_i (g(\varphi_i) - \max_{i \in [d]}\max\{g(\varphi_i), h(\varphi_i) \}) = 0, \text{ for all } i \in [d], \\
    (v) \quad & \lambda_{d+i} (h(\varphi_i) - \max_{i \in [d]}\max\{g(\varphi_i), h(\varphi_i) \}) = 0, \text{ for all } i \in [d], \\
    (vi) \quad & g(\varphi_i) \leq \max_{i \in [d]}\max\{g(\varphi_i), h(\varphi_i) \}), \text{ for all } i \in [d], \\
    (vii) \quad & h(\varphi_i) \leq \max_{i \in [d]}\max\{g(\varphi_i), h(\varphi_i) \}), \text{ for all } i \in [d], \\
     (viii) \quad & \mu_K \geq 0, \text{ for all } K \subseteq [d] \text{ with }|K| \text{ odd}, \\
     (ix) \quad & \sum\limits_{i \in K} \varphi_i \leq (|K| -1 ) \pi + \sum\limits_{i \not\in K} \varphi_i \text{ for all } K \subseteq [d] \text{ with }|K| \text{ odd}, \\
     (x) \quad & \mu_K\left(\sum\limits_{i \in K} \varphi_i - (|K| -1 ) \pi - \sum\limits_{i \not\in K} \varphi_i\right) =0,  \text{ for all } K \subseteq [d] \text{ with }|K| \text{ odd}. \\
    \end{align*}

    Now, observe that conditions (iv) and (v) imply that, for all $i \in [d]$, either $g(\varphi_i)$ or $h(\varphi_i)$ reaches the maximum, meaning that the minimal $1/\lambda^*$ is equal to this common value. Indeed, if the original $d$-cycle is completable, this statement is trivial, since we must have $|\rho_i - \cos\varphi^*_i| = 0$. This is the only case in which we can have \[
    \frac{1-\rho_i}{1 - \cos\varphi_i} = \frac{1+\rho_i}{1+ \cos\varphi_i} = 1,
    \]
    meaning that when the $d$-cycle is incompatible, then either $\max_{i \in [d]}\max\{g(\varphi_i), h(\varphi_i) \} - g(\varphi_i) > 0$ or $\max_{i \in [d]}\max\{g(\varphi_i), h(\varphi_i) \} - h(\varphi_i) > 0$.  Indeed, if $R > 0$, either $\lambda_i$ or $\lambda_{d+i}$ must be equal to zero since either $g(\varphi_i)$ or $h(\varphi_i)$ has a strictly positive gap from $\max_{i \in [d]}\max\{g(\varphi_i), h(\varphi_i) \}$. If both $\lambda_i = 0$ and $\lambda_{i+d} = 0$, we would have \[
    \mathop{\sum_{|K| odd}}\limits_{i \in K} \mu_K = \mathop{\sum_{|K| odd}}\limits_{i \in K^c} \mu_K,\]
    which is a contradiction due to the fact that there exists a unique $\mu_K \neq 0$. To prove the existence part, observe that if $\mu_K = 0$ for all $K \subseteq [d]$ with $|K|$ odd, then we would have \[
    \left(\frac{\lambda_i}{1-\rho_i} - \frac{\lambda_{i+d}}{1+\rho_i}\right)\sin\varphi_i = 0
    \]
    for all $i \in [d]$, and since there exists at least a $j$ such that $\lambda_j + \lambda_{d+j} > 0$ due to (iii), this would imply that $\varphi_j \in \{0, \pi\}$, which leads to $\theta_j \in \{0, \pi\}$, which is excluded from our analysis. To prove the uniqueness part, suppose there exists another $[d] \supseteq M \neq K$, with $|M|$ odd, such that
    \[
    \begin{cases}
        \sum\limits_{i \in K} \varphi_i = (|K| -1 ) \pi + \sum\limits_{i \not\in K} \varphi_i \\
        \sum\limits_{i \in M} \varphi_i = (|M| -1 ) \pi + \sum\limits_{i \not\in M} \varphi_i, 
    \end{cases}
    \]
    hence summing these equalities gives
    \[
    2\left(\sum\limits_{i \in K\cap M} \varphi_i - \sum\limits_{i \in K^c \cap M^c} \varphi_i \right) = (|K|+|M|-2)\pi.
    \]
    Now, if we suppose that $K^c \cap M^c = \varnothing$, meaning that $K\cup M = [d]$, it is easy to show that $2|K \cap M| \leq |K|+|M|-2$. Indeed, $|K \cap M| \leq |K| \wedge |M|$, with equality if and only if $M \subseteq K$ (or viceversa): in this case we must have $|K| \geq |M| + 2$, otherwise they would be equal, hence $2|K \cap M| \leq 2(|K|\wedge|M|) = 2|M|$ while $|K|+|M|-2 \geq |M|+2+|M|-2 = 2|M|$. If the equality is not reached, $2|K \cap M| \leq 2(|K|\wedge|M|-1) = 2|M| - 2$, while $|K|+|M|-2 \geq |M| + |M| - 2 = 2|M| - 2$. This shows that $2|K \cap M| \leq |K|+|M|-2$, which implies that the equality above can be verified only if $\varphi_i = \pi$ for all $i \in K \cap M$, which is excluded from our analysis. Furthermore, if $K^c \cap M^c \neq \varnothing$, this is even worse unless $\varphi_i = 0$ for all $i \in K^c \cap M^c$, which is again excluded from our analysis. This completes the proof of the fact that for all $i \in [d]$, if $R > 0$, exactly one between $\lambda_i$ and $\lambda_{d_i}$ is greater than zero. As a corollary, we have that the optimal $(\varphi_1^*, \ldots, \varphi_d^*)$ satisfies \[
    1 - \lambda^* = |\rho_i - \lambda^*\cos\varphi_i^*|, \text{ for all } i \in [d],
    \]
    as required.
    
    (ii) The primal set is strictly feasible, hence we know that $R$ is attained in the dual set, which is enough to prove existence. As for uniqueness, suppose there exists two optimal $\Sigma_1, \Sigma_2$ such that 
    \[
    \begin{cases}
        \Sigma_{\mathbb{S}} = \lambda^*A\Sigma_1 + (1-\lambda^*)\Sigma_{\mathbb{S}}' \\
    \Sigma_{\mathbb{S}} = \lambda^*A\Sigma_2 + (1-\lambda^*)\Sigma_{\mathbb{S}}''. 
    \end{cases}
    \]
    This implies that for all $\mu \in (0,1)$ 
    \[
    \Sigma_{\mathbb{S}} = \lambda^*A(\mu\Sigma_1+(1-\mu)\Sigma_2) + (1-\lambda^*)(\mu\Sigma_{\mathbb{S}}'+(1-\mu)\Sigma_{\mathbb{S}}''),
    \]
    meaning that $\mu\Sigma_1+(1-\mu)\Sigma_2$ is optimal. By the optimality of $\Sigma_1$ and $\Sigma_2$ we must have that $\Sigma_{\mathbb{S}}'$ and $\Sigma_{\mathbb{S}}''$ are maximally incompatible, which means they must all be singular, as stated in Example \ref{ex:explicit_theta_zero}. Now, observe that if there exists $i \in [d]$ such that $\Sigma_{\{i,i+1\}}' \neq \Sigma_{\{i,i+1\}}''$, \[
    \mu\Sigma_{\{i,i+1\}}' + (1-\mu)\Sigma_{\{i,i+1\}}'' = \begin{pmatrix}
        1 & \pm(2\mu - 1)\\
        \pm(2\mu - 1) & 1
    \end{pmatrix},
    \]
    which means that $\mu\Sigma_{\mathbb{S}}'+(1-\mu)\Sigma_{\mathbb{S}}''$ can never be maximally incompatible since $\mu \in (0,1)$. This implies that $\Sigma_{\mathbb{S}}' = \Sigma_{\mathbb{S}}''$, which in turn implies that $\boldsymbol{\varphi}_1^* = \boldsymbol{\varphi}_2^*$. As for the continuity of $\boldsymbol{\varphi}^*(\theta_1, \ldots, \theta_d)$, observe that $1 - \lambda^* = |\rho_j - \lambda^*\cos\varphi_j^*|, \text{ for all } j \in [d]$ in point (i) means that there exist $\{\epsilon_j = \pm 1\}_{j \in [d]}$ such that \[
    \lambda^* = \frac{1-\epsilon_j \cos\theta_j}{1-\epsilon_j \cos\varphi_j^*}, \quad \text{ for all } j \in [d].
    \]
    Now, let $\left\{\boldsymbol{\theta}^{(n)} = \left(\theta_{1,n}, \ldots, \theta_{d,n}\right)\right\}_{n \in \mathbb{N}} \longrightarrow \boldsymbol{\theta} = (\theta_1, \ldots, \theta_d)$, and consider the associated sequence of optimal $\left\{\boldsymbol{\varphi}_n^{*} = (\varphi_{1,n}^*, \ldots, \varphi_{d,n}^*)\right\}_{n \in \mathbb{N}}$, meaning that \[
    \lambda_n^* =  \frac{1-\epsilon_{j,n} \cos\theta_{j,n}}{1-\epsilon_{j,n} \cos\varphi_{j,n}^*}, \quad \text{ for all } j \in [d].
    \]
    Taking the limit on both sides, since $\lambda^*$ is continuous due to Proposition \ref{prop:properties} (ii), we get that 
    \[
    \lambda^* = \frac{1 \pm \cos\theta_j}{1 \pm \cos\left(\lim_n \varphi_{j,n}^* \right)}, \quad \text{ for all } j \in [d].
    \]
    This shows that the limit $\lim_n \varphi_{j,n}^*$ exists, and by uniqueness (i), we can conclude that $\lim_n \varphi_{j,n}^* = \varphi_j^*$, showing that $\boldsymbol{\varphi}$ is continuous.
    
    As for (iii), supposing without loss of generality that $\theta_1 = \max_{i \in [d]} \theta_i$, with at most $\theta_1 > \pi/2$, observe that incompatibility is equivalent to having $\theta_1 - \sum_{i=2}^d \theta_i > 0$, hence in order to make $\lambda^*$ as big as possible we should choose $\rho_j = \lambda^*\cos\varphi_j^* + (1-\lambda^*)$ for all $ j \in \{2, \ldots, d\}$, and $\rho_1 = \lambda^*\cos\varphi_1^* - (1-\lambda^*)$. This would imply that the optimal choice of signs for a general $d$-cycle is $\epsilon_d = (-1, +\boldsymbol{1}_{d-1}^T)$, and this turns out to be true indeed. To see why, start by considering the case $d=3$, and observe that from (ii) we know that there exists a unique $K \subseteq [3]$ with $|K|$ odd such that \[
    \mathop{\sum_{|K| odd}}\limits_{i \in K} \mu_K = \mathop{\sum_{|K| odd}}\limits_{i \in K^c} \mu_K.\]
    The possible values of $K$ are $\{1\}, \{2\}, \{3\}$ and $\{1,2,3\}$, and these are associated to the vectors of signs $(-1,1,1)$,$ (1,-1,1)$, $(1,1,-1)$ and $(-1,-1,-1)$, respectively. Hence, in order to prove the statement it is necessary and sufficient to show that $K = \{1\}$ leads to the optimal $\lambda^*$, meaning that $\lambda^*_1 \geq \lambda^*_2$ and $\lambda^*_1 \geq \lambda^*_4$, where \[
    \lambda^*_1 = \frac{1+\cos\theta_1}{1+\cos\varphi_1^*} = \frac{1-\cos\theta_2}{1-\cos\varphi_2^*} = \frac{1-\cos\theta_3}{1-\cos(\varphi_1^*-\varphi_2^*)},
    \]
    \[
    \lambda^*_2 = \frac{1-\cos\theta_1}{1-\cos\tilde{\varphi}_1^*} = \frac{1+\cos\theta_2}{1+\cos\tilde{\varphi}_2^*} = \frac{1-\cos\theta_3}{1-\cos(\tilde{\varphi}_2^*-\tilde{\varphi}_1^*)}, 
    \]
\[
    \lambda^*_4 = \frac{1+\cos\theta_1}{1+\cos\tilde{\varphi}_1^*} = \frac{1+\cos\theta_2}{1+\cos\tilde{\varphi}_2^*} = \frac{1+\cos\theta_3}{1+\cos(2\pi - \tilde{\varphi}_1^*-\tilde{\varphi}_2^*)}. 
    \]
        Now, for $\lambda^*_1 < \lambda^*_2$ to be true it is necessary to have \[ 
    \begin{cases}
        \cos(\varphi_1^*-\varphi_2^*) < \cos(\tilde{\varphi}_1^*-\tilde{\varphi}_2^*) \\
        \\
        \cos\tilde{\varphi}_1^* > 1 - \frac{1-\cos\theta_1}{1+\cos\theta_1}(1+\cos\varphi_1^*)\\
        \\
        \cos\tilde{\varphi}_2^* < -1 + \frac{1+\cos\theta_2}{1-\cos\theta_2}(1-\cos\varphi_2^*),\\
    \end{cases}
    \]
    with $(\varphi_1^*, \varphi_2^*), (\tilde{\varphi}_1^*, \tilde{\varphi}_2^*) \in [0, \pi]^2$ that need to simultaneously satisfy
\begin{align*}
    \begin{cases}
        \cos\varphi_2^* = 1 - \frac{1-\cos\theta_2}{1+\cos\theta_1}(1+\cos\varphi_1^*)\\
        \\
        \cos\tilde{\varphi}_1^* = 1 - \frac{1-\cos\theta_1}{1+\cos\theta_2}(1+\cos\tilde{\varphi}_2^*),\\
    \end{cases} & \text{ and } \quad \quad  \begin{cases}
        \varphi_1^* \leq \theta_1, \quad \quad \varphi_2^* \geq \theta_2 \\
        \\
        \tilde{\varphi}_1^* \geq \theta_1, \quad \quad \tilde{\varphi}_2^* \leq \theta_2, \\
    \end{cases}
\end{align*}
    to ensure $R(\Sigma_{\mathbb{S}_3}) \in [0,1]$. This system of inequalities has no solution in $(\tilde{\varphi}_1^*, \tilde{\varphi}_2^*)$ for fixed $(\varphi_1^*, \varphi_2^*)$ and $\theta_1 > \theta_2$. The same reasoning shows that $\lambda^*_1 < \lambda^*_4$ can never be satisfied as well, showing that the optimal choice of signs for $d=3$ is indeed $\epsilon_3 = (-1, +1, +1)$. For general $d$, it is sufficient to proceed by induction: indeed, suppose that $\epsilon_j = (-1, +\boldsymbol{1}_{j-1}^T)$ for all $j \in \{3, \ldots, d-1\}$, and consider $\lambda^* = \lambda^*(\theta_d)$ as a function of $\theta_d$, for fixed $\theta_1, \ldots, \theta_{d-1}$. This function is continuous over $[0, \theta_1 - \sum_{i = 2}^{d-1}\theta_i)$, because is the restriction of $\lambda^* = 1 - R(\Sigma_{\mathbb{S}_d})$, which is continuous by Proposition \ref{prop:properties} (ii), onto the last coordinate. Now, $\lambda^*(\theta_d)$ uniquely identifies a vector of signs for varying $\theta_d \in [0, \theta_1 - \sum_{i = 2}^{d-1}\theta_i)$, call it $\epsilon(\theta_d)$, taking values in $\{+1, -1\}^d$. This vector is unique because we supposed the cycle to be incompatible, hence either $g(\varphi_i)$ or $h(\varphi_i)$ in the KKT conditions has a strictly positive optimal gap, so that there exists a unique $\mu_K \neq 0$. We will show that this vector is constant for all  $\theta_d \in [0, \theta_1 - \sum_{i = 2}^{d-1}\theta_i)$, that is to say that each component of $\epsilon(\theta_d)$ is continuous in  $\theta_d \in [0, \theta_1 - \sum_{i = 2}^{d-1}\theta_i)$. Indeed, consider without loss of generality the first component of $\epsilon(\theta_d)$, and suppose by contradiction that $\epsilon(\theta_d)_1$ is not continuous in $\tilde{\theta}_d$. This implies that there exists a sequence of angles $\{\theta_{d,n}\}$ converging to $\tilde{\theta}_d$ such that \[
    \lim\limits_{n \rightarrow +\infty}\epsilon(\theta_{d,n})_1  = \epsilon_{\mathrm{lim}} = -\epsilon(\tilde{\theta}_d)_1.
    \]  
    Without loss of  generality, assume $\epsilon_{\mathrm{lim}} = +1$ and $\epsilon(\tilde{\theta}_d)_1 = -1$. But we must have by continuity \begin{align*}
         \frac{1+ \cos\theta_1}{1+\cos\varphi_1^*} & =\frac{1-\epsilon(\tilde{\theta}_d)_1\cos\theta_1}{1-\epsilon(\tilde{\theta}_d)_1\cos\varphi_1^*} = \lambda^*(\theta_d) = \lim\limits_{n \rightarrow +\infty} \lambda^*(\theta_{d,n}) = \lim\limits_{n \rightarrow +\infty} \frac{1-\epsilon(\theta_{d,n})_1\cos\theta_1}{1-\epsilon(\theta_{d,n})_1\cos\varphi_{1,n}^{*}} \\
        & = \frac{1- \cos\theta_1\lim\limits_{n \rightarrow +\infty}\epsilon(\theta_{d,n})_1 }{1-\cos\left(\lim\limits_{n \rightarrow +\infty}\varphi_{1,n}^{*}\right) \lim\limits_{n \rightarrow +\infty}\epsilon(\theta_{d,n})_1 } = \frac{1-\epsilon_{\mathrm{lim}}\cos\theta_1}{1-\epsilon_{\mathrm{lim}}\cos\varphi_1^*} =  \frac{1-\cos\theta_1}{1-\cos\varphi_1^*},
    \end{align*}
    where $\cos\varphi_{1,n}^*$ and $\cos\varphi_1^*$ are the $(1,2)$-th entries of the optimal matrix of the dual in $\theta_{d,n}$ and $\tilde{\theta}_d$, respectively. This implies that $\lambda^*(\tilde{\theta}_d)$ admits both representations, one with the plus sign, and one with the minus sign, and this can happen only if the cycle is compatible, which cannot be the case for $\tilde{\theta}_d \in [0, \theta_1 - \sum_{i = 2}^{d-1}\theta_i)$ since $\theta_1 > \sum_{i=2}^d \theta_i$. This means that $\epsilon(\theta_d)_j$ is continuous for all $j \in [d]$ for varying $\theta_d \in [0, \theta_1 - \sum_{i = 2}^{d-1}\theta_i)$, which implies that the vector $\epsilon(\theta_d)$ is constant on $[0, \theta_1 - \sum_{i = 2}^{d-1}\theta_i)$, so that the behaviour of $\epsilon(\theta_d)$ is uniquely determined by $\epsilon(0)$. But we do know that \[\epsilon(0) = (\epsilon_{d-1}, \epsilon(0)_d) = (-1, +\boldsymbol{1}_{d-2}^T, \epsilon(0)_d) \]
    due to Proposition \ref{prop:reduction_2} and the induction step: this, together with the fact that $\epsilon(0)_d = +1$ in order to make $\Sigma'_{\mathbb{S}}$ maximally incompatible, completes the proof. \\
\end{proof}

\subsection{Proofs for Section \ref{subsec:block}}

\begin{proof}[Proof of Theorem \ref{block_minimax}]
Formally, we are looking at
         \[
         H_0: R(\Sigma_\mathbb{S}) = 0 \quad \text{ vs. } \quad H_1(\rho): R(\Sigma_\mathbb{S}) > \rho, 
        \]
         for fixed $\rho > 0$, and we aim at finding the smallest of such $\rho$'s for which we can have non-trivial power.
We prove the result by considering the two cases $d \geq 42$ and $d<42$ separately. For the first of these, we specialise $\Sigma_{\mathbb{S}}$ to be 
         \[
         \Sigma_{\mathbb{S}} = \Bigg\{\begin{pmatrix}
            I_d & P \\
            P^T & I_d
        \end{pmatrix},\begin{pmatrix}
            I_d & -P \\
            -P^T & I_d
        \end{pmatrix}, \begin{pmatrix}
            I_d & \beta I_d \\
            \beta I_d & I_d
        \end{pmatrix}\Bigg\},
        \]
        and since $R(\Sigma_\mathbb{S}) \geq  \frac{3}{4d}\sum_{j = 1}^d (\sigma_j^2(P) -  \frac{1-\beta}{2})_+$ by Proposition \ref{prop:block_3_cycle}, it is sufficient to study the testing problems
        \[
         H_0': \sum_{j = 1}^d (\theta_i)_+ \leq 0 \quad \text{ vs. } \quad H_1'(\rho'): \sum_{j = 1}^d (\theta_i)_+ > \rho', 
         \]
         where $\theta_i = \sigma_j^2(P) -  \frac{1-\beta}{2}$, find the smallest $\rho'$ for which we have non-trivial power, and use the relationship $\rho = \frac{3}{4d} \rho'$. More precisely, focusing on the latter testing problem, we want to lower bound the minimax testing risk
\[
\rho^*(n_{\mathbb{S}}, \eta) := \inf\left\{\rho > 0: \exists \varphi_{\mathbb{S}} \in \Psi_{\mathbb{S}} : \sup_{P_{\mathbb{S},0} \in \bar{\mathcal{P}}_{\mathbb{S}}(0)} P_{\mathbb{S},0}^{\otimes n_{\mathbb{S}}}(\varphi_{\mathbb{S}} = 1) + \sup_{P_{\mathbb{S},1} \in \mathcal{P}_{\mathbb{S}}(\rho)} P_{\mathbb{S},1}^{\otimes n_{\mathbb{S}}} (\varphi_{\mathbb{S}} = 0) \leq \eta \right\}
\]
        where $\bar{\mathcal{P}}_{\mathbb{S}}(0) = \{P_{\mathbb{S}} : \operatorname{Corr}(P_{\mathbb{S}}) = \Sigma_{\mathbb{S}} \text{ and } R(\Sigma_\mathbb{S}) = 0 \}$, $\mathcal{P}_{\mathbb{S}}(\rho) =  \{P_{\mathbb{S}} : \operatorname{Corr}(P_{\mathbb{S}}) = \Sigma_{\mathbb{S}} \text{ and } R(\Sigma_\mathbb{S}) > \rho \}$, and $\Psi_{\mathbb{S}}$ is the set of collection of tests coherent with $\mathbb{S}$. To this aim, we start by defining two prior distributions $\mu_0, \mu_1$ for $P$. First, there exist two measures $\nu_0, \nu_1$ with matching moments up to the $M$-th order such that \begin{itemize}
        \item[I.] $\operatorname{supp}\left(\nu_0\right) \subseteq[-b, 0], \quad \operatorname{supp}\left(\nu_1\right) \subseteq[-b, 0] \cup\left\{\frac{b}{4 M^2}\right\}$
        \item[II.] $\nu_1\left(\left\{\frac{b}{4 M^2}\right\}\right) \geq \frac{1}{2}$
        \item[III.] $\forall k \in\{0,1, \ldots, M\}: \int z^k \nu_0(\mathrm{~d} z)=\int z^k \nu_1(\mathrm{~d} z)$.
    \end{itemize}
    This is proved in \cite{jud_nemiroski_matching_mom02} using ideas from the theory of best polynomial approximation. A different, but closely related version, was proved in \cite{cai2011testing} using similar techniques. Such prior distributions have been extensively used in the minimax literature in the last decade, and led to optimal, or nearly-optimal, lower bounds in many problems of interest such as optimal estimation of nonsmooth functionals \citep{cai2011testing, weissman16, verzelen2021schatten}, testing MCAR in a fully nonparametric setting \citep{berrett2022optimal}, and testing convex hypothesis \citep{carpentier18convex}. Let $\mathcal{U}(d)$ denote the (normalised) Haar measure over the Lie group of orthogonal matrices $SO(d) = \{ U \in \mathbb{R}^{d,d} : U^T U = UU^T = I_d\}$, and let $\nu_0, \nu_1$ the distributions with matching moments up to the order $M$ defined above. Calling $\delta_0$ the Dirac measure in zero, we define $\mu_i$ to be the distribution of $P=U^T \Lambda U$, where $U \sim \mathcal{U}(d)$, and $\Lambda = \operatorname{diag}(\sigma_{1:d})$, with $\sigma_{1:d} \sim \nu_i^{\otimes \lceil d/2 \rceil } \otimes \delta_0 ^{\otimes (d - \lceil d/2 \rceil)}$, for $i \in \{0, 1\}$. Observe now that the support of $\mu_1$ also contains elements in $(\mathcal{P}_\mathbb{S}(\rho))^C$. In order to overcome this, we will consider the conditional measure $\mu_1|\xi$, where $\xi$ is the event \[
\xi=\left\{\sum_{i=1}^d \mathbbm{1}_{\left\{\mu_i=b/4M^2\right\}} \geq \frac{d}{3}\right\},
\]
which ensures that $\mu_1|\xi$ is supported on the alternative. Now, given $P$, we use the shorthand
\[
N_\mathbb{S}^{\otimes n_\mathbb{S}} \equiv N_\mathbb{S}^{\otimes n_\mathbb{S}}(P) = \left(N^{\otimes n_1}\left(0, \begin{pmatrix}
            I_d & P \\
            P^T & I_d
        \end{pmatrix}\right), N^{\otimes n_2}\left(0, \begin{pmatrix}
            I_d & -P \\
            -P^T & I_d
        \end{pmatrix}\right), N^{\otimes n_3}\left(0, \begin{pmatrix}
            I_d & \beta I_d \\
            \beta I_d & I_d
        \end{pmatrix}\right)\right).
\]
The marginal distribution of the data when $P$ is generated according to $\mu_1|\xi$ is then given by the mixture distribution
\[
\mathbb{E}_{\mu_1|\xi}N_\mathbb{S}^{\otimes n_\mathbb{S}} = \left(\mathbb{E}_{\mu_1|\xi}N^{\otimes n_1}\left(0, \begin{pmatrix}
            I_d & P \\
            P^T & I_d
        \end{pmatrix}\right), \mathbb{E}_{\mu_1|\xi}N^{\otimes n_2}\left(0, \begin{pmatrix}
            I_d & -P \\
            -P^T & I_d
        \end{pmatrix}\right), N^{\otimes n_3}\left(0, \begin{pmatrix}
            I_d & \beta I_d \\
            \beta I_d & I_d
        \end{pmatrix}\right)\right).
\]
Similarly, the marginal distribution of the data when $P$ is generated according to $\mu_i$ is then given by the mixture distribution
\[
\mathbb{E}_{\mu_i}N_\mathbb{S}^{\otimes n_\mathbb{S}} = \left(\mathbb{E}_{\mu_i}N^{\otimes n_1}\left(0, \begin{pmatrix}
            I_d & P \\
            P^T & I_d
        \end{pmatrix}\right), \mathbb{E}_{\mu_i}N^{\otimes n_2}\left(0, \begin{pmatrix}
            I_d & -P \\
            -P^T & I_d
        \end{pmatrix}\right), N^{\otimes n_3}\left(0, \begin{pmatrix}
            I_d & \beta I_d \\
            \beta I_d & I_d
        \end{pmatrix}\right)\right),
\]
for $i \in \{0,1 \}$.
For every test collection $\varphi_{\mathbb{S}} \in \Psi_{\mathbb{S}}$, and for prior distributions $\mu_0, \mu_1|\xi$, we can bound the total error probability as
\begin{align*}
             \mathcal{R}(n_\mathbb{S}, \rho') &  = \sup_{P_{\mathbb{S},0} \in \bar{\mathcal{P}}_{\mathbb{S}}(0)} P_{\mathbb{S},0}^{\otimes n_{\mathbb{S}}}(\varphi_{\mathbb{S}} = 1) + \sup_{P_{\mathbb{S},1} \in \mathcal{P}_{\mathbb{S}}(\rho)} P_{\mathbb{S},1}^{\otimes n_{\mathbb{S}}} (\varphi_{\mathbb{S}} = 0) \\
             & \geq \mathbb{E}_{\mu_0}N_\mathbb{S}^{\otimes n_\mathbb{S}}(\varphi_{\mathbb{S}} = 1) + \mathbb{E}_{\mu_1|\xi}N_\mathbb{S}^{\otimes n_\mathbb{S}}(\varphi_{\mathbb{S}} = 0) \\
             & = \mathbb{E}_{\mu_0}N_\mathbb{S}^{\otimes n_\mathbb{S}}(\varphi_{\mathbb{S}} = 1)+ 1-\frac{\mathbb{E}_{\mu_1}N_\mathbb{S}^{\otimes n_\mathbb{S}}(\{\varphi_\mathbb{S}=1\} \cap \xi)}{\mu_1(\xi)} \\
& \geq \mathbb{E}_{\mu_0}N_\mathbb{S}^{\otimes n_\mathbb{S}}(\varphi_{\mathbb{S}} = 1) + 1-\frac{10}{9}\mathbb{E}_{\mu_1}N_\mathbb{S}^{\otimes n_\mathbb{S}}(\varphi_{\mathbb{S}} = 1) \\
& \geq \mathbb{E}_{\mu_0}N_\mathbb{S}^{\otimes n_\mathbb{S}}(\varphi_{\mathbb{S}} = 1) + \frac{10}{9}\mathbb{E}_{\mu_1}N_\mathbb{S}^{\otimes n_\mathbb{S}}(\varphi_{\mathbb{S}} = 0)-\frac{1}{9} \\
& \geq \mathbb{E}_{\mu_0}N_\mathbb{S}^{\otimes n_\mathbb{S}}(\varphi_{\mathbb{S}} = 1) + \mathbb{E}_{\mu_1}N_\mathbb{S}^{\otimes n_\mathbb{S}}(\varphi_{\mathbb{S}} = 0)-\frac{1}{9} \\
& \geq 1 - \operatorname{TV}\left(\mathbb{E}_{\mu_0}N_\mathbb{S}^{\otimes n_\mathbb{S}}, \mathbb{E}_{\mu_1}N_\mathbb{S}^{\otimes n_\mathbb{S}} \right) -\frac{1}{9}.
\end{align*}
The second inequality follows from Hoeffding's inequality, which ensures that for all $d \geq 42$,
\[
\mu_1\left(\xi\right) \geq \frac{9}{10},
\]
since $\mu_1(\{b/4M^2\}) \geq 1/2$ by II. This shows that it is now sufficient to control the total variation distance between the marginals of $N_\mathbb{S}^{\otimes n_\mathbb{S}}$ with respect to the unconditional priors $\mu_0, \mu_1$ by finding $b/4M^2$ such that $\operatorname{TV}\left(\mathbb{E}_{\mu_0}N_\mathbb{S}^{\otimes n_\mathbb{S}}, \mathbb{E}_{\mu_1}N_\mathbb{S}^{\otimes n_\mathbb{S}} \right) \leq 1/2 - 1/9$. This would imply that $\mathcal{R}(n_\mathbb{S}, \rho') \geq 1/2$, and would lead to  \[
\rho' = \frac{d}{3} \frac{b}{4M^2},\]
where the extra $d/3$ factor comes from conditioning on the event $\xi$. Hence, let us now focus on controlling $\operatorname{TV}\left(\mathbb{E}_{\mu_0}N_\mathbb{S}^{\otimes n_\mathbb{S}}, \mathbb{E}_{\mu_1}N_\mathbb{S}^{\otimes n_\mathbb{S}} \right)$. We have 
        \begin{align*}
               \operatorname{TV}&\left(\mathbb{E}_{\mu_0}N_\mathbb{S}^{\otimes n_\mathbb{S}}, \mathbb{E}_{\mu_1}N_\mathbb{S}^{\otimes n_\mathbb{S}} \right) \\
             & = \operatorname{TV}\Bigg\{\left(\mathbb{E}_{\mu_0} N^{\otimes n_1}\left(0, \begin{pmatrix}
            I_d & P \\
            P^T & I_d
        \end{pmatrix}\right), \mathbb{E}_{\mu_0}N^{\otimes n_2}\left(0, \begin{pmatrix}
            I_d & -P \\
            -P^T & I_d
        \end{pmatrix}\right), N^{\otimes n_3}\left(0, \begin{pmatrix}
            I_d & \beta I_d \\
            \beta I_d & I_d
        \end{pmatrix}\right) \right), \\
        & \quad \quad \quad \quad \quad  \left(\mathbb{E}_{\mu_1} N^{\otimes n_1}\left(0, \begin{pmatrix}
            I_d & P \\
            P^T & I_d
        \end{pmatrix}\right), \mathbb{E}_{\mu_1}N^{\otimes n_2}\left(0, \begin{pmatrix}
            I_d & -P \\
            -P^T & I_d
        \end{pmatrix}\right), N^{\otimes n_3}\left(0, \begin{pmatrix}
            I_d & \beta I_d \\
            \beta I_d & I_d
        \end{pmatrix}\right) \right) \Bigg\} \\
        & = \operatorname{TV}\Bigg\{\left(\mathbb{E}_{\mu_0} N^{\otimes n_1}\left(0, \begin{pmatrix}
            I_d & P \\
            P^T & I_d
        \end{pmatrix}\right), \mathbb{E}_{\mu_0}N^{\otimes n_2}\left(0, \begin{pmatrix}
            I_d & -P \\
            -P^T & I_d
        \end{pmatrix}\right) \right), \\
        & \quad \quad \quad \quad \quad  \left(\mathbb{E}_{\mu_1} N^{\otimes n_1}\left(0, \begin{pmatrix}
            I_d & P \\
            P^T & I_d
        \end{pmatrix}\right), \mathbb{E}_{\mu_1}N^{\otimes n_2}\left(0, \begin{pmatrix}
            I_d & -P \\
            -P^T & I_d
        \end{pmatrix}\right) \right) \Bigg\} \\
        & \leq \operatorname{TV}\Bigg\{\mathbb{E}_{\mu_0}N^{\otimes n_1} \left(0, \begin{pmatrix}
            I_d & P \\
            P^T & I_d
        \end{pmatrix}\right), \mathbb{E}_{\mu_1}N^{\otimes n_1} \left(0, \begin{pmatrix}
            I_d & P \\
            P^T & I_d
        \end{pmatrix}\right)\Bigg\} \\
        & + \operatorname{TV}\Bigg\{\mathbb{E}_{\mu_0}N^{\otimes n_2} \left(0, \begin{pmatrix}
            I_d & -P \\
            -P^T & I_d
        \end{pmatrix}\right), \mathbb{E}_{\mu_1}N^{\otimes n_2} \left(0, \begin{pmatrix}
            I_d & -P \\
            -P^T & I_d
        \end{pmatrix}\right) \Bigg\}.
        \end{align*}
         Dealing with such $\mu_0, \mu_1$ is not straightforward, due to the presence of the integrals with respect to the Haar measure. Nonetheless, following similar ideas as in \cite{verzelen2021schatten}, we upper bound the total variance distance above using the following two lemmata, where we suppose $P$ to be symmetric. 
        \begin{lemma}\label{lemma:haar}
    Let $P$ be symmetric, with spectral decomposition $P = U^T \Lambda U$. Let $\mathcal{U}(d)$ denote the (normalised) Haar measure over the Lie group of orthogonal matrices $SO(d) = \{ U \in \mathbb{R}^{d,d} : U^T U = UU^T = I_d\}$, and let $\nu_0, \nu_1$ the distributions with matching moments up to the order $M$ defined above. Denote by $\mu_i$ the distribution of $U^T \Lambda U$, where $U \sim \mathcal{U}(d)$, and $\Lambda = \operatorname{diag}(\sigma_{1:d})$, with $\sigma_{1:d} \sim \nu_i^{\otimes \lceil d/2 \rceil } \otimes \delta_0 ^{\otimes (d - \lceil d/2 \rceil)}$. Then \begin{align*}
       \operatorname{TV}\Bigg\{\mathbb{E}_{\mu_0}& N^{\otimes n} \left(0, \begin{pmatrix}
            I_d & P \\
            P^T & I_d
        \end{pmatrix}\right), \mathbb{E}_{\mu_1}N^{\otimes n} \left(0, \begin{pmatrix}
            I_d & P \\
            P^T & I_d
        \end{pmatrix}\right)\Bigg\} \\
    & \leq \lceil d/2 \rceil \operatorname{TV}\Bigg\{\mathbb{E}_{\tilde{\pi}_0}\left\{N^{\otimes n}\left(\boldsymbol{0}_{2d}, \begin{pmatrix}
        I_{d} & \eta u u^T \\
        \eta u u^T & I_{d}
    \end{pmatrix}\right) \right\}, \mathbb{E}_{\tilde{\pi}_1}\left\{N^{\otimes n}\left(\boldsymbol{0}_{2d}, \begin{pmatrix}
        I_{d} & \eta' u' u'^T \\
        \eta' u' u'^T & I_{d}
    \end{pmatrix}\right) \right\} \Bigg\},
    \end{align*}
    where $\tilde{\pi}_0$ (resp. $\tilde{\pi}_1$) is the distribution of $\eta u u^T$ (resp. $\eta' u' u'^T$), where $\eta \sim \nu_0$ (resp. $\eta' \sim \nu_1$) and $u = (u_{d'}, \boldsymbol{0}^T_{d-d'})$ (resp. $u' = (u'_{d'}, \boldsymbol{0}^T_{d-d'})$) is such that $u_{d'}$ (resp $u'_{d'}$) is a uniform sample from the $d'$-dimensional sphere $\mathcal{S}^{d'-1} = \{x \in \mathbb{R}^{d'}:\|x\|_2 = 1 \}$, with $d' = d+1-\lceil d/2 \rceil$.
\end{lemma}

\begin{lemma}\label{lemma:determinant}
    With the same notation as above, then 
    \begin{align*}
\operatorname{TV}^2&\Bigg\{\mathbb{E}_{\tilde{\pi}_0}\left\{N^{\otimes n}\left(\boldsymbol{0}_{2d}, \begin{pmatrix}
        I_d & \eta u u^T \\
        \eta u u^T & I_d
    \end{pmatrix}\right) \right\},  \mathbb{E}_{\tilde{\pi}_1}\left\{N^{\otimes n}\left(\boldsymbol{0}_{2d}, \begin{pmatrix}
        I_d & \eta' u' u'^T \\
        \eta' u' u'^T & I_d
    \end{pmatrix}\right) \right\}\Bigg\} \\
    & \quad \quad \quad \leq \sum\limits_{k = M +1}^\infty \binom{k+n-1}{n-1}\mathbb{E}[u_1^{2k}] \left(\int \eta^{k} [\nu_0(d\eta)-\nu_1(d\eta)] \right)^2.
    \end{align*}
\end{lemma}
Applying these lemmata, it follows that
    \begin{align*}    
         \operatorname{TV}&\Bigg\{\mathbb{E}_{\mu_0}N^{\otimes n_1} \left(0, \begin{pmatrix}
            I_d & P \\
            P^T & I_d
        \end{pmatrix}\right), \mathbb{E}_{\mu_1}N^{\otimes n_1} \left(0, \begin{pmatrix}
            I_d & P \\
            P^T & I_d
        \end{pmatrix}\right)\Bigg\} \ \\
         & \leq \lceil d/2 \rceil \operatorname{TV}\Bigg\{\mathbb{E}_{\tilde{\pi}_0}\left\{N^{\otimes n_1}\left(\boldsymbol{0}_{2d}, \begin{pmatrix}
        I_d & \eta u u^T \\
        \eta u u^T & I_d
    \end{pmatrix}\right) \right\}, \mathbb{E}_{\tilde{\pi}_1}\left\{N^{\otimes n_1}\left(\boldsymbol{0}_{2d}, \begin{pmatrix}
        I_d & \eta' u' u'^T \\
        \eta' u' u'^T & I_d
    \end{pmatrix}\right) \right\}  \Bigg\} \\
    & \leq \lceil d/2 \rceil \sqrt{\sum_{k = M+1}^\infty \binom{k+n_1-1}{n_1-1} \mathbb{E}[u_1^{2k}] \left(\int \eta^{k} [\nu_0(d\eta)-\nu_1(d\eta)] \right)^2} ,
    \end{align*}
where the first inequality comes from Lemma \ref{lemma:haar}, and the second from Lemma \ref{lemma:determinant}. Here $u_1$ is the first coordinate of a uniform random vector in the $d'$-dimensional unit sphere, and $\nu_0, \nu_1$ are the distributions with matching moments up the order $M$ defined above. Now, observe that $u_1^2 \overset{d}{=} Z_1^2/\sum_{i=1}^{d'} Z_i^2$ where $Z_i \overset{\text{i.i.d.}}{\sim} N(0,1)$,  due to the fact that the standard normal distribution is isotropic. Hence $u_1^2 \sim \operatorname{Beta}(\frac{d'-1}{2}, \frac{1}{2})$ since if $X \sim \chi^2(\alpha)$ and $Y \sim \chi^2(\beta)$ are independent, then $ \frac{X}{X+Y} \sim \operatorname{Beta}\left(\frac{\alpha}{2}, \frac{\beta}{2}\right)$. It follows that
\begin{align*}
    \mathbb{E}[u_1^{2k}] & = \int_{-1}^{1} \frac{u^{2k}(1-u^2)^{\frac{d'-3}{2}}}{B(\frac{d'-1}{2}, \frac{1}{2})} du = \int_{0}^{1} \frac{v^{k-\frac{1}{2}}(1-v)^{\frac{d'-1}{2}-1}}{B(\frac{d'-1}{2}, \frac{1}{2})} dv \\
    & = \frac{B(\frac{d'-1}{2}, k +\frac{1}{2})}{B(\frac{d'-1}{2}, \frac{1}{2})} = \frac{\Gamma(k+\frac{1}{2})}{\Gamma(\frac{1}{2})} \frac{\Gamma(\frac{d'}{2})}{\Gamma(\frac{d'}{2} +k)}.
\end{align*}
Moreover \[
\left(\int \eta^{k} [\nu_0(d\eta)-\nu_1(d\eta)] \right)^2 \leq \left(b^k\left(1+\frac{1}{4^kM^{2k}} \right)\right)^2 \leq 4b^{2k}.
\]
If we choose $b^2 = \frac{d'}{4n}$, we have 
    \begin{align*}
       \sum_{k = M+1}^\infty \binom{k+n-1}{n-1} & \mathbb{E}[u_1^{2k}] \left(\int \eta^{k} [\nu_0(d\eta)-\nu_1(d\eta)] \right)^2 \\
       & \leq 4 \sum_{k = M+1}^\infty \binom{k+n-1}{n-1} \frac{\Gamma(k+\frac{1}{2})}{\Gamma(\frac{1}{2})} \frac{\Gamma(\frac{d'}{2})}{\Gamma(\frac{d'}{2} +k)} b^{2k} \\
       & = 4 \sum_{k = M+1}^\infty \frac{\Gamma(n+k)}{\Gamma(n)\Gamma(k+1)} \frac{\Gamma(k+\frac{1}{2})}{\Gamma(\frac{1}{2})} \frac{\Gamma(\frac{d'}{2})}{\Gamma(\frac{d'}{2} +k)} b^{2k} \\
       & = 4 \sum_{k = M+1}^\infty \frac{\Gamma(k+\frac{1}{2})}{\Gamma(k+1)\Gamma(\frac{1}{2})} \frac{\Gamma(\frac{d'}{2})}{\Gamma(\frac{d'}{2} +k)} \frac{\Gamma(n+k)}{\Gamma(n)} \frac{(d')^k}{4^kn^k} \\
       & = 4 \sum_{k = M+1}^\infty \frac{\Gamma(k+\frac{1}{2})}{\Gamma(k+1)\Gamma(\frac{1}{2})} \frac{(\frac{d'}{2})^k \Gamma(\frac{d'}{2})}{\Gamma(\frac{d'}{2} +k)} \frac{\Gamma(n+k)}{n^k \Gamma(n)} \frac{1}{2^k} \\
       & \leq  4 \sum_{k = M+1}^\infty 
        \frac{(\frac{d'}{2})^k \Gamma(\frac{d'}{2})}{\Gamma(\frac{d'}{2} +k)} \frac{\Gamma(n+k)}{n^k \Gamma(n)} 2^{-k} \leq 4 \cdot 2^{-(M+1)} = 2^{-(M-1)}.
    \end{align*}
     The second inequality follows from the fact that $\frac{\Gamma(k+\frac{1}{2})}{\Gamma(k+1)\Gamma(\frac{1}{2})} \leq 1$, while the third one follows from the fact that
    \[
        \frac{(\frac{d}{2})^k \Gamma(\frac{d}{2})}{\Gamma(\frac{d}{2} +k)} \frac{\Gamma(n+k)}{n^k \Gamma(n)} \leq 1.
    \] 
    Indeed, writing $\psi$ for the digamma function, the function $x \mapsto \log \frac{a^k \Gamma(a)}{\Gamma(a+k)}$ has derivative $\psi(a)-\psi(a+k)+k/a > 0$, and is therefore increasing. Thus, whenever $n \geq d/2 \geq d'/2$ the inequality follows. 
    Summing up, if we set $b/4M^2 = d/\{4(n_1 \wedge n_2)\}$,  we have that 
    \begin{align*}
             \operatorname{TV}\left(\mathbb{E}_{\mu_0}N_\mathbb{S}^{\otimes n_\mathbb{S}}, \mathbb{E}_{\mu_1}N_\mathbb{S}^{\otimes n_\mathbb{S}} \right) & \leq \operatorname{TV}\Bigg\{\mathbb{E}_{\mu_0}N^{\otimes n_1} \left(0, \begin{pmatrix}
            I_d & P \\
            P^T & I_d
        \end{pmatrix}\right), \mathbb{E}_{\mu_1}N^{\otimes n_1} \left(0, \begin{pmatrix}
            I_d & P \\
            P^T & I_d
        \end{pmatrix}\right)\Bigg\}  \\
        & + \operatorname{TV}\Bigg\{\mathbb{E}_{\mu_0}N^{\otimes n_2} \left(0, \begin{pmatrix}
            I_d & -P \\
            -P^T & I_d
        \end{pmatrix}\right), \mathbb{E}_{\mu_1}N^{\otimes n_2} \left(0, \begin{pmatrix}
            I_d & -P \\
            -P^T & I_d
        \end{pmatrix}\right) \Bigg\} \\
        & \leq 2 \lceil d/2 \rceil 2^{-\frac{M-1}{2}} \leq (d+1)2^{-\frac{M-1}{2}},
        \end{align*}
which is upper bounded by $1/2 - 1/9$ if and only if \[
M > 2\frac{\log (d+1) - \log(1/2 - 1/9)}{\log2} + 1.
\]
Hence, this shows that \[
\frac{b}{4M^2} \gtrsim \sqrt{\frac{d}{(n_1 \wedge n_2) \log^4 d}}
\] is sufficient to have $\operatorname{TV}\left(\mathbb{E}_{\mu_0}N_\mathbb{S}^{\otimes n_\mathbb{S}}, \mathbb{E}_{\mu_1}N_\mathbb{S}^{\otimes n_\mathbb{S}} \right) \leq 1/2-1/9$, which implies that $\mathcal{R}(n_\mathbb{S}, \rho') =  \mathcal{R}(n_\mathbb{S}, \frac{d}{3}\frac{b}{4M^2}) \geq 1/2$. This allows us to conclude that \[
\rho^*  \geq \frac{3}{4d}\rho' = \frac{b}{16M^2} \gtrsim \sqrt{\frac{d}{(n_1 \wedge n_2) \log^4 d}}.\]

We finally turn to the simpler case $d < 42$ and show that testing 
 \[
 H_0: R(\Sigma_\mathbb{S}) = 0 \quad \text{ vs. } \quad H_1(\rho): R(\Sigma_\mathbb{S}) > \rho, 
\]
requires at least a separation of the order $\sqrt{1/(n_1 \wedge n_2)}$. To this aim, we consider $\Sigma_\mathbb{S} = \{I_{2d} + \Delta, I_{2d}, I_{2d} \}$, with 
\[
\Delta_{ij} = \begin{cases}
    \delta, \quad \text{ if } (i,j) \in \{(1,2), (2,1)\} \\
    0, \quad \text{otherwise},
\end{cases}
\]
and show that $R(\Sigma_\mathbb{S}) \geq \delta/2$. This follows from the fact that $\tilde{X}_\mathbb{S} = \{\Theta_1 -\frac{1}{2}I_{2d}, \Theta_2 -\frac{1}{2}I_{2d}, -\frac{1}{2}I_{2d} \}$, where 
\[
\Theta_1 = \begin{pNiceArray}{cccccc}[margin,columns-width=auto]
  +\frac{3d}{4} & -\frac{3d}{4} & \Block{2-2}{\boldsymbol{O}_{2,2d-2}} & \\
  -\frac{3d}{4} & +\frac{3d}{4} & & \\
  \Block{2-2}{\boldsymbol{O}_{2d-2,2}}
    &  & \Block{2-2}{\boldsymbol{O}_{2d-2}} & \\
\\
\end{pNiceArray} \text{ and } \Theta_2 = \begin{pNiceArray}{cccccc}[margin,columns-width=auto]
  +\frac{3d}{4} & +\frac{3d}{4} & \Block{2-2}{\boldsymbol{O}_{2,2d-2}} & \\
  +\frac{3d}{4} & +\frac{3d}{4} & & \\
  \Block{2-2}{\boldsymbol{O}_{2d-2,2}}
    &  & \Block{2-2}{\boldsymbol{O}_{2d-2}} & \\
\\
\end{pNiceArray},
\]
is feasible as $\tilde{X}_\mathbb{S} + \tilde{X}_\mathbb{S}^0 = \{\Theta_1, \Theta_2, \boldsymbol{O}_{2d} \}$ and $A^*\tilde{X}_\mathbb{S}$ is diagonal with $\operatorname{tr}(A^*\tilde{X}_\mathbb{S}) = 0$.
We then bound the total error probability by choosing 
\[
P_{\mathbb{S},0}^{\otimes n_\mathbb{S}} = \left(N^{\otimes n_1}(\boldsymbol{0}_{2d}, I_{2d}), N^{\otimes n_2}(\boldsymbol{0}_{2d}, I_{2d}), N^{\otimes n_3}(\boldsymbol{0}_{2d}, I_{2d}) \right) \in \bar{\mathcal{P}}_{\mathbb{S}}(0),
\]
and
\[
P_{\mathbb{S},1}^{\otimes n_\mathbb{S}} = \left(N^{\otimes n_1}(\boldsymbol{0}_{2d}, I_{2d}+\Delta), N^{\otimes n_2}(\boldsymbol{0}_{2d}, I_{2d}), N^{\otimes n_3}(\boldsymbol{0}_{2d}, I_{2d}) \right) \in \mathcal{P}_{\mathbb{S}}(\delta),
\]
and using the fact that
\begin{align*}
    \mathcal{R}(n_\mathbb{S}, \delta) & =\inf_{\varphi_\mathbb{S}} \biggl\{ \sup_{P_{\mathbb{S},0} \in \bar{\mathcal{P}}_{\mathbb{S}}(0)} P_{\mathbb{S},0}^{\otimes n_{\mathbb{S}}}(\varphi-\mathbb{S} = 1)  + \sup_{P_{\mathbb{S},1} \in \mathcal{P}_{\mathbb{S}}(\delta)} P_{\mathbb{S},1}^{\otimes n_{\mathbb{S}}} (\varphi-\mathbb{S} = 0) \biggr\} \geq 1 - \operatorname{TV}\bigl( P_{\mathbb{S},0}^{\otimes n_{\mathbb{S}}} , P_{\mathbb{S},1}^{\otimes n_{\mathbb{S}}} \bigr) \\
    & \geq  1 - \operatorname{TV}\bigl(N^{\otimes n_1}(\boldsymbol{0}_{2d}, I_{2d}), N^{\otimes n_1}(\boldsymbol{0}_{2d}, I_{2d}+\Delta) \bigr).
\end{align*}
Now, if $P_0 := N^{\otimes n_1}\left(\boldsymbol{0}_{2}, I_{2}\right), P_1 := N^{\otimes n_1}\left(\boldsymbol{0}_{2},\begin{pmatrix}
    1 & \delta \\
    \delta & 1
\end{pmatrix}\right)$, we have
\begin{align*}
    4 \operatorname{TV}&\bigl(N^{\otimes n_1}(\boldsymbol{0}_{2d}, I_{2d}), N^{\otimes n_1}(\boldsymbol{0}_{2d}, I_{2d}+\Delta) \bigr)^2 \\
    & = 4 \operatorname{TV}(P_0, P_1)^2 \leq \operatorname{\chi^2}\left(P_0, P_1 \right) = \int \left(\frac{dP_1}{dP_0}\right)^2 dP_0 - 1 \\
    & = \frac{1}{\cos^{2n_1}\delta} \int \left(\prod_{i = 1}^{n_1} \exp\left\{ -\frac{1}{2} \langle \begin{pmatrix}
        x_i \\
        y_i
    \end{pmatrix}\begin{pmatrix}
        x_i \\
        y_i
    \end{pmatrix}^T, \frac{2}{\cos^2 \delta} \begin{pmatrix}
        1 & \sin\delta \\
        \sin \delta & 1
    \end{pmatrix} - 2I_2\rangle \right\}\right)^2 dP_0 - 1 \\
    & = \{(2-\cos^2\delta)^2 - 4\sin^2\delta \}^{-n_1/2}  - 1\\
    & \leq (1-\delta^2)^{-n_1/2} - 1\leq e^{n_1\delta^2/2} - 1,
\end{align*}
from which we see that $\operatorname{TV}\left(P_0, P_1 \right) \leq 1/2$ if $\delta \leq \sqrt{2\log 2/n_1}$. The same holds true if we switch the roles of $n_1$ and $n_2$. The above bound on the total variation distance demonstrates that we may choose $\delta=\sqrt{2\log 2/n_1 \wedge n_2}$, and hence that we have
\[
    \rho^* \geq \delta = \sqrt{\frac{2\log 2}{n_1 \wedge n_2}},
\]
as claimed.
\end{proof}
\begin{proof}[Proof of Lemma \ref{lemma:haar}]
    Let $\varphi(z)$ denote the density of the $d$-dimensional Gaussian law with respect to the Lebesgue measure. By the triangle inequality we have that
    \begin{align*}
       \operatorname{TV} & \Bigg\{\mathbb{E}_{\mu_0}\left\{ N^{\otimes n}\left(\boldsymbol{0}_{2d}, \begin{pmatrix}
        I_d & U^T \Lambda U \\
        U^T \Lambda U & I_d
    \end{pmatrix}\right) \right\}, \mathbb{E}_{\mu_1}\left\{N^{\otimes n}\left(\boldsymbol{0}_{2d}, \begin{pmatrix}
        I_d & U^T \Lambda U \\
        U^T \Lambda U & I_d
    \end{pmatrix}\right) \right\} \Bigg\} \\
    & \leq \sum_{j = 0}^{\lceil d/2 \rceil-1} \operatorname{TV}\Bigg\{\mathbb{E}_{\pi_j}\left\{N^{\otimes n}\left(\boldsymbol{0}_{2d}, \begin{pmatrix}
        I_d & U^T \Lambda U \\
        U^T \Lambda U & I_d
    \end{pmatrix}\right) \right\}, \mathbb{E}_{\pi_{j+1}}\left\{N^{\otimes n}\left(\boldsymbol{0}_{2d}, \begin{pmatrix}
        I_d & U^T \Lambda U \\
        U^T \Lambda U & I_d
    \end{pmatrix} \right) \right\}\Bigg\}, 
    \end{align*}
where $\pi_j$ is distribution of $U^T \Lambda U$, where $U \sim \mathcal{U}(d)$ is common for all $\pi_j$, while $\Lambda = \operatorname{diag}(\sigma_{1:d})$, with $\sigma_{1:d} \sim \nu_0^{\otimes (\lceil d/2 \rceil -j)} \otimes \nu_1^{\otimes j} \otimes \delta_0 ^{\otimes \lfloor d/2 \rfloor}$, for $j \in \{0, \ldots, \lceil d/2 \rceil - 1\}$. Observe that $\pi_0 = \mu_0$ and $\pi_{\lceil d/2 \rceil} = \mu_1$, so that this inequality essentially interpolates $\mu_0$ and $\mu_1$ with $\lceil d/2 \rceil$ intermediate measures such that, for every $j \in \{0, \ldots, \lceil d/2 \rceil - 1\}$, $\pi_j$ differs from $\pi_{j+1}$ only for the distribution of $\sigma_j$ in $\Lambda$. Now, consider a generic $j \in \{0, \ldots, \lceil d/2 \rceil - 1\}$ and define $S:= \{1, \ldots, \lceil d/2 \rceil -j-1, \lceil d/2 \rceil -j+1, \ldots, \lceil d/2 \rceil  \}$. We will show that we can bound each term of the summation above by \[\operatorname{TV}\Bigg\{\mathbb{E}_{\tilde{\pi}_0}\left\{N^{\otimes n}\left(\boldsymbol{0}_{2d}, \begin{pmatrix}
        I_d & \eta u u^T \\
        \eta u u^T & I_d
    \end{pmatrix}\right) \right\}, \mathbb{E}_{\tilde{\pi}_1}\left\{N^{\otimes n}\left(\boldsymbol{0}_{2d}, \begin{pmatrix}
        I_d & \eta' u' u'^T \\
        \eta' u' u'^T & I_d
    \end{pmatrix}\right) \right\} \Bigg\}, \]
    with $\tilde{\pi}_0, \tilde{\pi}_1$ defined in the statement, and this would conclude the proof. To this aim, observe that if \[
    \begin{pmatrix}
        X \\
        Y
    \end{pmatrix} \sim N\left(\boldsymbol{0}_{2d}, \begin{pmatrix}
        I_d & U^T \Lambda U \\
        U^T \Lambda U & I_d
    \end{pmatrix} \right),
    \]
    then \[
    \begin{cases}
        X|Y \sim N(U^T \Lambda U Y, I_d - U^T \Lambda^2 U)\\
        \\
        Y \sim N(\boldsymbol{0}_{d}, I_d).
    \end{cases}
    \]
    This allows us to write 
    \begin{align*}
TV_0 & := \operatorname{TV} \Bigg\{\mathbb{E}_{\pi_j}\left\{N^{\otimes n}\left(\boldsymbol{0}_{2d}, \begin{pmatrix}
        I_d & U^T \Lambda U \\
        U^T \Lambda U & I_d
    \end{pmatrix}\right) \right\}, \mathbb{E}_{\pi_{j+1}}\left\{N^{\otimes n}\left(\boldsymbol{0}_{2d}, \begin{pmatrix}
        I_d & U^T \Lambda U \\
        U^T \Lambda U & I_d
    \end{pmatrix} \right) \right\} \Bigg\} \\
    & = \int  \prod_{i =1}^n \varphi(y_i) \bigg|\mathbb{E}_{\pi_j}\left\{\prod_{i =1}^n |I - U^T \Lambda^2 U|^{-1/2}\varphi((I - U^T \Lambda^2 U)^{-1/2}(x_i - U^T \Lambda U y_i))\right\} \\
    & \quad \quad \quad \quad - \mathbb{E}_{\pi_{j+1}}\left\{\prod_{i =1}^n |I - U^T \Lambda^2 U|^{-1/2}\varphi((I - U^T \Lambda^2 U)^{-1/2}(x_i - U^T \Lambda U y_i))\right\} \bigg|  dx dy,
    \end{align*}
    where $dx = dx_1 \ldots dx_n$, and similarly for $dy$. Now, let $U_S$ be the restriction  of $U$ to its columns in $S$. By definition of $\pi_{j}$ (resp. $ \pi_{j+1}$), we can write $U^T \Lambda U$ as $U^T \Lambda U = \eta u u^T + U_S^T \operatorname{diag}(\sigma_{-j}) U_S$
where $\sigma_{-j} \sim \nu_0^{\otimes (\lceil d/2 \rceil -j-1)} \otimes \delta_0 \otimes \nu_1^{\otimes j} \otimes \delta_0^{\otimes d - \lceil d/2 \rceil}$, where $\delta_0$ is the Dirac measure in $0$. Write $\pi$ for the distribution of $\left(U_S, \operatorname{diag}\left(\sigma_{-j}\right)\right)$ and $f_0$ (resp. $f_1$) for the conditional distribution of $(u, \eta)$ given $\left(U_S, \operatorname{diag}\left(\sigma_{-j}\right)\right)$: this is given by $\eta \sim \nu_0$ (resp. $\nu_1$), while $u|U_S$ is sampled uniformly from $\mathcal{S}^{d-1} \cap U_S^\perp$, i.e. the intersection between the $d$-dimensional unit sphere $\mathcal{S}^{d-1}:= \{x \in \mathbb{R}^d:\|x\|_2 = 1 \}$ and the orthogonal complement of the columns spanned by $U_S$. First, observe that $\operatorname{dim}(\mathcal{S}^{d-1} \cap U_S^\perp) = d+1-\lceil d/2 \rceil$. Secondly, observe that for every measurable function $h$,  \[
\mathbb{E}h(P) = \int h(P) \pi_j(dP) = \int h(P) f_0(du, d\eta)\pi\left(dU_S, d\sigma_{-j}\right),
\]
and similarly for $\pi_{j+1}$. This allows to bound the TV distance above one step further as
\begin{multline*}
TV_0 \leq \int \prod_{i =1}^n \varphi(y_i) \bigg|\int \prod_{i =1}^n |I - U^T \Lambda^2 U|^{-1/2}\varphi((I - U^T \Lambda^2 U)^{-1/2}(x_i - U^T \Lambda U y_i) f_0(du, d\eta) \\
     - \int \prod_{i =1}^n |I - U'^T \Lambda'^2 U'|^{-1/2}\varphi((I - U'^T \Lambda'^2 U')^{-1/2}(x_i - U'^T \Lambda' U' y_i) f_1(du', d\eta')\bigg| \pi\left(dU_S, d\sigma_{-j}\right) dx dy \\
= \int \prod_{i =1}^n \varphi(y_i) \bigg(\int \bigg|\int \prod_{i =1}^n |I - U^T \Lambda^2 U|^{-1/2}\varphi((I - U^T \Lambda^2 U)^{-1/2}(x_i - U^T \Lambda U y_i) f_0(du, d\eta) \\
- \int \prod_{i =1}^n |I - U'^T \Lambda'^2 U'|^{-1/2}\varphi((I - U'^T \Lambda'^2 U')^{-1/2}(x_i - U'^T \Lambda' U' y_i) f_1(du', d\eta')\bigg| dx \bigg) \pi\left(dU_S, d\sigma_{-j}\right) dy,  
\end{multline*}
where in the first step we used Jensen's inequality, bringing the common $\pi$ outside the absolute value, while in the last step we used Fubini-Tonelli theorem with positive integrand. Consider now the innermost integral \begin{align*}
    \int & \bigg|\int \prod_{i =1}^n |I - U^T \Lambda^2 U|^{-1/2}\varphi((I - U^T \Lambda^2 U)^{-1/2}(x_i - U^T \Lambda U y_i) f_0(du, d\eta) \\
    & \quad - \int \prod_{i =1}^n |I - U'^T \Lambda'^2 U'|^{-1/2}\varphi((I - U'^T \Lambda'^2 U')^{-1/2}(x_i - U'^T \Lambda' U' y_i) f_1(du', d\eta')\bigg| dx \\
    & = \int \bigg|\int \prod_{i =1}^n |I - U^T \Lambda^2 U|^{-1/2}\varphi((I - U^T \Lambda^2 U)^{-1/2}(x_i - \eta u u^T y_i - \sum_{k \neq j}\sigma_k u_k u_k^T y_i) f_0(du, d\eta) \\
    & \quad - \int \prod_{i =1}^n |I - U'^T \Lambda'^2 U'|^{-1/2}\varphi((I - U'^T \Lambda'^2 U')^{-1/2}(x_i - \eta' u' u'^T y_i - \sum_{k \neq j}\sigma_k u_k u_k^T y_i) f_1(du', d\eta')\bigg| dx
\end{align*}
for fixed $U_S, \sigma_{-j}, y$. This can be simplified to \begin{align*}
    \int & \bigg|\int \prod_{i =1}^n |I - U^T \Lambda^2 U|^{-1/2}\varphi((I - U^T \Lambda^2 U)^{-1/2}(x_i - \eta u u^T y_i)) f_0(du, d\eta) \\
    & \quad - \int \prod_{i =1}^n |I - U'^T \Lambda'^2 U'|^{-1/2}\varphi((I - U'^T \Lambda'^2 U')^{-1/2}(x_i - \eta' u' u'^T y_i)) f_1(du', d\eta')\bigg| dx
\end{align*} after the change of variable $x_i' = x_i - \sum_{k \neq j}\sigma_k u_k u_k^T y_i$, for all $i \in [n]$. Now, observe that, under $f_0$, we have \[
(I - U^T \Lambda^2 U)^{-1/2} = \frac{1}{\sqrt{1-\eta^2}}uu^T + \sum_{i \neq j} \frac{1}{\sqrt{1-\sigma_i^2}} u_i u_i^T,
\]
which yields \[|I - U^T \Lambda^2 U|^{-1/2} = \frac{1}{{\sqrt{1-\eta^2}} \prod\limits_{i \neq j}\sqrt{1-\sigma_i^2}};\] 
similarly under for $(I - U'^T \Lambda'^2 U')^{-1/2}$ under $\pi_1$, with $\eta', u'$ in place of $\eta, u$. Perform the change of variables \[
x_i = \left(I - \sum_{k \neq j}\sigma_k^2 u_k u_k^T\right)^{1/2}z_i, 
\] 
for all $i \in [n]$, whose Jacobian is \[
\prod_{i =1}^n \left|I - \sum_{k \neq j}\sigma_k^2 u_k u_k^T\right|^{1/2} =  \prod_{i =1}^n \prod\limits_{k \neq j}\sqrt{1-\sigma_k^2}.
\]
We get 
\begin{align*}
    \int & \bigg|\int \prod_{i =1}^n \frac{1}{\sqrt{1-\eta^2}} \varphi\left(\left(\sum_{i \neq j} u_i u_i^T + \frac{1}{\sqrt{1-\eta^2}}uu^T \right) z_i - \frac{\eta}{\sqrt{1-\eta^2}} u u^T y_i\right) f_0(du, d\eta) \\
    & \quad - \int \prod_{i =1}^n \frac{1}{\sqrt{1-\eta'^2}} \varphi\left(\left(\sum_{i \neq j} u_i u_i^T + \frac{1}{\sqrt{1-\eta'^2}}u'u'^T \right) z_i - \frac{\eta'}{\sqrt{1-\eta'^2}} u' u'^T y_i\right) f_1(du', d\eta')\bigg| dx \\
    & = \int \bigg|\int \prod_{i =1}^n \frac{1}{\sqrt{1-\eta^2}} \varphi\left(\left(\sum_{i \neq j} u_i u_i^T + \frac{1}{\sqrt{1-\eta^2}}uu^T \right)(z_i - \eta u u^T y_i)\right) f_0(du, d\eta) \\
    & \quad - \int \prod_{i =1}^n \frac{1}{\sqrt{1-\eta'^2}} \varphi\left(\left(\sum_{i \neq j} u_i u_i^T + \frac{1}{\sqrt{1-\eta'^2}}u'u'^T \right)(z_i - \eta' u' u'^T y_i)\right) f_1(du', d\eta')\bigg| dx \\
    & = \int \bigg|\int \prod_{i =1}^n \frac{1}{\sqrt{1-\eta^2}} \varphi\left(\left(I - \eta^2 uu^T \right)^{-1/2}(z_i - \eta u u^T y_i)\right) f_0(du, d\eta) \\
    & \quad - \int \prod_{i =1}^n \frac{1}{\sqrt{1-\eta'^2}} \varphi\left(\left(I - \eta'^2 u'u'^T \right)^{-1/2} (z_i - \eta' u' u'^T y_i)\right) f_1(du', d\eta')\bigg| dx \\
    & = \int \bigg|\int \prod_{i =1}^n \left|I - \eta^2 uu^T \right|^{-1/2}\varphi\left(\left(I - \eta^2 uu^T \right)^{-1/2}(z_i - \eta u u^T y_i)\right) f_0(du, d\eta) \\
    & \quad - \int \prod_{i =1}^n \left|I - \eta'^2 u'u'^T \right|^{-1/2} \varphi\left(\left(I - \eta'^2 u'u'^T \right)^{-1/2} (z_i - \eta' u' u'^T y_i)\right) f_1(du', d\eta')\bigg| dx \\
    & = \operatorname{TV}\Bigg\{ \mathbb{E}_{f_0}\{N^{\otimes n}(\eta u u^T y, I - (\eta uu^T)(\eta uu^T)^T)\}, \mathbb{E}_{f_1}\{N^{\otimes n}(\eta' u' u'^T y, I - (\eta' u'u'^T)(\eta' u'u'^T)^T)\}\Bigg\}. 
\end{align*} 
Coming back to the initial TV distance we wish to bound, we get that 
\begin{align*}
& \operatorname{TV} \Bigg\{\mathbb{E}_{\pi_j}\left\{N^{\otimes n}\left(\boldsymbol{0}_{2d}, \begin{pmatrix}
        I_d & U^T \Lambda U \\
        U^T \Lambda U & I_d
    \end{pmatrix}\right) \right\}, \mathbb{E}_{\pi_{j+1}}\left\{N^{\otimes n}\left(\boldsymbol{0}_{2d}, \begin{pmatrix}
        I_d & U^T \Lambda U \\
        U^T \Lambda U & I_d
    \end{pmatrix} \right) \right\} \Bigg\} \\
        & \leq \int \operatorname{TV}\Bigg\{ \mathbb{E}_{f_0}\{N^{\otimes n}(\eta u u^T y, I - (\eta uu^T)(\eta uu^T)^T)\}, \mathbb{E}_{f_1}\{N^{\otimes n}(\eta' u' u'^T y, I - (\eta' u'u'^T)(\eta' u'u'^T)^T)\}\Bigg\} \\
        & \quad \quad \quad \quad \quad \quad \quad \quad \quad \quad \quad \quad \quad \quad \quad \quad \quad \quad \quad \quad \quad \quad \quad \quad \quad \quad \quad \quad \quad \quad \varphi(y) \pi\left(dU_S, d\sigma_{-j}\right) dy \\
         & = \int \operatorname{TV}\Bigg\{\mathbb{E}_{f_0}\left\{N^{\otimes n}\left(\boldsymbol{0}_{2d}, \begin{pmatrix}
        I_{d} & \eta u u^T \\
        \eta u u^T & I_{d}
    \end{pmatrix}\right) \right\}, \mathbb{E}_{f_1}\left\{N^{\otimes n}\left(\boldsymbol{0}_{2d}, \begin{pmatrix}
        I_{d} & \eta' u' u'^T \\
        \eta' u' u'^T & I_{d}
    \end{pmatrix}\right) \right\} \Bigg\} \pi\left(dU_S, d\sigma_{-j}\right) \\
        & = \operatorname{TV}\Bigg\{\mathbb{E}_{\tilde{\pi}_0}\left\{N^{\otimes n}\left(\boldsymbol{0}_{2d}, \begin{pmatrix}
        I_{d} & \eta u u^T \\
        \eta u u^T & I_{d}
    \end{pmatrix}\right) \right\}, \mathbb{E}_{\tilde{\pi}_1}\left\{N^{\otimes n}\left(\boldsymbol{0}_{2d}, \begin{pmatrix}
        I_{d} & \eta' u' u'^T \\
        \eta' u' u'^T & I_{d}
    \end{pmatrix}\right) \right\} \Bigg\}, 
    \end{align*}
    where $\tilde{\pi}_0$ (resp. $\tilde{\pi}_1$) is the distribution of $\eta u u^T$ (resp. $\eta' u' u'^T$), where $\eta \sim \nu_0$ (resp. $\eta' \sim \nu_1$) and $u$ (resp. $u'$) is sampled uniformly from a $d'$-dimensional unit sphere embedded in $\mathbb{R}^d$, with $d' = d + 1 - \lceil d/2 \rceil$. Now, since the Gaussian distribution is invariant under orthogonal transformation, we might assume that $u = (u_{d'}, \boldsymbol{0}^T_{d-d'})$, with $u_{d'}$ uniformly sampled from the $d'$-dimensional sphere $\mathcal{S}^{d'-1}$, and the result follows.
\end{proof}

\begin{proof}[Proof of Lemma \ref{lemma:determinant}]
Consider \[
\int \prod_{i =1}^n \varphi(y_i) \bigg|\int \prod_{i =1}^n \frac{1}{\sqrt{1-\eta^2}} \varphi\left(\left(I - \eta^2 uu^T \right)^{-1/2}(z_i - \eta u u^T y_i)\right) [\tilde{\pi}_0(du, d\eta) - \tilde{\pi}_1(du, d\eta)]\bigg| dz dy,
\]
and observe that \[
\left(I - \eta^2 uu^T \right)^{-1} = I + \frac{\eta^2}{1-\eta^2}uu^T.
\]
Hence, 
\begin{align*}
    & \varphi\left(\left(I - \eta^2 uu^T \right)^{-1/2}(z_i - \eta u u^T y_i)\right) = (2\pi)^{-\frac{d}{2}} \exp\left\{-\frac{1}{2}(z_i - \eta u u^T y_i)^T \left(I + \frac{\eta^2}{1-\eta^2}uu^T \right) (z_i - \eta u u^T y_i) \right\} \\
    & = (2\pi)^{-\frac{d}{2}} \exp\left\{-\frac{1}{2}z_i^T z_i - \frac{1}{2}\frac{\eta^2}{1-\eta^2} z_i^T uu^T z_i - \frac{1}{2} \eta^2 y_i^T uu^T y_i -\frac{1}{2} \frac{\eta^4}{1-\eta^2} y_i^T uu^T y_i  \right. \\
    & \left. \hspace{11.1cm} + \eta z_i^T uu^T y_i + \frac{\eta^3}{1-\eta^2}z_i^T u u^T y_i \right\} \\
    & = (2\pi)^{-\frac{d}{2}} \exp\left\{-\frac{1}{2}z_i^T z_i - \frac{1}{2}\frac{\eta^2}{1-\eta^2} z_i^T uu^T z_i - \frac{1}{2} \frac{\eta^2}{1-\eta^2} y_i^T uu^T y_i + \frac{\eta}{1-\eta^2}z_i^T u u^T y_i\right\} \\
    & = \varphi(z_i)\exp\left\{ - \frac{1}{2}\frac{\eta^2}{1-\eta^2} z_i^T uu^T z_i - \frac{1}{2} \frac{\eta^2}{1-\eta^2} y_i^T uu^T y_i + \frac{\eta}{1-\eta^2}z_i^T u u^T y_i\right\} \\
    & = \varphi(z_i)\exp \left\{ - \frac{1}{2} \langle \begin{pmatrix}
       \frac{\eta^2}{1-\eta^2} uu^T & -\frac{\eta}{1-\eta^2} uu^T \\
       -\frac{\eta}{1-\eta^2} uu^T & \frac{\eta^2}{1-\eta^2}uu^T
    \end{pmatrix}, \begin{pmatrix}
        z_i \\ y_i
    \end{pmatrix}\begin{pmatrix}
        z_i \\ y_i
    \end{pmatrix}^T \rangle \right\} =: \varphi(z_i) g(\eta, u, z_i, y_i).
\end{align*}
Hence, 
\begin{align*}
    TV_1 &:= \int \prod_{i =1}^n \varphi(y_i) \bigg|\int \prod_{i =1}^n \frac{1}{\sqrt{1-\eta^2}} \varphi\left(\left(I - \eta^2 uu^T \right)^{-1/2}(z_i - \eta u u^T y_i)\right) [\tilde{\pi}_0(du, d\eta) - \tilde{\pi}_1(du, d\eta)]\bigg| dx dy \\
    & = \int \prod_{i =1}^n \varphi(z_i) \varphi(y_i) \bigg|\int \prod_{i =1}^n \frac{1}{\sqrt{1-\eta^2}} g(\eta, u, z_i, y_i) [\tilde{\pi}_0(du, d\eta) - \tilde{\pi}_1(du, d\eta)]\bigg| dz dy \\
    & \leq \sqrt{\int \prod_{i =1}^n \varphi(z_i) \varphi(y_i) \bigg[\int \prod_{i =1}^n \frac{1}{\sqrt{1-\eta^2}} g(\eta, u, z_i, y_i) [\tilde{\pi}_0(du, d\eta) - \tilde{\pi}_1(du, d\eta)]\bigg]^2 dz dy},
\end{align*}
where we used Cauchy-Schwartz inequality in the last step. Thus, it follows that 
\begin{align*}
    TV_1^2 & \leq \int \prod_{i =1}^n \varphi(z_i) \varphi(y_i) \sum_{k = 0,1} \sum_{j = 0,1}(-1)^{k+j} \\
    & \quad \quad \quad \left(\int \int \prod_{i =1}^n \frac{1}{\sqrt{1-\eta^2}\sqrt{1-\eta'^2}} g(\eta, u, z_i, y_i) g(\eta', u', z_i, y_i)\tilde{\pi}_k(du, d\eta)\tilde{\pi}_j(du', d\eta')\right) dz dy \\
    & = \sum_{k = 0,1} \sum_{j = 0,1}(-1)^{k+j}\int \int (1-\eta^2)^{-n/2} (1-\eta'^2)^{-n/2} \\
    & \quad \quad \quad \left(\int \prod_{i =1}^n g(\eta, u, z_i, y_i) g(\eta', u', z_i, y_i) \varphi(z_i) \varphi(y_i) dz dy \right)\tilde{\pi}_k(du, d\eta)\tilde{\pi}_j(du', d\eta') \\
    & = \sum_{k = 0,1} \sum_{j = 0,1}(-1)^{k+j}\int \int (1-\eta^2)^{-n/2} (1-\eta'^2)^{-n/2} \\
    & \quad \quad \quad \left(\prod_{i =1}^n \int  g(\eta, u, z_i, y_i) g(\eta', u', z_i, y_i) \varphi(z_i) \varphi(y_i) dz_i dy_i \right)\tilde{\pi}_k(du, d\eta)\tilde{\pi}_j(du', d\eta'), 
\end{align*}
where in the second equality we used Fubini-Tonelli's theorem to change the order of integration, and Fubini's theorem to factorise independent integrands in the last one. Let us consider a generic \[
\int g(\eta, u, z_i, y_i) g(\eta', u', z_i, y_i) \varphi(z_i) \varphi(y_i)dz_i dy_i,
\]
bearing in mind that $u = (u_{d'}, \boldsymbol{0}^T_{d-d'}), u' = (u'_{d'},\boldsymbol{0}^T_{d-d'})$, with $u_{d'}, u'_{d'}$ being independent and uniform samples from the $d'$-dimensional unit sphere, where $d' = d + 1 - \lceil d/2 \rceil$.  We have \begin{align*}
     \int & g(\eta, u, z_i, y_i) g(\eta', u', z_i, y_i) \varphi(z_i) \varphi(y_i)dz_i dy_i \\
    & = \int (2\pi)^{-d} \exp \left\{ - \frac{1}{2} \langle \begin{pmatrix}
       \frac{\eta^2}{1-\eta^2} uu^T & -\frac{\eta}{1-\eta^2} uu^T \\
       -\frac{\eta}{1-\eta^2} uu^T & \frac{\eta^2}{1-\eta^2}uu^T
    \end{pmatrix} + \right. \\
    & \left. \hspace{4cm} + \begin{pmatrix}
       \frac{\eta'^2}{1-\eta'^2} u'u'^T & -\frac{\eta'}{1-\eta'^2} u'u'^T \\
       -\frac{\eta'}{1-\eta'^2} u'u'^T & \frac{\eta'^2}{1-\eta'^2}u'u'^T
    \end{pmatrix} + \begin{pmatrix}
       I_d & \boldsymbol{O}_d \\
       \boldsymbol{O}_d & I_d
    \end{pmatrix}, \begin{pmatrix}
        z_i \\ y_i
    \end{pmatrix}\begin{pmatrix}
        z_i \\ y_i
    \end{pmatrix}^T \rangle \right\} \\
    & = \int (2\pi)^{-d} \exp \left\{ - \frac{1}{2} \langle \begin{pmatrix}
       I_d + \frac{\eta^2}{1-\eta^2} uu^T + \frac{\eta'^2}{1-\eta'^2} u'u'^T & -\frac{\eta}{1-\eta^2} uu^T -\frac{\eta'}{1-\eta'^2} u'u'^T \\
       -\frac{\eta}{1-\eta^2} uu^T  -\frac{\eta'}{1-\eta'^2} u'u'^T & I_d + \frac{\eta^2}{1-\eta^2}uu^T \frac{\eta'^2}{1-\eta'^2}u'u'^T
    \end{pmatrix}, \begin{pmatrix}
        z_i \\ y_i
    \end{pmatrix}\begin{pmatrix}
        z_i \\ y_i
    \end{pmatrix}^T \rangle \right\} \\
    & =: \int (2\pi)^{-d} \exp \left\{ - \frac{1}{2} \langle K, \begin{pmatrix}
        z_i \\ y_i
    \end{pmatrix}\begin{pmatrix}
        z_i \\ y_i
    \end{pmatrix}^T \rangle \right\} = |K|^{-1/2}.
\end{align*}
Now $K$ takes the form \[
K = \begin{pmatrix}
       I_d + \frac{\eta^2}{1-\eta^2} uu^T + \frac{\eta'^2}{1-\eta'^2} u'u'^T & -\frac{\eta}{1-\eta^2} uu^T -\frac{\eta'}{1-\eta'^2} u'u'^T \\
       -\frac{\eta}{1-\eta^2} uu^T  -\frac{\eta'}{1-\eta'^2} u'u'^T & I_d + \frac{\eta^2}{1-\eta^2}uu^T \frac{\eta'^2}{1-\eta'^2}u'u'^T
    \end{pmatrix}.
\]
It is straightforward to show that \[
|K| = \frac{(1 - (u^Tu')^2 \eta \eta')^2}{(1-\eta^2)(1-\eta'^2)} \overset{d}{=} \frac{(1 - u_1^2 \eta \eta')^2}{(1-\eta^2)(1-\eta'^2)},
\]
but, since it requires some lengthy algebraic computations, we defer its proof to Lemma \ref{lemma:computing_det} below. Now, it follows that 
\begin{align*}
    TV_1^2 & \leq \sum_{k = 0,1} \sum_{j = 0,1}(-1)^{k+j}\int \int (1-\eta^2)^{-n/2} (1-\eta'^2)^{-n/2} \\
    & \quad \quad \quad \quad \left(\prod_{i =1}^n \int  g(\eta, u, z_i, y_i) g(\eta', u', z_i, y_i) \varphi(z_i) \varphi(y_i) dz_i dy_i \right)\tilde{\pi}_k(du, d\eta)\tilde{\pi}_j(du', d\eta') \\
    & = \sum_{k = 0,1} \sum_{j = 0,1}(-1)^{k+j}\int \int (1-\eta^2)^{-n/2} (1-\eta'^2)^{-n/2} |K|^{-n/2}\tilde{\pi}_k(du, d\eta)\tilde{\pi}_j(du', d\eta') \\
    & = \sum_{k = 0,1} \sum_{j = 0,1}(-1)^{k+j}\int \int \frac{1}{(1-u_1^2 \eta\eta')^n}\tilde{\pi}_k(du, d\eta)\tilde{\pi}_j(du', d\eta') \\
    & = \sum_{h = 0}^\infty \sum_{k = 0,1} \sum_{j = 0,1}(-1)^{k+j}\int \int \binom{h +n-1}{n-1} u_1^{2h}\eta^h\eta'^h\tilde{\pi}_0(du, d\eta)\tilde{\pi}_1(du', d\eta') \\
    & = \sum_{k = 0}^\infty \binom{k+n-1}{n-1} 
 \mathbb{E}[u_1^{2k}] \left(\int \eta^{k} [\nu_0(d\eta)-\nu_1(d\eta)] \right)^2 \\
 & = \sum_{k = M+1}^\infty \binom{k+n-1}{n-1} \mathbb{E}[u_1^{2k}] \left(\int \eta^{k} [\nu_0(d\eta)-\nu_1(d\eta)] \right)^2,
\end{align*}
since $\nu_0, \nu_1$ share the first $M$ moments.
\end{proof}
\begin{lemma}\label{lemma:computing_det}
    Let \[K = \begin{pmatrix}
       I_d + \frac{\eta^2}{1-\eta^2} uu^T + \frac{\eta'^2}{1-\eta'^2} u'u'^T & -\frac{\eta}{1-\eta^2} uu^T -\frac{\eta'}{1-\eta'^2} u'u'^T \\
       -\frac{\eta}{1-\eta^2} uu^T  -\frac{\eta'}{1-\eta'^2} u'u'^T & I_d + \frac{\eta^2}{1-\eta^2}uu^T \frac{\eta'^2}{1-\eta'^2}u'u'^T
    \end{pmatrix},
\]
where $u,u'$ are $d$-dimensional unit vectors. Then, \[
|K| = \frac{(1 - (u^Tu')^2 \eta \eta')^2}{(1-\eta^2)(1-\eta'^2)} 
\]
\end{lemma}
\begin{proof}
Let $\alpha = \eta^2/(1-\eta^2), \alpha' = \eta'^2/(1-\eta'^2), \beta = -\eta/(1-\eta^2), \beta' = -\eta'/(1-\eta'^2)$. We aim at finding $|K|$, where \[
K = \begin{pmatrix}
       I_d + \alpha uu^T + \alpha' u'u'^T & \beta uu^T + \beta' u'u'^T \\
       \beta uu^T  +\beta' u'u'^T & I_d + \alpha uu^T \alpha u'u'^T
    \end{pmatrix}.
\]
    First, observe that by Schur's complement
    \begin{align*}
    |K| &= \bigg| I_d + \alpha uu^T + \alpha' u'u'^T \bigg| \\
    & \quad \quad \quad \quad \quad \times \bigg| I_d + \alpha uu^T + \alpha' u'u'^T - \left(\beta uu^T + \beta' u'u'^T\right)\left( I_d + \alpha uu^T + \alpha' u'u'^T \right)^{-1}\left(\beta uu^T + \beta' u'u'^T\right)\bigg|,
\end{align*}
and that we may assume without loss of generality that $u' = \boldsymbol{e}_1$. Indeed, let $R$ be any orthogonal matrix in $\mathbb{R}^{d,d}$, and consider $Ru, Ru'$ in place of $u,u'$ respectively. Then \begin{align*}
    \bigg| I_d + \alpha R uu^T R^T + \alpha' R u'u'^T R^T \bigg| &= \bigg| R(I_d + \alpha uu^T + \alpha' u'u'^T)R^T \bigg| = \bigg|R\bigg| \bigg|I_d + \alpha uu^T + \alpha' u'u'^T\bigg| \bigg|R^T \bigg| \\
    & \bigg| I_d + \alpha uu^T + \alpha' u'u'^T \bigg|,
\end{align*}
and it is easy to check that the same happens for 
\begin{align*}
    \bigg| I_d & + \alpha Ruu^T R^T + \alpha' Ru'u'^TR^T \\
    & - \left(\beta Ruu^TR^T + \beta' Ru'u'^TR^T\right)\left( I_d + \alpha Ruu^TR^T + \alpha'Ru'u'^T R^T\right)^{-1}\left(\beta Ruu^TR^T + \beta' Ru'u'^TR^T\right)\bigg|.
\end{align*}
This is not necessary for the proof, but it helps with the notation, and also explains why $u^Tu' \overset{d}{=} u_1$ when $u$ and $u'$ are sampled as described when we apply the result. Now, for all $\alpha, \alpha' \in \mathbb{R}$, \begin{align*}
    \left(I + \alpha u u^T + \alpha' \boldsymbol{e}_1 \boldsymbol{e}_1^T \right)^{-1} &= I - \frac{\alpha}{1+\alpha-\frac{\alpha\alpha'}{1+\alpha'}\gamma^2} uu^T - \frac{\alpha'}{1+\alpha'-\frac{\alpha\alpha'}{1+\alpha'}\gamma^2}\boldsymbol{e}_1 \boldsymbol{e}_1^T + \\
    & + \gamma \frac{\alpha'}{1+\alpha'}\frac{\alpha}{1+\alpha-\frac{\alpha\alpha'}{1+\alpha'}\gamma^2}\boldsymbol{e}_1 u^T + \gamma \frac{\alpha}{1+\alpha}\frac{\alpha'}{1+\alpha'-\frac{\alpha\alpha'}{1+\alpha'}\gamma^2} u \boldsymbol{e}_1^T,
\end{align*}
where $\gamma = \boldsymbol{e}_1^T u$, and
\begin{align*}
    I + \alpha u u^T + \alpha' \boldsymbol{e}_1 \boldsymbol{e}_1^T & - \left(\beta uu^T + \beta' \boldsymbol{e}_1 \boldsymbol{e}_1^T\right)\left(I + \alpha u u^T + \alpha' \boldsymbol{e}_1 \boldsymbol{e}_1^T \right)^{-1}\left(\beta uu^T + \beta' \boldsymbol{e}_1 \boldsymbol{e}_1^T\right) \\
    & = I +\frac{\gamma^2 \eta^2 \eta'^2}{1-\gamma^2 \eta^2 \eta'^2}(uu^T + \boldsymbol{e}_1 \boldsymbol{e}_1^T) -\frac{\gamma \eta \eta}{1-\gamma^2 \eta^2 \eta'^2}(\boldsymbol{e}_1 u^T + u \boldsymbol{e}_1^T).
\end{align*}
Calling $x = (\gamma^2 \eta^2 \eta'^2)/(1-\gamma^2 \eta^2 \eta'^2)$ and $c = \gamma \eta \eta'$, we thus have 
\begin{align*}
    |K| &= \bigg| I_d + \alpha uu^T + \alpha' \boldsymbol{e}_1 \boldsymbol{e}_1^T \bigg| \bigg| I_d + x uu^T + x \boldsymbol{e}_1 \boldsymbol{e}_1^T - \frac{x}{c} \begin{pmatrix}
        \vertbar & \vertbar \\
        u & \boldsymbol{e}_1 \\
        \vertbar & \vertbar
    \end{pmatrix} \begin{pmatrix}
        \horzbar & \boldsymbol{e}_1 & \horzbar \\
        \horzbar & u & \horzbar
    \end{pmatrix}\bigg|.
\end{align*}
In order to compute these determinants, we will repeatedly make use of the fact that, if $A$ is an invertible $n\times n$ matrix, $U,V$ are $n \times m$ matrices, then \[
|A + uu^T| = |I_m + V^T A^{-1} U| |A|.
\]
If $A = I_n$, this is commonly referred as the Weinstein–Aronszajn identity. Now, 
\[
\bigg| I_d + \alpha uu^T + \alpha' \boldsymbol{e}_1 \boldsymbol{e}_1^T \bigg| = \bigg| I_d + \alpha uu^T \bigg| \bigg|1 +  \alpha' \boldsymbol{e}^T_1(I - \alpha uu^T/(1+\alpha))\boldsymbol{e}_1 \bigg| = (1+\alpha)\left(1+\alpha'-\frac{\alpha \alpha'}{1+\alpha}\gamma^2\right),
\]
and \begin{align*}
    \bigg| & I_d + x uu^T + x \boldsymbol{e}_1 \boldsymbol{e}_1^T - \frac{x}{c} \begin{pmatrix}
        \vertbar & \vertbar \\
        u & \boldsymbol{e}_1 \\
        \vertbar & \vertbar
    \end{pmatrix} \begin{pmatrix}
        \horzbar & \boldsymbol{e}_1 & \horzbar \\
        \horzbar & u & \horzbar
    \end{pmatrix}\bigg| \\
    & = \bigg| I_d + x uu^T + x \boldsymbol{e}_1 \boldsymbol{e}_1^T \bigg|\bigg| I -\frac{x}{c}\begin{pmatrix}
        \horzbar & \boldsymbol{e}_1 & \horzbar \\
        \horzbar & u & \horzbar
    \end{pmatrix}(I_d + x uu^T + x \boldsymbol{e}_1 \boldsymbol{e}_1^T)^{-1} \begin{pmatrix}
        \vertbar & \vertbar \\
        u & \boldsymbol{e}_1 \\
        \vertbar & \vertbar
    \end{pmatrix}\bigg| \\
    & = (1+x)(1+x-\frac{x^2}{1+x}\gamma^2)\bigg|I -\frac{x}{c}\begin{pmatrix}
        \horzbar & \boldsymbol{e}_1 & \horzbar \\
        \horzbar & u & \horzbar
    \end{pmatrix} \\
    & \quad \times \left(I - \frac{x}{1+x-\frac{x^2}{1+x}\gamma^2} (uu^T + \boldsymbol{e}_1 \boldsymbol{e}_1^T) + \gamma \frac{x}{1+x} \frac{x}{1+x-\frac{x^2}{1+x}\gamma^2}\begin{pmatrix}
        \vertbar & \vertbar \\
        u & \boldsymbol{e}_1 \\
        \vertbar & \vertbar
    \end{pmatrix} \begin{pmatrix}
        \horzbar & \boldsymbol{e}_1 & \horzbar \\
        \horzbar & u & \horzbar
    \end{pmatrix}\right)\begin{pmatrix}
        \vertbar & \vertbar \\
        u & \boldsymbol{e}_1 \\
        \vertbar & \vertbar
    \end{pmatrix} \bigg| \\
    & = (1+x)(1+x-\frac{x^2}{1+x}\gamma^2) \left|\begin{pmatrix}
        1 - \frac{x}{c}(\gamma + 2\tau_1 \gamma + \tau_2(1+\gamma^2)) & -\frac{x}{c}(1+\tau_1(1+\gamma^2)+2\tau_2\gamma) \\
        -\frac{x}{c}(1+\tau_1(1+\gamma^2)+2\tau_2\gamma) &  1 - \frac{x}{c}(\gamma + 2\tau_1 \gamma + \tau_2(1+\gamma^2)) 
    \end{pmatrix} \right|,
\end{align*}
where $\tau_1 = - x/(1+x-\frac{x^2}{1+x}\gamma^2), \tau_2  = \gamma x \tau_1/(1+x)$. Putting all the pieces together, 
\begin{align*}
    |K| &= (1+\alpha)\left(1+\alpha'-\frac{\alpha \alpha'}{1+\alpha}\gamma^2\right) (1+x)\left(1+x-\frac{x^2}{1+x}\gamma^2\right) \\
    & \quad \quad \times \left((1 - \frac{x}{c}(\gamma + 2\tau_1 \gamma + \tau_2(1+\gamma^2)))^2 -\frac{x^2}{c^2}(1+\tau_1(1+\gamma^2)+2\tau_2\gamma))^2 \right),
\end{align*}
and substituting the expressions of $\alpha, \alpha', x, c, \tau_1, \tau_2$ as functions of $\eta, \eta', \gamma$ gives \[
|K| = \frac{(1 - (u^Tu')^2 \eta \eta')^2}{(1-\eta^2)(1-\eta'^2)},
\]
as claimed.
\end{proof}

\begin{proof}[Proof of Proposition \ref{prop:block_3_cycle}]
 We start by proving the first statement. Since $\Sigma_\mathbb{S}$ is consistent, we have that
\begin{align*}
    \Sigma_\mathbb{S} \text{ is compatible } \quad  & \text{if and only if } \quad \begin{pmatrix}
    I_d & P & -P \\
    P^T & I_d & \beta I_d \\
    -P^T & \beta I_d & I_d
\end{pmatrix} \succeq 0 \\
    &  \text{if and only if } \quad  \begin{pmatrix}
    I_d & \beta I_d \\
    \beta I_d & I_d
\end{pmatrix} - 
\begin{pmatrix}
    P^T P & -P^T P \\
    -P^T P & P^T P 
\end{pmatrix} \succeq 0,
\end{align*}
where the second equivalence follows by standard properties of Schur complements. However, we can see that
\begin{align*}
    \inf \biggl\{ \begin{pmatrix} x \\ y \end{pmatrix}^T \biggl\{ \begin{pmatrix}
    I_d & \beta I_d \\
    \beta I_d & I_d
\end{pmatrix} &- 
\begin{pmatrix}
    P^T P & -P^T P \\
    -P^T P & P^T P 
\end{pmatrix} \biggr\} \begin{pmatrix} x \\ y \end{pmatrix} : x, y \in \mathbb{R}^d \biggr\} \\
    & = \inf \bigl\{ \|x-y\|_2^2 + 2(1+\beta)x^Ty - \|P(x-y)\|_2^2 : x, y \in \mathbb{R}^d \bigr\} \\
    & = \inf \bigl\{ \|v\|_2^2 + 2(1+\beta)(v+y)^Ty - \|Pv\|_2^2 : v, y \in \mathbb{R}^d \bigr\} \\
    & = \inf \biggl\{ \frac{1-\beta}{2}\|v\|_2^2 - v^TP^TPv : v  \in \mathbb{R}^d \biggr\} \\
    & = \inf \biggl\{ \biggl( \frac{1-\beta}{2} - \|P\|_2^2\biggr)v^2 : v \in [0,\infty) \biggr\},
\end{align*}
where the third equality follows on noting that the minimising choice of $y$ is given by $-v/2$. It is now clear that $\Sigma_\mathbb{S}$ is compatible if and only if $\|P\|_2^2 \leq (1-\beta)/2$, as claimed.

As for the second part of the statement, let $\boldsymbol{v}_1, \ldots, \boldsymbol{v}_d$ the orthonormal eigenvectors or $P^TP$ with eigenvalues $\lambda_1 \geq \ldots \geq \lambda_d$, and let $L$ be the maximal $l$ such that $\lambda_l \geq (1-\beta)/2$. For $l \in [L]$, define 
\[
        X_{\mathbb{S}}^{(l)} = \frac{c}{4}\Bigg\{\begin{pmatrix}
            2P\boldsymbol{v}_l\boldsymbol{v}_l^T P^T & -2P\boldsymbol{v}_l\boldsymbol{v}_l^T \\
            -2\boldsymbol{v}_l\boldsymbol{v}_l^T P^T & \boldsymbol{v}_l\boldsymbol{v}_l^T/2
        \end{pmatrix},\begin{pmatrix}
            2P\boldsymbol{v}_l\boldsymbol{v}_l^T P^T & 2P\boldsymbol{v}_l\boldsymbol{v}_l^T \\
            2\boldsymbol{v}_l\boldsymbol{v}_l^T P^T & \boldsymbol{v}_l\boldsymbol{v}_l^T/2
        \end{pmatrix}, \begin{pmatrix}
            \boldsymbol{v}_l\boldsymbol{v}_l^T/2 & -\boldsymbol{v}_l\boldsymbol{v}_l^T \\
            -\boldsymbol{v}_l\boldsymbol{v}_l^T & \boldsymbol{v}_l\boldsymbol{v}_l^T/2
        \end{pmatrix}\Bigg\},
        \]
with $0< c \leq 5/6 + \sqrt{73}/6$, and define $ X_{\mathbb{S}} = \sum_{l=1}^L X_{\mathbb{S}}^{(l)}$. We first show that $X_\mathbb{S}$ is a feasible solution for our primal optimisation problem. We have 
\[A^*X_{\mathbb{S}} = \frac{c}{4} \sum_{l=1}^L
\begin{pmatrix}
            4P\boldsymbol{v}_l\boldsymbol{v}_l^T P^T & -2P\boldsymbol{v}_l\boldsymbol{v}_l^T & 2P\boldsymbol{v}_l\boldsymbol{v}_l^T\\
            -2\boldsymbol{v}_l\boldsymbol{v}_l^T P^T & \boldsymbol{v}_l\boldsymbol{v}_l^T & - \boldsymbol{v}_l\boldsymbol{v}_l^T \\
            2\boldsymbol{v}_l\boldsymbol{v}_l^T P^T & - \boldsymbol{v}_l\boldsymbol{v}_l^T &  \boldsymbol{v}_l\boldsymbol{v}_l^T
        \end{pmatrix} = \frac{c}{4} \sum_{l=1}^L \begin{pmatrix}
            2P\boldsymbol{v}_l \\
            -\boldsymbol{v}_l \\
            \boldsymbol{v}_l
        \end{pmatrix}
        \begin{pmatrix}
            2P\boldsymbol{v}_l \\
            -\boldsymbol{v}_l \\
            \boldsymbol{v}_l
        \end{pmatrix}^T \succeq 0,
\]
and since $X_{\mathbb{S}}^{(0)} = \frac{1}{2} (I_{2d}, I_{2d}, I_{2d})$, 
\begin{align*}
    X_{\mathbb{S}} + X_{\mathbb{S}}^{(0)} & = \frac{1}{2}\left(\begin{pmatrix}
    I_d + cP(\sum_{l=1}^L\boldsymbol{v}_l \boldsymbol{v}_l^T) P^T & -cP(\sum_{l=1}^L\boldsymbol{v}_l \boldsymbol{v}_l^T) \\
    -c(\sum_{l=1}^L\boldsymbol{v}_l \boldsymbol{v}_l^T) P^T & I_d + \frac{c}{4} \sum_{l=1}^L\boldsymbol{v}_l \boldsymbol{v}_l^T 
\end{pmatrix}, \right. \\
& \left. \quad \quad  \begin{pmatrix}
    I_d + cP(\sum_{l=1}^L\boldsymbol{v}_l \boldsymbol{v}_l^T) P^T & +cP(\sum_{l=1}^L\boldsymbol{v}_l \boldsymbol{v}_l^T) \\
    +c(\sum_{l=1}^L\boldsymbol{v}_l \boldsymbol{v}_l^T) P^T & I_d + \frac{c}{4} \sum_{l=1}^L\boldsymbol{v}_l \boldsymbol{v}_l^T 
\end{pmatrix}, \begin{pmatrix}
    I_d + \frac{c}{4} \sum_{l=1}^L\boldsymbol{v}_l \boldsymbol{v}_l^T & -\frac{c}{2} \sum_{l=1}^L\boldsymbol{v}_l \boldsymbol{v}_l^T \\
    -\frac{c}{2} \sum_{l=1}^L\boldsymbol{v}_l \boldsymbol{v}_l^T & I_d + \frac{c}{4} \sum_{l=1}^L\boldsymbol{v}_l \boldsymbol{v}_l^T
\end{pmatrix} \right).
\end{align*}
It remains to show that $X_{\mathbb{S}} + X_{\mathbb{S}}^{(0)} \succeq_\mathbb{S} 0$. Now, as for the first component of $X_{\mathbb{S}} + X_{\mathbb{S}}^{(0)}$, observe that the bottom-right block \[
I_d + \frac{c}{4} \sum_{l=1}^L\boldsymbol{v}_l \boldsymbol{v}_l^T \succeq 0,
\]
and it is invertible due to the fact that $\|c \sum_{l=1}^L\boldsymbol{v}_l \boldsymbol{v}_l^T/4\|_2 \leq c/4 < 1$, since the $\boldsymbol{v}_l$'s are orthonormal. The inverse is 
\begin{align*}
    \left(I_d + \frac{c}{4} \sum_{l=1}^L\boldsymbol{v}_l \boldsymbol{v}_l^T \right)^{-1}  &= \left(I_d -  \left(-\frac{c}{4} \sum_{l=1}^L\boldsymbol{v}_l \boldsymbol{v}_l^T\right)\right)^{-1} = \sum_{k=0}^\infty (-1)^k \left(\frac{c}{4}\right)^k \left(\sum_{l=1}^L\boldsymbol{v}_l \boldsymbol{v}_l^T \right)^k \\
    &= I_d + \sum_{k=1}^\infty (-1)^k \left(\frac{c}{4}\right)^k \left(\sum_{l=1}^L\boldsymbol{v}_l \boldsymbol{v}_l^T \right)^k = I_d + \left(\sum_{l=1}^L\boldsymbol{v}_l \boldsymbol{v}_l^T \right)\sum_{k=1}^\infty (-1)^k \left(\frac{c}{4}\right)^k \\
    & = I_d + \left(\sum_{l=1}^L\boldsymbol{v}_l \boldsymbol{v}_l^T \right)\left( \frac{1}{1+c/4} - 1\right) = I_d - \frac{c}{c+4}\sum_{l=1}^L\boldsymbol{v}_l \boldsymbol{v}_l^T,
\end{align*}
where the fourth equality comes from the fact that $\sum_{l=1}^L\boldsymbol{v}_l \boldsymbol{v}_l^T$ is idempotent again by the orthonormality of the $\boldsymbol{v}_l$'s. Hence, the first component of $X_{\mathbb{S}} + X_{\mathbb{S}}^{(0)}$ is positive semi-definite if and only if \[
I_d + cP\left(\sum_{l=1}^L\boldsymbol{v}_l \boldsymbol{v}_l^T\right) P^T \succeq c^2 P \left(\sum_{l=1}^L\boldsymbol{v}_l \boldsymbol{v}_l^T\right) \left(I_d -\frac{c}{4+c}\sum_{l=1}^L\boldsymbol{v}_l \boldsymbol{v}_l^T\right) \left(\sum_{l=1}^L\boldsymbol{v}_l \boldsymbol{v}_l^T\right) P^T,
\]
which is equivalent to \[
I_d \succeq \left(\frac{4c^2}{4+c} - c \right)P\left(\sum_{l=1}^L\boldsymbol{v}_l \boldsymbol{v}_l^T\right)P^T = \left(\frac{4c^2}{4+c} - c \right)P\left(\sum_{l=1}^L\boldsymbol{v}_l \boldsymbol{v}_l^T\right)^2P^T,
\]
which is satisfied if and only if $4c^2/(4+c) - c \leq 1$, due to the fact that $\|P\left(\sum_{l=1}^L\boldsymbol{v}_l \boldsymbol{v}_l^T\right)^2P^T\|_2 = \|P\left(\sum_{l=1}^L\boldsymbol{v}_l \boldsymbol{v}_l^T\right)\|_2^2 \leq 1$ again by orthonormality. This implies that the first component of $X_{\mathbb{S}} + X_{\mathbb{S}}^{(0)} \succeq_\mathbb{S}$ is PSD if and only if $0< c \leq 5/6 + \sqrt{73}/6$, and of course the same is true for the second component of $X_{\mathbb{S}} + X_{\mathbb{S}}^{(0)}$. As for the third component,  using an analogous idea, it is easy to show that it is positive semi-definite if and only if \[
I_d \succeq \left(\frac{c^2}{4+c} - \frac{c}{4} \right)  \sum_{l=1}^L\boldsymbol{v}_l \boldsymbol{v}_l^T,
\]
which is satisfied if and only if $0 < c \leq 4$. Summing up, this shows that $X_{\mathbb{S}}$ is feasible for $0< c \leq 5/6 + \sqrt{73}/6$, and leads to 
\begin{align*}
    R(\Sigma_\mathbb{S}) & \geq -\frac{c}{3d} \sum_{l=1}^L \begin{pmatrix}
            2P\boldsymbol{v}_l \\
            -\boldsymbol{v}_l \\
            \boldsymbol{v}_l
        \end{pmatrix}
        \begin{pmatrix}
            I_d & P & -P \\
            P & I_d & \beta I_d \\
            -P^T \beta I_d & I_d 
        \end{pmatrix}
        \begin{pmatrix}
            2P\boldsymbol{v}_l \\
            -\boldsymbol{v}_l \\
            \boldsymbol{v}_l
        \end{pmatrix}^T \\
        & = \frac{c}{3d} \sum_{l=1}^L \left(\lambda_l - \frac{1-\beta}{2}\right) = \frac{c}{3d} \sum_{l=1}^d \left(\lambda_l - \frac{1-\beta}{2}\right)_+ \\
        & = \frac{c}{3d} \sum_{l=1}^d \left(\sigma^2_l(P)^2 - \frac{1-\beta}{2}\right)_+ > \frac{3}{4d} \sum_{l=1}^d \left(\sigma^2_l(P)^2 - \frac{1-\beta}{2}\right)_+,
\end{align*}
since $ 5/6 + \sqrt{73}/6 > 9/4$.
\end{proof}

\bibliography{bib}{}

\begin{thebibliography}{71}
\providecommand{\natexlab}[1]{#1}
\providecommand{\url}[1]{\texttt{#1}}
\expandafter\ifx\csname urlstyle\endcsname\relax
  \providecommand{\doi}[1]{doi: #1}\else
  \providecommand{\doi}{doi: \begingroup \urlstyle{rm}\Url}\fi

\bibitem[Albert(1972)]{albert72pseudo}
Arthur Albert.
\newblock Regression and the {M}oore-{P}enrose pseudoinverse.
\newblock \emph{Mathematics in science and engineering}, 94, 1972.

\bibitem[Aleksić(2024)]{aleksic2023novel}
Danijel Aleksić.
\newblock A novel test of {M}issing {C}ompletely at {R}andom:
  U-statistics-based approach.
\newblock \emph{Statistics}, 0\penalty0 (0):\penalty0 1--20, 2024.
\newblock \doi{10.1080/02331888.2024.2386361}.
\newblock URL \url{https://doi.org/10.1080/02331888.2024.2386361}.

\bibitem[Aleksić et~al.(2023)Aleksić, Cuparić, and Milošević]{milosevic23}
Danijel Aleksić, Marija Cuparić, and Bojana Milošević.
\newblock {N}on-degenerate {U}-statistics for data missing completely at random
  with application to testing independence.
\newblock \emph{Stat}, 12\penalty0 (1):\penalty0 e634, 2023.
\newblock \doi{https://doi.org/10.1002/sta4.634}.
\newblock URL \url{https://onlinelibrary.wiley.com/doi/abs/10.1002/sta4.634}.

\bibitem[Andrews(2000)]{andrews2000bootstrap}
Donald W.~K. Andrews.
\newblock Inconsistency of the bootstrap when a parameter is on the boundary of
  the parameter space.
\newblock \emph{Econometrica}, 68\penalty0 (2):\penalty0 399--405, 2000.
\newblock \doi{https://doi.org/10.1111/1468-0262.00114}.
\newblock URL
  \url{https://onlinelibrary.wiley.com/doi/abs/10.1111/1468-0262.00114}.

\bibitem[Barrett et~al.(1993)Barrett, Johnson, and Tarazaga]{barrett1993real}
Wayne Barrett, Charles~R Johnson, and Pablo Tarazaga.
\newblock The real positive definite completion problem for a simple cycle.
\newblock \emph{Linear Algebra Appl.}, 192:\penalty0 3--31, 1993.

\bibitem[Bentler and Berkane(1986)]{kurtosisElliptical86}
P.~M. Bentler and M.~Berkane.
\newblock {Greatest lower bound to the elliptical theory kurtosis parameter}.
\newblock \emph{Biometrika}, 73\penalty0 (1):\penalty0 240--241, 04 1986.
\newblock ISSN 0006-3444.
\newblock \doi{10.1093/biomet/73.1.240}.
\newblock URL \url{https://doi.org/10.1093/biomet/73.1.240}.

\bibitem[Berrett and Samworth(2023)]{berrett2022optimal}
Thomas~B Berrett and Richard~J Samworth.
\newblock Optimal nonparametric testing of {M}issing {C}ompletely {A}t
  {R}andom, and its connections to compatibility.
\newblock \emph{Ann. Statist.}, 51\penalty0 (5):\penalty0 2170--2193, 2023.

\bibitem[Berrett et~al.(2022)Berrett, Bordino, Duisenbekov, Jaffe, and
  Samworth]{berrett2022MCARtest}
Thomas~B. Berrett, Alberto Bordino, Danat Duisenbekov, Sean Jaffe, and
  Richard~J. Samworth.
\newblock \emph{MCARtest: Optimal nonparametric testing of {M}issing
  {C}ompletely {A}t {R}andom}, 2022.
\newblock URL
  \url{https://cran.r-project.org/web/packages/MCARtest/index.html}.
\newblock R package version 1.2.

\bibitem[Blanchard et~al.(2018)Blanchard, Carpentier, and
  Gutzeit]{carpentier18convex}
Gilles Blanchard, Alexandra Carpentier, and Maurilio Gutzeit.
\newblock {Minimax Euclidean separation rates for testing convex hypotheses in
  $\mathbb{R}^{d}$}.
\newblock \emph{Electronic Journal of Statistics}, 12\penalty0 (2):\penalty0
  3713 -- 3735, 2018.
\newblock \doi{10.1214/18-EJS1472}.
\newblock URL \url{https://doi.org/10.1214/18-EJS1472}.

\bibitem[Blekherman et~al.(2012)Blekherman, Parrilo, and
  Thomas]{blekherman2012sdp}
Grigoriy Blekherman, Pablo~A. Parrilo, and Rekha~R. Thomas.
\newblock \emph{Semidefinite Optimization and Convex Algebraic Geometry}.
\newblock Society for Industrial and Applied Mathematics, Philadelphia, PA,
  2012.
\newblock \doi{10.1137/1.9781611972290}.
\newblock URL \url{https://epubs.siam.org/doi/abs/10.1137/1.978161197229}.

\bibitem[Boyd and Vandenberghe(2004)]{boyd_conex_opt_2004}
Stephen Boyd and Lieven Vandenberghe.
\newblock \emph{Convex {O}ptimization}.
\newblock {Cambridge University Press}, March 2004.
\newblock ISBN 0521833787.
\newblock URL
  \url{https://www.cambridge.org/gb/universitypress/subjects/statistics-probability/optimization-or-and-risk/convex-optimization?format=HB}.

\bibitem[Brown and Forsythe(1974)]{brown12variances}
Morton~B. Brown and Alan~B. Forsythe.
\newblock Robust tests for the equality of variances.
\newblock \emph{Journal of the American Statistical Association}, 69\penalty0
  (346):\penalty0 364--367, 1974.
\newblock \doi{10.1080/01621459.1974.10482955}.
\newblock URL
  \url{https://www.tandfonline.com/doi/abs/10.1080/01621459.1974.10482955}.

\bibitem[Bunea and Xiao(2015)]{Bunea_2015}
Florentina Bunea and Luo Xiao.
\newblock On the sample covariance matrix estimator of reduced effective rank
  population matrices, with applications to {fPCA}.
\newblock \emph{Bernoulli}, 21\penalty0 (2), may 2015.
\newblock URL \url{https://doi.org/10.3150%2F14-bej602}.

\bibitem[Cai and Low(2011)]{cai2011testing}
T~Tony Cai and Mark~G Low.
\newblock Testing composite hypotheses, {H}ermite polynomials and optimal
  estimation of a nonsmooth functional.
\newblock \emph{Ann. Statist.}, 39:\penalty0 1012--1041, 2011.

\bibitem[Cai and Zhang(2019)]{cai2018high}
T~Tony Cai and Linjun Zhang.
\newblock High-dimensional linear discriminant analysis: Optimality, adaptive
  algorithm, and missing data.
\newblock \emph{J. Roy. Statist. Soc. Ser. B}, 81\penalty0 (4):\penalty0
  675--705, 2019.

\bibitem[Cai et~al.(2011)Cai, Liu, and Luo]{cai2011precision}
Tony Cai, Weidong Liu, and Xi~Luo.
\newblock A {C}onstrained $l_1$ {M}inimization {A}pproach to {S}parse
  {P}recision {M}atrix {E}stimation.
\newblock \emph{Journal of the American Statistical Association}, 106\penalty0
  (494):\penalty0 594--607, 2011.
\newblock ISSN 01621459.
\newblock URL \url{http://www.jstor.org/stable/41416395}.

\bibitem[Cand\`es and Tao(2010)]{candes2010completion}
Emmanuel~J. Cand\`es and Terence Tao.
\newblock The {P}ower of {C}onvex {R}elaxation: {N}ear-{O}ptimal {M}atrix
  {C}ompletion.
\newblock \emph{IEEE Transactions on Information Theory}, 56\penalty0
  (5):\penalty0 2053--2080, 2010.
\newblock \doi{10.1109/TIT.2010.2044061}.

\bibitem[Candès and Recht(2009)]{candes_recht_2009}
Emmanuel~J. Candès and Benjamin Recht.
\newblock Exact {M}atrix {C}ompletion via {C}onvex {O}ptimization.
\newblock \emph{Foundations of Computational Mathematics}, 9\penalty0
  (6):\penalty0 717–772, Dec 2009.
\newblock ISSN 1615-3375.
\newblock Funding by NSF.

\bibitem[Cavaliere et~al.(2017)Cavaliere, Nielsen, and
  Rahbek]{cavaliere17bootstrap}
Giuseppe Cavaliere, Heino~Bohn Nielsen, and Anders Rahbek.
\newblock On the consistency of bootstrap testing for a parameter on the
  boundary of the parameter space.
\newblock \emph{Journal of Time Series Analysis}, 38\penalty0 (4):\penalty0
  513--534, 2017.
\newblock \doi{https://doi.org/10.1111/jtsa.12214}.
\newblock URL \url{https://onlinelibrary.wiley.com/doi/abs/10.1111/jtsa.12214}.

\bibitem[Dempster et~al.(1977)Dempster, Laird, and Rubin]{rubin77EM}
A.~P. Dempster, N.~M. Laird, and D.~B. Rubin.
\newblock Maximum {L}ikelihood from {I}ncomplete {D}ata {V}ia the {EM}
  {A}lgorithm.
\newblock \emph{Journal of the Royal Statistical Society: Series B
  (Methodological)}, 39\penalty0 (1):\penalty0 1--22, 1977.
\newblock URL
  \url{https://rss.onlinelibrary.wiley.com/doi/abs/10.1111/j.2517-6161.1977.tb01600.x}.

\bibitem[Elsener and van~de Geer(2019)]{elsner2019}
Andreas Elsener and Sara van~de Geer.
\newblock Sparse spectral estimation with missing and corrupted measurements.
\newblock \emph{Stat}, 8\penalty0 (1):\penalty0 e229, 2019.
\newblock \doi{https://doi.org/10.1002/sta4.229}.
\newblock URL \url{https://onlinelibrary.wiley.com/doi/abs/10.1002/sta4.229}.

\bibitem[Follain et~al.(2022)Follain, Wang, and Samworth]{follain2021}
Bertille Follain, Tengyao Wang, and Richard~J. Samworth.
\newblock {High-dimensional Changepoint Estimation with Heterogeneous
  Missingness}.
\newblock \emph{Journal of the Royal Statistical Society Series B: Statistical
  Methodology}, 84\penalty0 (3):\penalty0 1023--1055, 07 2022.
\newblock ISSN 1369-7412.
\newblock \doi{10.1111/rssb.12540}.
\newblock URL \url{https://doi.org/10.1111/rssb.12540}.

\bibitem[Fuchs(1982)]{fuchs1982maximum}
Camil Fuchs.
\newblock Maximum likelihood estimation and model selection in contingency
  tables with missing data.
\newblock \emph{J. Amer. Statist. Assoc.}, 77:\penalty0 270--278, 1982.

\bibitem[Gastwirth et~al.(2009)Gastwirth, Gel, and Miao]{miao09levene}
Joseph~L. Gastwirth, Yulia~R. Gel, and Weiwen Miao.
\newblock {The Impact of Levene’s Test of Equality of Variances on
  Statistical Theory and Practice}.
\newblock \emph{Statistical Science}, 24\penalty0 (3):\penalty0 343 -- 360,
  2009.
\newblock \doi{10.1214/09-STS301}.
\newblock URL \url{https://doi.org/10.1214/09-STS301}.

\bibitem[Grone et~al.(1984)Grone, Johnson, Sá, and Wolkowicz]{GRONE1984109}
Robert Grone, Charles~R. Johnson, Eduardo~M. Sá, and Henry Wolkowicz.
\newblock Positive definite completions of partial {H}ermitian matrices.
\newblock \emph{Linear Algebra and its Applications}, 58:\penalty0 109--124,
  1984.
\newblock ISSN 0024-3795.
\newblock \doi{https://doi.org/10.1016/0024-3795(84)90207-6}.
\newblock URL
  \url{https://www.sciencedirect.com/science/article/pii/0024379584902076}.

\bibitem[Han and Liu(2017)]{Han_2017}
Fang Han and Han Liu.
\newblock Statistical analysis of latent generalized correlation matrix
  estimation in transelliptical distribution.
\newblock \emph{Bernoulli}, 23\penalty0 (1), feb 2017.
\newblock \doi{10.3150/15-bej702}.
\newblock URL \url{https://doi.org/10.3150%2F15-bej702}.

\bibitem[Hawkins(1981)]{Hawkins1981}
{Douglas M.} Hawkins.
\newblock A new test for multivariate normality and homoscedasticity.
\newblock \emph{Technometrics}, 23\penalty0 (1):\penalty0 105--110, 1981.
\newblock ISSN 0040-1706.
\newblock \doi{10.1080/00401706.1981.10486244}.

\bibitem[Hofert et~al.(2020)Hofert, Kojadinovic, Maechler, and Yan]{R_copula}
Marius Hofert, Ivan Kojadinovic, Martin Maechler, and Jun Yan.
\newblock \emph{copula: Multivariate Dependence with Copulas}, 2020.
\newblock URL \url{https://CRAN.R-project.org/package=copula}.
\newblock R package version 1.0-0.

\bibitem[Jamshidian and Jalal(2010)]{jamshidian2010tests}
Mortaza Jamshidian and Siavash Jalal.
\newblock Tests of homoscedasticity, normality, and missing completely at
  random for incomplete multivariate data.
\newblock \emph{Psychometrika}, 75:\penalty0 649--674, 2010.

\bibitem[Jiao et~al.(2016)Jiao, Han, and Weissman]{weissman16}
Jiantao Jiao, Yanjun Han, and Tsachy Weissman.
\newblock Minimax {E}stimation of the {$L_1$} {D}istance.
\newblock In \emph{2016 IEEE International Symposium on Information Theory
  (ISIT)}, pages 750--754, 2016.
\newblock \doi{10.1109/ISIT.2016.7541399}.

\bibitem[Juditsky and Nemirovski(2002)]{jud_nemiroski_matching_mom02}
Anatoli Juditsky and Arkadi Nemirovski.
\newblock {On nonparametric tests of positivity/monotonicity/convexity}.
\newblock \emph{The Annals of Statistics}, 30\penalty0 (2):\penalty0 498 --
  527, 2002.
\newblock \doi{10.1214/aos/1021379863}.
\newblock URL \url{https://doi.org/10.1214/aos/1021379863}.

\bibitem[Kellerer(1984)]{kellerer1984duality}
Hans~G Kellerer.
\newblock Duality theorems for marginal problems.
\newblock \emph{Z. Wahrscheinlichkeit.}, 67:\penalty0 399--432, 1984.

\bibitem[Khachiyan and Porkolab(1997)]{porkolab}
L.~Khachiyan and L.~Porkolab.
\newblock Computing integral points in convex semi-algebraic sets.
\newblock In \emph{Proceedings 38th Annual Symposium on Foundations of Computer
  Science}, pages 162--171, 1997.
\newblock \doi{10.1109/SFCS.1997.646105}.

\bibitem[Kim and Bentler(2002)]{kim2002tests}
Kevin~H Kim and Peter~M Bentler.
\newblock Tests of homogeneity of means and covariance matrices for
  multivariate incomplete data.
\newblock \emph{Psychometrika}, 67:\penalty0 609--623, 2002.

\bibitem[Laurent(2009)]{laurent2009matrix}
Monique Laurent.
\newblock Matrix {C}ompletion {P}roblems.
\newblock \emph{Encyclopedia of Optimization}, 3:\penalty0 221--229, 2009.

\bibitem[Laurent and Poljak(1996)]{laurent_elliptope96}
Monique Laurent and Svatopluk Poljak.
\newblock On the {F}acial {S}tructure of the {S}et of {C}orrelation {M}atrices.
\newblock \emph{SIAM Journal on Matrix Analysis and Applications}, 17\penalty0
  (3):\penalty0 530--547, 1996.
\newblock \doi{10.1137/0617031}.
\newblock URL \url{https://doi.org/10.1137/0617031}.

\bibitem[Lehmann(1999)]{lehmann99largesample}
Erich~L Lehmann.
\newblock \emph{Elements of {L}arge-{S}ample {T}heory}.
\newblock Springer Science \& Business Media, 1999.

\bibitem[Li and Yu(2015)]{li2015nonparametric}
Jun Li and Yao Yu.
\newblock A nonparametric test of missing completely at random for incomplete
  multivariate data.
\newblock \emph{Psychometrika}, 80:\penalty0 707--726, 2015.

\bibitem[Little and Rubin(2002)]{little2002statistical}
R.J.A. Little and D.B. Rubin.
\newblock \emph{Statistical analysis with missing data}.
\newblock Wiley series in probability and mathematical statistics. Probability
  and mathematical statistics. Wiley, 2002.
\newblock ISBN 9780471183860.
\newblock URL \url{http://books.google.com/books?id=aYPwAAAAMAAJ}.

\bibitem[Little(1988)]{little1988test}
Roderick~JA Little.
\newblock A test of {M}issing {C}ompletely at {R}andom for multivariate data
  with missing values.
\newblock \emph{J. Amer. Statist. Assoc.}, 83:\penalty0 1198--1202, 1988.

\bibitem[Loh and Wainwright(2012)]{loh2012high}
Po-Ling Loh and Martin~J Wainwright.
\newblock High-dimensional regression with noisy and missing data: Provable
  guarantees with nonconvexity.
\newblock \emph{Ann. Statist.}, 40\penalty0 (3):\penalty0 1637--1664, 2012.

\bibitem[Lounici(2014)]{lounici14}
Karim Lounici.
\newblock {High-dimensional covariance matrix estimation with missing
  observations}.
\newblock \emph{Bernoulli}, 20\penalty0 (3):\penalty0 1029 -- 1058, 2014.
\newblock \doi{10.3150/12-BEJ487}.
\newblock URL \url{https://doi.org/10.3150/12-BEJ487}.

\bibitem[Lov{\'a}sz(2003)]{Lovasz2003}
L.~Lov{\'a}sz.
\newblock \emph{Semidefinite Programs and Combinatorial Optimization}, pages
  137--194.
\newblock Springer New York, New York, NY, 2003.
\newblock ISBN 978-0-387-22444-2.
\newblock URL \url{https://doi.org/10.1007/0-387-22444-0_6}.

\bibitem[Meinshausen and Bühlmann(2006)]{buhlmann2006lasso}
Nicolai Meinshausen and Peter Bühlmann.
\newblock High-{D}imensional {G}raphs and {V}ariable {S}election with the
  {L}asso.
\newblock \emph{The Annals of Statistics}, 34\penalty0 (3):\penalty0
  1436--1462, 2006.
\newblock ISSN 00905364.
\newblock URL \url{http://www.jstor.org/stable/25463463}.

\bibitem[Mosteller and Fisher(1948)]{fisher48method}
Frederick Mosteller and R.~A. Fisher.
\newblock Questions and answers.
\newblock \emph{The American Statistician}, 2\penalty0 (5):\penalty0 30--31,
  1948.
\newblock ISSN 00031305.
\newblock URL \url{http://www.jstor.org/stable/2681650}.

\bibitem[Muirhead(1982)]{Muirhead1982AspectsOM}
Robb~J. Muirhead.
\newblock Aspects of multivariate statistical theory.
\newblock In \emph{Wiley Series in Probability and Statistics}, 1982.
\newblock URL \url{https://api.semanticscholar.org/CorpusID:123513635}.

\bibitem[Nesterov and Nemirovskii(1994)]{nesterov94}
Yurii Nesterov and Arkadii Nemirovskii.
\newblock \emph{Interior-Point Polynomial Algorithms in Convex Programming}.
\newblock Society for Industrial and Applied Mathematics, 1994.
\newblock \doi{10.1137/1.9781611970791}.
\newblock URL \url{https://epubs.siam.org/doi/abs/10.1137/1.9781611970791}.

\bibitem[Oliveira(2010)]{oliveira2010concentration}
Roberto~Imbuzeiro Oliveira.
\newblock Concentration of the adjacency matrix and of the {L}aplacian in
  random graphs with independent edges.
\newblock \emph{arXiv preprint arXiv:0911.0600}, 2010.

\bibitem[Polak(2012)]{polak_2012_optimization}
Elijah Polak.
\newblock \emph{Optimization}.
\newblock Springer Science and Business Media, 12 2012.

\bibitem[Ramana(1997)]{Ramana1997AnED}
Motakuri~V. Ramana.
\newblock An exact duality theory for semidefinite programming and its
  complexity implications.
\newblock \emph{Mathematical Programming}, 77:\penalty0 129--162, 1997.

\bibitem[Recht(2011)]{recht2011}
Benjamin Recht.
\newblock A {S}impler {A}pproach to {M}atrix {C}ompletion.
\newblock \emph{J. Mach. Learn. Res.}, 12:\penalty0 3413–3430, dec 2011.
\newblock ISSN 1532-4435.

\bibitem[Rockafellar(1970)]{rockafellar-1970a}
R.~Tyrrell Rockafellar.
\newblock \emph{Convex analysis}.
\newblock Princeton Mathematical Series. Princeton University Press, Princeton,
  N. J., 1970.

\bibitem[Rockel(2020)]{R_missMethods}
Tobias Rockel.
\newblock \emph{missMethods: Methods for Missing Data}, 2020.
\newblock URL \url{https://CRAN.R-project.org/package=missMethods}.
\newblock R package version 0.2.0.

\bibitem[Rudelson and Vershynin(2007)]{rudelson07rank1}
Mark Rudelson and Roman Vershynin.
\newblock Sampling from large matrices: An approach through geometric
  functional analysis.
\newblock \emph{J. ACM}, 54\penalty0 (4):\penalty0 21–es, jul 2007.
\newblock ISSN 0004-5411.
\newblock \doi{10.1145/1255443.1255449}.
\newblock URL \url{https://doi.org/10.1145/1255443.1255449}.

\bibitem[Samworth(2003)]{samworth03restoring}
Richard Samworth.
\newblock {A note on methods of restoring consistency to the bootstrap}.
\newblock \emph{Biometrika}, 90\penalty0 (4):\penalty0 985--990, 12 2003.
\newblock ISSN 0006-3444.
\newblock \doi{10.1093/biomet/90.4.985}.
\newblock URL \url{https://doi.org/10.1093/biomet/90.4.985}.

\bibitem[Santos et~al.(2019)Santos, Pereira, Costa, Soares, Santos, and
  Abreu]{generateMAR}
Miriam~Seoane Santos, Ricardo~Cardoso Pereira, Adriana~Fonseca Costa,
  Jastin~Pompeu Soares, João Santos, and Pedro~Henriques Abreu.
\newblock Generating {S}ynthetic {M}issing {D}ata: A {R}eview by {M}issing
  {M}echanism.
\newblock \emph{IEEE Access}, 7:\penalty0 11651--11667, 2019.
\newblock \doi{10.1109/ACCESS.2019.2891360}.

\bibitem[Seaman et~al.(2013)Seaman, Galati, Jackson, and Carlin]{seaman13}
Shaun Seaman, John Galati, Dan Jackson, and John Carlin.
\newblock {What Is Meant by “Missing at Random”?}
\newblock \emph{Statistical Science}, 28\penalty0 (2):\penalty0 257 -- 268,
  2013.
\newblock \doi{10.1214/13-STS415}.
\newblock URL \url{https://doi.org/10.1214/13-STS415}.

\bibitem[Sell et~al.(2024)Sell, Berrett, and Cannings]{sell2023nonparametric}
Torben Sell, Thomas~B. Berrett, and Timothy~I. Cannings.
\newblock {Nonparametric classification with missing data}.
\newblock \emph{The Annals of Statistics}, 52\penalty0 (3):\penalty0 1178 --
  1200, 2024.
\newblock \doi{10.1214/24-AOS2389}.
\newblock URL \url{https://doi.org/10.1214/24-AOS2389}.

\bibitem[Spohn et~al.(2021)Spohn, N{\"a}f, Michel, and
  Meinshausen]{michel2021pklm}
Meta-Lina Spohn, Jeffrey N{\"a}f, Loris Michel, and Nicolai Meinshausen.
\newblock {PKLM}: A flexible {MCAR} test using {C}lassification.
\newblock \emph{arXiv preprint arXiv:2109.10150}, 2021.

\bibitem[Stekhoven and Bühlmann(2011)]{buhlmann_impuation_2011}
Daniel~J. Stekhoven and Peter Bühlmann.
\newblock {MissForest—non-parametric missing value imputation for mixed-type
  data}.
\newblock \emph{Bioinformatics}, 28\penalty0 (1):\penalty0 112--118, 10 2011.
\newblock ISSN 1367-4803.
\newblock \doi{10.1093/bioinformatics/btr597}.
\newblock URL \url{https://doi.org/10.1093/bioinformatics/btr597}.

\bibitem[Th{\'e}paut and Verzelen(2024)]{verzelen2021schatten}
Sol{\`e}ne Th{\'e}paut and Nicolas Verzelen.
\newblock Optimal estimation of {S}chatten norms of a rectangular matrix.
\newblock \emph{The Annals of Statistics}, 52\penalty0 (4):\penalty0 1334 --
  1359, 2024.
\newblock \doi{10.1214/24-AOS2374}.
\newblock URL \url{https://doi.org/10.1214/24-AOS2374}.

\bibitem[Tierney and Cook(2023)]{naniar}
Nicholas Tierney and Dianne Cook.
\newblock Expanding tidy data principles to facilitate missing data
  exploration, visualization and assessment of imputations.
\newblock \emph{Journal of Statistical Software}, 105\penalty0 (7):\penalty0
  1--31, 2023.
\newblock \doi{10.18637/jss.v105.i07}.

\bibitem[van Buuren and Groothuis-Oudshoorn(2011)]{van2011mice}
Stef van Buuren and Karin Groothuis-Oudshoorn.
\newblock mice: {M}ultivariate imputation by chained equations in {R}.
\newblock \emph{Journal of Statistical Software}, 45:\penalty0 1--67, 2011.

\bibitem[Vandenberghe and Boyd(1996)]{boyd_vandeberghe}
Lieven Vandenberghe and Stephen Boyd.
\newblock Semidefinite programming.
\newblock \emph{SIAM Review}, 38\penalty0 (1):\penalty0 49--95, 1996.
\newblock \doi{10.1137/1038003}.
\newblock URL \url{https://doi.org/10.1137/1038003}.

\bibitem[Vershynin(2019)]{vershynin19hdp}
Roman Vershynin.
\newblock \emph{High-Dimensional Probability}.
\newblock Cambridge University Press, 2019.
\newblock URL
  \url{https://www.math.uci.edu/~rvershyn/papers/HDP-book/HDP-book.pdf}.

\bibitem[Waghmare and Panaretos(2022)]{waghmare2022completion}
Kartik~G Waghmare and Victor~M Panaretos.
\newblock The completion of covariance kernels.
\newblock \emph{The Annals of Statistics}, 50\penalty0 (6):\penalty0
  3281--3306, 2022.

\bibitem[Wainwright(2019)]{wainwright2019high}
Martin~J Wainwright.
\newblock \emph{High-dimensional Statistics: A Non-asymptotic Viewpoint}.
\newblock Cambridge University Press, 2019.

\bibitem[Wilks(1946)]{wilks1946means}
S.~S. Wilks.
\newblock {Sample Criteria for Testing Equality of Means, Equality of
  Variances, and Equality of Covariances in a Normal Multivariate
  Distribution}.
\newblock \emph{The Annals of Mathematical Statistics}, 17\penalty0
  (3):\penalty0 257 -- 281, 1946.
\newblock \doi{10.1214/aoms/1177730940}.
\newblock URL \url{https://doi.org/10.1214/aoms/1177730940}.

\bibitem[Yanagida(2024)]{misty24Rpackage}
Takuya Yanagida.
\newblock \emph{misty: {M}iscellaneous {F}unctions}, 2024.
\newblock URL \url{https://CRAN.R-project.org/package=misty}.
\newblock R package version 0.6.7.

\bibitem[Yates(1933)]{yates33}
F.~Yates.
\newblock The {A}nalysis of {R}eplicated {E}xperiments {W}hen the {F}ield
  {R}esults {A}re {I}ncomplete.
\newblock \emph{Empirical Journal of Experimental Agriculture 1 (2): 129–42},
  1933.

\bibitem[Zhu et~al.(2022)Zhu, Wang, and Samworth]{zhu2019}
Ziwei Zhu, Tengyao Wang, and Richard~J. Samworth.
\newblock {High‐dimensional principal component analysis with heterogeneous
  missingness}.
\newblock \emph{Journal of the Royal Statistical Society Series B}, 84\penalty0
  (5):\penalty0 2000--2031, November 2022.
\newblock \doi{10.1111/rssb.12550}.
\newblock URL
  \url{https://ideas.repec.org/a/bla/jorssb/v84y2022i5p2000-2031.html}.

\end{thebibliography}

\begin{appendices}
    Appendix \ref{sec:extra_prop_R} contains further properties of our measure of incompatibility $R$ that were not investigated in the main body of the text. Appendix \ref{sec:test_cov} contains another oracle test based on a different measure of incompatibility, which acts on covariance matrices normalised in such a way to have fixed scale. Appendix \ref{sec:SDP} contains auxiliary results in Semi-definite Programming, while classical tail bounds are contained in Appendix~\ref{sec:technical_ineq}.

\section{Further properties of $R(\cdot)$}\label{sec:extra_prop_R}
We explore the properties of $R(\cdot)$ in the cycle example. In particular, we provide explicit expression in simple cases, we discuss the meaning of maximal incompatibility and we prove a result showing that $R(\cdot)$ is bounded below by the maxima of suitable linear functions if $\Sigma_{\mathbb{S}_d}$ is not too singular. Here we write $\Sigma_{\mathbb{S}_d}:= (\Sigma_{\{1,2\}}, \cdots, \Sigma_{ \{d,1\}})$ for a collection of $2 \times 2$ correlation matrices with
\[
\Sigma_{j,j+1} = \begin{pmatrix} 1 & \rho_{j} \\ \rho_{j} & 1 \end{pmatrix}.
\]
Our next result shows that singular matrices can be removed from $\Sigma_{\mathbb{S}_d}$ when $d\geq4$, without affecting the value of $R(\cdot)$, reducing the length of the cycle.

\begin{prop}\label{prop:reduction_2}
    Fix $d \geq 3$ and $k \geq 1$. Let $\Sigma_{\mathbb{S}_{d+k}}$ be a $(d+k)$-cycle with correlations $(\rho_{\{1,2\}},\ldots,\rho_{\{d+k,1\}})$ such that  $|\rho_{j,j+1}| = 1$ for all $j \in \{d+1, \cdots, d+k \}$ and let $\Sigma_{\mathbb{S}_{d}}$ represent a $d$-cycle with correlations $(\bar{\rho}_{\{1,2\}},\ldots,\bar{\rho}_{\{d,1\}})$ such that $\bar{\rho}_{j,j+1} = \rho_{j, j+1}, \text{ for all } j \in [d-1]$ and
    \[
    \begin{cases}
        \bar{\rho}_{d,1} = \rho_{d, d+1} \text{ if } \prod\limits_{j = d+1}^{d+k} \rho_{j,j+1} = 1\\
        \bar{\rho}_{d,1} = -\rho_{d, d+1} \text{ if } \prod\limits_{j = d+1}^{d+k} \rho_{j,j+1} = -1.
    \end{cases}
    \]
    Then we have $R(\Sigma_{\mathbb{S}_{d+k}}) = R(\Sigma_{\mathbb{S}_{d}})$.
\end{prop}

\begin{proof}[Proof of Proposition \ref{prop:reduction_2}]
    We will prove the result using the dual characterisation, which allows expressing $R(\Sigma_{\mathbb{S}_{d}})$ as \[
   1 - \frac{1}{d} \sup\{ \mathrm{tr}(\Sigma) : \Sigma \in \mathcal{P}^*, \Sigma_{11} = \cdots = \Sigma_{dd}, \Sigma_{\mathbb{S}_{d}} - A \Sigma \succeq_{\mathbb{S}_d} 0\}. \] 
   Suppose that $\prod_{j = d+1}^{d+k} \rho_{j,j+1} = 1$. We will show that $R(\Sigma_{{\mathbb{S}}_{d+k}}) = R(\Sigma_{\mathbb{S}_{d}})$ by proving both $R(\Sigma_{{\mathbb{S}}_{d+k}}) \geq R(\Sigma_{\mathbb{S}_{d}})$ and $R(\Sigma_{{\mathbb{S}}_{d+k}}) \leq R(\Sigma_{\mathbb{S}_{d}})$. As for the first of these, for every $\tilde{\Sigma}$ optimal for $\Sigma_{\mathbb{S}_{d+k}}$, we will show that $\Sigma = \tilde{\Sigma}_{|[d]}$ is feasible for $\Sigma_{\mathbb{S}_d}$. Now, $\Sigma \succeq 0$ since $\tilde{\Sigma} \succeq 0$, and $\Sigma_{11} = \ldots = \Sigma_{dd}$ by definition of $\tilde{\Sigma}$. As for $\Sigma_{\mathbb{S}_d} - A_{\mathbb{S}_d}\Sigma \succeq_{\mathbb{S}} 
0$, observe that $\Sigma_{\mathbb{S}_d}$ contains exactly the first $d$ matrices in $\Sigma_{\mathbb{S}_{d}}$, but $\mathbb{S}_d$ contains just $d-1$ patterns of $\Sigma_{\mathbb{S}_{d+k}}$. This is due to the fact that $\mathbb{S}_d$ has $\{d,1\}$ in place of $\{d,d+1\}$, which prevents us from employing Proposition \ref{prop:properties} (ii). Nonetheless, observe that $\tilde{\Sigma}_{1,d} = \tilde{\Sigma}_{d,d+1}$, due to the fact that $\tilde{\Sigma}_{j,j+1} = \rho_{j,j+1}\tilde{\Sigma}_{11} = \pm \tilde{\Sigma}_{11}$ for all $j \in \{d+1, \ldots, d+k\}$. Indeed, $\Sigma_{\mathbb{S}_{d+k}} - A_{\mathbb{S}_{d+k}}\tilde{\Sigma} \succeq_{\mathbb{S}_{d+k}} 0 $ implies that \[
\begin{pmatrix}
    1 & \pm 1 \\
    \pm 1 & 1
\end{pmatrix} - 
\begin{pmatrix}
    \tilde{\Sigma}_{11} & \tilde{\Sigma}_{j,j+1} \\
    \tilde{\Sigma}_{j,j+1} & \tilde{\Sigma}_{11}
\end{pmatrix} \succeq 0
\]
for all $j \in \{d+1, \ldots, d+k\}$, which can be satisfied if and only if $\tilde{\Sigma}_{j,j+1} = \pm \tilde{\Sigma}_{11} = \rho_{j,j+1}\tilde{\Sigma}_{11}$, since we must also have $|\tilde{\Sigma}_{j,j+1}| \leq \tilde{\Sigma}_{11}$ in order to have $\tilde{\Sigma} \succeq 0$. The fact that $\tilde{\Sigma}_{j,j+1} = \rho_{j,j+1} \tilde{\Sigma}_{11}$ for all $j \in \{d+1, \ldots, d+k\}$ implies that  $\operatorname{Var}(X_{j+1}- \Tilde{\Sigma}_{j+1,j+2} X_{j+2}) = 0$,
for all $j \in \{d, \cdots, d+k-1\}$, since $A_{\mathbb{S}_{d+k}}\tilde{\Sigma}$ is compatible. By induction, this gives $\operatorname{Var}(X_{d+1}- \prod_{j = d+1}^{d+k} \tilde{\Sigma}_{j,j+1} X_{1}) = 0$ by which 
\[
\tilde{\Sigma}_{1,d} = \frac{1}{\prod\limits_{j = d+1}^{d+k} \tilde{\Sigma}_{j,j+1}}\tilde{\Sigma}_{d,d+1} = \frac{1}{\prod\limits_{j = d+1}^{d+k} \rho_{j,j+1}}\tilde{\Sigma}_{d,d+1} = \tilde{\Sigma}_{d,d+1}. 
\]
Since $\tilde{\Sigma}_{1,d} = \tilde{\Sigma}_{d,d+1}$, we know that $\Sigma_{\mathbb{S}_{d+k}} - A_{\mathbb{S}_{d+k}}\tilde{\Sigma} \succeq_{\mathbb{S}_{d+k}} 0 $ implies that $\Sigma_{\mathbb{S}_d} - A_{\mathbb{S}_d}\Sigma \succeq_{\mathbb{S}} 
0$.

To show the reverse inequality, consider an optimal $\Sigma$ for $\Sigma_{\mathbb{S}_{d}}$ coming from the dual formulation above, and define \[ \tilde{\Sigma} :=
   \begin{pmatrix}
       \Sigma & B \\
       B^T & U
   \end{pmatrix},
   \] where \[
   U := \Sigma_{11}\begin{pmatrix}
    U_{11} & \cdots & U_{1k} \\
    \vdots & & \vdots \\
    U_{k1} & \cdots & U_{kk}
   \end{pmatrix}
   \]
   is such that $U = U^T$, $U_{ii} = 1$ for $i \in [k]$, $U_{i,i+1} = \rho_{d+i,d+i+1}$ for $i \in [k-1]$, $U_{1k} = U_{k1}$ is either $+1$ or $-1$ to make this $(k-1)$-cycle completable, and the other entries are again $+1$ or $-1$ to make the cycle consistent; and  \[B^T := \begin{pmatrix}
    B_{11} & \cdots & B_{1d} \\
    \vdots & & \vdots \\
    B_{k1} & \cdots & B_{kd}
   \end{pmatrix}
   \]
   is such that $B_{ij} = \Sigma_{1j} \cdot U_{i1}$ for $i \in [k], j \in [d]$. If such a $\tilde{\Sigma}$ is feasible for $\Sigma_{{\mathbb{S}}_{d+k}}$, then the result would follow from the fact that $R(\Sigma_{{\mathbb{S}}_{d+k}}) \leq 1-\operatorname{tr}(\tilde{\Sigma})/(d+k) = 1-\Sigma_{11} = R(\Sigma_{{\mathbb{S}}_{d}})$. The condition $\Sigma_{\mathbb{S}_{d+k}} - A_{\mathbb{S}_{d+k}}\tilde{\Sigma} \succeq_{\mathbb{S}_{d+k}} 0 $ is implied by $\Sigma_{\mathbb{S}_{d}} - A\Sigma \succeq_{\mathbb{S}_{d}} 0$, which is satisfied by hypothesis, and $(1-\Sigma_{11})\Sigma_{\{i, i+1\}} \succeq 0$ for $i \in \{d+1, \cdots, d+k-1 \}$, which is again satisfied since $\Sigma_{11} \in [0,1]$ and $\Sigma_{\{i, i+1\}} \succeq 0$. Moreover, being a symmetric block matrix, $\tilde{\Sigma}$ is positive semi-definite if and only if $\Sigma \succeq 0$, which is true by hypothesis, $U - B^T\Sigma^{\dagger}B \succeq 0$, and $(I - \Sigma\Sigma^{\dagger})B = 0$, where $\Sigma^\dagger$ is the Moore-Penrose inverse of $\Sigma$. As for the first of these last two conditions, observe that the $(k,h)$-th entry of $B^T\Sigma^{\dagger}B$ is given by \[
   (B^T\Sigma^{\dagger}B)_{kh} = U_{k1}\Sigma_1^T \Sigma^\dagger \Sigma_1 U_{h1} = U_{k1}U_{h1}\Sigma_1^T \Sigma^\dagger \Sigma_1 = U_{kh}\Sigma_1^T \Sigma^\dagger \Sigma_1,
   \]
   where $\Sigma_1$ is the first column of $\Sigma$. What is left to prove is to check that $\Sigma_1^T \Sigma^\dagger \Sigma_1 = \operatorname{tr}(\Sigma_1^T \Sigma^\dagger \Sigma_1) = \Sigma_{11}$ and, to this aim, we will use the limit characterisation of the pseudoinverse (see pag. 19 in \cite{albert72pseudo}), which allows writing $\Sigma^{\dagger}$ as $\lim_{\delta \rightarrow 0}\Sigma^T(\Sigma\Sigma^T + \delta^2 I)^{-1}$. With this in mind, and calling $\lambda_i$ the eigenvalues of $\Sigma$, and $\boldsymbol{v}_i$ the associated orthonormal eigenvectors,
   \begin{align*}
       \operatorname{tr}(\Sigma_1^T \Sigma^\dagger \Sigma_1) &= \operatorname{tr}(\Sigma^\dagger \Sigma_1 \Sigma_1^T) = \operatorname{tr}\left(\lim_{\delta \rightarrow 0}\Sigma^T(\Sigma\Sigma^T + \delta^2 I)^{-1} \Sigma_1 \Sigma_1^T\right) \\
       &= \lim_{\delta \rightarrow 0}\operatorname{tr}\left((\Sigma \Sigma^T + \delta^2 I)^{-1} \Sigma_1 \Sigma_1^T \Sigma^T\right) = \lim_{\delta \rightarrow 0}\operatorname{tr}\left(\sum_{i=1}^d \frac{1}{\lambda_i^2 +\delta^2}\boldsymbol{v}_i \boldsymbol{v}_i^T \Sigma_1 \Sigma_1^T \Sigma^T\right)\\
       &= \lim_{\delta \rightarrow 0}\sum_{i=1}^d \frac{1}{\lambda_i^2 +\delta^2} \operatorname{tr}\left(\boldsymbol{v}_i \boldsymbol{v}_i^T \Sigma_1 \Sigma_1^T \Sigma^T\right) = \lim_{\delta \rightarrow 0}\sum_{i,j=1}^d \frac{\lambda_j}{\lambda_i^2 +\delta^2} \operatorname{tr}\left( \boldsymbol{v}_j \boldsymbol{v}_j^T \boldsymbol{v}_i \boldsymbol{v}_i^T \Sigma_1 \Sigma_1^T \right) \\
       &= \lim_{\delta \rightarrow 0}\sum_{i=1}^d \frac{\lambda_i}{\lambda_i^2 +\delta^2} \| \boldsymbol{v}_i \boldsymbol{v}_i^T \Sigma_1\|_2^2 = \lim_{\delta \rightarrow 0}\sum_{i=1}^d \frac{\lambda_i}{\lambda_i^2 +\delta^2} \| \boldsymbol{v}_i \boldsymbol{v}_i^T \sum_{j=1}^d \lambda_j v_{j1} \boldsymbol{v}_j\|_2^2 \\
       &= \lim_{\delta \rightarrow 0}\sum_{i=1}^d \frac{\lambda_i^3}{\lambda_i^2 +\delta^2} \| v_{i1} \boldsymbol{v}_i\|_2^2 = \sum_{i=1}^d \lambda_i v_{i1}^2 = \Sigma_{11}.
   \end{align*}
   The other condition can be checked easily using again the limit characterisation of $\Sigma^\dagger$ and the spectral decomposition of $\Sigma$. This concludes the proof for the case $\prod_{j = d+1}^{d+k} \rho_{j,j+1} = +1$. On the other hand, if $\prod_{j = d+1}^{d+k} \rho_{j,j+1} = -1$, $R(\Sigma_{{\mathbb{S}}_{d+k}}) \geq R(\Sigma_{\mathbb{S}_{d}})$ follows after noticing that, if $\tilde{\Sigma}$ is optimal for $\Sigma_{{\mathbb{S}}_{d+k}}$, we must have $\tilde{\Sigma}_{d,d+1} = -\tilde{\Sigma}_{d,1}$. As for $R(\Sigma_{{\mathbb{S}}_{d+k}}) \leq R(\Sigma_{\mathbb{S}_{d}})$, the proof follows the exact same line as the one above, with the only exception that $B_{ij}$ should now be defined as $ - \Sigma_{1j} \cdot U_{i1}$ for $i \in [k], j \in [d]$.
\end{proof}

This reduction applies when the correlations associated to an edge belonging to the path from node $d+1$ to node $1$ are either $+1$ or $-1$. In this setting, we are allowed to identify node $1$ with node $d+1$ in such a way that the incompatibility measure of the $d$-cycle $\Sigma_{\mathbb{S}_{d}}$ is the same as the one of the original $(d+k)$-cycle $\Sigma_{\mathbb{S}_{d+k}}$. This is to be expected, as $\rho_{j,j+1}=\pm 1$ means that variables $j$ and $j+1$ can be identified, up to change in scale, and the dimensionality of the problem can be reduced. Clearly, this result is invariant under cyclic permutations of the nodes' labels.

We now give some explicit expressions for $R(\cdot)$ in special cases and discuss a case for which $R(\Sigma_{\mathbb{S}}) = 1$, meaning that $\Sigma_{\mathbb{S}}$ is maximally incompatible. It will be convenient for the rest of the subsection to reparametrise the correlations as $\rho_j = \cos\theta_j$, with $\theta_j \in [0, \pi]$.

\begin{eg}\label{ex:explicit_theta_zero}
    If there are $\theta_1,\theta_2 \in [0,\pi]$ such that $\theta_1 \geq \theta_2$ and
        \[
        \Sigma_{\mathbb{S}_3} = \Bigg\{\begin{pmatrix}
            1 & \cos\theta_1 \\
            \cos\theta_1 & 1
        \end{pmatrix},\begin{pmatrix}
            1 & \cos\theta_2 \\
            \cos\theta_2 & 1
        \end{pmatrix}, \begin{pmatrix}
            1 & 1 \\
            1 & 1
        \end{pmatrix} \Bigg\},
        \]
        then $R(\Sigma_{\mathbb{S}_3}) = (\cos\theta_2 - \cos\theta_1)/2$. In particular, setting $\theta_2=0$ we see that if 
        \[
        \Sigma_{\mathbb{S}_3} = \Bigg\{\begin{pmatrix}
            1 & \cos\theta_1 \\
            \cos\theta_1 & 1
        \end{pmatrix}, \begin{pmatrix}
            1 & 1 \\
            1 & 1
        \end{pmatrix}, \begin{pmatrix}
            1 & 1 \\
            1 & 1
        \end{pmatrix}\Bigg\},
        \]
        then $R(\Sigma_{\mathbb{S}_3}) = \sin^2(\theta_1/2)$. Moreover, assuming without loss of generality that at most one correlation is negative, as justified in Proposition \ref{prop:reduciton3} below, for a general $3$-cycle $\Sigma_{\mathbb{S}_3}$ we have $R(\Sigma_{\mathbb{S}_3}) = 1$ if and only if $$\Sigma_{{\mathbb{S}}_3} := \left\{\begin{pmatrix}
        1 & -1 \\ -1 & 1
    \end{pmatrix}, \begin{pmatrix}
        1 & 1 \\ 1 & 1
    \end{pmatrix}, \begin{pmatrix}
        1 & 1 \\ 1 & 1
    \end{pmatrix} \right\}.$$ 
    These results can be extended to a general $d$ using Proposition \ref{prop:reduction_2}.
\end{eg}

\begin{proof}[Proof of Example \ref{ex:explicit_theta_zero}]
    Start by considering a $3$-cycle. In the first case, the optimal $\Sigma$ of the dual representation \[
        R(\Sigma_{\mathbb{S}_3}) = 1 - \frac{1}{d} \sup\{ \mathrm{tr}(\Sigma) : \Sigma \in \mathcal{P}^*, \Sigma_\mathbb{S} - A \Sigma \succeq_\mathbb{S} 0, \Sigma_{11}=\Sigma_{22}=\Sigma_{33}  \}
        \]
        must be of the form \[
        \begin{pmatrix}
            \lambda & x & \lambda \\
            x & \lambda & y \\
            \lambda & y & \lambda
        \end{pmatrix}
        \]
        for some $\lambda \in [0,1]$ and some $x,y \in [-\lambda, \lambda]$ in order to satisfy $\Sigma_\mathbb{S} - A \Sigma \succeq_\mathbb{S} 0$. Furthermore, since $\operatorname{det}(\Sigma) = -\lambda(x-y)^2$, we must have $x=y$ in order to satisfy $\Sigma \succeq 0$. It follows that \begin{align*}
            R(\Sigma_{\mathbb{S}_3}) & = 1 - \sup\{\lambda \in [0,1]: 1- \lambda \geq \max\{|x-\rho_1|,|x-\rho_2| \}, \text{ with }  x \in [-\lambda, \lambda]\} \\
            & = \inf\{\epsilon \in [0,1] : \epsilon \geq \max\{|x-\rho_1|,|x-\rho_2| \}, \epsilon \leq 1 - |x| \} \\
            & = \inf\{\max\{|x-\rho_1|,|x-\rho_2| \} \in [0,1] : \max\{|x-\rho_1|,|x-\rho_2| \} \leq 1 - |x| \}
        \end{align*}
        which is equal to $(\cos\theta_2 - \cos\theta_1)/2$. As for the second case with $d=3$, setting $\rho_2=1$ in the above we see that if \[
        \Sigma_{\mathbb{S}_3} = \Bigg\{\begin{pmatrix}
            1 & \rho_1 \\
            \rho_1 & 1
        \end{pmatrix}, \begin{pmatrix}
            1 & 1 \\
            1 & 1
        \end{pmatrix}, \begin{pmatrix}
            1 & 1 \\
            1 & 1
        \end{pmatrix}\Bigg\},
        \]
        then $R(\Sigma_{\mathbb{S}_3}) = (1-\rho_1)/2 = (1-\cos\theta_1)/2 = \sin^2(\theta_1/2)$. Plugging in $\theta_1 = \pi$ gives the sufficiency part of the second statement. As for the necessity part, Proposition \ref{prop:KKT_cycle} (i) implies that it is necessary that $|\rho_i| = 1$ for all $i \in [d]$ for $R$ to be $1$.
\end{proof}

Related to the last claim of Example \ref{ex:explicit_theta_zero}, another important property for a $d$-cycle is that we can always assume without loss of generality that at most one $\theta_i$ is larger than $\pi/2$, which is equivalent to having at most one negative $\rho_i$, without changing the value of $R(\cdot)$. This is shown in the following result.
\begin{prop}\label{prop:reduciton3}
    Consider the $d$-cycle with $\mathbb{S}_d=\{\{1,2\},\ldots, \{d,1\}\}$ and $\Sigma_{\mathbb{S}_d}:= (\Sigma_{\{1,2\}}, \cdots, \Sigma_{ \{d,1\}})$, where the correlations $\rho_j = \cos\theta_j$ are uniquely determined by $0 \leq \theta_1, \ldots, \theta_d \leq \pi$. Then, there exists another sequence of angles $0 \leq \tilde{\theta}_1, \ldots, \tilde{\theta_d} \leq \pi$ with at most one $\tilde{\theta}_i$ larger than $\pi/2$ such that the corresponding $d$-cycle $\tilde{\Sigma}_{\mathbb{S}_d}:= (\tilde{\Sigma}_{\{1,2\}}, \cdots, \tilde{\Sigma}_{ \{d,1\}})$ satisfies $R(\tilde{\Sigma}_{\mathbb{S}_d}) = R(\Sigma_{\mathbb{S}_d})$.
\end{prop}

\begin{proof}[Proof of Proposition \ref{prop:reduciton3}]
    It is easy to see that we can always transform the original $d$-cycle into a new one where at most $\cos{\theta_1}$ is negative by changing some $X_j$ into $-X_j$. To see why, let $\theta = (\theta_1, \ldots, \theta_d)$ be such that $\theta_j = \mathbbm{1}\{\rho_{j, j+1} \geq 0\}$, and observe that, if $\theta_{j-1} = 0$ and $\theta_{j} = 1$, changing the sign of $X_j$ corresponds to switching $\theta_{j-1}$ with $\theta_j$. Hence, it is easy to see that we can switch signs to some variables in order to reach a configuration of $\theta$ in which all the zeros are at the beginning, and all the ones at the end. It is now sufficient to couple the zeros starting from the end, and switch sign to make it both one, to get $\theta = (\theta_1, \boldsymbol{1}_{d-1})$, where $\theta_1 = +1$ if the number of original $\rho_{j,j+1}$ is even, and zero otherwise. As a by-product, this also shows that we can always assume without loss of generality that at most $\cos\theta_1$ is negative. Now, let $\tilde{\Sigma}_{\mathbb{S}_d}$ be this new $d$-cycle: what we want to show is that $R(\tilde{\Sigma}_{\mathbb{S}_d}) = R(\Sigma_{\mathbb{S}_d})$, and, in order to do so, we will show that we can construct feasible $\tilde{X}_{\mathbb{S}}$ and $\tilde{\Sigma}$ for primal and dual problems of $\tilde{\Sigma}_{\mathbb{S}_d}$ which lead to the same target values, using the optimal $X_{\mathbb{S}}$ and $\Sigma$  for $\Sigma_{\mathbb{S}_d}$.
    Starting from the dual problem, let $M$ be a diagonal matrix such that $M_{jj} = -1$ if $X_j$ was replaced with $-X_j$, and $+1$ otherwise. Then, it is easy to see $\tilde{\Sigma} = M \Sigma M$ has the same trace as $\Sigma$, and it is feasible for $\tilde{\Sigma}_{\mathbb{S}_d}$: indeed, $\tilde{\Sigma} \succeq 0$ since it has the same spectrum as $\Sigma$, being similar matrices, and $\tilde{\Sigma}_{\mathbb{S}_d} - A\tilde{\Sigma} \succeq_{\mathbb{S}_d} 0$ because for every $j \in [d]$ we have
    \begin{align*}
    \tilde{\Sigma}_{\{j,j+1 \}} - &\tilde{\Sigma}_{|\{j,j+1 \}}  = \begin{pmatrix}
         1 & \tilde{\rho}_j \\
        \tilde{\rho}_j & 1
    \end{pmatrix} - \begin{pmatrix}
        \tilde{\Sigma}_{j,j}  & \tilde{\Sigma}_{j,j+1} \\
        \tilde{\Sigma}_{j,j+1} & \tilde{\Sigma}_{j+1,j+1}
    \end{pmatrix} \\
    & = \begin{pmatrix}
         1 & M_{j,j}M_{j+1,j+1}\rho_j \\
         M_{j,j}M_{j+1,j+1}\rho_j  & 1
    \end{pmatrix} - \begin{pmatrix}
         \Sigma_{j,j} & M_{j,j}M_{j+1,j+1}\Sigma_{j,j+1} \\
         M_{j,j}M_{j+1,j+1}\Sigma_{j,j+1}  & \Sigma_{j+1,j+1}
    \end{pmatrix} \\
    & = \begin{pmatrix}
         1 - \Sigma_{j,j} & M_{j,j}M_{j+1,j+1}(\rho_j - \Sigma_j) \\
         M_{j,j}M_{j+1,j+1}(\rho_j - \Sigma_j)  & 1 - \Sigma_{j+1,j+1}
    \end{pmatrix} \succeq 0,
    \end{align*}
    since $|M_{j,j}M_{j+1,j+1}(\rho_j - \Sigma_j)| = |\rho_j - \Sigma_j| \leq 1 - \Sigma_{j,j} =  1 - \Sigma_{j+1,j+1}$, due to the fact that $\Sigma$ is feasible for $\Sigma_\mathbb{S}$.
    As for the primal problem, it is sufficient to define $\tilde{X}_{\mathbb{S}} = AM \cdot X_{\mathbb{S}} \cdot AM$, where $\cdot$ acts pointwise, which essentially consists in changing the signs of the off-diagonal entries of $X_{\mathbb{S}}$ according to $M$. Let \[X_{\{j,j+1 \}} = \begin{pmatrix}
       x_{j, 11} & x_{j,12} \\
       x_{j,21} & x_{j,22}
    \end{pmatrix} \text{ and }
    \tilde{X}_{\{j,j+1 \}} = \begin{pmatrix}
       \tilde{x}_{j, 11} &\tilde{x}_{j,12} \\
       \tilde{x}_{j,21} &\tilde{x}_{j,22}
    \end{pmatrix} = \begin{pmatrix}
        x_{j, 11} &  M_{j,j}M_{j+1,j+1} x_{j, 12}\\
        M_{j,j}M_{j+1,j+1} x_{j, 21} & x_{j, 22}
    \end{pmatrix},
    \]
    for all $j \in [d]$. It is easy to show that $\tilde{X}_{\mathbb{S}}$ is feasible, and clearly leads to $
    \langle \tilde{X}_{\mathbb{S}}, \tilde{\Sigma}_{\mathbb{S}} \rangle_{\mathbb{S}} =  \langle X_{\mathbb{S}}, \Sigma_{\mathbb{S}} \rangle_{\mathbb{S}}
    $ since, for a generic pattern $j \in [d]$, we have \begin{align*}
        \langle \tilde{X}_{\{j,j+1 \}}, \tilde{\Sigma}_{\{j,j+1 \}} \rangle & = 
    \langle \begin{pmatrix}
       \tilde{x}_{j, 11} &\tilde{x}_{j,12} \\
       \tilde{x}_{j,21} &\tilde{x}_{j,22}
    \end{pmatrix}, \begin{pmatrix}
         1 & \tilde{\rho}_j \\
        \tilde{\rho}_j & 1
    \end{pmatrix}\rangle \\
    & = \langle \begin{pmatrix}
        x_{j, 11} &  M_{j,j}M_{j+1,j+1} x_{j, 12}\\
        M_{j,j}M_{j+1,j+1} x_{j, 21} & x_{j, 22}
    \end{pmatrix}, \begin{pmatrix}
         1 & M_{j,j}M_{j+1,j+1}\rho_j \\
         M_{j,j}M_{j+1,j+1}\rho_j  & 1
    \end{pmatrix}\rangle \\
    & = x_{j, 11} + x_{j, 22} + M_{j,j}^2M_{j+1,j+1}^2 x_{j, 12}\rho_j + M_{j,j}^2M_{j+1,j+1}^2 x_{j, 21}\rho_j \\
    & = x_{j, 11} + x_{j, 22} + x_{j, 12} \rho_j + x_{j, 21} \rho_j \\
    & = \langle \begin{pmatrix}
        x_{j, 11} & x_{j, 12} \\
        x_{j, 21} & x_{j, 22}
    \end{pmatrix}, \begin{pmatrix}
         1 & \rho_j \\
        \rho_j & 1
    \end{pmatrix}\rangle = \langle X_{\{j,j+1 \}}, \Sigma_{\{j,j+1 \}} \rangle.
    \end{align*}
    This completes the proof.
\end{proof} 

The last result we present on the $d$-cycle gives an explicit lower bound for $R$ in the case that $\Sigma_{\mathbb{S}_d}$ is incompatible. This is related to the results of \cite{barrett1993real} characterising exactly when the partial correlation matrix
\[\Sigma_{partial} = 
	\begin{pmatrix} 1 & \cos \theta_1 & \ast & \cdots & \cos \theta_d \\ \cos \theta_1 & 1 & \cos \theta_2 & \cdots & \ast \\ \ast & \cos \theta_2 & 1 & \ldots & \ast \\ \vdots & \vdots & \vdots & \ddots &\vdots \\ \cos \theta_d & \ast & \ast & \cdots & 1 \end{pmatrix}
\]
has a positive semi-definite completion. \cite{barrett1993real} shows that this is the case if and only if
\begin{align}
    \label{eq:barrett_charcterisation}
    \sum\limits_{j \in K} \theta_j \leq (|K| -1 ) \pi + \sum\limits_{j \not\in K} \theta_j
\end{align}
for all $K \subseteq [d]$ with $|K|$ odd. Remarkably, this shows that the parametrisation $\rho_{j,j+1}=\cos \theta_j$ allows us to characterise the feasibility of positive semi-definite matrix completion in terms of a finite number of linear inequalities. If we know that $0 \leq \theta_d \leq \theta_{d-1} \leq \ldots \leq \theta_1 \leq \pi$ then this reduces to checking
\[
	\sum_{j=1}^k \theta_j \leq (k-1) \pi + \sum_{j=k+1}^d \theta_j
\]
for all odd $k \in [d]$. Furthermore, if $0 \leq \theta_1, \ldots, \theta_d \leq \pi$ with at most one $\theta_j$ larger than $\pi/2$, then $\Sigma_{partial}$ has a positive semi-definite completion if and only if \[
2\max_{j \in [d]}\theta_j \leq \sum_{j=1}^d \theta_j.
\]
Proposition \ref{prop:reduciton3} shows that we can always work under this setting, so that the problem of whether $\Sigma_{partial}$ has a PSD completion or not is determined by one condition only, namely $2\max_{j \in [d]}\theta_j \leq \sum_{j=1}^d \theta_j$. This is a novel contribution \textit{per se}, since it is not present in \cite{barrett1993real}. 

We now show that, provided not too many of our input matrices are close to being singular, $R(\Sigma_{\mathbb{S}_d})$ can be bounded below by a finite maximum of linear functionals that is zero if and only if $\Sigma_{\mathbb{S}_d}$ is compatible. This lower bound constitutes another sanity check for our measure $R(\cdot)$, since the quantities appearing in the lower bound are a natural quantitive version of the qualitative conditions given in \cite{barrett1993real} to check whether the partial matrix $\Sigma_{partial}$ defined above admits a PSD completion. 

\begin{prop}\label{ex:cycle_LB}

    Consider the $d$-cycle with $\mathbb{S}=\mathbb{S}_d=\{\{1,2\},\ldots, \{d,1\}\}$ and suppose that
    \[
        \Sigma_{\{j,j+1\}} = \begin{pmatrix} 1 & \cos\theta_j \\ \cos\theta_j & 1 \end{pmatrix}.
    \]
   Assume further that there exist $c>0$ and two indices $k,j \in [d]$ such that $1-\rho_j^2 \geq c, 1-\rho_k^2 \geq c$, so that $\Sigma_{\{j,j+1\}}$ and $\Sigma_{\{k,k+1\}}$ are bounded away from singularity. Then, whenever $\Sigma_{\mathbb{S}_d}$ is incompatible, we have
    \[
        R(\Sigma_{\mathbb{S}_d}) \geq c' \mathop{\max\limits_{K \subseteq [d]}}\limits_{|K| \text{ odd}}\left(\sum\limits_{i \in K}\theta_i - (|K| -1) \pi - \sum\limits_{i \in K^c}\theta_i\right),
    \]
    where $c'>0$ depends only on $c$. 
\end{prop}

\begin{proof}[Proof of Proposition \ref{ex:cycle_LB}]
            We will prove the statement by induction, with base cases $d=3$, and $d=4$:
    \begin{itemize}
        \item[$d=3$] Suppose without loss of generality that $\theta_1, \theta_2$ are bounded away from singularity. Also, assume without loss of generality that $\theta_2, \theta_3 \leq \pi/2$, using Proposition \ref{prop:reduciton3}, so that incompatibility means $\theta_1 > \theta_2 + \theta_3$. We will prove the base case \[
        R(\Sigma_{\mathbb{S}_3}) \gtrsim \theta_2 - \theta_1 -\theta_3.
        \]
       by showing that \[
        R(\Sigma_{\mathbb{S}_3}) \geq \frac{\theta_1-\theta_2-\theta_3}{\theta_1-\theta_2}\frac{\cos\theta_2-\cos\theta_1}{2},
        \] 
        and since $\cos\theta_2-\cos\theta_1 \gtrsim_c \theta_1 - \theta_2$ being bounded away from singularity, the result would follow.  Now, fix arbitrary $\theta_1, \theta_2$ satisfying the hypothesis of the statement, and suppose $\theta_1-\theta_2 \leq \pi/2$. Observe that for $\theta_3 = 0$ and $\theta_3 = \theta_1-\theta_2$ the lower bound is satisfied with equality sign due to Example \ref{ex:all_but_one_separation} and Barrett's characterisation (\ref{eq:barrett_charcterisation}), respectively. Now, call
\[
h = \frac{\theta_1-\theta_2-\theta_3}{\theta_1-\theta_2},
\]
and observe that the thesis is equivalent to \[
\lambda^* = 1-R(\Sigma_{\mathbb{S}_d}) \leq 1 -   \frac{h}{2}(\cos\theta_2-\cos\theta_1),
\]
Now, thanks to the KKT representation of the optimal $\lambda^*$, in order to have $\lambda^* > 1 - h(\cos\theta_2-\cos\theta_1)/2$ we must have 
\begin{align*}
    \begin{cases}
    \cos\theta_1 < \cos\varphi_1^* < \frac{\cos\theta_1 + h(\cos\theta_2-\cos\theta_1)/2}{1-h(\cos\theta_2-\cos\theta_1)/2} \\
    \\
    \cos\theta_2 > \cos\varphi_2^* > \frac{\cos\theta_2 - h(\cos\theta_2-\cos\theta_1)/2}{1-h(\cos\theta_2-\cos\theta_1)/2}\\
    \\
    \cos\theta_3 > \cos\varphi_3^* > \frac{\cos\theta_3 - h(\cos\theta_2-\cos\theta_1)/2}{1-h(\cos\theta_2-\cos\theta_1)/2},
    \end{cases}
\end{align*}
with $\varphi^*_1 = \varphi^*_2 + \varphi^*_3$ due to Proposition \ref{prop:KKT_cycle} (iii). We see numerically that this system of inequalities can never be satisfied for $\theta_3 \in (0, \theta_1-\theta_2)$. Finally, taking into account all the possible ways in which a generic $3$-cycle can be reduced to a $3$-cycle with at most one negative correlation, as stated in Proposition \ref{prop:reduciton3}, we get \[
        R(\Sigma_{\mathbb{S}_3}) \gtrsim \max(\theta_1-\theta_2-\theta_3, \theta_2-\theta_1-\theta_3,\theta_3-\theta_1-\theta_2,\theta_1+\theta_2+\theta_3 - 2 \pi).
        \]
     
    \item[$d=4$] Suppose without loss of generality that one of the two angles bounded away from singularity is $\theta_1$, with $\theta_1 > \theta_2 + \theta_3 + \theta_4$, and $\theta_2, \theta_3, \theta_4 \in [0, \pi/2]$. As shown in Figure \ref{fig:barrettLB}, there are two possible cases: the first one (on the left) is when the two angles bounded away from singularity are adjacent, and the second one (on the right) when they are opposite to each other. We will use the following lemma: 
\begin{lemma}\label{prop:reduction_1}
    Consider the $d$-cycle with $\mathbb{S}_d=\{\{1,2\},\ldots, \{d,1\}\}$ and $\Sigma_{\mathbb{S}_d}:= (\Sigma_{\{1,2\}}, \cdots, \Sigma_{ \{d,1\}})$. Then, for every optimal $\Sigma$ of the dual problem, i.e. $\Sigma_{\mathbb{S}_d} = \lambda^* A\Sigma + (1-\lambda^*)\Sigma_{\mathbb{S}_d}'$, and for every $d \geq 4$,
    $$R(\Sigma_{\mathbb{S}_d}) \geq R(B_{\mathbb{S}_{d-1}}(\phi)) + R(E_{\mathbb{S}_3}(\phi)), \quad \forall \phi \in [\lambda^*\Sigma_{1,d-1}-R, \lambda^*\Sigma_{1,d-1}+R],$$
    where $B_{\mathbb{S}_{d-1}}(\phi)= (\Sigma_{\{1,2\}}, \cdots, \Sigma_{ \{d-2,d-1\}}, \Sigma_{ \{d-1,1\}}(\phi))$, $E_{\mathbb{S}_3}(\phi)= (\Sigma_{ \{d-1, d\}}, \Sigma_{ \{d,1\}}, \Sigma_{ \{d-1,1\}}(\phi))$ and $\Sigma_{ \{d-1,1\}}(\phi)$ is the $2 \times 2$ correlation matrix with off-diagonal entries equal to $\phi$. 
\end{lemma}

\begin{proof}[Proof of Lemma \ref{prop:reduction_1}]
 Suppose without loss of generality that $\theta_d \geq \theta_{d-1}$, and that $\theta_1 = \max_{i \in [d]} \theta_i$, with $\theta_2, \ldots, \theta_d \leq \pi/2$. Let $R \equiv R(\Sigma_{\mathbb{S}_d}) $, and let \[
 \Sigma_{\mathbb{S}_d} = (1-R)A\Sigma + R\Sigma'_{\mathbb{S}}  = \lambda^*A\Sigma + (1-\lambda^*)\Sigma'_{\mathbb{S}}
 \] 
 be a (not necessarily unique) dual representation of $\Sigma_{\mathbb{S}_d}$, and denote by $\Sigma_{1,d-1}$ the entry $(1, d-1)$ of $\Sigma$. We will prove the statement in three steps:
    \begin{enumerate}
        \item $R(B_{\mathbb{S}_{d-1}}(\lambda^*\Sigma_{1,d-1}+R)) \leq R(\Sigma_{\mathbb{S}_d})$ and $R(E_{\mathbb{S}_{3}}(\lambda^*\Sigma_{1,d-1}+R)) = 0$,
        \item $R(E_{\mathbb{S}_{3}}(\lambda^*\Sigma_{1,d-1}-R)) \leq R(\Sigma_{\mathbb{S}_d})$ and $R(B_{\mathbb{S}_{d-1}}(\lambda^*\Sigma_{1,d-1}-R)) = 0$,
        \item $\Xi(\phi) := R(B_{\mathbb{S}_{d-1}}(\phi)) + R(E_{\mathbb{S}_3}(\phi))$ is convex for all $\phi \in [-1,1]$.
        
    \end{enumerate}
    \begin{enumerate}
        \item As for the fact that $R(E_{\mathbb{S}_{3}}(\lambda^*\Sigma_{1,d-1}+R)) = 0$ observe that \begin{align*}
A^*E_{\mathbb{S}_{3}}(\lambda^*\Sigma_{1,d-1}+R) - I_3  &= \begin{pmatrix}
    1 & \lambda^*\Sigma_{1,d-1} + R & \rho_d \\
    \lambda^*\Sigma_{1,d-1} + R & 1 & \rho_{d-1} \\
    \rho_{d} & \rho_{d-1} & 1  
\end{pmatrix} = \\
&=\begin{pmatrix}
    1 & \lambda^*\Sigma_{1,d-1} + R & \lambda^*\Sigma_{1,d} + R \\
    \lambda^*\Sigma_{1,d-1} + R & 1 & \lambda^*\Sigma_{d-1,d} + R \\
    \lambda^*\Sigma_{1,d} + R & \lambda^*\Sigma_{d-1,d} + R & 1  
\end{pmatrix} \\
& \quad \quad = \lambda^* \begin{pmatrix}
    1 & \Sigma_{1,d-1} & \Sigma_{1,d} \\
    \Sigma_{1,d-1} & 1 & \Sigma_{d-1,d} \\
    \Sigma_{1,d} & \Sigma_{d-1,d} & 1  
\end{pmatrix} + (1-\lambda^*)\begin{pmatrix}
    1 & 1 & 1 \\
    1 & 1 & 1 \\
    1 & 1 & 1
\end{pmatrix},
\end{align*}
where the second equality follows from  the optimal choice of signs given in Proposition \ref{prop:KKT_cycle}  (iii) under the hypothesis $\theta_1 = \max_{i \in [d]} \theta_i$, with $\theta_2, \ldots, \theta_d \leq \pi/2$. This implies $R(E_{\mathbb{S}_{3}}(\lambda^*\Sigma_{1,d-1}+R)) = 0$ since $A^*E_{\mathbb{S}_{3}}(\lambda^*\Sigma_{1,d-1}+R) - I_3$, which is the $3 \times 3$ correlation matrix whose $2 \times 2$ marginals are precisely those in $E_{\mathbb{S}_{3}}(\lambda^*\Sigma_{1,d-1}+R)$, is PSD being the sum of two PSD matrices. As for 
$R(B_{\mathbb{S}_{d-1}}(\lambda^*\Sigma_{1,d-1}+R)) \leq R(\Sigma_{\mathbb{S}_d})$, observe that, if $\Sigma_{\mathbb{S}_d} = \lambda^* A\Sigma + (1-\lambda^*)\Sigma_{\mathbb{S}_d}'$, then 
\[
B_{\mathbb{S}_{d-1}}(\lambda^*\Sigma_{1,d-1}+R) = \lambda^* A\Sigma_{|(-d)} + (1-\lambda^*)\Sigma_{B_{\mathbb{S}_{d-1}}}'',
\]
where \[
\Sigma_{B_{\mathbb{S}_{d-1}}}'' = (\Sigma_{\{1,2 \}}', \ldots, \Sigma_{\{d-1,d-2\}}', \boldsymbol{1}_2 \boldsymbol{1}_2^T),
\]
which is maximally incompatible. To see why, observe that $\Sigma'_{\mathbb{S}_d}$ is maximally incompatible by definition of the dual representation, and since $\theta_1 = \max_{i \in [d]} \theta_i$, with $\theta_2, \ldots, \theta_d \leq \pi/2$, Proposition \ref{prop:KKT_cycle}  (iii) ensures that 
\[
\Sigma'_{\mathbb{S}_d} = (-\boldsymbol{1}_2 \boldsymbol{1}_2^T, +\boldsymbol{1}_2 \boldsymbol{1}_2^T, \ldots, +\boldsymbol{1}_2 \boldsymbol{1}_2^T),
\]
which leads to \[
\Sigma_{B_{\mathbb{S}_{d-1}}}'' =  (-\boldsymbol{1}_2 \boldsymbol{1}_2^T, +\boldsymbol{1}_2 \boldsymbol{1}_2^T, \ldots, +\boldsymbol{1}_2 \boldsymbol{1}_2^T).
\]
This shows that $\Sigma_{|(-d)}$ is feasible for $B_{\mathbb{S}_{d-1}}(\lambda^*\Sigma_{1,d-1}+R)$, and implies that $R(B_{\mathbb{S}_{d-1}}(\lambda^*\Sigma_{1,d-1}+R)) \leq R(\Sigma_{\mathbb{S}_d})$.
\item The arguments in the proof above can be followed \textit{mutatis mutandis} to show that $R(E_{\mathbb{S}_{3}}(\lambda^*\Sigma_{1,d-1}-R)) \leq R(\Sigma_{\mathbb{S}_d})$ and $R(B_{\mathbb{S}_{d-1}}(\lambda^*\Sigma_{1,d-1}-R)) = 0$.
\item  In order to show that $R(\Sigma_{\mathbb{S}_d}) \geq R(B_{\mathbb{S}_{d-1}}(\phi)) + R(E_{\mathbb{S}_3}(\phi)), \forall \phi \in I,$ we will make use of the fact that $R$ is convex and continuous, as stated in Proposition \ref{prop:properties} (i) (ii), i.e. \[
R\left(\mu\Sigma_\mathbb{S}^{(1)} + (1-\mu)\Sigma_\mathbb{S}^{(2)}\right) \leq \mu R\left(\Sigma_\mathbb{S}^{(1)}\right) + (1-\mu) R\left(\Sigma_\mathbb{S}^{(2)}\right), \text{ for all } \mu \in [0,1].
\]
Now, define \[
\Xi(\phi) = R(B_{\mathbb{S}_{d-1}}(\phi)) + R(E_{\mathbb{S}_3}(\phi)), \text{ for all } \phi \in I = [-1,1].
\]
It is easy to see that $\Xi(\phi)$ is convex in $I$ since, for all $\phi_1, \phi_2 \in I$, for all $\mu \in [0,1]$,
\begin{align*}
    \Xi(\mu\phi_1 + (1-\mu)\phi_2) & =  R(B_{\mathbb{S}_{d-1}}(\mu\phi_1 + (1-\mu)\phi_2) + R(E_{\mathbb{S}_3}(\mu\phi_1 + (1-\mu)\phi_2)) \\
    & = R(\mu B_{\mathbb{S}_{d-1}}(\phi_1) + (1-\mu) B_{\mathbb{S}_{d-1}}(\phi_2)) + R(\mu E_{\mathbb{S}_{3}}(\phi_1) + (1-\mu) E_{\mathbb{S}_{3}}(\phi_2)) \\
    & \leq \mu R(B_{\mathbb{S}_{d-1}}(\phi_1)) + (1-\mu) R(B_{\mathbb{S}_{d-1}}(\phi_2)) + \mu R(E_{\mathbb{S}_{3}}(\phi_1)) + (1-\mu) R(E_{\mathbb{S}_{3}}(\phi_2)) \\
    & = \mu\Xi(\phi_1) + (1-\mu)\Xi(\phi_2).
\end{align*}
This, implies that, for all $\mu \in [0,1]$, \begin{align*}
    R & \geq \mu R(E_{\mathbb{S}_{d-1}}(\lambda^*\Sigma_{1,d-1}-R)) + (1-\mu)R(B_{\mathbb{S}_{d-1}}(\lambda^*\Sigma_{1,d-1}-R)) \\
    & = \mu\Xi(\lambda^*\Sigma_{1,d-1}-R) + (1-\mu)\Xi(\lambda^*\Sigma_{1,d-1}+R) \geq \Xi(\lambda^*\Sigma_{1,d-1}+1 - 2\mu R) \\
    & =: \Xi(\phi) = R(B_{\mathbb{S}_{d-1}}(\phi)) + R(E_{\mathbb{S}_3}(\phi)),
\end{align*}
for all $\phi \in [\lambda^*\Sigma_{1,d-1}-R, \lambda^*\Sigma_{1,d-1}+R]$, as claimed. For general angles $(\theta_1, \ldots, \theta_d)$, it is sufficient to perform the transformation outlined in Proposition \ref{prop:reduciton3}, find $\phi$ and $I$ as above, and perform the inverse transformation. 
\end{enumerate}
\begin{figure}
    \centering
    \includegraphics[scale=0.4]{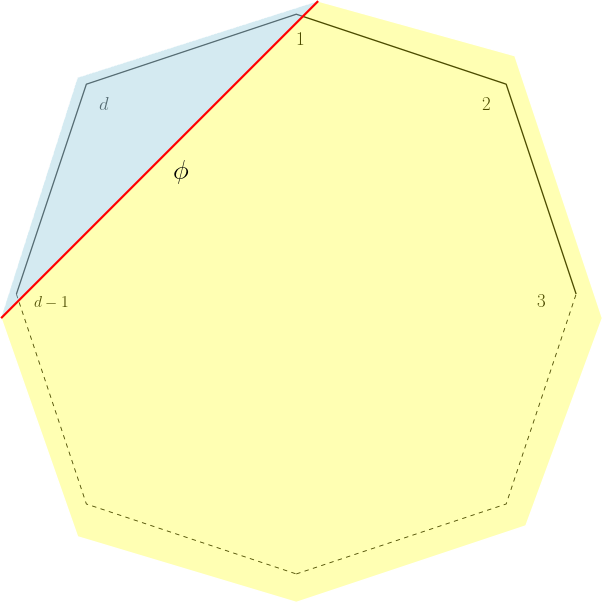}
    \caption{Illustration of Proposition \ref{prop:reduction_1}. We split the original $d$-cycle into two smaller cycles, adding the extra edge $\{1, d-1\}$ with associated correlation $\phi$. We end up with a $(d-1)$-cycle $B_{\mathbb{S}_{d-1}}(\phi)$ in yellow, and a $3$-cycle $E_{\mathbb{S}_{3}}(\phi)$ in blue, such that $R(\Sigma_{\mathbb{S}_d}) \geq R(B_{\mathbb{S}_{d-1}}(\phi)) + R(E_{\mathbb{S}_3}(\phi))$ for all $\phi \in [\lambda^*\Sigma_{1,d-1}-R, \lambda^*\Sigma_{1,d-1}+R]$.}
    \label{fig:red1}
\end{figure}
As we can see from Figure \ref{fig:red1}, this reduction corresponds to adding an edge in correspondence to $\{1, d-1\}$, so that the $d$-cycle $\Sigma_{{\mathbb{S}}_d}$ is divided into two smaller cycles, $B_{\mathbb{S}_d}(\phi)$ of dimension $d-1$, and $E_{\mathbb{S}_d}(\phi)$ of dimension 3. The result ensures the possibility of adding a correlation $\rho_{1, d-1} = \phi$ for the edge $\{1,d-1\}$ to make $B_{\mathbb{S}_d}(\phi)$ and $E_{\mathbb{S}_d}(\phi)$ maximally compatible, or better, at least as compatible as the original $d$-cycle, since $R(\Sigma_{\mathbb{S}_d}) \geq R(B_{\mathbb{S}_d}(\phi)) + R(E_{\mathbb{S}_d}(\phi))$.
\end{proof}

    In the first case, suppose we add an edge between $(2,4)$ with correlation $\cos(\theta_3+\theta_4)$. We first show that this is a valid choice of $\phi$ to invoke Proposition \ref{prop:reduction_1}. In this regard, observe that $R(E_{\mathbb{S}_3}(\phi)) = 0$ for all $\phi \in [\cos(\theta_{d-1} + \theta_d), \cos(\theta_{d-1} - \theta_d)]$, hence, since we proved $R(E_{\mathbb{S}_{d-1}}(\lambda^*\Sigma_{1,d-1}+R)) = 0$ in the proof of the lemma above, we must have $\cos(\theta_{d-1} + \theta_d) \leq \lambda^*\Sigma_{1,d-1}+R$. Similarly, $\cos(\theta_1 - \sum_{i = 2}^{d-2} \theta_i) \geq \lambda^*\Sigma_{1,d-1}-R$. This, together with the fact that $\cos(\theta_1 - \sum_{i = 2}^{d-2} \theta_i) \leq \cos(\theta_{d-1} + \theta_d)$ since $\theta_1 > \sum_{i = 2}^{d} \theta_i$, allows concluding that $\lambda^*\Sigma_{1,d-1}-R \leq \cos(\theta_1 - \sum_{i = 2}^{d-2} \theta_i) < \cos(\theta_{d-1} + \theta_d) \leq \lambda^*\Sigma_{1,d-1}+R$. Now, Proposition \ref{prop:reduction_1} ensures that \[R(\theta_1, \theta_2, \theta_3, \theta_4) \geq R(\theta_1, \theta_2, \theta_3 + \theta_4), \]
    and since $\theta_1, \theta_2$ are bounded away from singularity, we can employ the lower bound we found for $d=3$, and conclude
    \[
    R(\Sigma_{\mathbb{S}_4}) \geq \frac{\theta_1 - \theta_2 - (\theta_3+\theta_4)}{\theta_1 - \theta_2}\frac{\cos\theta_2-\cos\theta_1}{2},
    \]
    which gives the desired result. In the second case, we can proceed in the same way as before, and get
    \[
    R(\Sigma_{\mathbb{S}_4}) \geq \frac{\theta_1 - \theta_2 - (\theta_3+\theta_4)}{\theta_1 - (\theta_3+\theta_4)}\frac{\cos(\theta_3+\theta_4)-\cos\theta_1}{2}.
    \]
    Now, if $\sin^2(\theta_3+\theta_4) \geq c$ we are done, otherwise, $\theta_4$ must be bounded away from singularity. Indeed, since $\theta_3, \theta_4 \in [0, \pi/2]$, and $\sin^2(\theta_3) \geq c$ by hypothesis, in order to have $\sin^2(\theta_3+\theta_4) < c$ we must have $\sin^2\theta_4 \geq 1-c$. Now, since we can assume that $c$ is small enough, say $c \leq 1/2$, we conclude $\sin^2\theta_4 \geq 1-c > c$. This implies that $\theta_4$ is bounded away from singularity, and since it is adjacent to $\theta_1$, we can proceed as in the first case to get the desired result. 
    \begin{figure}
        \centering
        \includegraphics[scale=0.3]{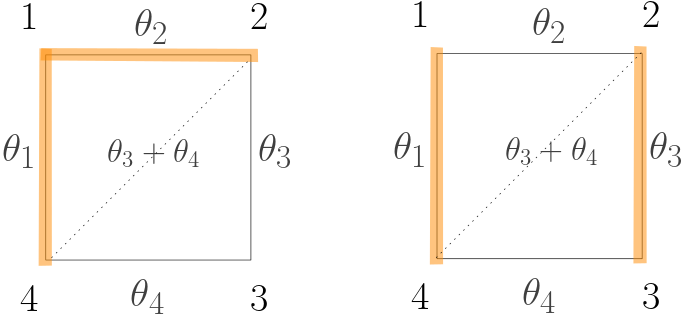}
        \caption{The two possible configurations of the two angles bounded away from singularity when $d=4$. On the left, the two angles are adjacent, while on the right they are opposite to each other.}
        \label{fig:barrettLB}
    \end{figure}
    \item[$d \geq 5$] Suppose again without loss of generality that $\theta_1$ is bounded away from singularity, and call $\theta_j$ the other one. Now, since $d \geq 5$, we can find $k \neq 1,j$ such that $\theta_k, \theta_{k+1}$ are not necessarily assumed to be bounded away from singularity. Then, proceeding as before, thanks to Proposition \ref{prop:reduction_1}, we have \[R(\theta_1, \ldots, \theta_d) \geq R(\theta_1, \ldots, \theta_{k-1}, \theta_{k}+\theta_{k+1}, \theta_{k+2}, \ldots, \theta_d), \]
    so that the induction step gives immediately that \begin{align*}
        R(\theta_1, \ldots, \theta_d) &\geq R(\theta_1, \ldots, \theta_{k-1}, \theta_{k}+\theta_{k+1}, \theta_{k+2}, \ldots, \theta_d) \\
        & \geq c' \left(\theta_1 -  (\theta_k + \theta_{k+1}) - \sum\limits_{i \neq 1,j,k,k+1} \theta_i\right) = c' \left(\theta_1  - \sum\limits_{i = 2}^d \theta_i\right),
    \end{align*}
    where $c'$ is a constant depending on $c$ only. Finally, taking into account all the possible ways in which a generic $d$-cycle can be reduced to a $d$-cycle with at most one negative correlation, as stated in Proposition \ref{prop:reduciton3}, we get \[
        R(\Sigma_{\mathbb{S}_d}) \geq c' \mathop{\max\limits_{K \subseteq [d]}}\limits_{|K| \text{ odd}}\left(\sum\limits_{i \in K}\theta_i - (|K| -1) \pi - \sum\limits_{i \in K^c}\theta_i\right),
    \]
    where $c'>0$ depends only on $c$.
    \end{itemize}
    \end{proof}
    
First, observe that this lower bound reduces to
\[
c' \left(\theta_1 - \sum_{i = 2}^d \theta_i \right)_+,
\] 
in the case that $\theta_1 = \max_{j \in [d]} \theta_j$ and $\theta_2,\ldots,\theta_d \leq \pi/2$, which we have already argued that we may assume without loss of generality. Furthermore, as a sanity check, the simple explicit expressions found in Example \ref{ex:all_but_one_separation}, in which we have seen that
        \[
        R\left(\Bigg\{\begin{pmatrix}
            1 & \cos\theta_1 \\
            \cos\theta_1 & 1
        \end{pmatrix},\begin{pmatrix}
            1 & \cos\theta_2 \\
            \cos\theta_2 & 1
        \end{pmatrix}, \begin{pmatrix}
            1 & 1 \\
            1 & 1
        \end{pmatrix}, \ldots, \begin{pmatrix}
            1 & 1 \\
            1 & 1
        \end{pmatrix}\Bigg\}\right) = (\cos\theta_2 - \cos\theta_1)/2,
        \]
        is in accordance with Proposition \ref{ex:cycle_LB}, since $\cos\theta_2 - \cos\theta_1 \gtrsim_c \theta_1 - \theta_2$ when $\theta_1, \theta_2$ are bounded away from $\{0, \pi \}$. \\

We now analyse two further examples, where the missingness patterns $\mathbb{S}$ is more complex. These examples are interesting \textit{per se}, but we have decided not to include them in the main body because they would have disrupted the flow of the presentation. We start from the case where we observe all possible patterns of cardinality $d-1$, and nothing else.

\begin{eg}\label{ex:all_except_one}
    Consider the set of patterns $\mathbb{S} = \{S_{(-1)}, \cdots, S_{(-d)}\}$, with $d \geq 2$, where $S_{(-i)} = \{1, \ldots, i-1, i+1, \ldots, d\}$. We show how $R(\Sigma_{\mathbb{S}})$ can be lower-bounded by the maximal inconsistency, or, more precisely, \[
    R(\Sigma_{\mathbb{S}}) \geq \frac{1}{2} \max_{i > j} \mathop{\max_{k > h}}_{k,h \neq i,j} |\rho_{ij}^{(k)} - \rho_{ij}^{(h)}| =: \Theta,
    \]
    where $\rho_{ij}^{(k)}$ is the correlation between $X_i$ and $X_j$ for the pattern $S_{(-k)}$, for $k \in [d] \setminus \{i,j\}$. 
    \end{eg}

    \begin{proof}[Proof of Example \ref{ex:all_except_one}]
    In order to prove the statement, suppose the maximum is $|\rho_{ij}^{(k)} - \rho_{ij}^{(h)}|$, and consider $X_{\mathbb{S}} = dY_{\mathbb{S}} - X_{\mathbb{S}}^{(0)}$, where $X_{\mathbb{S}}^{(0)} = \frac{1}{d-1}(I_{d-1}, \ldots, I_{d-1})$ and \[
    Y_{\mathbb{S}} = (0, \ldots, \underbrace{A_1}_{h}, 0, \ldots, 0, \underbrace{A_2}_{k}, 0, \ldots, 0),
    \]
    with \[
    (A_1)_{\tilde{i}, \tilde{j}} = \begin{cases}
        1/4 \quad \text{ if } (\tilde{i}, \tilde{j}) \in \{(i,i), (j,j), (i,j), (j,i)\} \\
        0 \quad \text{ otherwise},
    \end{cases}
    \]
    and 
    \[
    (A_2)_{\tilde{i}, \tilde{j}} = \begin{cases}
        1/4 \quad \text{ if } (\tilde{i}, \tilde{j}) \in \{(i,i), (j,j) \} \\
        -1/4 \quad \text{ if } (\tilde{i}, \tilde{j}) \in \{(i,j), (j,i) \} \\
        0 \quad \text{ otherwise}. 
    \end{cases}
    \]
    Then, provided $X_{\mathbb{S}}$ is feasible, we get precisely that 
    \[
    R(\Sigma_{\mathbb{S}}) \geq -\frac{1}{d}\langle X_{\mathbb{S}}, \Sigma_{\mathbb{S}}\rangle_{\mathbb{S}} = \frac{1}{2} \max_{i > j} \mathop{\max_{k > h}}_{k,h \neq i,j} |\rho_{ij}^{(k)} - \rho_{ij}^{(h)}|.
    \]
    All is left to prove is that $X_{\mathbb{S}}$ is indeed feasible: $X_{\mathbb{S}} + X_{\mathbb{S}}^{(0)} = dY_{\mathbb{S}} \succeq_{\mathbb{S}} 0$, and $A^*X_{\mathbb{S}}$ is diagonal with trace zero, hence we can choose $Y = -A^*X_{\mathbb{S}}  \in \mathcal{Y}$ in the primal characterisation so that $A^*X_{\mathbb{S}} + Y = \boldsymbol{O} \succeq 0$. \\
\end{proof}

    Observe that the same is true in the case where we also have a complete case pattern, i.e. $\mathbb{S} = \{[d], S_{(-1)}, \cdots, S_{(-d)}\}$, with $d \geq 2$, meaning that using the same strategy we can control $R$ with the maximal inconsistency. Related to this, it would be interesting to know if there is a case in which the incompatibility value $\Theta$ controls $R(\Sigma_{\mathbb{S}})$ both from above and below, meaning that $\Theta$ fully characterises $R(\Sigma_{\mathbb{S}})$.  In this regard, we have the following:

\begin{eg}\label{ex:all_but_one_equality}
     Consider $\mathbb{S} = \{ S_{(-1)}, \cdots, S_{(-d)}\}$ and $\Sigma_{\mathbb{S}} = (I_{d-1}, \ldots, I_{d-1}, A)$, where \[A =
    \begin{pmatrix}
        1 & \epsilon_1/2 & 0 & 0 & \cdots & 0 \\
        \epsilon_1/2 & 1 & \epsilon_2/2 & 0 & \cdots & 0 \\
        \vdots & \ddots & \ddots & \ddots \\
        \\
        \vdots & & &\ddots & \ddots & \ddots \\
        \\
        0 & \cdots & 0 & 0 & \epsilon_{d-1}/2 & 1
    \end{pmatrix},
    \]
    with $\epsilon_i \in [-1, 1]$. Then, \[
     R(\Sigma_{\mathbb{S}}) = \Theta = \frac{1}{2}\max_{i \in [d-1]}|\epsilon_i|.\]
\end{eg}

\begin{proof}[Proof of Example \ref{ex:all_but_one_equality}]
    If we consider \[\Sigma =
    \begin{pmatrix}
        1-\Theta & \epsilon_1/4 & 0 & \cdots & 0 & 0 & 0 & 0\\
        \epsilon_1/4 & 1-\Theta & \epsilon_2/4 & \cdots & 0 & 0 & 0 & 0\\
        \ddots & \ddots & \ddots \\
        \\
        \\
        & & & \ddots & \ddots & \ddots \\
        0 & \cdots & 0 & \cdots & \epsilon_{d-2}/4 & 1-\Theta & \epsilon_{d-1}/4 &  0 \\
        0 & \cdots & 0 & 0 & \cdots & \epsilon_{d-1}/4 & 1-\Theta & 0 \\
        0 & \cdots & 0 & \cdots & 0 & 0 & 0 & 1-\Theta
    \end{pmatrix} \in \mathbb{R}^{d,d},
    \]
    if $\Sigma$ were feasible we would be able to conclude $R(\Sigma_{\mathbb{S}}) = \frac{1}{2}\max_{i \in [d]}|\epsilon_i|$ being \[
    \Theta \geq R(\Sigma_{\mathbb{S}}) \geq \Theta = \frac{1}{2}\max_{i \in [d-1]}|\epsilon_i|.
    \]
    All is left to prove is that $\Sigma$ is feasible. First, $\Sigma \succeq 0$ since it is diagonally dominant, being $1-\max_{i \in [d]}|\epsilon_i|/2 \in [1/2, 1]$ and $\epsilon_i/4 \in [-1/4, 1/4]$. Finally, a generic element in $\Sigma_{\mathbb{S}} - A\Sigma$ is given by \[
    \begin{pmatrix}
        \max_{i \in [d]}|\epsilon_i|/2 & \alpha_1/4 & 0 & 0 & \cdots & 0\\
        \alpha_1/4 & \max_{i \in [d]}|\epsilon_i|/2 & \alpha_2/4 & 0 & \cdots & 0\\
        & \ddots & \ddots & \ddots \\
        & & \ddots & \ddots & \ddots \\
        & & &\ddots & \ddots & \ddots \\
        \\
        0 & \cdots & 0 & 0 & \alpha_{d-1}/4 & \max_{i \in [d]}|\epsilon_i|/2
    \end{pmatrix} \in \mathbb{R}^{d-1,d-1},\] where $\alpha_i \in \{ \pm \epsilon_i, 0\}$. This is again diagonally dominant since $\max_{i \in [d]}|\epsilon_i|/2 \geq |\alpha_j|/4 + |\alpha_{j+1}|/4$ for all $j \in [d-1]$, by definition of the maximum.
\end{proof}

This example is particularly important since it clearly shows that, in this case, testing compatibility is at least as hard as testing consistency. Indeed, $\Theta$ is a pointwise measure of consistency, and equals $0$ if and only if $\Sigma_\mathbb{S}$ is consistent. Nonetheless, the equality $R(\Sigma_\mathbb{S}) = \Theta$ holds for a very specific subclass of $\Sigma_\mathbb{S}$, while, in general, there could be cases for which $R(\Sigma_\mathbb{S})>0$, while $\Theta = 0$.

\begin{eg}\label{Ex:schatten1}
    Consider $\mathbb{S}=\{[d-2]\cup \{d-1\}, [d-2]\cup \{d\}\}$. Call $S_1$ and $S_2$ the two patterns, respectively, and suppose we observe the collection of correlation matrices given by $\Sigma_\mathbb{S} = (\Sigma_{S_1}, \Sigma_{S_2})$. If we call \[\tilde{\Sigma} = (\Sigma_{S_2})_{|[d-2]} - (\Sigma_{S_1})_{|[d-2]},\] where $(\Sigma_{S_i})_{|[d-2]}$ is the restriction of $\Sigma_{S_i}$ on the set $[d-2]$, for $i \in \{1,2\}$, then  \[
    R(\Sigma_{\mathbb{S}}) \geq \frac{1}{2d} \|\tilde{\Sigma}\|_*,
    \]
    where $\|\cdot\|_*$ is the nuclear norm, also known as the Schatten-1 norm. 
\end{eg}

\begin{proof}[Proof of Example \ref{Ex:schatten1}]
    Define \[
    X_{\mathbb{S}} = \left( \begin{pmatrix}
        X & \boldsymbol{0}_{d-2} \\
        \boldsymbol{0}_{d-2}^T & 0
    \end{pmatrix}, \begin{pmatrix}
        -X & \boldsymbol{0}_{d-2} \\
        \boldsymbol{0}_{d-2}^T & 0
    \end{pmatrix} \right),
    \]
    where $X \in \mathbb{R}^{d-2,d-2}$ and $\|X\|_2 \leq 1/2$. Observe that this choice of $ X_{\mathbb{S}}$ is feasible since $A^*X_{\mathbb{S}} = \boldsymbol{O} \succeq 0$, and \[
    X_{\mathbb{S}} + X_{\mathbb{S}}^0 = \left( \begin{pmatrix}
        X + \frac{1}{2}I_{d-2} & \boldsymbol{0}_{d-2} \\
        \boldsymbol{0}_{d-2}^T & 1
    \end{pmatrix}, \begin{pmatrix}
        -X + \frac{1}{2}I_{d-2} & \boldsymbol{0}_{d-2} \\
        \boldsymbol{0}_{d-2}^T & 1
    \end{pmatrix} \right) \succeq_{\mathbb{S}} 0,
    \]
    since $\|X\|_2 \leq 1/2$. It follows that  \[
    R(\Sigma_{\mathbb{S}}) \geq \mathop{\sup_{X=X^T}}\limits_{\|X\|_2 \leq 1/2} -\frac{1}{d} \langle X_{\mathbb{S}}, \Sigma_{\mathbb{S}} \rangle_{\mathbb{S}} = \mathop{\sup_{X=X^T}}\limits_{\|X\|_2 \leq 1/2} \frac{1}{d} \langle X, \tilde{\Sigma} \rangle = \frac{1}{2d} \|\tilde{\Sigma}\|_*,
    \]
    where $\|\cdot\|_*$ is the nuclear norm and equality follows since the spectral norm and the nuclear norm are dual with respect to the Frobenius inner product.
\end{proof}

\section{Another test under trace normalisation}\label{sec:test_cov}
In the main body we were dealing with the incompatibility measure $R$, which acts on correlation matrices, normalised in a such a way that diagonal elements are all equal to one.  Nonetheless, other standardisations are possible, and these lead to different compatibility measures. In this section, we will define another measure of compatibility $\tilde{R}(\cdot)$, study its properties, and use it to define a testing procedure. Similarly to Table \ref{Table:Notation} in the main body, refer to Table \ref{tab:definitions_2} for all the new algebraic definitions needed in this section.
\begin{table}[!ht]
    \centering
    \begin{tabularx}{\textwidth}{|c|X|X|}
        \hline
        \textbf{Notation} & \textbf{Definition} & \textbf{Meaning}\\
        \hline
       $\bar{\mathrm{tr}} : \mathcal{M}_\mathbb{S} \rightarrow \mathbb{R}$ &  $\bar{\mathrm{tr}}(X_\mathbb{S}) = \sum_{j=1}^d |\mathbb{S}_j|^{-1} \sum_{S \in \mathbb{S}_j} (X_S)_{jj}$ & Generalisation of the trace such that, if $\Sigma_\mathbb{S}$ is compatible, then $\bar{\mathrm{tr}}(\Sigma_\mathbb{S})$ is equal to the trace of the underlying true covariance matrix \\
       $\tilde{\mathcal{P}}$ & $ \{ \Sigma \in \mathcal{P}^* : \operatorname{tr}(\Sigma) = d\}$ & Set of PSD matrices with fixed scale \\
        $\mathcal{\tilde{P}}_\mathbb{S}$ & $ \{ \Sigma_\mathbb{S} \in \mathcal{P}_\mathbb{S}^* : \bar{\mathrm{tr}}(\Sigma_\mathbb{S}) = d\}$ & collections of PSD matrices with scale fixed \\
	$\tilde{\mathcal{P}}_\mathbb{S}^0$ & $ \{ A \Sigma : \Sigma \in \tilde{\mathcal{P}} \}$ & Same as $\mathcal{P}_\mathbb{S}^0$, but $\Sigma$ has fixed scale \\
        \hline
    \end{tabularx}
    \caption{Table with all the definitions needed in Appendix \ref{sec:test_cov}.}
    \label{tab:definitions_2}
\end{table}
\\

The linear operator $\bar{\mathrm{tr}}$ satisfies the following:
\begin{prop}
\label{Prop:BasicLinearProperties2}
The following hold:
\begin{itemize}
	\item[(i)] If we define $X_\mathbb{S}^0 \in \mathcal{M}_\mathbb{S}$ by taking $X_S^0$ to be the diagonal matrix with $(X_S^0)_{jj} = 1/|\mathbb{S}_j|$, we have
	\[
		\bar{\mathrm{tr}}(X_\mathbb{S}) = \langle X_\mathbb{S}, X_\mathbb{S}^0 \rangle_\mathbb{S}
	\]
	for all $X_\mathbb{S} \in \mathcal{M}_\mathbb{S}$.
	\item[(ii)] Suppose that $X_\mathbb{S}$ is consistent, meaning that $(X_{S_1})_{jj'} = (X_{S_2})_{jj'}$ whenever $S_1,S_2 \in \mathbb{S}_{jj'}$, and write $X^\mathrm{partial}$ for the incomplete $d \times d$ matrix with $(X^\mathrm{partial})_{jj'} = (X_S)_{jj'}$ for any $S \in \mathbb{S}_{jj'}$. Then
	\[
		\bar{\mathrm{tr}}(X_\mathbb{S}) = \mathrm{tr}(X^\mathrm{partial}) \quad \text{and} \quad \langle X_\mathbb{S}, Y_\mathbb{S} \rangle_\mathbb{S} = \langle  X^\mathrm{partial}, A^*Y_\mathbb{S} \rangle
	\]
for any  $Y_\mathbb{S} \in \mathcal{M}_\mathbb{S}$.
\end{itemize}
\end{prop}

\begin{proof}[Proof of Proposition \ref{Prop:BasicLinearProperties2}]
    Now for any $X_\mathbb{S} \in \mathcal{M}_\mathbb{S}$ we see that
\[
	\langle X_\mathbb{S}, X_\mathbb{S}^0 \rangle_\mathbb{S} = \sum_{S \in \mathbb{S}} \sum_{j \in S} (X_S)_{jj} (X_S^0)_{jj} = \sum_{j=1}^d |\mathbb{S}_j|^{-1} \sum_{S \in \mathbb{S}} \mathbbm{1}_{\{j \in S\}} (X_S)_{jj} = \bar{\mathrm{tr}}(X_\mathbb{S}),
\]
proving property (i). The first part of (ii) can be seen immediately from the definition of $\bar{\mathrm{tr}}$. For the second part, write
\[
	\langle X_\mathbb{S}, Y_\mathbb{S} \rangle_\mathbb{S} = \sum_{S \in \mathbb{S}} \sum_{j,j' \in S} (X_S)_{jj'} (Y_S)_{jj'} = \sum_{j,j'=1}^d (X^\mathrm{partial})_{jj'} \sum_{S \in \mathbb{S}} \mathbbm{1}_{\{j,j' \in S\}} (Y_S)_{jj'} = \langle  X^\mathrm{partial}, A^*Y_\mathbb{S} \rangle.
\]
\end{proof}

Now, suppose that $\Sigma_\mathbb{S}$ is such that $\bar{\mathrm{tr}}(\Sigma_\mathbb{S})=d$, where $\bar{\mathrm{tr}}(X_\mathbb{S}) = \sum_{j=1}^d |\mathbb{S}_j|^{-1} \sum_{S \in \mathbb{S}_j} (X_S)_{jj}$, with $\mathbb{S}_j:=\{S \in \mathbb{S} : j \in S\}$, and define 
$$\tilde{R}(\Sigma_\mathbb{S}) := \sup\biggl\{ -\frac{1}{d} \langle X_\mathbb{S}, \Sigma_\mathbb{S} \rangle_\mathbb{S} : X_\mathbb{S} + X_\mathbb{S}^0 \succeq_\mathbb{S} 0, A^* X_\mathbb{S} \succeq 0 \biggr\}.$$
This new measure of incompatibility has the following dual representation:
\begin{prop}
\label{Prop:TraceDuality}
For $\Sigma_\mathbb{S} \in \tilde{\mathcal{P}}_\mathbb{S}$ we have
\[
	\tilde{R}(\Sigma_\mathbb{S}) = \inf\{ \epsilon \in [0,1] : \Sigma_\mathbb{S} \in (1-\epsilon) \tilde{\mathcal{P}}_\mathbb{S}^0 + \epsilon \tilde{\mathcal{P}}_\mathbb{S} \}.
\]
\end{prop}

\begin{proof}[Proof of Proposition \ref{Prop:TraceDuality}]
As in the proof of Proposition \ref{Prop:CorrDuality}, the strategy is to write this optimisation problem 
\begin{equation}\label{eq:primal_cov}
    \sup\biggl\{ -\frac{1}{d} \langle X_\mathbb{S}, \Sigma_\mathbb{S} \rangle_\mathbb{S} : X_\mathbb{S} + X_\mathbb{S}^0 \succeq_\mathbb{S} 0, A^* X_\mathbb{S} \succeq 0 \biggr\}
\end{equation}
in standard SDP form, prove that the dual problem is precisely
\begin{equation}\label{eq:dual_cov}
    1 - \frac{1}{d} \sup\{ \mathrm{tr}(\Sigma) : \Sigma \in \mathcal{P}^*, \Sigma_\mathbb{S} - A \Sigma \succeq_\mathbb{S} 0\},
\end{equation}
and then show that Slater's condition is satisfied for the primal problem (\ref{eq:primal_cov}). Calling $Y_\mathbb{S} = X_\mathbb{S} + X_\mathbb{S}^0$, we have that 
\begin{align*}
     \sup\biggl\{& -\frac{1}{d} \langle X_\mathbb{S}, \Sigma_\mathbb{S} \rangle_\mathbb{S} : X_\mathbb{S} + X_\mathbb{S}^0 \succeq_\mathbb{S} 0, A^* X_\mathbb{S} \succeq 0 \biggr\} \\
     &= \sup\biggl\{ -\frac{1}{d} \langle Y_\mathbb{S}, \Sigma_\mathbb{S} \rangle_\mathbb{S} + \frac{1}{d} \underbrace{\langle X_\mathbb{S}^0, \Sigma_\mathbb{S} \rangle_\mathbb{S}}_{ =\bar{\mathrm{tr}}(\Sigma_\mathbb{S}) = d} : Y_\mathbb{S} \succeq_\mathbb{S} 0, A^* Y_\mathbb{S} \succeq I_d \biggr\} \\
     &= 1 -\frac{1}{d}\inf \biggl\{\langle Y_\mathbb{S}, \Sigma_\mathbb{S} \rangle_\mathbb{S}: Y_\mathbb{S} \succeq_\mathbb{S} 0, A^* Y_\mathbb{S} - Z = I_d, \text{ for some } Z \succeq 0 \biggr\}.
\end{align*}
We write this optimisation problem in standard SDP form as follows: enumerate $\mathbb{S}$ as $\{ S_1, \ldots, S_m \}$, and define 
\[
   X:=  \begin{pmatrix} Y_{S_1} & \cdots & 0 & 0 \\ \vdots & \ddots & \vdots & \vdots \\ 0 & \cdots &  Y_{S_m} & 0 \\  0 & \cdots & 0 & Z \end{pmatrix},
\]
 so that $\langle Y_\mathbb{S}, \Sigma_\mathbb{S} \rangle_\mathbb{S} = \langle X, C \rangle$, where 
\[
   C:=  \begin{pmatrix} \Sigma_{S_1} & \cdots & 0 & 0 \\ \vdots & \ddots & \vdots & \vdots \\ 0 & \cdots &  \Sigma_{S_m} & 0 \\  0 & \cdots & 0 & 0 \end{pmatrix}.
\]
As for the constraints, they are equivalent to $X \succeq 0$ and $\langle X, A^{jj'} \rangle = \delta_{jj'}$, for $j,j' \in [d]$, with
\[
   A^{jj'}:=  \begin{pmatrix} E_{S_1,jj'} & \cdots & 0 & 0 \\ \vdots & \ddots & \vdots & \vdots \\ 0 & \cdots &  E_{S_m, jj'} & 0 \\  0 & \cdots & 0 & -E_{jj'} \end{pmatrix},
\]
where $E_{jj'} =(\boldsymbol{e}_j\boldsymbol{e}_{j'}^T +\boldsymbol{e}_{j'}\boldsymbol{e}_j^T)/2$ is the symmetric matrix of the same dimension as $Z$ with its only non-zero entries being in the $(j,j')$-th and $(j',j)$-th positions, and where $E_{S,jj'} =(\boldsymbol{e}_{S,j}\boldsymbol{e}_{S,j'}^T +\boldsymbol{e}_{S,j'}\boldsymbol{e}_{S,j}^T)/2$ is the symmetric matrix of the same dimension as $Y_S$ with its only non-zero entries being in the $(j,j')$-th and $(j',j)$-th positions of $Y_S$. Then, the standard dual problem is 
\begin{align*}
    \sup\biggl\{& \sum\limits_{j,j' \in [d]} \delta_{j,j'} Y_{j,j'} : C - \sum\limits_{j,j' \in [d]}  Y_{j,j'} A^{jj'} \succeq 0 \biggr\} \\
    &=  \sup\biggl\{\tr(Y) : \Sigma_{\mathbb{S}} - \frac{1}{2}A\left(Y + Y^T\right) \succeq_{\mathbb{S}} 0, \left(Y + Y^T\right) \succeq 0\biggr\} \\
    &=  \sup\biggl\{\tr(W) : \Sigma_{\mathbb{S}} - AW \succeq_{\mathbb{S}} 0, W \succeq 0\biggr\},
\end{align*}
where we made the substitution $W = \left(Y + Y^T\right)/2$ and used the fact that $\tr(W) = \tr(Y)/2 + \tr(Y^T)/2 = \tr(Y)$. This shows that (\ref{eq:dual_cov}) is the dual problem of (\ref{eq:primal_cov}). As in the proof of Proposition \ref{Prop:CorrDuality}, the result follows upon noticing that the primal problem (\ref{eq:primal_cov}) is strictly feasible, since
$Y_{\mathbb{S}} = X^0_{\mathbb{S}} \succ_{\mathbb{S}} 0$ is such that $A^*Y_{\mathbb{S}} = I_d \succeq_{\mathbb{S}} I_d$, which ensures that strong duality holds.
\end{proof}
As before, we can prove some properties for $\tilde{R}(\cdot)$
\begin{prop}\label{prop:properties2}
    The following hold:
    \begin{enumerate}
        \item[(i)] $\tilde{R}$ is convex.
        \item[(ii)] $\tilde{R}$ is continuos.
        \item[(iii)] If $\mathbb{S} \subseteq \mathbb{S}'$ and $\Sigma_{\mathbb{S}} \subseteq \Sigma_{\mathbb{S}'}$, then $\tilde{R}(\Sigma_{\mathbb{S}}) \leq d' \tilde{R}(\Sigma_{\mathbb{S}'})/d$, where $d' = \operatorname{card}(\cup_{S \in \mathbb{S'}}S)$ and $d = \operatorname{card}(\cup_{S \in \mathbb{S}}S)$.
    \end{enumerate}
\end{prop}

\begin{proof}[Proof of Proposition \ref{prop:properties2}]
     (i) and (ii) are essentially the same as in Proposition \ref{prop:properties}. To prove (iii), let $\tilde{X}_{\mathbb{S}}^{(1)}$ be a feasible point of $\{X_{\mathbb{S}} + X_{\mathbb{S}}^0 \succeq_{\mathbb{S}} 0 , A^*X_{\mathbb{S}} \succeq 0\}$, and define $\tilde{X}_{\mathbb{S}'}^{(2)}:= (\tilde{X}_{\mathbb{S}}^{(1)}, \boldsymbol{O}, \cdots, \boldsymbol{O})$, where we added a compatible zero matrix $\boldsymbol{O}$ for every element in $\mathbb{S}^C \cap \mathbb{S}'$. Then, $\tilde{X}_{\mathbb{S}'}^{(2)} + X_{\mathbb{S}'}^0 \succeq_{\mathbb{S}'} 0$ is equivalent to $\tilde{X}_{\mathbb{S}}^{(1)} + X_{\mathbb{S}}^0 \succeq_{\mathbb{S}} 0$ and $X_{\mathbb{S}' \setminus \mathbb{S}}^0 \succeq_{\mathbb{S}' \setminus \mathbb{S}} 0$, which are satisfied, while $A^*_{\mathbb{S}'} \tilde{X}_{\mathbb{S}'}^{(2)} = A^*_{\mathbb{S}} \tilde{X}_{\mathbb{S}}^{(1)} \succeq 0$, since $\tilde{X}_{\mathbb{S}}^{(1)}$ is feasible. Hence, $\tilde{X}_{\mathbb{S}'}^{(2)}$ is feasible for $\Sigma_{\mathbb{S}'}^{(2)}$, and the thesis follows from the fact that the normalising constant changes from $1/d$ to $1/d'$. Observe that the dual representation given by Proposition \ref{Prop:TraceDuality} allows proving the statement differently. Indeed, let $\Sigma \succeq 0 \in \mathbb{R}^{d',d'}$ be such that $\operatorname{tr}(\Sigma) = d'$ and \[
 \Sigma_{\mathbb{S}'}  = (1- \lambda')A_{\mathbb{S}'} \Sigma' + \lambda' \tilde{\Sigma}_{\mathbb{S}'},
 \]
  where $\tilde{\Sigma}_{\mathbb{S}'} \succeq_{\mathbb{S}'} 0$ and $\lambda' = R(\Sigma_{\mathbb{S}'})$. Then, since $\mathbb{S} \subseteq \mathbb{S}'$, we can automatically write also $\Sigma_{S}$ in this form as
  \[
  \Sigma_{\mathbb{S}}  = (1- \lambda')A_{\mathbb{S}} \Sigma + \lambda' \tilde{\Sigma}_{\mathbb{S}},
  \]
  where $\Sigma$ results from deleting all rows and columns of $\Sigma'$ associated to every element $i \notin \cup_{S \in \mathbb{S'}}S \setminus \cup_{S \in \mathbb{S}}S$, and is ensured to be non-negative definite by Cauchy interlacing theorem. Then, calling $\sigma_i^2$ the diagonal elements of $(1- \lambda')\Sigma'$,
  \begin{align*}
      \tilde{R}(\Sigma_{\mathbb{S}}) & \leq 1 - \frac{1}{d}\sum\limits_{i \in \cup_{S \in \mathbb{S}}S} \sigma_i^2 \leq 1 - \frac{1}{d}\left(\sum\limits_{i \in [d']} \sigma_i^2 - (d'-d)\right) \\
      & =  1 - \frac{1}{d}\left(d(1-\tilde{R}(\Sigma_{\mathbb{S}'})) - (d'-d)\right)  = \frac{d'}{d} \tilde{R}(\Sigma_{\mathbb{S}'}).
  \end{align*}
\end{proof}

We conclude the analysis of $\tilde{R}$ showing that it can be highly complex even for very simple settings.
\begin{figure}[ht]
    \centering
    \includegraphics[scale=0.3]{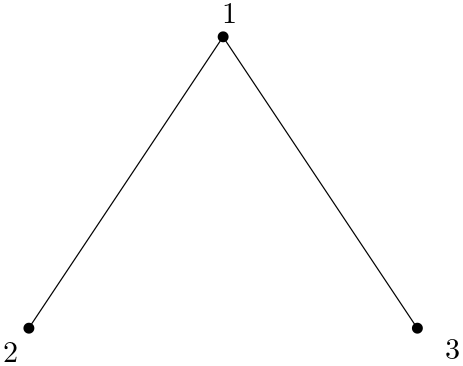}
    \caption{Graph associated to the pattern $\mathbb{S}=\{\{1,2\},\{1,3\}\}$.}
    \label{fig:ex123}
\end{figure}

\begin{eg}
\label{Ex:1213}
Consider $\mathbb{S}=\{\{1,2\},\{1,3\}\}$, which is associated to the graph in Figure \ref{fig:ex123}, and suppose without loss of generality that we observe
\[
    \Sigma_{\{1,2\}} = \begin{pmatrix} \sigma_1^2 & \rho_{12} \sigma_1 \sigma_2 \\ \rho_{12} \sigma_1 \sigma_2 & \sigma_2  \end{pmatrix}, \quad  \Sigma_{\{1,3\}} = \begin{pmatrix} \tilde{\sigma}_1^2 & \rho_{13} \tilde{\sigma}_1 \sigma_2 \\ \rho_{12} \tilde{\sigma}_1 \sigma_2 & \sigma_3^2  \end{pmatrix}
\]
with $\tilde{\sigma}_1^2 \geq \sigma_1^2$. Let $\theta,\phi \in [0,\pi/2]$ be such that $\cos \theta = \sigma_1/\tilde{\sigma_1}$ and $\cos \phi = |\rho_{13}|$. Then we have
\[
  R(\Sigma_\mathbb{S}) = \frac{1}{6}(\tilde{\sigma}_1^2 - \sigma_1^2) + \frac{1}{3} \sigma_3^2 \sin^2 \bigl( (\theta - \phi)_+ \bigr).
\]

\end{eg}
\begin{proof}[Proof of Example \ref{Ex:1213}]
    We prove this statement by giving an optimal choice of $X_\mathbb{S}$ for the primal problem and an optimal choice of $\Sigma$ for the dual problem. It turns out that the optimal $X_\mathbb{S}$ is of the form
\[
    X_\mathbb{S} = \left( \begin{pmatrix} \lambda & 0 \\ 0 & 0 \end{pmatrix}, uu^T - \begin{pmatrix} \lambda & 0 \\ 0 & 0 \end{pmatrix} \right)
\]
for $\lambda \in \mathbb{R}$ and $v \in \mathbb{R}^2$. Given $v \in \mathbb{R}^2$, we take $\lambda=1/2+3v_1^2/(1+3v_2^2)$ as this is the maximal value for which $X_\mathbb{S} + X_\mathbb{S}^0 \succeq_\mathbb{S} 0$. It is clear that $A^* X_\mathbb{S}  \succeq 0$, so this choice of $\lambda$ always leads to a feasible $X_\mathbb{S}$. When $\phi \geq \theta$ we will simply take $\lambda=1/2$ and $v=0$ to recover the same feasible solution as for $\bar{R}$ and the simple lower bound $R(\Sigma_\mathbb{S}) \geq (1/6)(\tilde{\sigma}_1^2 - \sigma_1^2)$. When $\phi=0$ (so that $|\rho_{13}|=1$) we take $v=\mu(\sigma_3/\tilde{\sigma}_1,-\operatorname{sgn}(\rho_{13}))$ with $\mu \rightarrow \infty$ to see that
\[
    R(\Sigma_\mathbb{S}) \geq\sup_{\mu \geq 0} \biggl( \frac{1}{6} + \mu^2 \frac{\sigma_3^2/\tilde{\sigma}_1^2}{1+3 \mu^2} \biggr) (\tilde{\sigma}_1^2 - \sigma_1^2) = (1/6)(\tilde{\sigma}_1^2 - \sigma_1^2) + (1/3) \sigma_3^2 (1-\sigma_1^2/\tilde{\sigma}_1^2),
\] 
which matches our claim. When $\theta > \phi>0$ we choose
\[
    v = \sqrt{\frac{\sin(\theta-\phi)}{\cos(\theta) \sin(\phi)}} \begin{pmatrix} (\sigma_3/\tilde{\sigma}_1)\cos(\theta-\phi) \\ -\operatorname{sgn}(\rho_{13})\cos(\theta) \end{pmatrix}.
\]
Using trigonometric identities, it can be seen that $\lambda = 1/2 + (\sigma_3^2/\tilde{\sigma}_1^2)\frac{\sin(\theta-\phi)\cos(\theta-\phi)}{\sin(\theta) \cos(\theta)}$ and
\begin{align*}
    &R(\Sigma_\mathbb{S}) \geq -(1/3) \langle X_\mathbb{S}, \Sigma_\mathbb{S} \rangle_\mathbb{S} \\
    & = (1/3) \bigl\{\lambda \tilde{\sigma}_1^2 \sin^2(\theta) - v_1^2 \tilde{\sigma}_1^2 - 2v_1v_2 \tilde{\sigma}_1 \sigma_3 \cos (\phi) \operatorname{sgn}(\rho_{13}) - v_2^2 \sigma_3^2 \bigr\} \\
    & = \frac{1}{6} \tilde{\sigma}_1^2 \sin^2(\theta) + \frac{1}{3} \sigma_3^2 \frac{\sin(\theta-\phi)}{\cos(\theta)\sin(\phi)} \bigl\{  \cos(\theta-\phi) \sin(\theta) \sin(\phi) - \cos^2(\theta-\phi) \\
    &\hspace{200pt} + 2 \cos(\theta-\phi) \cos(\theta)\cos(\phi) - \cos^2(\theta) \bigr\} \\
    & = \frac{1}{6} \tilde{\sigma}_1^2 \sin^2(\theta) + \frac{1}{3} \sigma_3^2 \sin^2(\theta-\phi).
\end{align*}
We have now provided the required lower bound in all cases, and turn to the upper bound through the dual problem. Start first with the case that $\phi \geq \theta$. Then $\tilde{\sigma}_1 |\rho_{13}| \leq \sigma_1$ so that
\[
    \Sigma = \begin{pmatrix} \sigma_1^2 &  \rho_{12} \sigma_1 \sigma_2 &\sigma_1 \sigma_3 \frac{\tilde{\sigma}_1\rho_{13}}{\sigma_1}   \\ \rho_{12} \sigma_1 \sigma_2 & \sigma_2^2 & \sigma_2 \sigma_3 \rho_{12} \frac{\tilde{\sigma}_1\rho_{13}}{\sigma_1} \\ \sigma_1 \sigma_3 \frac{\tilde{\sigma}_1\rho_{13}}{\sigma_1} & \sigma_2 \sigma_3 \rho_{12} \frac{\tilde{\sigma}_1\rho_{13}}{\sigma_1} & \sigma_3^2\end{pmatrix}
\]
is a valid covariance matrix. We have
\[
    \Sigma_\mathbb{S} -A\Sigma  = \left(\boldsymbol{O}_{2,2}, \begin{pmatrix} \tilde{\sigma}_1^2 - \sigma_1^2 & 0 \\ 0 & 0 \end{pmatrix} \right) \succeq_\mathbb{S} 0
\]
so $\Sigma$ is feasible. Thus, when $\phi \geq \theta$, we have
\[
    R(\Sigma_\mathbb{S}) \leq 1 - \frac{1}{3} \mathrm{tr}(\Sigma) = \frac{\tilde{\sigma}_1^2 + \sigma_1^2}{6} + \frac{\sigma_2^2 + \sigma_3^2}{3} - \frac{1}{3}(\sigma_1^2 + \sigma_2^2 + \sigma_3^2) = \frac{1}{6}(\tilde{\sigma}_1^2 - \sigma_1^2)
\]
as required. When $\phi<\theta$ we consider
\[
    \Sigma = \begin{pmatrix} \sigma_1^2 &  \rho_{12} \sigma_1 \sigma_2 &\sigma_1 \sigma_3 \cos(\theta-\phi) \operatorname{sgn}(\rho_{13})   \\ \rho_{12} \sigma_1 \sigma_2 & \sigma_2^2 & \sigma_2 \sigma_3 \rho_{12}  \cos(\theta-\phi) \operatorname{sgn}(\rho_{13}) \\ \sigma_1 \sigma_3 \cos(\theta-\phi) \operatorname{sgn}(\rho_{13}) & \sigma_2 \sigma_3 \rho_{12}  \cos(\theta-\phi) \operatorname{sgn}(\rho_{13}) & \sigma_3^2 \cos^2(\theta-\phi)\end{pmatrix},
\]
which is a covariance matrix so $\Sigma \succeq 0$. Clearly $(A\Sigma)_{\{1,2\}} = \Sigma_{\{1,2\}}$. It follows from trigonometric identities that
\begin{align*}
    \tilde{\sigma}_1 \sigma_3 \rho_{13} - \sigma_1 \sigma_3 \cos(\theta-\phi) \operatorname{sgn}(\rho_{13}) &= \tilde{\sigma}_1 \sigma_3\operatorname{sgn}(\rho_{13}) \{ \cos(\phi) - \cos(\theta) \cos(\theta-\phi) \} \\
    &= \tilde{\sigma}_1 \sigma_3\operatorname{sgn}(\rho_{13}) \sin(\theta) \sin(\theta - \phi)
\end{align*}
so that
\[
    \Sigma_{\{1,3\}}-(A \Sigma)_{\{1,3\}} = \begin{pmatrix} \tilde{\sigma}_1^2 \sin^2(\theta) & \tilde{\sigma}_1 \sigma_3\operatorname{sgn}(\rho_{13}) \sin(\theta) \sin(\theta - \phi) \\ \tilde{\sigma}_1 \sigma_3\operatorname{sgn}(\rho_{13}) \sin(\theta) \sin(\theta - \phi) & \sigma_3^2 \sin^2(\theta-\phi) \end{pmatrix},
\]
which is a covariance matrix so is positive semi-definite. Thus $\Sigma$ is feasible and when $\phi < \theta$ we have
\begin{align*}
    R(\Sigma_\mathbb{S}) &\leq 1 - \frac{1}{3} \mathrm{tr}(\Sigma) = \frac{\tilde{\sigma}_1^2 + \sigma_1^2}{6} + \frac{\sigma_2^2 + \sigma_3^2}{3} - \frac{1}{3}\{\sigma_1^2 + \sigma_2^2 + \sigma_3^2 \cos^2(\theta-\phi)\} \\
    &= \frac{1}{6}(\tilde{\sigma}_1^2 - \sigma_1^2) + \frac{1}{3} \sin^2(\theta-\phi),
\end{align*}
as required.
\end{proof}

Now, the goal of this subsection is to develop an analogous oracle test for the measure $\tilde{R}$, under the usual hypothesis of $\Sigma_{\mathbb{S}} \succeq_{\mathbb{S}} c I_{\mathbb{S}}$, with $ c > 0$. In this case, the maximum is attained in the set $$ \mathcal{H}_c := \{X_{\mathbb{S}} + X_{\mathbb{S}}^0 \succeq_{\mathbb{S}} 0 , A^*X_{\mathbb{S}} \succeq 0, \langle X_{\mathbb{S}} + X_{\mathbb{S}}^0, cI_{\mathbb{S}} \rangle_{\mathbb{S}} \leq d\},$$
hence the only difference with $\mathcal{F}_c$ is that $A^*X_{\mathbb{S}} + Y \succeq 0 \text{ for some } Y \in \mathcal{Y}$ is substituted by $A^*X_{\mathbb{S}} \succeq 0$. Hence, since in the previous subsection we discarded the condition $A^*X_{\mathbb{S}} + Y \succeq 0 \text{ for some } Y \in \mathcal{Y}$, if we now discard $A^*X_{\mathbb{S}} \succeq 0$, all the previous steps remain valid for controlling $\mathbb{P}_{H_0} \left\{\tilde{R}(\hat{\Sigma}_{\mathbb{S}}) \geq C_{\alpha}  \right\}$, so that we can again reduce this problem to bounding $\max_{S \in \mathbb{S}} \|\hat{\Sigma}_{S} - \Sigma_{S}\|_2$, with the only difference that now $\Sigma_{S}$ is related to the corresponding covariance matrix through a different normalisation. Repeating the same steps that lead to the proof of Theorem \ref{thm:oracle_test}, we can prove the following result, which gives the right separation to test compatibility based on $\tilde{R}$. 
\begin{prop}\label{Prop:oracle_test_cov}
    Suppose we observe $X_{S,1}, \ldots, X_{S, n_S} \overset{\text{i.i.d.}}{\sim} P_S, \forall S \in \mathbb{S}$ independently, where each $P_S$ is $\nu$-subgaussian with mean $\mu_S$ and $\nu \gg 1$, with the collection of population covariance matrices $\Sigma_{\mathbb{S}}$ satisfying $\operatorname{\bar{tr}}(\Sigma_\mathbb{S}) = d$, and $\Sigma_{\mathbb{S}} \succeq_\mathbb{S} cI_\mathbb{S}$, for a given $c >0$. Let $\hat{\Sigma}_{\mathbb{S}}$ be the collection of sample covariance matrix associated to each pattern $S \in \mathbb{S}$, $n_\mathbb{S}$ the collection of sample sizes, and suppose that also $\hat{\Sigma}_{\mathbb{S}}$ are normalised so that $\operatorname{\bar{tr}}(\hat{\Sigma}_\mathbb{S}) = d$.
Then, for all $\alpha \in (0,1)$, the test that rejects $H_0 : \tilde{R}(\Sigma_\mathbb{S}) = 0$ if and only if $\tilde{R}(\hat{\Sigma}_{\mathbb{S}}) \geq C_{\alpha}$ has Type I error bounded by $\alpha$, where 
    \[
C_{\alpha} = \frac{C_1 \nu^2}{c} \max_{S \in \mathbb{S}} \sqrt{\frac{|S| + \log(|\mathbb{S}|/\alpha)}{n_S}} \vee \frac{|S| + \log(|\mathbb{S}|/\alpha)}{n_S},
    \]
    and $C_1 > 0$ is a universal constant. Moreover, for $\beta \in (0,1)$, if $\tilde{R}(\Sigma_\mathbb{S}) > C_{\alpha} + C_{\beta}$, then $\mathbb{P}\{\tilde{R}(\hat{\Sigma}_{\mathbb{S}}) \leq C_{\alpha}\} \leq \beta$.
\end{prop}

The proof is essential analogous to the one of Theorem \ref{thm:oracle_test}, except for the fact that now we used directly Proposition \ref{prop:concentration_cov} in Appendix \ref{sec:technical_ineq} instead of Proposition \ref{prop:corr_concentration}. Also, observe that the separation rate in this case is analogous to the one we found in Theorem \ref{thm:oracle_test}, being of the order of 
\[
C_{\alpha} \lesssim  \max_{S \in \mathbb{S}} \sqrt{\frac{|S| + \log(|\mathbb{S}|/\alpha)}{n_S}},
\]
under $n_S \gtrsim |S|$ for all $|S| \in \mathbb{S}$, which is necessary to have a consistent test. As far as the drawbacks are concerned, notice that here we need to normalise the sample covariance matrix a priori, so that $\operatorname{\bar{tr}}(\hat{\Sigma}_{\mathbb{S}}) = d$, which is somehow annoying. What is even more disturbing is the hypothesis that the subgaussian proxy $\nu^2$ needs to be significantly bigger than one, due to the fact that for a $\nu$-subgaussian random variable $X$ we have $\operatorname{Var}[X] \leq \nu^2$. Hence, the hypothesis $\nu^2 \gg 1$ is necessary to have a little flexibility in the variances, while still satisfying $\operatorname{\bar{tr}}(\hat{\Sigma}_{\mathbb{S}}) = d$. There is no reason to assume that $\nu^2 \gg 1$, so that this is another point in favour of the incompatibility measure $R$. 

\section{Auxiliary results in SDP}\label{sec:SDP}
Semi-definite programs are linear optimisation problems over spectrahedra, i.e. sets of the form $$
S=\left\{\left(x_1, \ldots, x_m\right) \in \mathbb{R}^m: A_0+\sum_{i=1}^m A_i x_i \succeq 0\right\},
$$
for some given symmetric matrices $A_0, A_1, \ldots, A_m$. An SDP problem in standard primal form is written as
$$
\begin{aligned}
\begin{cases}
    \operatorname{minimize} \quad \langle C, X\rangle \\
\text {subject to } \quad X \succeq 0 \text{ and }\left\langle A_i, X\right\rangle=b_i, \quad i \in [m], \\
\end{cases}
\end{aligned}
$$
where $C, A_i$ are given symmetric matrices, and $b_i$ are given scalars. For every semi-definite program in primal form, there is another associated SDP, called the dual problem, that can be stated as
$$
\begin{aligned}
\begin{cases}
    \text { maximize } & b^T y \\
\text { subject to } & \sum_{i=1}^m A_i y_i \preceq C,
\end{cases}
\end{aligned}
$$
where $b=\left(b_1, \ldots, b_m\right)$, and $y=\left(y_1, \ldots, y_m\right)$ are the dual decision variables. As in linear programming, the so-called weak duality holds, meaning that if $X$ and $y$ are any two feasible solutions of the primal and dual problems respectively, we have
$$
\langle C, X\rangle-b^T y=\langle C, X\rangle-\sum_{i=1}^m y_i\left\langle A_i, X\right\rangle=\left\langle C-\sum_{i=1}^m A_i y_i, X\right\rangle \geq 0.
$$
Unfortunately, the equality is not always satisfied in general (see Example 2.14. in \cite{blekherman2012sdp}), but under some mild conditions, strong duality holds. One of such conditions is Slater's condition, where either the primal or the dual problem is required to be strictly feasible, meaning that there exists either $X \succ 0$ for the primal problem satisfying $\langle A_i, X \rangle=b_i$, for $ i \in [m]$, or $y$ for the dual satisfying $\sum_{i=1}^m A_i y_i \prec C$. If this is the case, it can be shown that strong duality holds (Theorem 2.15. in  \cite{blekherman2012sdp}, Theorem 3.1. in \cite{boyd_vandeberghe}). Furthermore, if the primal is strictly feasible, then the dual optimum is attained, and viceversa. In the proof, we show that it is possible to define $R$ as the optimal value of an SDP problem written in primal form, find its dual and show that Slater's condition is satisfied. This, apart from enabling us to prove Proposition \ref{Prop:CorrDuality}, ensures that $R$ can be computed explicitly using standard SDP libraries, which are available for almost all programming languages. As for the computational cost, for SDP problems in their general setting, without extra assumptions like strict complementarity, no polynomial-time algorithms are known, and there are examples of SDPs for which every solution needs exponential space \citep{porkolab}. Moreover, \cite{Ramana1997AnED} showed that SDP lies either in the intersection of NP and co-NP, or outside the union of NP and co-NP, and nothing better than this is known. Luckily, if Slater's condition is satisfied, like in our case, then the primal-dual interior point method has a computational complexity which is polynomial in the number of constraints and the dimension of the unknown square matrix (Section 6.4.1. of \cite{nesterov94}, Section 5.7. of \cite{boyd_vandeberghe}), which ensures that $R$ can be always computed efficiently without additional assumptions.

Finally, we recall Farkas' lemma for SDP problems, and its proof, following Lemma 6.3.2 in \cite{Lovasz2003}.
\begin{prop}[Farkas' lemma for Semi-definite Programming]\label{Prop:farkas_lemma}
Let $A_1, \ldots, A_n$ be symmetric $m \times m$ matrices. The system
\[
x_1 A_1+\ldots+x_n A_n \succ 0
\]
has no solution in $x_1, \ldots, x_n$ if and only if there exists a symmetric matrix $Y \neq 0$ such that
\[
\begin{aligned}
\begin{cases}
    \langle A_1, Y \rangle =0 \\
    \langle A_2, Y \rangle =0 \\
    \quad \vdots \\
    \langle A_n, Y \rangle =0 \\
    Y \succeq 0 . 
\end{cases}
\end{aligned}
\]
\end{prop}

\begin{proof}
The set $\mathcal{P}^*_m$ of $m \times m$ positive semi-definite matrices forms a closed convex cone. If
$$
x_1 A_1+\ldots+x_n A_n \succ 0
$$
has no solution, then the linear subspace $\mathcal{L}$ of matrices of the form $x_1 A_1+\ldots x_n A_n$ is disjoint from the interior of $\mathcal{P}^*_m$, which in turn implies that $\mathcal{L}$ is contained in a hyperplane that is disjoint from the interior of $\mathcal{P}^*_m$. This hyperplane can be described as $\{X \in \mathcal{P}^*_m: \langle Y, X \rangle =0\}$ for a certain symmetric $Y$, where we may assume that $\langle Y, X \rangle \geq 0$ for every $X \in\mathcal{P}^*_m$. Then, since a matrix $A$ is positive semi-definite if and only if $\langle A, B \rangle \geq 0$ for every positive semi-definite matrix $B$, we conclude that $Y \neq 0$, $Y \succeq 0$, and, since  $A_i$ belong to $\mathcal{L}$, that $\langle A_i, Y \rangle =0$. 
\end{proof}

\section{Technical inequalities}\label{sec:technical_ineq}
\begin{prop}[Tail bound for a sum of subgausssian RVs]\label{prop:hoeffding}
Suppose that the variables $X_i, i=1, \ldots, n$, are independent, and $X_i$ has mean $\mu_i$ and sub-Gaussian parameter $\sigma_i$. Then for all $t \geq 0$, we have
$$
\mathbb{P}\left\{\sum_{i=1}^n\left(X_i-\mu_i\right) \geq t\right\} \leq \exp \left\{-\frac{t^2}{2 \sum_{i=1}^n \sigma_i^2}\right\}.
$$
\end{prop}
\begin{proof}
    See Proposition 2.5 in \cite{wainwright2019high}.
\end{proof}

\begin{prop}[Euclidean norm of a subgaussian RV]\label{prop:norm_SG}
    Let $X \in \mathbb{R}^d$ be a subgaussian random vector with proxy $\sigma^2$. Then, for all $\delta \in (0,1)$, we have
    \[
    \mathbb{P}\left\{\| X \|_2 > 4\sigma \sqrt{d} + 2\sigma \sqrt{\log (1/\delta)}\right\} \leq \delta.
    \]
    Equivalently, for all $t > 0$ we have 
    \[
    \mathbb{P}\left\{\| X \|_2 > t \right\} \leq 5^d \exp \left\{-t^2/8\sigma^2 \right\}.
    \]
\end{prop}
\begin{proof}
    We break the proof up into two steps:  use a discretisation argument to reduce the problem to the task of computing the maximum of finitely many random variables, and then use standard concentration inequalities. Firstly, let $N_\epsilon$ be an $\epsilon$-net of the $d$-dimensional sphere $\mathbb{S}^{d-1}$. Then, 
    \[
    \| X \|_2 \leq \frac{1}{1-\epsilon} \max_{v \in N_\epsilon} v^T X.
    \]
    This follows from a discretization argument, similar to the one used in the proof of Proposition \ref{prop:concentration_cov}. Choosing $\epsilon = 1/2$ gives
    \[
    \mathbb{P}\left\{\| X \|_2 > t \right\} \leq |N_{1/2}| \exp \left\{-\frac{t^2}{8 \sigma^2} \right\} \leq 5^d \exp \left\{-\frac{t^2}{8 \sigma^2} \right\}.
    \]
    Inverting the bound yields the first claim.
\end{proof}

\begin{prop}[Tail bound for Binomial RVs]\label{prop:ineq_binomial}
    Let $X \sim \operatorname{Bin}(n, n^{-1})$. Then, for all $t \geq 1$, we have 
    \[
    \mathbb{P}\left\{X > t\right\} \leq \frac{e^{t-1}}{t^t}.
    \]
\end{prop}
\begin{proof}
    By the standard Chernoff argument, for all $\lambda > 0$ we have \begin{align*}
        \mathbb{P}\left\{X > t\right\} & \leq e^{-\lambda t} \mathbb{E}[e^{\lambda X}] = e^{-\lambda t} (1 - n^{-1} + n^{-1} e^\lambda)^n \\
        & \leq  e^{-\lambda t} e^{e^\lambda - 1} = e^{-\lambda t + e^\lambda - 1},
    \end{align*}
    where in the last inequality we used the fact that $1+x \leq e^x$. Choosing $\lambda = \log t$ concludes the proof.
\end{proof}

\begin{prop}[Tail bound for a sum of subexponential RVs]\label{prop:SE_tail_bound}
Consider an independent sequence $\left\{X_k\right\}_{k=1}^n$ of random variables, such that $X_k$ has mean $\mu_k$, and is sub-exponential with parameters $\left(v_k, \alpha_k\right)$. Then, $\sum_{k=1}^n\left(X_k-\mu_k\right)$ is sub-exponential with the parameters $\left(v_*, \alpha_*\right)$, where
$$
\alpha_*:=\max _{k=1, \ldots, n} \alpha_k \quad \text { and } \quad v_*:=\sqrt{\sum_{k=1}^n v_k^2},
$$
and
$$
\mathbb{P}\left\{ \left|\frac{1}{n} \sum_{i=1}^n\left(X_k-\mu_k\right) \right|\geq t\right\} \leq \begin{cases} 2e^{-\frac{n^2 t^2}{2 v_*^2}} & \text { for } 0 \leq t \leq \frac{v_*^2}{n \alpha_*} \\ 2e^{-\frac{n t}{2 \alpha_*}} & \text { for } t>\frac{v_*^2}{n \alpha_*}.\end{cases}
$$
\end{prop}
\begin{proof}
    See equation (2.18) in \cite{wainwright2019high}.
\end{proof}

\begin{prop}[The square of a subgaussian is subexponential]\label{prop:squareSG}
    If $X$ is $\sigma$-subgaussian, then $X^2$ is subexponential with parameters $(\nu, \alpha) = (4\sqrt{2}\sigma^2, 4\sigma^2)$.
\end{prop}

\begin{proof}
Using the definitions of the Orlicz norm $\|\cdot\|_{\psi_1}$ and $\|\cdot\|_{\psi_2}$ (see \cite{wainwright2019high, vershynin19hdp}), it is easy to prove that the product of two subgaussian RVs is subexponential (Lemma 2.7.7. in \cite{vershynin19hdp}), and that $X$ is subgaussian if and only if $X^2$ is subexponential (Lemma 2.7.6. in \citet{vershynin19hdp}). As for its subexponential parameters, assuming WLOG that $X$ has mean zero, we know that
$$
\mathbb{E}\left[e^{\lambda X}\right] \leq e^{\frac{1}{2} \lambda^2 \sigma^2}, \quad \text { for all } \lambda \in \mathbb{R}. 
$$
Our goal is to find a similar bound for the moment generating function of $X^2$, and, to this aim, we will make use of the fact that the moments of $X$ are bounded as follows
$$ \mathbb{E}\left[|X|^r\right] \leq r 2^{r / 2} \sigma^r \Gamma(r / 2), \quad \text { for all } r > 0,
$$ where $\Gamma(r)$ is the Gamma function.
Now, calling $\mu=\mathbb{E}[X^2]$, by power series expansion and since $\Gamma(r)=(r-1)!$ for an integer $r$, we have
$$
\begin{aligned}
\mathbb{E}\left[e^{\lambda\left(X^2-\mu\right)}\right] & =1+\lambda \mathbb{E}\left[X^2-\mu\right]+\sum_{r=2}^{\infty} \frac{\lambda^r \mathbb{E}\left[\left(X^2-\mu\right)^r\right]}{r !} \\
& \leq 1+\sum_{r=2}^{\infty} \frac{\lambda^r \mathbb{E}\left[|X|^{2 r}\right]}{r !} \leq 1+\sum_{r=2}^{\infty} \frac{\lambda^r 2 r 2^r \sigma^{2 r} \Gamma(r)}{r !} \\
& =1+\sum_{r=2}^{\infty} \lambda^r 2^{r+1} \sigma^{2 r} =1+\frac{8 \lambda^2 \sigma^4}{1-2 \lambda \sigma^2}.
\end{aligned}
$$
By making $|\lambda| \leq 1 /4 \sigma^2$, we have $1 /\left(1-2 \lambda \sigma^2\right) \leq 2$. Finally, since for every $\alpha \in \mathbb{R}$ it holds $1+\alpha \leq e^\alpha$, we have that the MGF of $X^2$ satisfies
$$
\mathbb{E}\left[e^{\lambda\left(X^2-\mathbb{E}\left[X^2\right]\right)}\right] \leq e^{16 \lambda^2 \sigma^4}, \quad \text{ for all } |\lambda| \leq 1 /\left(4 \sigma^2\right).
$$
Thus, we obtained a bound for the moment generating function of the subexponential variable $X^2$, that is similar to that of subgaussian variables but holds only for a small range of $\lambda$.
\end{proof}

\begin{prop}[Concentration inequality for Covariance Matrices]\label{prop:concentration_cov}
Let $X_1, \ldots, X_n$ be an i.i.d sequence of $\sigma$-subgaussian random vectors with zero mean and covariance matrix $\Omega$, and let $\hat{\Omega}_n:=\frac{1}{n} \sum_{i=1}^n X_i X_i^T$ be the sample covariance matrix. Then there exists a universal constant $C>0$ such that, for $\delta \in(0,1)$, with probability at least $1-\delta$
$$
\|\hat{\Omega}_n-\Omega\|_2 \leq C \sigma^2 \max \left\{\sqrt{\frac{d+\log (2 / \delta)}{n}}, \frac{d+\log (2 / \delta)}{n}\right\}.
$$
Equivalently, for all $t > 0$ we have
\[
\mathbb{P}\left\{\|\hat{\Omega}_n-\Omega\|_2 \geq t\right\} \leq 2 \cdot 9^d \exp \left\{-n \min \left\{\left(\frac{t}{16 \sigma^2}\right)^2, \frac{t}{16 \sigma^2}\right\}\right\}.
\]
\end{prop}

\begin{proof}
    We break the proof up into two steps:  use a discretisation argument to reduce the problem to the task of computing the maximum of finitely many random variables, and then use standard concentration inequalities. Firstly, let $A \in S^{d \times d}$ and let $N_\epsilon$ be an $\epsilon$-net of the $d$-dimensional sphere $\mathbb{S}^{d-1}$. Then
$$
\|A\|_2 \leq \frac{1}{1-2 \epsilon} \max _{y \in N_\epsilon}\left|y^T A y\right|. 
$$
Indeed, let $y \in N_\epsilon$ satisfy $\|x-y\| \leq \epsilon$. Then
$$
\begin{aligned}
\left|xAx - y^T A y\right| & =\left|x^T A(x-y)+y^T A(x-y)\right| \\
& \leq\left|x^T A(x-y)\right|+\left|y^T A(x-y)\right|
\end{aligned}
$$
Looking at $\left|x^T A(x-y)\right|$ we have
$$
\begin{aligned}
\left|x^T A(x-y)\right| & \leq\|A(x-y) \mid\|\|x\| \\
& \leq\|A\|_2 \underbrace{\|x-y\|}_{\leq \epsilon} \underbrace{\|x\|}_{=1} \\
& \leq\|A\|_2 \epsilon
\end{aligned}
$$
Applying the same argument to $\left|y^T A(x-y)\right|$ gives us $\left|x^{A} x-y^T A y\right| \leq 2 \epsilon\| A \|_2$. To complete the proof, we see that $\|A\|_2=\max _{x \in \mathbb{S}^{d-1}} x^T A x \leq 2 \epsilon\|A\|_2+\max _{y \in N_\epsilon} y^T A y$. Rearranging the equation gives $\|A\|_2 \leq$ $\frac{1}{1-2 \epsilon} \max _{y \in N_\epsilon} y^T A y$ as desired. Then, if we apply this result to $\hat{\Omega}_n-\Omega$ with $\epsilon=1 / 4$ we have
$$
\|\hat{\Omega}_n-\Omega\|_2 \leq 2 \max _{v \in N_{1 / 4}}\left|v^T\left(\hat{\Omega}_n-\Omega\right) v\right|
$$
Additionally, we know that $\operatorname{card}(N_{1 / 4}) \leq 9^d$ (see Lemma 5.7 and Example 5.8 in \cite{wainwright2019high}. From here, we can apply standard concentration tools to get
$$
\begin{aligned}
\mathbb{P}\left\{\|\hat{\Omega}_n-\Omega\|_2 \geq t\right\} & \leq \mathbb{P}\left(\max _{v \in N_{1 / 4}}\left|v^T\left(\hat{\Omega}_n-\Omega\right) v\right| \geq t / 2\right) \\
& \leq \operatorname{card}(N_{1 / 4}) \cdot \mathbb{P}\left(\left|v_i^T\left(\hat{\Omega}_n-\Omega\right) v_i\right| \geq t / 2\right),
\end{aligned}
$$
where $v_i$ is a unit vector on the $d$-dimensional sphere.
Now, $v_i^T\left(\hat{\Omega}_n-\Omega\right) v_i$ can be rewritten as 
$$
\begin{aligned}
v_i^T\left(\hat{\Omega}_n-\Omega\right) v_i & =\frac{1}{n} \sum_{j=1}^n\left(v_i^T X_j\right)^2-\mathbb{E}\left[\left(v_i^T X_j\right)^2\right] \\
& =\frac{1}{n} \sum_{j=1}^n Z_j-\mathbb{E}\left[Z_j\right],
\end{aligned}
$$
where the $Z_j-\mathbb{E}\left[Z_j\right]$ are independent subexponential of parameters $(\nu, \alpha) = \left(4\sqrt{2}\sigma^2, 4 \sigma^2\right)$, since $v_i^T X_j$ are $\sigma$-subgaussian by definition of subgaussian random vector. Applying the subexponential tail bound in Proposition \ref{prop:SE_tail_bound} gives us
$$
\mathbb{P}\left\{\left|v_i^T\left(\hat{\Omega}_n-\Omega\right) v_i\right| \geq t / 2\right\} \leq 2 \exp \left\{-n \min \left\{\left(\frac{t}{16 \sigma^2}\right)^2, \frac{t}{16 \sigma^2}\right\}\right\}.
$$
so that
$$
\mathbb{P}\left\{\|\hat{\Omega}_n-\Omega\|_2 \geq t\right\} \leq 2 \cdot 9^d \exp \left\{-n \min \left\{\left(\frac{t}{16 \sigma^2}\right)^2, \frac{t}{16 \sigma^2}\right\}\right\}.
$$
Inverting the bound gives the first result. For further reference, please refer to Chapter 3 in \cite{wainwright2019high}.
\end{proof}

    \begin{prop}[Theorem 3.1 in \cite{rudelson07rank1}]\label{prop:rudelson}
    Let $X$ be a random vector in $\mathbb{R}^d$, which is uniformly bounded almost everywhere: $\|X\|_2 \leq M$. Assume for normalisation that $\|\mathbb{E}X X^T\|_2 \leq 1$. Let $X_1 \ldots X_n$ be independent copies of $X$. Then, for every $t \in(0,1)$, there exists a universal constant $K > 0$ such that
$$
\mathbb{P}\left\{\left\|\frac{1}{n} \sum_{i=1}^n X_i X_i^T -\mathbb{E} X X^T \right\|_2>t\right\} \leq 2 \exp \left\{- \frac{K n t^2}{M^2 \log n}\right\}.
$$
\end{prop}

\end{appendices}
\end{document}